\documentclass[12pt]{article}

\newcommand{\argmin}{\mathop{\arg\min}}
\newcommand{\pprime}{{\prime\prime}}
\newcommand{\dd}{\text{d}}
\newcommand{\sgn}{\text{sgn}}
\newcommand{\supp}{\text{supp}}
\providecommand{\inner}[1]{\langle#1\rangle}
\newcommand{\vertiii}[1]{{\left\vert\kern-0.25ex\left\vert\kern-0.25ex\left\vert #1
    \right\vert\kern-0.25ex\right\vert\kern-0.25ex\right\vert}}

\newcommand{\Acal}{\mathcal{A}}

\newcommand{\Dcal}{\mathcal{D}}
\newcommand{\Ecal}{\mathcal{E}}
\newcommand{\Fcal}{\mathcal{F}}
\newcommand{\Gcal}{\mathcal{G}}
\newcommand{\Hcal}{\mathcal{H}}

\newcommand{\Lcal}{\mathcal{L}}

\newcommand{\Ncal}{\mathcal{N}}
\newcommand{\Ocal}{\mathcal{O}}

\newcommand{\Xcal}{\mathcal{X}}

\newcommand{\Cscr}{\mathscr{C}}

\newcommand{\Lscr}{\mathscr{L}}

\newcommand{\EE}{\mathbb{E}}

\newcommand{\GG}{\mathbb{G}}

\newcommand{\II}{\mathbb{I}}

\newcommand{\NN}{\mathbb{N}}

\newcommand{\PP}{\mathbb{P}}

\newcommand{\RR}{\mathbb{R}}

\newcommand{\ZZ}{\mathbb{Z}}

\newcommand{\Nfrak}{\mathfrak{N}}

\newcommand{\afrak}{\mathfrak{a}}
\newcommand{\bfrak}{\mathfrak{b}}

\newcommand{\sfrak}{\mathfrak{s}}

\newcommand{\wfrak}{\mathfrak{w}}

\newcommand{\xbf}{\mathbf{x}}
\newcommand{\ybf}{\mathbf{y}}
\newcommand{\zbf}{\mathbf{z}}

\usepackage{amsthm,amsmath,amssymb,multirow,mathrsfs,titling,enumitem,geometry}
\usepackage{graphicx}

\setlist{leftmargin=5mm}
\usepackage[hidelinks,hypertexnames=false,hyperfootnotes=false]{hyperref}

\newtheoremstyle{mystyle}{}{}{}{}{\bfseries}{.}{ }
{\thmname{#1}\thmnumber{ #2}\thmnote{ (#3)}}
\theoremstyle{mystyle}
\newtheorem{theorem}{Theorem}[section]
\newtheorem*{theorem*}{Theorem}
\newtheorem{definition}[theorem]{Definition}
\newtheorem{lemma}[theorem]{Lemma}
\newtheorem{corollary}[theorem]{Corollary}

\newtheorem{proposition}[theorem]{Proposition}
\newtheorem{assumption}[theorem]{Assumption}
\newtheorem{remark}[theorem]{Remark}

\usepackage[page,toc,titletoc,title]{appendix}

\makeatletter
\def\keywords{\xdef\@thefnmark{}\@footnotetext}
\makeatother

\usepackage[backend=biber, style=alphabetic, sorting=nty]{biblatex}
\usepackage{titletoc}

\newcommand{\worktitle}{New Empirical Process Tools and Their Applications to Robust Deep ReLU Networks and Phase Transitions for Nonparametric Regression}

\begin{document}

\title{\worktitle}
\author{\large Yizhe Ding, Runze Li and Lingzhou Xue}
\date{\large Department of Statistics, The Pennsylvania State University\\[0.4em]
\normalsize First version: November 2024;\quad This version: November 2026}

\keywords{\emph{MSC2020 Subject Classifications:} 60E15, 62G08, 62G35.}
\keywords{\emph{Keywords and phrases:} Heavy-tailed, Robustness, ReLU Neural Networks, non-Donsker, Huber Regression.}
\maketitle

\begin{abstract}
This paper introduces new empirical process tools for analyzing a broad class of statistical learning models under heavy-tailed noise and complex function classes. Our primary contribution is the derivation of two Dudley-type maximal inequalities for expected empirical processes that remove restrictive assumptions such as light tails and uniform boundedness of the function class. These inequalities enlarge the scope of empirical process theory available for statistical learning and nonparametric estimation. Exploiting the new bounds, we establish robustness guarantees for deep ReLU network estimators in Huber and quantile regression. In particular, we prove a unified non-asymptotic sub-Gaussian concentration bound that remains valid even under infinite-variance noise and provide a comprehensive analysis of non-asymptotic robustness for deep Huber estimators across all noise regimes. For deep quantile regression, we provide the first non-asymptotic sub-Gaussian bounds without requiring moment assumptions. As an additional application, our framework yields estimation error bounds for nonparametric least-squares estimators that simultaneously accommodate infinite-variance noise, non-Donsker function classes, and approximation error. Moreover, unlike prior approaches based on specialized multiplier processes, our framework extends to broader empirical risk minimization problems, including the nonparametric generalized linear models and the ``set-structured'' models.

\end{abstract}

\startcontents[main]
\printcontents[main]{}{1}{}

\section{Introduction}\label{sec: Introduction}

Over the past decade, deep neural networks (DNNs) have catalyzed a new era of artificial intelligence (AI), powering revolutionary advances in autonomous driving, natural sciences, large language models (LLMs), and image and video generation. Motivated by these empirical breakthroughs, the statistical community has developed theoretical frameworks and methodological extensions of DNNs across a broad spectrum of research problems, including nonparametric regression \cite{bauer2019deep, schmidt2020nonparametric, kohler2021rate}, survival analysis \cite{zhong2021deep, zhong2022deep}, causal inference \cite{farrell2021deep, chen2024causal}, factor and interaction models \cite{fan2024factor, bhattacharya2024deep}, repeated measurement models \cite{yan2025deep}, and robust regression \cite{shen2021robust, padilla2022quantile, shen2021deep, zhong2024neural, feng2024deep, fan2024noise,ding2026mitigating}. Despite these advances, the theoretical understanding of DNNs under adverse data distributions remains limited, particularly in the presence of heavy-tailed observations, where extreme values occur far more frequently than in light-tailed settings. Heavy-tailedness is a pervasive characteristic of modern high-dimensional statistics \cite{fan2021shrinkage}, manifesting itself in various domains such as finance \cite{cont2001empirical}, macroeconomics \cite{stock2002macroeconomic}, and medical imaging \cite{eklund2016cluster}. This phenomenon also presents fundamental challenges in the modern AI era. For example, long-tailed class frequencies in computer vision tasks can markedly degrade model performance \cite{cui2019class}.

From the perspective of classical estimation theory, a key insight for addressing heavy-tailed data lies in the robust loss design. While the least squares loss is sensitive to outliers, robust alternatives such as the Huber loss~\cite{huber1973robust} and the quantile loss~\cite{koenker2005quantile} effectively mitigate the impact of extreme observations. Motivated by this principle, several studies have examined DNNs trained with robust losses \cite{shen2021robust, padilla2022quantile, shen2021deep, zhong2024neural, feng2024deep}. However, their analyses establish only asymptotic consistency or convergence in expectation, and thus fall short of providing finite-sample guarantees that remain valid under heavy-tailed noise.

In contrast, the framework of Catoni (2012)  \cite{catoni2012challenging} offers a paradigm of nonasymptotic robustness, ensuring sub-Gaussian concentration for the mean estimator even when the data exhibit heavy tails. Building on this perspective, Fan et al. (2024)~\cite{fan2024noise} analyzed deep ReLU networks trained with the Huber loss and showed that the resulting deep Huber estimator attains such Catoni-type robustness: its estimation error satisfies a nonasymptotic sub-Gaussian bound under a suitably chosen Huber parameter.
While \cite{fan2024noise} established an important connection between deep learning and robust estimation, the analysis does not encompass several practically relevant regimes. First, when the Huber parameter $\tau$ is moderately large, the nonasymptotic error bounds in \cite{fan2024noise} degenerate from sub-Gaussian to polynomial deviation rates.
In practice, when $\tau$ is determined in a data-driven manner, it may fall within a regime where existing theory provides no uniform sub-Gaussian guarantees. Second, the analysis in~\cite{fan2024noise} assumes the noise possesses finite variance. In contrast, Sun et al. (2020)~\cite{sun2020adaptive} showed that adaptive Huber regression for linear models can retain uniform robustness without finite-variance assumptions, albeit with different convergence behavior. Whether analogous robustness properties extend to deep Huber estimators, and what rates may be attained in the infinite-variance setting, remains an open question.
These considerations motivate the following foundational problem in robust deep learning:
\begin{center}
    \it (Q1) Do deep ReLU network estimators uniformly gain robustness from employing robust losses, such as the Huber or quantile loss, in the presence of heavy-tailed noise?
\end{center}

A close examination of \cite{fan2024noise} indicates that its limitations originate from the empirical process tools employed in the theoretical analysis, which are not sufficiently refined to establish robustness across all regimes. The key argument used to bound the estimation error of the Huber estimator (their Lemma A.1) controls the rate through tail probabilities of an associated empirical process, thereby introducing unnecessary additional uncertainty. As a result, although their analysis was developed for Huber regression, it effectively treats the regime where the Huber parameter $\tau$ is moderately large as least squares estimation, a setting known to lack robustness. This observation explains why, for such values of $\tau$, the polynomial deviation bounds in \cite{fan2024noise} coincide with those obtained under the least squares loss \cite{han2019convergence, kuchibhotla2022least}. Moreover, the maximal inequality used in their Lemma A.8 to control the expected empirical process relies on $L^2$-integrability, precluding its applicability to settings with infinite-variance noise. While \cite{fan2024noise} provides a careful treatment of the approximation properties of ReLU networks, its empirical process arguments limit both the tightness and generality of the robustness guarantees and do not extend readily to other loss functions. These considerations underscore the need for more general and sharper empirical process techniques.

To address these limitations, it is natural to revisit the empirical process framework, a cornerstone of statistical learning theory \cite{koltchinskii2011oracle, boucheron2013concentration, vaart2023empirical, gine2021mathematical, geer2000empirical}. Despite its central role in analyzing the estimation error of empirical risk minimizers and M-estimators, existing tools remain insufficient for establishing uniform non-asymptotic robustness for deep architectures. Bridging this gap calls for new empirical process inequalities that capture the diverse loss functions, the structure of DNNs, and the heavy-tailed nature of modern data distributions.

While (Q1) focuses on the robustness of deep Huber and quantile estimators, a broader perspective reveals that the underlying challenges are not unique to neural networks. More generally, existing empirical process results remain inadequate for characterizing the convergence behavior of nonparametric least-squares estimators (NPLSEs) under heavy-tailed noise and non-Donsker function classes, and offer limited insight into nonparametric generalized linear models. For these canonical settings, the absence of a unified empirical process framework highlights the extent to which theory still trails the rapid methodological advances of modern statistics and machine learning. These difficulties arise from the intrinsic interplay between heavy-tailed noise and the complexity of the candidate function class.

In the nonparametric least squares regression under heavy-tailed noise, \cite{mendelson2016upper, mendelson2017local} analyzed function classes possessing special structural properties, thereby excluding many practically relevant examples such as those mentioned in \cite{han2017sharp}. Subsequent works \cite{han2019convergence, kuchibhotla2022least} relaxed some of these assumptions but continued to rely on the Donsker property, effectively limiting the scope to function classes that are sufficiently ``small'' to admit favorable empirical process behavior. Yet many classes of central importance in modern applications are far too ``large'' to be Donsker, leaving their statistical properties less explored. For instance, reproducing kernel Hilbert spaces (RKHSs) with certain kernels, used in reinforcement learning \cite{yang2020function}, are non-Donsker, and the class of multivariate convex functions, central to financial engineering \cite{ait2003nonparametric} and stochastic optimization \cite{lim2012consistency}, is also non-Donsker when the dimension $d \geq 5$ \cite{kur2024convex}.

Beyond the Donsker restriction, existing analyses of nonparametric regression often impose additional distributional assumptions, such as independence between covariates and noise or finite noise variance \cite{mendelson2017local, han2019convergence, kuchibhotla2022least}. They also largely neglect approximation error, leaving unresolved how statistical and approximation errors interact under general conditions. These limitations again trace back to the empirical process techniques underlying current theory. Classical tools typically require (i) light-tailed data and (ii) uniformly bounded or otherwise well-behaved function classes, which fail in many modern settings. To relax these conditions, \cite{han2019convergence, kuchibhotla2022least} developed the \textit{multiplier process} approaches tailored for least squares regression. Yet, the approach in \cite{han2019convergence} presumes independence between covariates and noise, while the maximal inequality in Proposition B.1 of \cite{kuchibhotla2022least} requires square integrability, thereby excluding infinite-variance regimes. Interestingly, our analyses also reveal a deeper link: the non-Donsker regime can effectively be viewed as a non–square-integrable regime (see Section~\ref{sec: Empirical Process with Expected Linfty covering entropy}), explaining why existing techniques struggle to handle the Donsker restriction and the finite-variance condition simultaneously.

In summary, existing techniques remain limited in both generality and applicability. They do not fully accommodate heavy-tailed noise, non-Donsker function classes, or approximation error, and they cannot be readily extended beyond least squares regression, even to generalized linear models. These gaps obscure a deeper structural understanding of how data irregularities and model complexity interact, particularly in shaping qualitative shifts in estimator performance across regimes. This motivates the following fundamental question:
\begin{center}
    \it (Q2) How do heavy-tailedness and function class complexity jointly determine the rates and phase transitions of nonparametric regression estimators?
\end{center}

\subsection{Our Main Contributions}\label{sec: main contributions}

From the motivating challenge of establishing robustness for deep Huber estimators across all regimes to the broader task of characterizing phase transitions in nonparametric regression, these gaps point to the need for a more general empirical process framework. Our work develops a new suite of empirical process tools designed to overcome these constraints. The proposed framework extends empirical process theory to accommodate heavy-tailed noise, non-Donsker function classes, and general loss structures, thereby addressing the open questions raised above and laying a foundation for future advances in robust and nonparametric learning. Our main contributions are as follows:

\begin{enumerate}
    \item \textbf{New Empirical Process Tools.}
    We develop new empirical process techniques that remove several restrictive assumptions of classical theory, thereby enabling analysis under broader conditions. Theorem~\ref{theorem: convergence of EP with L^1 integrable functions} establishes a Dudley-type maximal inequality that bounds the expected empirical process by the expected $L^{1+\kappa}(\PP_n)$ covering entropy, which provides a foundation for deriving convergence rates of various estimators. Theorem~\ref{theorem: convergence of EP with L^infty integrable functions} strengthens this framework by incorporating both the $L^{1+\kappa}(P)$ and weighted uniform covering entropy terms, yielding a refined inequality that exploits additional structural properties of complex function classes. These results accommodate non-Donsker and non–square-integrable function classes, heavy-tailed distributions, and general loss functions. The proposed tools substantially broaden the applicability of empirical process theory to problems involving heavy-tailed data, complex function classes, and diverse loss functions. \end{enumerate}

Building on these theoretical developments, we systematically address the two motivating questions and extend the analysis to a broad class of nonparametric regression models.

\begin{enumerate}
    \addtocounter{enumi}{1}
    \item \textbf{Robust Estimation with Deep ReLU Networks.}
    We first apply the proposed empirical process framework (Theorem~\ref{theorem: convergence of EP with L^1 integrable functions}) to Huber regression with deep ReLU networks, addressing the open question (Q1). Theorem~\ref{thm: convergence rate of Huber regression} establishes a unified sub-Gaussian concentration bound that holds uniformly across all regimes, including infinite-variance noise and moderately large $\tau$. This result advances the analysis of \cite{fan2024noise} by demonstrating that deep Huber estimators retain Catoni-type robustness uniformly, without requiring an oracle choice of $\tau$. It further reveals that, in the infinite-variance noise setting, convergence behavior differs and favors a more conservative choice of $\tau$. Table~\ref{table: assumptions in the literature of Huber regression} summarizes the assumptions and robustness guarantees in the existing literature on Huber regression.

\begin{table}[ht]
    \centering
    \caption{Comparison of assumptions and conditions on Huber parameter $\tau$ for establishing robustness guarantees of Huber estimators in existing works and in our result (Theorem~\ref{thm: convergence rate of Huber regression}).}
    \resizebox{\linewidth}{!}{\begin{tabular}{c|c|c|c}
    \hline
     & Noise Assumption & Function Class  & Robustness Guarantee \\ \hline
     Sun et al. (2020) \cite{sun2020adaptive} & $\EE[|\xi_i|^m]<\infty$ for $m>1$ & Linear models & For all large $\tau$ \\
     Fan et al. (2024) \cite{fan2024noise} & $\|\EE[|\xi_i|^m|X_i]\|_{L^\infty}<\infty$ for $m\geq 2$ & ReLU networks & Only for $\tau\ll \tilde{n}^{1/m}$ \\
     Theorem \ref{thm: convergence rate of Huber regression} (Ours) & $\|\EE[|\xi_i|^m|X_i]\|_{L^\infty}<\infty$ for $m\geq 1$ & ReLU networks & For all large $\tau$ \\ \hline
    \end{tabular}}
    \label{table: assumptions in the literature of Huber regression}
\end{table}

    Since the proposed maximal inequalities are not specific to the Huber loss, Theorem~\ref{thm: quantile regression convergence rate} provides the first non-asymptotic sub-Gaussian concentration bound for deep quantile regression, without any moment assumptions on the response. These results show the intrinsic robustness of deep Huber and quantile regression estimators and the generality of the proposed empirical process tools for analyzing models trained with nonsmooth losses.

\begin{figure}[htbp]

  \begin{minipage}{0.45\textwidth}
    \centering
    \includegraphics[width=\textwidth]{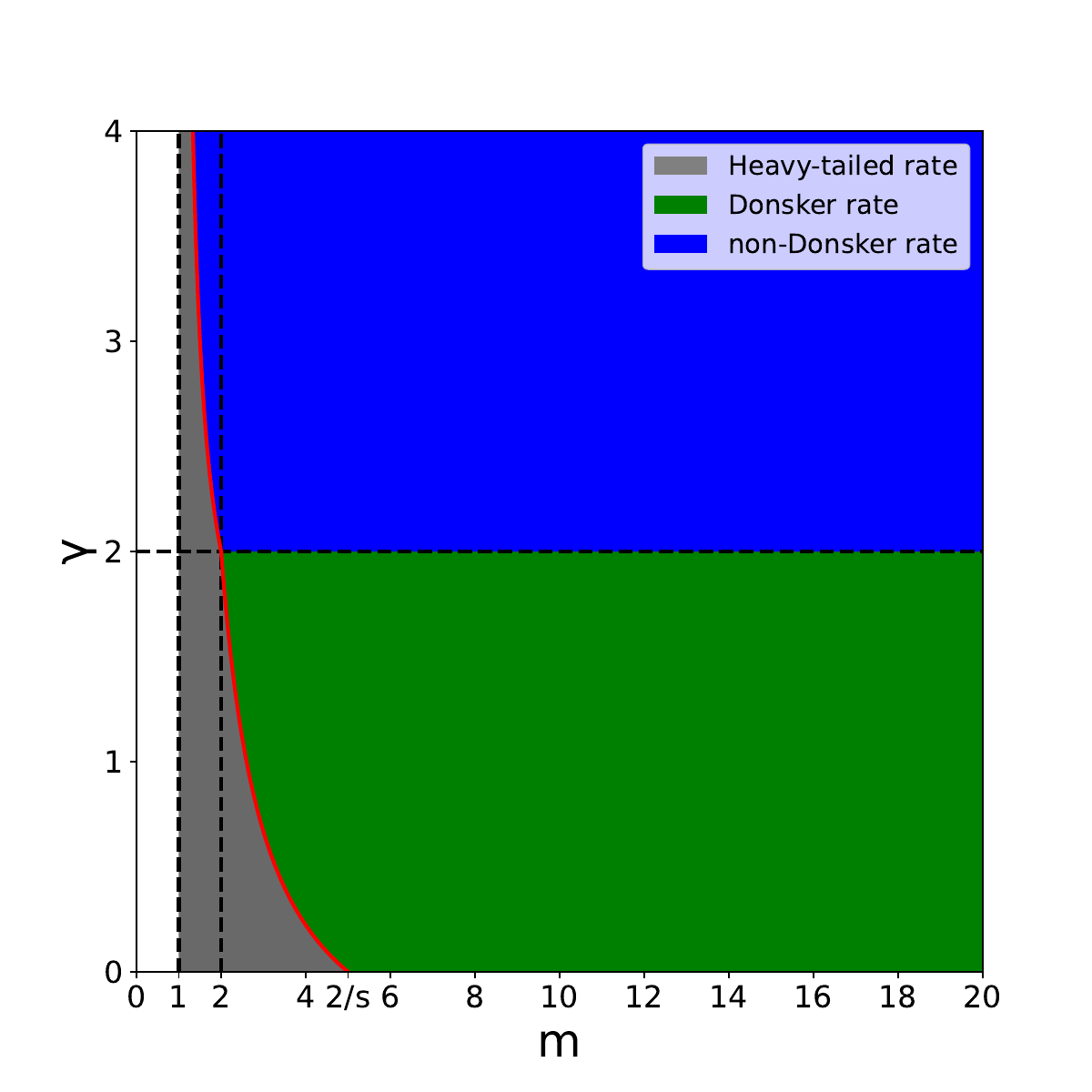}
    (a)
  \end{minipage}

  \begin{minipage}{0.45\textwidth}
    \centering
    \includegraphics[width=\textwidth]{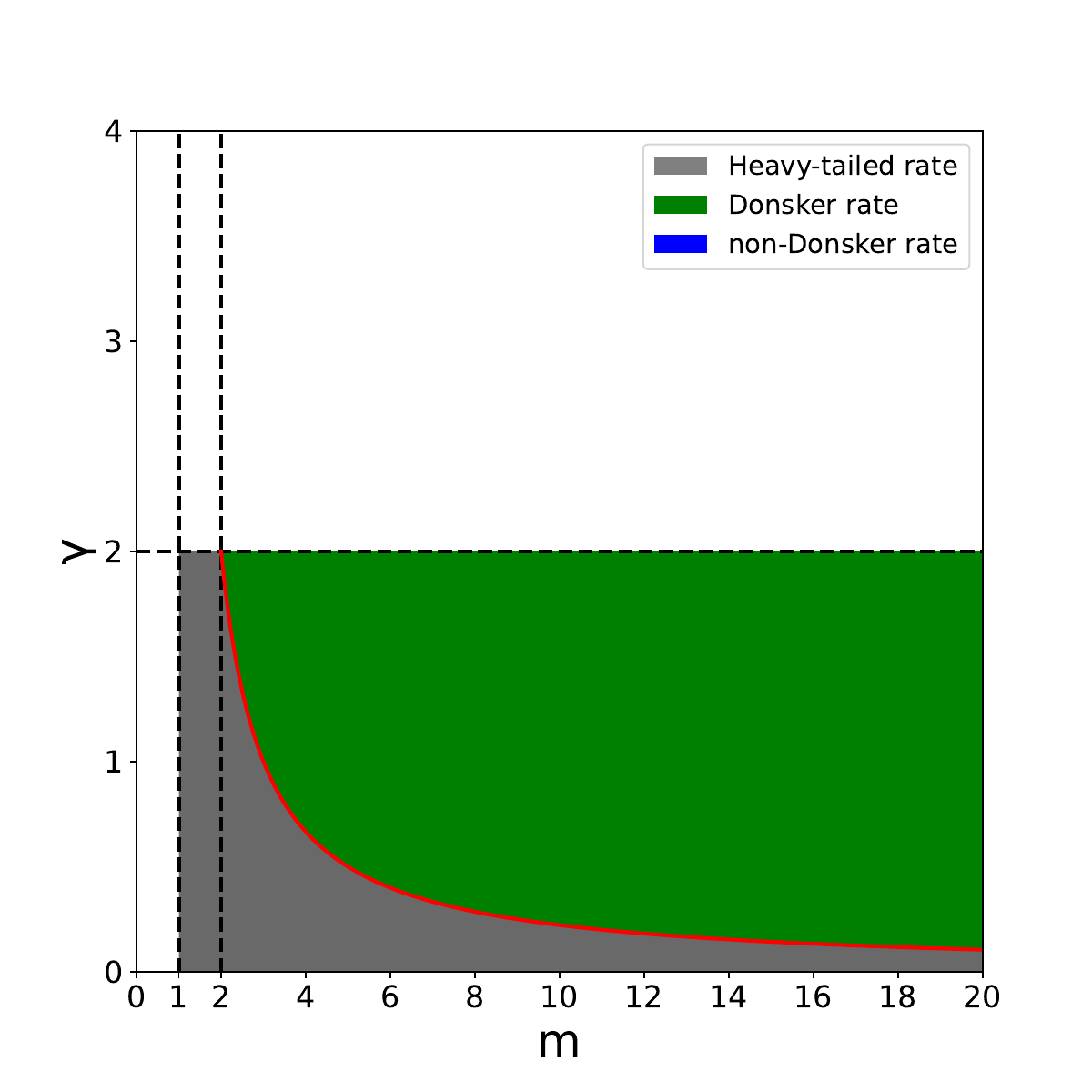}
    (b)
  \end{minipage}
  \hspace{0.05\textwidth}
  \begin{minipage}{0.45\textwidth}
    \centering
    \includegraphics[width=\textwidth]{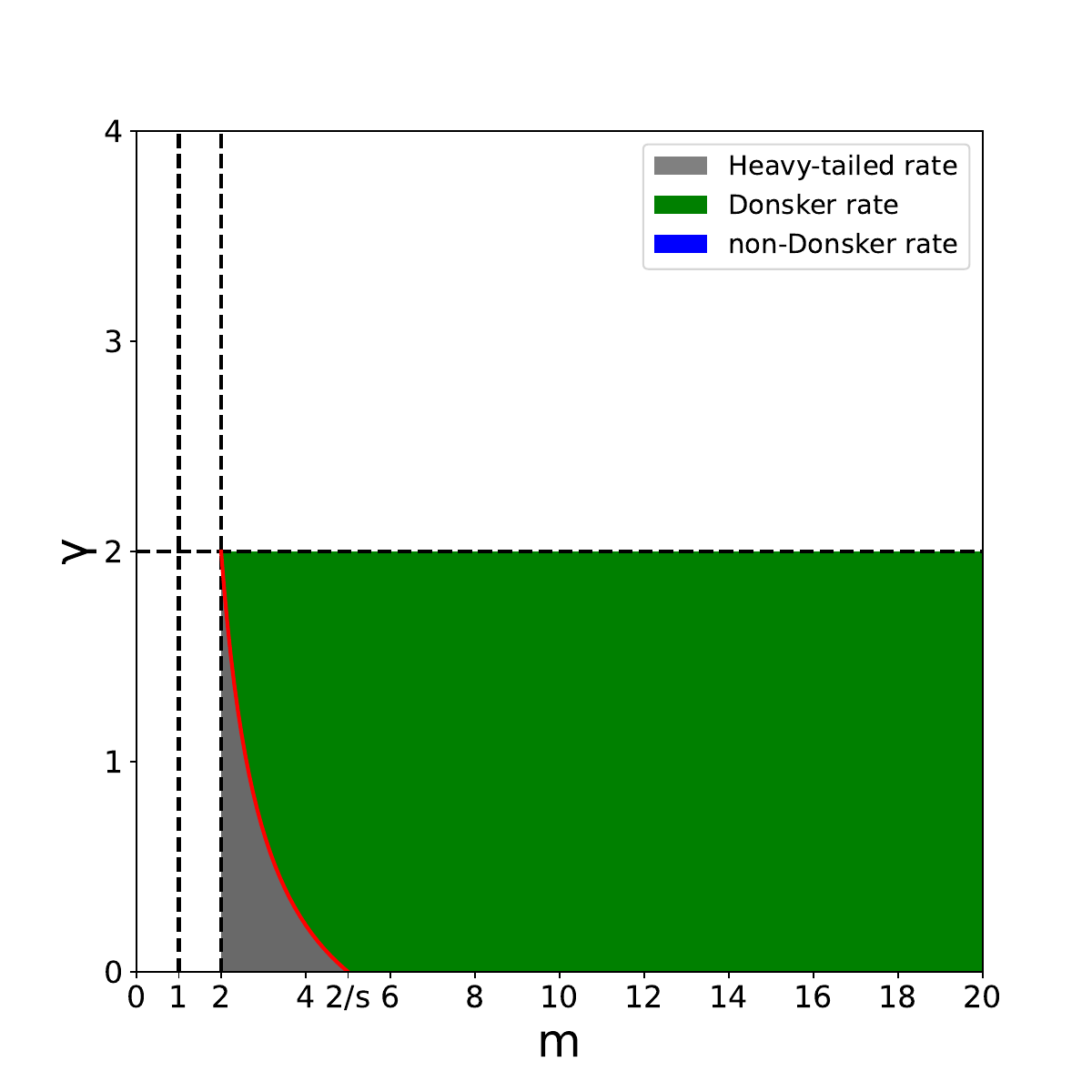}
    (c)
  \end{minipage}
\caption{
Phase transitions of the statistical error for NPLSEs: (a) depicts results in Theorems~\ref{theorem: least squares estimation convergence rate with Linfty covering entropy} and \ref{theorem: general ERM convergence rate with Linfty covering entropy}, (b) depicts results in \cite{han2019convergence}, and (c) depicts results in Theorem~4.1 of \cite{kuchibhotla2022least}.
The horizontal axis $m$ denotes the moment order of the noise distribution, and the vertical axis $\gamma$ measures function class complexity.
The red solid curves indicate the phase transition boundaries. Three regimes emerge:
\textbf{(1) Green}: Donsker rate $\tilde n^{-1/(2+\gamma)}$,
\textbf{(2) Blue}: non‐Donsker rate $\tilde n^{-1/(2\gamma)}$,
\textbf{(3) Gray}: heavy‐tailed rate $\tilde{n}^{-(\frac{2 - s}{1 - 1/m}+s\gamma)^{-1}}$.
}
\label{fig: phase transition of least squares regression}
\end{figure}

   \item \textbf{Phase Transitions in Heavy-Tailed Nonparametric Regression.}
   To address the fundamental question (Q2), we first derive finite-sample guarantees for nonparametric least squares estimators (NPLSEs) under mild conditions on heavy-tailed noise and general function class complexity. Building on Theorem~\ref{theorem: convergence of EP with L^infty integrable functions}, Theorem~\ref{theorem: least squares estimation convergence rate with Linfty covering entropy} recovers and strengthens the results of \cite{han2019convergence, kuchibhotla2022least}, while extending the analysis to infinite-variance and non-Donsker regimes and explicitly incorporating approximation error. As illustrated in Figure~\ref{fig: phase transition of least squares regression}, Theorem~\ref{theorem: least squares estimation convergence rate with Linfty covering entropy} reveals a substantially richer phase-transition landscape for NPLSEs, characterizing how tail behavior, function class complexity, and smoothness jointly govern statistical performance, well beyond the scope of previous analyses.

   To further address (Q2), in contrast to prior results relying on multiplier processes specialized to least squares regression, the proposed empirical process framework in Theorem~\ref{theorem: convergence of EP with L^infty integrable functions} applies more broadly. The same phase-transition phenomena and convergence behaviors extend naturally to a wide class of empirical risk minimizers, including the nonparametric generalized linear models (see Theorem~\ref{theorem: general ERM convergence rate with Linfty covering entropy}). Using the proposed empirical process framework, we also derive a convergence rate for the NPLSE for unbounded ``set-structured'' function classes \cite{han2021set} (see Theorem~\ref{theorem: indicator regression convergence rate with L2 covering entropy}).
\end{enumerate}

\subsection{Organization}

This paper is organized as follows. In Section~\ref{sec: Background and Overview of Our Main Results}, we review the background and highlight the intrinsic connection between empirical process theory and our motivating questions, followed by an overview of the main results.
In Section~\ref{sec: The New Empirical Processes}, we develop new empirical process tools, which serve as the foundation of theories established in Sections~\ref{sec: Robustness of Deep Huber Regression} and \ref{sec: Heavy-tailed Nonparametric Least Squares Regression}.
In Section~\ref{sec: Robustness of Deep Huber Regression}, we establish unified sub-Gaussian concentration bounds for ReLU network estimators under Huber and quantile loss.
In Section~\ref{sec: Heavy-tailed Nonparametric Least Squares Regression}, we prove the new convergence guarantees, together with extensions to nonparametric generalized linear models.  Section \ref{conclusion} includes a few concluding remarks.
All technical proofs are given in the supplement.

\subsection{Notations and Terminologies}

We write $a \lesssim b$ if there exists a constant $C>0$, independent of $a$ and $b$, such that $a \leq C b$. Similarly, $a \asymp b$ means that both $a \lesssim b$ and $b \lesssim a$ hold. To suppress logarithmic factors, we use the notations $\lesssim_{\log}$, $\gtrsim_{\log}$, and $\asymp_{\log}$. We denote $a \lor b$ and $a \land b$ as the maximum and minimum of $a$ and $b$, respectively.

For $n \in \NN$, let $[n]=\{1,2,\dots,n\}$. For $x \in \RR$, define $\log_+(x)=1 \lor \log(x)$ and ceiling function $\lceil x \rceil = \min\{n \in \NN: n \geq x\}$. For $x \in \RR^d$, we denote its $\ell_2$-norm by $\|x\|_2$, its augmented Euclidean norm by $\inner{x} = 1 + \|x\|_2$, and $\|x\|_0$ as the number of its nonzero elements.
\textbf{Covering number:} For a metric space $(M,d)$, a subset $K\subset M$ and a positive $h$, a subset $C\subseteq K$ is called an $h$-covering of $K$ with metric $d$, if $K\subseteq\bigcup_{x\in C}\{y\in M:d(x,y)\leq h\}$. The covering number of $K$ with $h>0$ is defined as the smallest cardinality of $h$-covering of $M$, denoted as $\Ncal(h,K,d)$. Its logarithm, $h\mapsto \log\Ncal(h,K,d)$, is called the covering entropy of $K$. An integral involving a covering entropy is called a covering integral.

\textbf{Sub-Weibull:} A random vector $X \in \RR^d$ is called sub-Weibull for some $K, \theta > 0$, if $\PP(\|X\|_2 \geq x) \leq 2\exp(-(x/K)^\theta)$ for all $x\geq 0$.

\section{Overview of Main Results}\label{sec: Background and Overview of Our Main Results}

This section outlines our main theoretical contributions.
Section~\ref{sec: Literature Review of Empirical Process Theory} reviews the foundational role of empirical processes in statistical learning theory, highlights the restrictive assumptions that constrain their applicability, and summarizes our contributions.
Section~\ref{sec: Deep Huber Estimator} introduces the modern robustness framework based on non-asymptotic concentration bounds \cite{catoni2012challenging}, and summarizes our contributions.
Section~\ref{sec: Literature Review of Least Squares Regression} reviews recent developments in nonparametric least squares regression, emphasizing their key limitations, followed by a summary of our results.

\subsection{New Empirical Process Tools}\label{sec: Literature Review of Empirical Process Theory}

We begin with the general framework of statistical learning through the lens of empirical process theory. Let $\{\Dcal_i\}_{i=1}^n$ be i.i.d.\ observations from an unknown distribution $P$. The goal is to estimate a parameter $\theta_0$, possibly infinite-dimensional, e.g., a function, defined as the minimizer of the population risk:
\[
    \theta_0=\argmin_{\theta\in\Theta_0}\EE_P[\ell(\theta;\Dcal)] =: P\ell(\theta),
\]
for some loss $\ell$. Since $P$ is unknown, we consider the empirical risk minimizer (ERM):
\[
    \hat\theta_n\in\argmin_{\Theta_n}\frac{1}{n}\sum_{i=1}^n\ell(\theta;\Dcal_i) =: \PP_n\ell(\theta),
\]
where $\Theta_n\subseteq\Theta_0$ is a sieve parameter space that may not contain $\theta_0$, and $\PP_n=\frac{1}{n}\sum_{i=1}^n \delta_{\Dcal_i}$ is the empirical measure. The central question is to characterize the convergence rate of $d(\hat\theta_n,\theta_0)$ for some metric $d(\cdot,\cdot)$.

A standard assumption is quadratic curvature of the population loss at $\theta_0$:
\[
    d(\theta, \theta_0)^2 \leq P\ell(\theta) - P\ell(\theta_0), \quad \forall\ \theta \in \Theta_0.
\]
Combined with the fact that $\hat\theta_n$ minimizes the empirical risk, this yields a pessimistic upper bound: when $\theta_0 \in \Theta_n$,
\begin{equation} \label{eq: connection between convergence rates of ERM and EP, sec: overview}
    d(\hat\theta_n, \theta_0)^2
    \leq P(\ell(\hat\theta_n) - \ell(\theta_0)) + \PP_n(\ell(\theta_0) - \ell(\hat\theta_n))
    \leq \sup_{\theta \in \Theta_n} \Big|(\PP_n - P)(\ell(\theta) - \ell(\theta_0)) \Big|.
\end{equation}

Let \( \Lscr = \{\ell(\theta) - \ell(\theta_0): \theta \in \Theta_n\} \) denote the associated loss class. Then the convergence rate of $\hat\theta_n$ is governed by the empirical process
\[
    \sup_{l \in \Lscr} |(\PP_n - P)l| =: \|\PP_n - P\|_{\Lscr}.
\]
This establishes the fundamental link between the convergence of ERMs and empirical process bounds, typically derived via \emph{maximal inequalities}. In practice, not all $\theta \in \Theta_n$ can be the ERM with high probability, so one often restricts the supremum to a local neighborhood of $\theta_0$, a technique known as \emph{localization}, which is central to deriving sharp convergence rates for M-estimators \cite{vaart2023empirical, van2002m, van2017concentration}. For the remainder of this work, we generalize the notation and analyze the empirical process $\|\PP_n - P\|_\Fcal$ for an arbitrary function class $\Fcal$, emphasizing that our results extend beyond the ERM framework.

Building on Dudley’s seminal work on suprema of Gaussian processes \cite{dudley1967sizes}, maximal inequalities for empirical processes, via covering and bracketing entropy, have been developed extensively; see \cite{talagrand1994sharper, mendelson2002improving, gine2006concentration} and the overviews in \cite{gine2021mathematical, vaart2023empirical}. Yet these classical results typically require uniform boundedness of $\Fcal$, limiting applicability in modern settings. Subsequent relaxations assume either an $L^2(P)$-integrable envelope or the existence of (weighted) $L^\infty$ covering/bracketing entropy \cite{van2011local, dirksen2015tail, nickl2022polynomial}, but these regularity conditions are still too restrictive for theoretical analysis of modern statistical learning procedures.

In this work, we develop two maximal inequalities that dispense with these restrictive conditions, thereby broadening the scope of empirical process theory in analyzing modern statistical learning procedures. As an illustration, we first display the following informal version of Theorem~\ref{theorem: convergence of EP with L^1 integrable functions} and focus on the square-integrable setting for simplicity.
\begin{theorem}[Informal version of Theorem~\ref{theorem: convergence of EP with L^1 integrable functions}]\label{thm: Informal version of theorem: convergence of EP with L^1 integrable functions}
    Let $F\geq 0$ be the envelope function of $\Fcal$, and assume $\|f\|_{L^{2}(P)}\leq\sigma$ for all $f\in\Fcal$. Then,
    \begin{equation}
        \begin{aligned}\label{eq: Informal version of theorem: convergence of EP with L^1 integrable functions}
            \EE^\ast\|\PP_n-P\|_\Fcal
            \lesssim& \inf_{\epsilon\in (0, \sigma)}\Bigl[\epsilon + \frac{1}{\sqrt{n}}\int_{\epsilon}^{\sigma} \sqrt{\EE[\log (2\Ncal(x,\Fcal,L^{2}(\PP_n)))]}\, \dd x\Bigr]\\
            &\quad +\inf_{M\geq 0}\Bigl[\frac{M}{n} \EE[\log(2\Ncal(\sigma, \Fcal, L^{2}(\PP_n)))] + \EE[F\cdot \II(F> M)]\Bigr].
        \end{aligned}
    \end{equation}
\end{theorem}

However, Theorem~\ref{theorem: convergence of EP with L^1 integrable functions} will be insufficient to establish tight estimation error bounds when additional information about $\Fcal$ is available, especially (weighted) uniform covering entropy. This is because they can provide additional information about the neighborhood of each function in $\Fcal$ in (weighted) uniform norms, providing tighter controls in establishing the upper bound of the expected empirical process. For this reason, Theorem~\ref{theorem: convergence of EP with L^infty integrable functions} is provided as a supplement to Theorem~\ref{theorem: convergence of EP with L^1 integrable functions} and helps establish the results in Section~\ref{sec: Heavy-tailed Nonparametric Least Squares Regression}. An informal version focusing on the square-integrable setting is provided below.
\begin{theorem}[Informal version of Theorem~\ref{theorem: convergence of EP with L^infty integrable functions}]
    Let $F\geq 0$ be the envelope function of $\Fcal$, and assume $\|f\|_{L^{2}(P)}\leq\sigma$ for all $f\in\Fcal$. Suppose that the weight function $w>0$ satisfies $1/w\in L^m (P)$ for some $m\geq2$ and $\|F\|_{L^\infty(w)}<\infty$. Then
        \[\begin{aligned}
            \EE^\ast\|\PP_n-P\|_\Fcal
            \lesssim& n^{-\frac{1}{2}}\int_{0}^{\sigma}\Bigl(1+\sqrt{\log\Ncal(x,\Fcal,L^{2}(P))}\Bigr)\, \dd x\\
            &\qquad
            +{n^{-\frac{m-1}{m}}}\int_{0}^{\|F\|_{L^\infty(w)}} \Bigl(1+\log\Ncal(x,\Fcal,L^{\infty}(w))^{\frac{m-1}{m}}\Bigr)\, \dd x.
        \end{aligned}\]
\end{theorem}
While we are not the first to use two covering entropy terms in tandem, the weighted uniform covering entropy here acts as a bridge between the complexity of the candidate function class and that of the associated loss classes. Therefore, in Section~\ref{sec: Heavy-tailed Nonparametric Least Squares Regression}, we can derive convergence rates for NPLSEs and extend the analysis seamlessly to a broader family of ERMs, thereby dispensing with tailored multiplier process tools required in prior work and marginally reducing technical overhead.

\subsection{Deep Huber Estimator}\label{sec: Deep Huber Estimator}

As a well-known principle in robust statistics, least squares loss is highly sensitive to outliers, making it particularly vulnerable to heavy-tailed noise. By contrast, robust alternatives, such as the Huber loss~\cite{huber1973robust} and quantile loss~\cite{koenker2005quantile}, often outperform least squares regression in practice.
From the classical M-estimation perspective, the success of these robust estimators stems from their bounded influence functions. For example, the Huber loss $\ell_\tau$ and its derivative $\ell_\tau^\prime$ are defined, for $\tau \in (0,\infty]$, by
\begin{equation}
    \ell_\tau(x)=\begin{cases}
        \frac{1}{2}x^2 & |x|\leq \tau,\\[3pt]
        \tau |x|-\frac{1}{2}\tau^2 & |x|>\tau,
    \end{cases}
    \qquad
    \ell_\tau^\prime(x)=\begin{cases}
        x & |x|\leq \tau,\\[3pt]
        \sgn(x)\tau & |x|>\tau.
    \end{cases}
\end{equation}
The boundedness of $\ell_\tau^\prime$ accounts for its robustness. However, asymptotic theory shows that both least squares and robust estimators with linear models are asymptotically normal under mild conditions, unable to fully explain the empirical superiority of robust methods.

The growing field of non-asymptotic analysis offers a new lens for explaining robustness. Catoni (2012)~\cite{catoni2012challenging} introduced a mean estimator that achieves sub-Gaussian concentration even under heavy-tailed data. Specifically, given $n$ i.i.d.\ observations from a distribution with mean $\mu$ and variance $\sigma^2$, Catoni’s estimator $\hat\mu_C(t)$ satisfies, for any $t>0$ and any $n$,
\[
    \PP\left(|\hat\mu_C(t)-\mu|\gtrsim t\,\sigma\, n^{-1/2}\right)\leq 2\exp(-t^2),
\]
in contrast to the sample mean, which only enjoys a polynomial deviation bound of order $1/t^2$. Motivated by this, we say an estimator is \emph{robust to heavy-tailedness} if it exhibits an exponential-type non-asymptotic concentration inequality.

Building on this principle, \cite{sun2020adaptive} showed that the adaptive Huber estimator for linear regression enjoys sub-Gaussian concentration. Later, \cite{fan2024noise} extended this analysis to ReLU networks under Huber loss. Specifically, given the data-generating process
\begin{equation}\label{eq: definition of least squares regression}
    Y_i=f_0(X_i)+\xi_i,\qquad \EE[\xi_i | X_i] = 0,
\end{equation}
where $f_0$ is the unknown regression function to be estimated. Assuming $\Fcal_n$ is the class of uniformly bounded ReLU networks, they proved that if $\|\EE[|\xi_i|^m\mid X_i]\|_{L^\infty}\leq v_m$ for some $m \geq 2$, then the deep Huber estimator
\begin{equation}\label{eq: definition of deep huber estimator}
    \hat{f}_n(\tau)\in\argmin_{f\in\Fcal_n}\frac{1}{n}\sum_{i=1}^n \ell_\tau\big(Y_i-f(X_i)\big)
\end{equation}
{also enjoys sub-Gaussian concentration when $\tau \ll \tilde{n}^{\frac{1}{m}}$, where $\tilde{n}$ is an effective sample size adjusted for the complexity of the deep ReLU network class $\Fcal_n$. It remains an open question whether the deep Huber estimator is robust uniformly across $\tau$ and $m\in(1,\infty)$.}

Using the newly developed empirical process tools, we can show that the deep Huber estimator is uniformly robust regardless of the Huber parameter $\tau$. Specifically, Theorem~\ref{thm: convergence rate of Huber regression} in Section~\ref{sec: our results, sec: Robustness of Deep Huber Regression} proves that
\begin{theorem}[Informal version of Theorem \ref{thm: convergence rate of Huber regression}]
    Under the same conditions as in \cite{fan2024noise},
    \[\begin{aligned}
        \PP\Bigl(\|\hat{f}_n(\tau)-f_0\|_{L^2(P)}&\gtrsim_{\log} t\sqrt\frac{\tau}{\tilde{n}}+\sqrt\frac{\tau}{\tilde{n}}\land \tilde{n}^{\frac{1}{2m}-\frac{1}{2}} \\
        &\quad + \inf_{f\in\Fcal_n}\|f-f_0\|_{L^2(P)} +\frac{v_m}{\tau^{m-1}}\Bigr)\leq 2\exp(-(DW)^2t^2).
    \end{aligned}\]
\end{theorem}
This sub-Gaussian inequality holds \emph{uniformly} over $\tau$, providing a complete robustness guarantee for ReLU network estimators under Huber loss and filling the gap left by \cite{fan2024noise}. In addition, our analysis extends beyond their setting to encompass infinite-variance noise. For a side-by-side comparison of assumptions and Huber-parameter regimes used to establish Catoni-type robustness, see Table~\ref{table: assumptions in the literature of Huber regression}.

\subsection{Heavy-Tailed Nonparametric Regression}\label{sec: Literature Review of Least Squares Regression}

We now turn to nonparametric least squares regression and more general regression models under heavy-tailed noise. Firstly, consider i.i.d.\ data $\{(X_i,Y_i)\}_{i=1}^n$ from $P$ as in \eqref{eq: definition of least squares regression}:
\[
    Y_i=f_0(X_i)+\xi_i, \quad \text{with}\quad \EE[\xi_i | X_i] = 0,
\]
where $f_0$ is the unknown regression function. The NPLSE over a candidate class $\Fcal_n$ is
\begin{equation}\label{eq: definition of NPLSE}
    \hat{f}_n \in \argmin_{f \in \Fcal_n} \frac{1}{n} \sum_{i=1}^n \bigl(Y_i - f(X_i)\bigr)^2.
\end{equation}

Classical rates for $\|\hat{f}_n-f_0\|_{L^2(P)}$ rely on strong regularity: early works (e.g., \cite{birge1993rates, birge1998minimum, vaart2023empirical}) assume independence between $X_i$ and $\xi_i$ and light-tailed (sub-Weibull) noise. While \cite{shen1994convergence, chen1998sieve} allow heavy tails, they focus on sieve spaces and do not quantify how tail behavior impacts rates.
More recent analyses target heavy-tailed noise directly. \cite{mendelson2016upper, mendelson2017local} obtained guarantees under moment assumptions on $\xi_i$ but require $\Fcal_n$ to be sub-Gaussian, thereby excluding many common classes (indicator, monotone/convex, Hölder) \cite[Prop.~3]{han2017sharp}. Later, \cite{han2018robustness, han2019convergence} replaced this assumption with uniform $L^2$ covering or bracketing entropy: when $\xi_i$ is independent of $X_i$, has finite $m$-th moment ($m\geq 1$), and $\Fcal_n$ has covering entropy exponent $\gamma\in(0,2)$,
\[
    \|\hat{f}_n - f_0\|_{L^2(P)}=\Ocal_\PP\left(n^{-\frac{1}{2+\gamma}} + n^{-\frac{1}{2}+\frac{1}{2m}}\right),
\]
where the first term reflects function class complexity and the second term captures the effect of noise tails.
More recently, \cite{kuchibhotla2022least} removed the independence assumption between $\xi_i$ and $X_i$: when $m\geq 2$, and the $L^\infty$ covering entropy exponent $\gamma \in (0,2)$,
\[
    \|\hat{f}_n - f_0\|_{L^2(P)}=\Ocal_\PP\Bigl(n^{-\frac{1}{2+\gamma}} + \tilde{n}^{-\big(\frac{2 - s}{1 - 1/m} + s\gamma\big)^{-1}}\Bigr),
\]
where the second term offers a sharper description of heavy-tailed effects by introducing the interpolation smoothness parameter $s\in[0,1]$ of $\Fcal_n$, which quantifies the interaction between noise tail thickness ($m$) and function class complexity ($\gamma$).

Despite these advances, the assumptions remain stringent: either independence or finite variance of $\xi_i$ is imposed, and $\Fcal_n$ is restricted to be Donsker by requiring $\gamma\leq 2$. Yet many modern classes are non-Donsker (e.g., certain RKHSs \cite{yang2020function} and multivariate convex functions for $d\geq 5$ \cite{kur2024convex}). Another limitation is the commonly imposed assumption $f_0\in\Fcal_n$. In practice, $\Fcal_n$ is often a (sieve) nonparametric function space (e.g., splines, wavelets, or deep neural networks), while $f_0$ is assumed only to satisfy some smoothness or structural conditions. This mismatch introduces an unavoidable approximation error, which must be controlled.

These constraints stem from the proof technique, \emph{multiplier processes} as a special empirical process, used in \cite{han2019convergence, kuchibhotla2022least}. For least squares, \eqref{eq: connection between convergence rates of ERM and EP, sec: overview} specializes to
\begin{equation}\label{eq: localized multiplier process of LSE, sec: introduction}
    \|\hat{f}_n - f_0\|_{L^2(P)}^2 \leq \sup_{f\in\Fcal_n} \bigl| (\PP_n - P)(|f_0 - f|^2(X) + 2\xi(f_0 - f)(X)) \bigr|.
\end{equation}
The first term on the right-hand side (RHS) is a standard empirical process, typically controllable under mild complexity assumptions on $\Fcal_n$. The difficulty lies in the \emph{multiplier} term
\[
    \sup_{f\in\Fcal_n} \big|(\PP_n-P)\bigl(\xi(f_0-f)(X)\bigr)\big|,
\]
whose analysis under heavy tails and for non-Donsker $\Fcal_n$ is delicate. As a result, existing techniques are tailored to least squares and do not readily extend to broader models.

With the newly developed empirical process tools, we can easily accommodate infinite-variance noise, non-Donsker classes, and approximation error. Moreover, our analysis is not confined to least squares regression but accommodates a wide range of empirical risk minimization problems, with nonparametric generalized linear models (NPGLMs) as the most notable example. These results are presented in Theorems~\ref{theorem: least squares estimation convergence rate with Linfty covering entropy} and~\ref{theorem: general ERM convergence rate with Linfty covering entropy}:
\begin{theorem}[Informal version of Theorems~\ref{theorem: least squares estimation convergence rate with Linfty covering entropy} and \ref{theorem: general ERM convergence rate with Linfty covering entropy}]
    For NPLSEs, if $m>1$ and the $L^\infty$ covering entropy of $\Fcal_n$ has exponent $\gamma\geq 0$,
    \[
         \EE\|\hat{f}_n-f_0\|_{L^2(P)}\lesssim_{\log n}\tilde{n}^{-(\frac{1}{2+\gamma}\land\frac{1}{2\gamma})} + \tilde{n}^{-\big(\frac{2 - s}{1 - 1/m} + s\gamma\big)^{-1}}+\inf_{f\in\Fcal_n}\|f-f_0\|_{L^2(P)}.
    \]
    This estimation error bound also holds for NPGLMs.
\end{theorem}

To highlight the advantages over \cite{han2019convergence,kuchibhotla2022least}, Table~\ref{table: comparion of NPLSE} contrasts assumptions on noise, function class complexity, and approximation error. Figure~\ref{fig: phase transition of least squares regression} provides a visual comparison of phase-transition regimes, underscoring the broader scope of our result.

\begin{table}[ht]
    \centering
    \caption{Comparison of the assumptions in existing literature and our result on the convergence rates of NPLSEs.}
    \resizebox{\linewidth}{!}{
    \begin{tabular}{c|cc|cc}
        \hline
        \multirow{2}{*}{ } & \multicolumn{2}{c|}{Distribution} & \multicolumn{2}{c}{Function Class $\Fcal_n$} \\ \cline{2-5}
        & \multicolumn{1}{c|}{$X_i$ and $\xi_i$} & $\EE[|\xi_i|^m]<\infty$ & \multicolumn{1}{c|}{non-Donsker} & {Approximation Error} \\ \hline
        Mendelson (2017) \cite{mendelson2017local} & \multicolumn{1}{c|}{Independent}& $m>2$ & \multicolumn{1}{c|}{---}  & Not Allowed \\
        Han and Wellner (2019) \cite{han2019convergence} & \multicolumn{1}{c|}{Independent} & $m\geq 1$ & \multicolumn{1}{c|}{Not Allowed} & Not Allowed  \\
        Kuchibhotla and Patra (2022) \cite{kuchibhotla2022least} &\multicolumn{1}{c|}{$\EE[\xi_i|X_i]=0$} &$m\geq 2$&\multicolumn{1}{c|}{Not Allowed}&Not Allowed\\
        Theorem \ref{theorem: least squares estimation convergence rate with Linfty covering entropy} (Ours)  & \multicolumn{1}{c|}{$\EE[\xi_i|X_i]=0$} & $m>1$ & \multicolumn{1}{c|}{Allowed} & Allowed \\ \hline
    \end{tabular}
    }
    \label{table: comparion of NPLSE}
\end{table}

\section{New Empirical Process Tools}\label{sec: The New Empirical Processes}

In this section, we establish the new empirical process tools that underpin our analysis of statistical learning models in subsequent sections. In Sections \ref{sec: Empirical Process with Expected L1+kappa covering entropy} and \ref{sec: Empirical Process with Expected Linfty covering entropy}, we present two maximal inequalities under different conditions: Theorems \ref{theorem: convergence of EP with L^1 integrable functions} and \ref{theorem: convergence of EP with L^infty integrable functions}, respectively. To provide further insight into the expected random covering entropy condition used in Theorem \ref{theorem: convergence of EP with L^1 integrable functions}, we illustrate it through Propositions \ref{proposition: average covering entropy domminated for parametrized class} and \ref{proposition: average covering entropy dominated by L-infty covering entropy}. Finally, in Section \ref{sec: Technical Proofs for Section sec: Upper Bounds of Expected Empirical Processes}, we outline our proof strategy for Theorem \ref{theorem: convergence of EP with L^1 integrable functions}.

\subsection{Maximal Inequality with Expected \texorpdfstring{$L^{1+\kappa}(\PP_n)$}{L-1+kappa} Covering Entropy}\label{sec: Empirical Process with Expected L1+kappa covering entropy}

Theorem \ref{theorem: convergence of EP with L^1 integrable functions} assumes only $L^{1+\kappa}(P)$ integrability of the function class $\Fcal$ for some $\kappa \in (0,1]$, along with $L^1(P)$ integrability of the envelope function. These assumptions are suited for regression under infinite variance noise. By leveraging the $\epsilon$-separation technique \cite{bousquet2002concentration, von2004distance}, this result also applies to non-Donsker classes, broadening the applicability of empirical process theory. We use $\EE^\ast$ to denote outer expectation, as the supremum of a collection of functions may no longer be measurable, and use $\II(\cdot)$ for the indicator function.

\begin{theorem}\label{theorem: convergence of EP with L^1 integrable functions}
    Denote $\PP_n$ as an empirical measure made up of $n$ i.i.d. observations from $P$. Given a function class $\Fcal$ with envelope function $F\geq 0$, assume $\sigma\geq \sup_{f\in\Fcal}\|f\|_{L^{1+\kappa}(P)}$ for some $\sigma>0$ and $\kappa\in(0,1]$. Then for any $M>0$, we have:
    \begin{equation}\label{eq: maximal inequality upper bound in theorem: convergence of EP with L^1 integrable functions}
        \begin{aligned}
            &\EE^\ast\|\PP_n-P\|_\Fcal
            \leq 52.8172\cdot \sqrt\frac{M^{1-\kappa}}{n}\cdot \sqrt{\lceil\log_2(\frac{1}{\sigma}\lor 2^2)\rceil} \cdot \Bigl(\sigma^\frac{1+\kappa}{2}\land 2^{-(1+\kappa)}\Bigr)\\
            &\qquad +\inf_{\epsilon\in (0, \sigma\land 2^{-2})}\Bigl[4\epsilon +60.3400\cdot\sqrt\frac{M^{1-\kappa}}{n}\int_{\epsilon/8}^{\frac{\sigma}{2}\land 2^{-3}} \sqrt{\frac{\EE[\log \Ncal(x,\Fcal,L^{1+\kappa}(\PP_n))]}{x^{1-\kappa}}}\, \dd x\Bigr]\\
            &\qquad + 6\cdot\sqrt\frac{ M^{1-\kappa}}{n} \sigma^\frac{1+\kappa}{2}\cdot  \sqrt{ \EE[\log(2\Ncal(\sigma/4\land2^{-4}, \Fcal, L^{1+\kappa}(\PP_n)))]}\\
            &\qquad +\frac{32 M}{n} \EE[\log(2\Ncal(\sigma/4\land2^{-4}, \Fcal, L^{1+\kappa}(\PP_n)))] + 2\EE[F\cdot \II(F> M)].
        \end{aligned}
    \end{equation}
\end{theorem}

While this upper bound consists of six terms, practical use typically requires computing the expected covering integral and identifying the dominant term by appropriately tuning $\epsilon$ and $M$.
Moreover, the first term includes a logarithmic factor, {$\sqrt{\lceil\log_2(\frac{1}{\sigma}\lor 2^2)\rceil}$}, which arises from bounding an incomplete Gamma function in the proof, and typically it can be dominated by the remaining terms. Consequently, by using the monotonicity of {$x\mapsto \EE[\log (2\Ncal(x,\Fcal,L^{1+\kappa}(\PP_n)))]$}, we obtain the following more interpretable bound.

\begin{corollary}\label{corollary: convergence of EP with L^1 integrable functions}
    Under the conditions of Theorem~\ref{theorem: convergence of EP with L^1 integrable functions}, suppose that $\kappa=1$ and $0<\sigma<1/4$. Then
    \[\begin{aligned}
        &\EE^\ast\|\PP_n-P\|_\Fcal
        \lesssim \inf_{0<\epsilon<\sigma}\Bigl[\epsilon + \frac{1}{\sqrt n}\int_{\epsilon/8}^{\sigma/2}
        \sqrt{\EE\bigl[\log\bigl(2\Ncal(x,\Fcal,L^2(\PP_n))\bigr)\bigr]}\,\dd x\Bigr]\\
        &\qquad + \inf_{M>0}\Bigl[\frac{M}{n}\EE\bigl[\log\bigl(2\Ncal(\sigma/4,\Fcal,L^2(\PP_n))\bigr)\bigr]
        + \EE[F\cdot\II(F>M)]\Bigr].
    \end{aligned}\]
\end{corollary}

In this simplified bound, the expected empirical process is controlled by a decomposition into two distinct components: the first term is driven primarily by the complexity of {$\Fcal$}, and the second term captures the tail behavior through the envelope {$F$}. A detailed discussion of these two terms is provided in Section~\ref{sec: Explanations to the Two Terms in Corollary~corollary: convergence of EP with L^1 integrable functions} of the supplement. We also compare this result with classical $L^2$-based maximal inequalities in Section~\ref{sec: Relationship between Corollary~corollary: convergence of EP with L^1 integrable functions and Classical L2 Maximal Inequalities} of the supplement.

{
\begin{remark}
[Interpretation of Theorem~\ref{theorem: convergence of EP with L^1 integrable functions} and Corollary~\ref{corollary: convergence of EP with L^1 integrable functions}]
The classical result in Theorem~2.14.2 of~\cite{vaart2023empirical} can be viewed as a special case of Theorem~\ref{theorem: convergence of EP with L^1 integrable functions} and Corollary~\ref{corollary: convergence of EP with L^1 integrable functions} when $\kappa=1$ and $\Fcal$ is uniformly bounded and Donsker. In contrast, Theorem~\ref{theorem: convergence of EP with L^1 integrable functions} requires only a uniform $L^{1+\kappa}(P)$ bound on~$\Fcal$ together with the integrability of its envelope~$F$, while quantifying complexity through the expected random $L^{1+\kappa}(\PP_n)$ covering entropy. Thus, these assumptions accommodate function classes that are not necessarily square-integrable. Furthermore, the cutoff parameter~$\epsilon$ accommodates entropy integrals that diverge at the origin, while the truncation level~$M$ balances function-class complexity against the tail behavior of the envelope~$F$, thereby allowing for both non-Donsker classes and heavy-tailed distributions.

However, empirical-process bounds based solely on the expected random $L^{1+\kappa}(\PP_n)$ covering entropy can be suboptimal for function classes possessing additional structural properties. To address this, Theorem~\ref{theorem: convergence of EP with L^infty integrable functions} provides a natural extension by incorporating weighted $L^\infty$ covering entropy, thereby exploiting such structures while remaining applicable to unbounded and non-$L^2(P)$ function classes.
\end{remark}
}

{In what follows, we establish a bridge between the $L^{1+\kappa}(\mathbb{P}_n)$ covering entropy and standard complexity measures to facilitate technical extensions of our framework.} Unlike classical maximal inequalities established based on uniform $L^2(P)$ covering or bracketing entropy, Theorem~\ref{theorem: convergence of EP with L^1 integrable functions} employs the expected $L^{1+\kappa}(\PP_n)$ covering entropy. While this quantity has appeared in \cite{von2004distance, koltchinskii2011oracle}, its generality warrants further exploration. Hence, the following propositions bridge this entropy and the more common ones.

{Proposition~\ref{proposition: average covering entropy domminated for parametrized class} facilitates extending Theorem~\ref{theorem: convergence of EP with L^1 integrable functions} beyond nonparametric function estimation in normed spaces. By providing a bridge from the expected $L^{1+\kappa}(\PP_n)$ covering entropy to more conventional complexity measures, it enables applications in settings with more general metrics, such as the analysis of Fréchet estimators in a metric space~\cite{petersen2019frechet,bhattacharjee2025nonlinear}. A detailed discussion is provided in Section \ref{sec: Extension of Theorem~theorem: convergence of EP with L^1 integrable functions to the Analysis of Weak Conditional Fréchet Mean} of the Supplementary Material.}

\begin{proposition}[Parameterized class]\label{proposition: average covering entropy domminated for parametrized class}
    Let $\Fcal=\{f_\theta:\theta\in\Theta\}$ be a function class parameterized by a metric space $(\Theta, d)$. Assume the following conditions:
    \begin{enumerate}
        \item The covering entropy of $(\Theta,d)$ satisfies: $\log\Ncal(h, \Theta, d)\leq D_\Theta\cdot h^{-\gamma}$.
        \item There is a function $G\in L^1(P)$ such that for any $x$, $|f_{\theta_1}(x)-f_{\theta_2}(x)|\leq d(\theta_1,\theta_2)\cdot G(x).$
        \item (Right-tail concentration) For some function $\psi_n:\RR_+\to[0,1]$ and all $m\geq 0$,
        \[
            \PP(\|G\|_{L^{1+\kappa}(\PP_n)}\geq m\cdot \|G\|_{L^{1+\kappa}(P)})\leq \psi_n(m).
        \]
    \end{enumerate}
    Then,
    \[
        \EE[\log\Ncal(h\|G\|_{L^{1+\kappa}(P)},\Fcal,L^{1+\kappa}(\PP_n))]\leq D_\Theta\cdot\bigl(\frac{1}{h}\bigr)^\gamma\cdot\Bigl( \int_0^\infty \psi_n(u)\, \dd (u^\gamma)\lor 1\Bigr).
    \]
\end{proposition}

However, many commonly used function classes do not admit a convenient parametric representation. To accommodate such cases, the next result bounds the expected covering entropy without a parametric structure, via a weighted uniform covering entropy.

\begin{proposition}[Weighted uniform covering entropy]\label{proposition: average covering entropy dominated by L-infty covering entropy}
    Assume that
    \begin{enumerate}
        \item For some $D_\Fcal>0$ and $\gamma,\eta\geq 0$, the weighted $L^\infty$ covering entropy of \(\Fcal\) satisfies that:
        \[
            \log\Ncal(x,\Fcal,L^\infty(\inner{\cdot}^{-\eta}))\leq D_\Fcal\cdot x^{-\gamma},\quad \text{for all $x$.}
        \]

        \item (Right-tail concentration) For some function $\psi_n:\RR_+\to[0,1]$, and some $\mu>0$,
        $$\PP(\PP_n\inner{x}^{(1+\kappa)\eta}\geq (m\cdot \mu)^{1+\kappa})\leq \psi_n(m),\quad\text{for any }m\geq 0.$$
    \end{enumerate}
    Then,
    \[
        \EE[\log\Ncal(h,\Fcal,L^{1+\kappa}(\PP_n))]\leq D_\Fcal\cdot\bigl(\frac{\mu}{h}\bigr)^\gamma\cdot\Bigl( \int_0^\infty \psi_n(u)\, \dd (u^\gamma)\lor 1\Bigr).
    \]
\end{proposition}

These two propositions provide tools for bounding the expected $L^{1+\kappa}(\PP_n)$ entropy via function class structure and right-tail thickness assumptions, allowing Theorem~\ref{theorem: convergence of EP with L^1 integrable functions} to be applied across a wider range of statistical learning models.

\subsection{Maximal Inequality with Weighted \texorpdfstring{$L^{\infty}$}{L-infty} Covering Entropy}\label{sec: Empirical Process with Expected Linfty covering entropy}

Despite Theorem \ref{theorem: convergence of EP with L^1 integrable functions} has greatly extended the scope of maximal inequality, it may yield suboptimal upper bounds for certain function classes. This is because the result relies solely on the expected $L^{1+\kappa}(\PP_n)$ random covering entropy, leaving other properties of the function class $\Fcal$ underutilized.

In this section, we refine Theorem \ref{theorem: convergence of EP with L^1 integrable functions} by incorporating weighted $L^\infty$ covering entropy of $\Fcal$. This refinement allows for sharper analysis by accounting for localized uniform behavior of the function class $\Fcal$ within small neighborhoods. Unlike the original $L^\infty$ covering entropy that implicitly requires uniformly boundedness, the weighted version accommodates unbounded functions, further broadening the applicability of empirical process theory.

\begin{theorem}\label{theorem: convergence of EP with L^infty integrable functions}
    Let $\PP_n$ be an empirical measure made up of $n$ i.i.d. observations from $P$. For $m>1$, denote $m\land 2=1+\kappa$ with $\kappa\in(0,1]$. Suppose the function class $\Fcal$ satisfies:
    \begin{enumerate}
        \item For some $\sigma>0$, $\sigma\geq \sup_{f\in\Fcal}\|f\|_{L^{1+\kappa}(P)};$
        \item There is a weight function $w>0$ such that $1/w\in L^m(P)$;
        \item $\Fcal$ has an envelope function $F$ with $\|F\|_{L^\infty(w)}<\infty$.
    \end{enumerate}
    Then,
    \begin{equation}\label{eq: expected empirical process, theorem: convergence of EP with L^infty integrable functions}
        \begin{aligned}
            &\EE^\ast\|\PP_n-P\|_\Fcal
            \leq \frac{127.63}{\kappa}\frac{\|1/w\|_{L^{1+\kappa}(P)}^{\frac{1-\kappa}{2}}}{n^{\frac{\kappa}{1+\kappa}}}\cdot  \Bigl(\int_{0}^{2\sigma}1+\log\Ncal(x,\Fcal,L^{1+\kappa}(P))^\frac{\kappa}{1+\kappa}\, \dd x\Bigr)^{\frac{1+\kappa}{2}}\\
            &\qquad\qquad\qquad\qquad\qquad\qquad\cdot\Bigl(\int_{0}^{2 \|F\|_{L^\infty(w)}} 1+\log\Ncal(x,\Fcal,L^{\infty}(w))^\frac{\kappa}{1+\kappa}\, \dd x\Bigr)^{\frac{1-\kappa}{2}}\\
            &\qquad\qquad\qquad+\frac{104.37}{\kappa} \frac{\|1/w\|_{L^m(P)}}{n^{1-\frac{1}{m}}}\cdot\Bigl(\int_{0}^{2 \|F\|_{L^\infty(w)}} 1+\log\Ncal(x,\Fcal,L^{\infty}(w))^{1-\frac{1}{m}}\, \dd x\Bigr).\\
        \end{aligned}
    \end{equation}
\end{theorem}

{Similar to Corollary~\ref{corollary: convergence of EP with L^1 integrable functions}, this maximal inequality consists of two terms that jointly characterize the impact of function class complexity and heavy-tailedness on the expected empirical process. When $\kappa=1$, the first term simplifies to a measure of the $L^2(P)$ covering entropy of $\Fcal$, effectively decoupling the complexity from the tail behavior, whereas the second term accounts for the tail index $m$ and the envelope function $F$ through its weighted $L^\infty(w)$ norm, capturing how the complexity of $\Fcal$ is modulated by the tail thickness of the distribution. }

{
\begin{remark}[Interpretation and extensions of Theorem~\ref{theorem: convergence of EP with L^infty integrable functions}]
This two-entropy formulation parallels the classical framework of Theorem~2.14.21 of~\cite{vaart2023empirical}, which combines $L^2(P)$ and $L^\infty$ entropy integrals. Our result extends this classical result to unbounded and potentially non-square-integrable function classes.

Building on the conditions of Theorem~\ref{theorem: convergence of EP with L^1 integrable functions}, Theorem~\ref{theorem: convergence of EP with L^infty integrable functions} incorporates a weight function~$w$ to provide finer control over the tail behavior of the envelope~$F$. The resulting weighted uniform covering entropy captures finer structural information regarding the complexity of~$\Fcal$.

A promising direction for future research is extending Theorem~\ref{theorem: convergence of EP with L^infty integrable functions} beyond the i.i.d.\ setting to weakly dependent observations, such as stationary $\beta$-mixing time series. Such an extension would further clarify the interplay among heavy-tailedness, function-class complexity, and temporal dependence in governing empirical-process behavior.
\end{remark}
}

\begin{remark}[Intrinsic connection between the extremely heavy-tailed and non-Donsker regimes]
    Unlike Theorem \ref{theorem: convergence of EP with L^1 integrable functions}, which leverages the $\epsilon$-separation technique~\cite{bousquet2002concentration, von2004distance}, this result requires the integrability of $\log\Ncal(\cdot,\Fcal,L^{1+\kappa}(P))^{\frac{\kappa}{1+\kappa}}$, $\log\Ncal(\cdot,\Fcal,L^{\infty}(w))^{\frac{\kappa}{1+\kappa}}$ and \(\log\Ncal(\cdot,\Fcal,L^\infty(w))^{1-\frac{1}{m}}\) near $x=0+$, thereby imposing restrictions on the covering entropy of \(\Fcal\). It is worth pointing out that there is an intrinsic link between non-Donsker function classes and extremely heavy-tailed distributions ($m\in(1,2)$). In \eqref{eq: expected empirical process, theorem: convergence of EP with L^infty integrable functions}, note that the covering integrals are not integrable for non-Donsker $\Fcal$ when $m$ is too large, which is counterintuitive, since this implies a sufficiently light-tailed distribution. In such cases, one can replace $m$ with a smaller value between $1$ and $2$ to ensure integrability. Interestingly, this trick may lead to the non-Donsker rate, $\tilde{n}^{-\frac{1}{2\gamma}}$, as demonstrated in Theorem~\ref{theorem: least squares estimation convergence rate with Linfty covering entropy}.
\end{remark}

\subsection{Proof Sketch for Theorem \ref{theorem: convergence of EP with L^1 integrable functions}}\label{sec: Technical Proofs for Section sec: Upper Bounds of Expected Empirical Processes}

This subsection outlines the proof strategy for Theorem~\ref{theorem: convergence of EP with L^1 integrable functions}. {We highlight the major technical differences between our proof and the classical approach for $L^2$-based maximal inequalities in Section~\ref{sec: Major Technical Differences from the Classical L2 Approach} of the Supplementary Material.}

\paragraph*{Step 1. Symmetrization and Truncation}
Let $\{\varepsilon_i\}_{i=1}^n$ be i.i.d. Rademacher random variables. By the symmetrization technique, it suffices to study the \textit{Rademacher complexity} of $\Fcal$:
$$\EE^\ast\|\PP_n-P\|_\Fcal\leq 2\EE^\ast\Bigl[\sup_{f\in\Fcal}\Bigl|\frac{1}{n}\sum_{i=1}^n\varepsilon_i\cdot f(X_i)\Bigr|\Bigr]=:2\EE^\ast\|\PP_n^\varepsilon\|_\Fcal.$$

To handle unbounded function class $\Fcal$, we adopt a truncation argument: for any $M > 0$,
\begin{equation}\label{eq: decomposition of truncated rademacher process, sec: Proof Outline of Theorem L1}
    \EE^\ast\|\PP_n^\varepsilon\|_{\Fcal}\leq \EE^\ast\Bigl[ \sup_{f\in\Fcal}|\PP_n^\varepsilon (f\cdot \II(F\leq M))|\Bigr]+\EE^\ast\Bigl[ \sup_{f\in\Fcal}|\PP_n^\varepsilon (f\cdot \II(F> M))|\Bigr],
\end{equation}
where the second term can be bounded by $\EE[F\cdot \II(F>M)]$.

\paragraph*{Step 2. Control the Truncated Rademacher average via Hoeffding's inequality}

For a fixed $f$, conditional Hoeffding's inequality and
\(
    \PP_n\!\left[f^2\II(F\leq M)\right] \leq M^{1-\kappa}\|f\|_{L^{1+\kappa}(\PP_n)}^{1+\kappa}
\)
give,
\begin{equation}\label{eq: Hoeffding concentration inequality for truncated average, sec: Proof Outline of Theorem L1}
    \PP \Bigl(|\PP_n^\varepsilon f\II(F\leq M)|\geq\sqrt{\frac{2t\cdot M^{1-\kappa}}{n}\|f\|_{L^{1+\kappa}(\PP_n)}^{1+\kappa}}\Big|\PP_n\Bigr)\leq 2\exp(-t),
\end{equation}
for every $t>0$, which explains the presence of the factor $\sqrt{M^{1-\kappa}/n}$ in \eqref{eq: maximal inequality upper bound in theorem: convergence of EP with L^1 integrable functions}.

\paragraph*{Step 3. Chaining Technique}

We construct a chain $\{\pi_s f\}_{s\in\mathbb{N}}$ for each $f \in \Fcal$, such that $\pi_0 f \equiv 0$, $\|\pi_{s+1}f-\pi_s f\|_{L^{1+\kappa}(\PP_n)}\leq 2^{-2s}$, and the cardinality of $\{\pi_s f: f\in\Fcal\}$ is at most $N_s:=\Ncal(2^{-2s},\Fcal,L^{1+\kappa}(\PP_n))$.
Then, given $f=\sum_{s=0}^\infty(\pi_{s+1}f-\pi_s f)$, we have $\EE^\ast\Bigl[ \sup_{f\in\Fcal}|\PP_n^\varepsilon (f\cdot \II(F\leq M))|\Bigr] = \EE^\ast\Bigl[ \sup_{f\in\Fcal}\Bigl|\sum_{s\in\NN}\PP_n^\varepsilon ((\pi_{s+1}f-\pi_s f)\cdot \II(F\leq M))\Bigr|\Bigr].$

With a union bound, applying \eqref{eq: Hoeffding concentration inequality for truncated average, sec: Proof Outline of Theorem L1} to each increment $\pi_{s+1}f-\pi_s f$ gives:
$$\begin{aligned}
    &\PP \Bigl(\sup_{f\in\Fcal}|\PP_n^\varepsilon (\pi_{s+1}f-\pi_s f)\II(F\leq M)|\geq 2^{-(1+\kappa)s}\sqrt{\tfrac{M^{1-\kappa}}{n}}\bigl(\sqrt{t}+\sqrt{2\log N_{s+1}}\bigr)\Big|\PP_n\Bigr)\\
    \leq & \ 2\exp(-t).
\end{aligned}$$

Replacing $t$ with $(1+s)(1+t)$ and summing over $s$ yields:
$$\begin{aligned}
    &\PP\Bigl(\sup_{f\in\Fcal}|\PP_n^\varepsilon f\cdot \II(F\leq M)|\geq \sqrt\frac{M^{1-\kappa}}{n}\sum_{s}\Bigl(\frac{\sqrt{(1+s)(1+t)}}{2^{(1+\kappa)s}}+\frac{\sqrt{2\log N_{s+1}}}{2^{(1+\kappa)s}}\Bigr)\Big|\PP_n\Bigr)\\ \leq  & \ 2\exp(-t).
\end{aligned}$$
The second summation in the RHS recovers the covering integral term in Theorem \ref{theorem: convergence of EP with L^1 integrable functions}:
\[
    \sum_{s}2^{-(1+\kappa)s}\cdot\sqrt{2\log N_{s+1}}\asymp \int\sqrt{\frac{2\log \Ncal(x,\Fcal,L^{1+\kappa}(\PP_n))}{x^{1-\kappa}}}\, \dd x.
\]
Finally, integrating the tail bound over $t > 0$ completes the proof of Theorem \ref{theorem: convergence of EP with L^1 integrable functions}.

\section{Robust Estimation with Deep ReLU Networks}\label{sec: Robustness of Deep Huber Regression}

In this section, we analyze the robustness of deep ReLU network estimators under Huber and quantile loss.
In Section~\ref{sec: our results, sec: Robustness of Deep Huber Regression}, we apply the new empirical process tools to study Huber regression with ReLU networks. We then investigate quantile regression with ReLU networks in Section~\ref{sec: Robustness of Quantile Regression}, establishing the first non-asymptotic sub-Gaussian concentration bounds in this setting without any moment assumptions on the response variables.

\subsection{Uniform Robustness of Deep Huber Regression}\label{sec: our results, sec: Robustness of Deep Huber Regression}

Using our empirical process framework, we establish a unified sub-Gaussian concentration bound for deep Huber estimators. Following \cite{fan2024noise}, we define the ReLU network class and introduce the following assumptions.

\begin{definition}[ReLU networks]
    A ReLU network with width $W$ and depth $D$ is a function $f:\RR^d \to \RR$ defined by, for $\sigma(x) = 0\lor x$ applied entrywise: $$f(x)=\Lcal_{D+1}\circ\sigma\circ\Lcal_{D}\circ\sigma\circ\cdots\circ\Lcal_2\circ\sigma\circ\Lcal_1(x),$$
    where $\Lcal_i(x)=A_i x+b_i$ for $A_i\in\RR^{d_i\times d_{i-1}}, b_i\in\RR^{d_i}$ and the dimensions are
    \[
        (d_0,d_1,\cdots,d_D,d_{D+1})=(d,W,\cdots,W,1).
    \]
    Let $M > 0$ be a truncation parameter. We define the truncated ReLU network class as
    $$\begin{aligned}
        &\Fcal(d,D,W,M)=\{x\mapsto \sgn(f(x))\cdot (|f(x)|\land M): f\text{ of the form above}\}.
    \end{aligned}$$
\end{definition}

\begin{assumption}[Distributional conditions]\label{assumption: distributional assumption for Huber regression}
    For the data-generating process in Equation \eqref{eq: definition of least squares regression}, assume $\|f_0\|_{L^\infty}\leq M$ for some $M>0$, and for some $m > 1$, there exists $v_m\in(0,\infty)$ such that $\|\EE[|\xi_i|^m|X_i]\|_{L^\infty}\leq v_m.$
\end{assumption}

{In contrast to existing work (e.g.,~\cite{fan2024noise}), which often assumes $m\ge 2$, our analysis accommodates noise $\xi_i$ with infinite second conditional moments by leveraging our new maximal inequality presented in Theorem~\ref{theorem: convergence of EP with L^1 integrable functions}. However, the requirement $m > 1$ represents a near-universal threshold for consistency. As demonstrated in \cite{sun2020adaptive}, even in the simpler context of linear models, if the noise possesses only a finite first moment, the minimax risk for the linear regression estimator is bounded below by a positive constant independent of the sample size. Thus, $m > 1$ is essentially necessary to ensure that the estimation error vanishes as $n \to \infty$.}

\begin{theorem}[Huber regression]\label{thm: convergence rate of Huber regression}
    Under Assumption \ref{assumption: distributional assumption for Huber regression} (if $m<2$, denote $v_2=\infty$), let $\tau\geq 2\max\{2M, (2v_m)^{1/m}\}$, $\Fcal_n=\Fcal(d,D,W,M)$ and $\tilde{n}=\frac{n}{(DW)^2\log(DW)}$. Then for any $\delta\in(0,1)$, the deep Huber estimator in \eqref{eq: definition of deep huber estimator} satisfies:
    \begin{equation}\label{eq: main result in thm: convergence rate of Huber regression}
        \PP\Bigl(\|\hat{f}_n(\tau)-f_0\|_{L^2(P)}\gtrsim \sqrt{\tau}\cdot \sqrt{\tau\land (\sqrt{v_2}+M)}\cdot \sqrt\frac{\log(2/\delta)}{n} + \delta_\sfrak + \delta_\afrak+\delta_\bfrak\Bigr)\leq \delta,
    \end{equation}
    where $\delta_\afrak:=\inf_{f\in\Fcal_n}\|f-f_0\|_{L^2(P)}$, $\delta_\bfrak:=\frac{v_m}{\tau^{m-1}}$, and
    $$\begin{aligned}
        &\delta_\sfrak:=_{\log n}\left\{\begin{aligned}
            & \sqrt{\tau}\cdot \sqrt{\tau\land (\sqrt{v_2}+M)}\cdot \tilde{n}^{-\frac{1}{2}} &  \text{when }\tilde{n}^{\frac{1}{m}}\cdot(M+v_m^{1/m})\geq \tau,\\
            & \sqrt{ M\cdot(M+v_m^{1/m}) }\cdot  \tilde{n}^{\frac{1}{2m}-\frac{1}{2}} & \text{when }\tilde{n}^{\frac{1}{m}}\cdot(M+v_m^{1/m})< \tau.\\
        \end{aligned}\right.
    \end{aligned}$$
\end{theorem}

We first explain the meaning of terms in~\eqref{eq: main result in thm: convergence rate of Huber regression} of this theorem. Specifically, this result separates the estimation error into three terms: the Huberization bias $\delta_\bfrak$, the approximation error $\delta_\afrak$, and the statistical error $\delta_\sfrak$. {
As the estimation error in~\eqref{eq: main result in thm: convergence rate of Huber regression} matches the lower bound established in~\cite{fan2024noise} when $m\geq 2$, it reflects the intrinsic difficulty posed by heavy-tailed noise.
}The Huberization bias arises from replacing the unbiased least squares loss with the Huber loss and therefore decreases as $\tau$ increases.

The approximation error \(\delta_\afrak\) measures the discrepancy between \(f_0\) and the ReLU network class. As width and depth grow, \(\delta_\afrak\) typically decreases, albeit at the cost of a larger statistical error \(\delta_\sfrak\). In particular, when \(f_0\) is a hierarchical composition model \cite{bauer2019deep, kohler2021rate}, exhibiting low-dimensional structure, \(\delta_\afrak\) admits explicit bounds, highlighting the strong approximation power of deep ReLU networks \cite{fan2024noise, lu2021deep}. Since approximation theory is not our primary focus, we keep \(\delta_\afrak\) in the form \(\inf_{f \in \Fcal_n} \|f - f_0\|_{L^2(P)}\). For explicit bounds in terms of \(D\) and \(W\) under hierarchical composition, see \cite{fan2024noise}.

The statistical error $\delta_\sfrak$ exhibits two distinct phase transitions, driven by the noise tail thickness (i.e., $m$) and the choice of the Huber parameter $\tau$:
$$\delta_\sfrak\asymp_{\log n}
\begin{cases}
    \sqrt{\tau/\tilde{n}} & \text{when }\tilde{n}^{\frac{1}{m}}\gg \tau,\text{ and }m\in[2,+\infty),\\
    \tau/\sqrt{\tilde{n}} & \text{when }\tilde{n}^{\frac{1}{m}}\gg \tau,\text{ and } m\in(1,2),\\
    \tilde{n}^{\frac{1}{2m}-\frac{1}{2}} & \text{when }\tilde{n}^{\frac{1}{m}}\ll \tau.
\end{cases}$$

\begin{enumerate}
    \item \textbf{First phase transition (moderate $\tau$, heavy tail noise).}
    When $\tilde{n}^{1/m}\gg \tau$, the behavior of $\delta_\sfrak$ bifurcates depending on $m$. For $m\geq 2$, we obtain the familiar $\sqrt{\tau/\tilde{n}}$ rate, reflecting the presence of finite variance. However, when $m<2$ (infinite variance regime), $\delta_\sfrak$ inflates to $\tau/\sqrt{\tilde{n}}$, since second-moment control is unavailable ($v_2=\infty$).
    To eliminate the Huberization bias $\delta_\bfrak$, one would ideally choose $\tau\to\infty$. Yet this worsens the statistical error in the infinite variance case, motivating a more conservative choice of $\tau$ when $m<2$. Notably, the $\tau/\sqrt{\tilde{n}}$ rate, reflecting infinite noise variance, was not identified in \cite{fan2024noise}.

    \item \textbf{Second phase transition (large $\tau$).}
    When $\tau\gg \tilde{n}^{1/m}$, all residuals $\{|Y_i-f(X_i)|\}_{i=1}^n$ lie below $\tau$ asymptotically, making the Huber estimator behave like NPLSE. In this regime, the statistical error reduces to $\tilde{n}^{\frac{1}{2m}-\frac{1}{2}}$, matching the rate of NPLSE with $s=\gamma=0$ in \eqref{eq: estimation error of NPLSE, theorem: least squares estimation convergence rate with Linfty covering entropy}.
\end{enumerate}

{
\begin{remark}[Comparison with Fan et al. (2024)~\cite{fan2024noise}]
Theorem~\ref{thm: convergence rate of Huber regression} provides a complete nonasymptotic characterization of the robustness of the deep Huber estimator. Under the same conditions and when $m \geq 2$, \cite{fan2024noise} established the following bound: for any $t\geq 1$,
\begin{equation}\label{eq: convergence rate of Huber ReLU NN in fan2024noise}
    \begin{aligned}
        \PP\Bigl(\|\hat{f}_n(\tau)-f_0\|_{L^2(P)}\gtrsim t\bigl(\sqrt{\frac{\tau \land \tilde{n}^{1/m}}{\tilde{n}}} + \delta_\afrak + \delta_\bfrak\bigr)\Bigr)
        \lesssim \exp\!\bigl(-(DW)^2t^2\bigr)+\frac{\II\{\tau \gg \tilde{n}^{\frac{1}{m}}\}}{t^{2m}},
    \end{aligned}
\end{equation}
which yields only a polynomial deviation bound when $\tau \gg \tilde{n}^{1/m}$. Compared to the existing results such as~\eqref{eq: convergence rate of Huber ReLU NN in fan2024noise}, our bound offers the following key improvements:
\begin{itemize}
    \item First, we provide a unified sub-Gaussian bound that holds uniformly over $\tau$, providing a consistent explanation of the robustness of ReLU network estimators under the Huber loss. In contrast, the previous results~\eqref{eq: convergence rate of Huber ReLU NN in fan2024noise} yielded only a polynomial deviation bound when $\tau \gg \tilde{n}^{1/m}$. Thus, the deep ReLU network estimator consistently benefits from using the Huber loss for heavy-tailed robustness across all regimes of $\tau$.
    \item Second, our results demonstrate a sharper dependence on uncertainty because the deviation parameter $t$ does not affect the approximation error or the Huberization bias terms, leading to a cleaner decomposition of the nonasymptotic estimation error.
    \item Finally, the proposed framework offers broader applicability and remains valid even in the regime when $m < 2$, which was not covered by previous characterizations in \eqref{eq: convergence rate of Huber ReLU NN in fan2024noise}.
\end{itemize}
The technical challenges underlying the uniform sub-Gaussian bound and the incorporation of optimization error are discussed in Sections~\ref{sec: Technical Advancements in Theorem~thm: convergence rate of Huber regression} and~\ref{sec: Incorporating optimization error}, respectively, of the supplement.
\end{remark}}

\subsection{Robustness of Deep Quantile Regression}\label{sec: Robustness of Quantile Regression}

In this section, we analyze the performance of ReLU-based network estimators under non-smooth loss functions, with a particular focus on quantile regression, also known as deep quantile networks (DQNs).

As an alternative to mean regression, quantile regression estimates the conditional quantile function \cite{koenker2005quantile}. For a fixed quantile level $\tau \in (0,1)$, suppose we observe i.i.d. samples $\{(X_i, Y_i)\}_{i=1}^n$ such that the $\tau$-th conditional quantile of $Y_i$ given $X_i = x$ is
$$
Q_\tau(Y_i \mid X_i = x) = f_0(x).
$$
Here, $f_0$ minimizes the population loss of quantile regression:
$
f_0 \in \argmin \EE[\rho_\tau(Y_1 - f(X_1))]
$
for the check loss $\rho_\tau(u) = u(\tau - \II(u < 0))$, which naturally leads to estimation via empirical risk minimization.
Define the DQN estimator as
    $$\hat{f}_n\in\argmin_{f\in\Fcal(d,D,W,M)}\sum_{i=1}^n \rho_\tau(Y_i-f(X_i)).$$

However, the non-smoothness of $\rho_\tau$ poses significant challenges for theoretical analysis. For instance, the quadratic-type lower bound condition (cf. Equation \eqref{eq: quadratic loss function, assumption: regularity conditions for general ERM}) required by Theorem \ref{theorem: general ERM convergence rate with Linfty covering entropy} is difficult to verify in this setting. Moreover, quantile regression generally targets a different object than mean regression, making it fundamentally distinct from the least squares and Huber regression analyses in Sections \ref{sec: least squares regression} and \ref{sec: our results, sec: Robustness of Deep Huber Regression}.

Despite these challenges, significant progress has been made in analyzing the convergence rates of (convolution-type smoothed) quantile regression estimators for linear models; see, e.g., \cite{belloni2011penalized, fan2014strong, fan2014adaptive, zheng2015globally, fan2016multitask, he2023smoothed, tan2022high}. More recently, attention has shifted from linear functions to deep ReLU networks \cite{padilla2022quantile, shen2021deep, shen2021robust, zhong2024neural, feng2024deep}. However, existing results are often limited to asymptotic regimes, convergence-in-expectation bounds, and some rely on moment conditions of the response variable. These limitations hinder a complete understanding of the non-asymptotic behavior of DQN estimators.

To establish a non-asymptotic concentration bound for ReLU network estimators under quantile regression, we introduce the following regularity condition, which ensures well-behaved conditional densities near the target quantile:

\begin{assumption}\label{assumption: adaptive self-calibration condition of conditional distribution}
    Denote $p_{Y|X=x}$ as the condition density of $Y$ given $X=x$. Assume there exists some $\delta>0$, such that $$0<\inf_{t\in[f_0(x)-\delta,f_0(x)+\delta]}p_{Y|X=x}(t)\leq \sup_{t\in\RR}p_{Y|X=x}(t)<\infty,\quad\text{almost surely.}$$
\end{assumption}

As a sufficient condition for the adaptive self-calibration condition in \cite{feng2024deep}, it facilitates a clean decomposition of statistical and approximation errors in quantile regression. Importantly, this condition is verifiable for many data-generating processes.

We are now ready to present our main result in this section:

\begin{theorem}[Quantile regression]\label{thm: quantile regression convergence rate}
    Let $\tilde{n}=\frac{n}{(DW)^2\log(DW)}$. Suppose Assumption~\ref{assumption: adaptive self-calibration condition of conditional distribution} holds.
    Then for any $\delta\in(0,1)$,
    $$\PP\Bigl(\|\hat{f}_n-f_0\|_{L^2(P)}\gtrsim_{\log} \sqrt{\log(\frac{2}{\delta})}\cdot\tilde{n}^{-\frac{1}{2}} + \inf_{f\in\Fcal_n}\|f-f_0\|_{L^2(P)}\Bigr)\leq \delta.$$
\end{theorem}

To the best of our knowledge, this is the first result that establishes a non-asymptotic concentration bound for DQN estimators. Unlike previous works that provide only asymptotic or convergence-in-expectation results, our analysis yields a sub-Gaussian concentration inequality, offering substantially stronger performance guarantees. Notably, our result does not require finite moments on the response variable, highlighting the intrinsic robustness of quantile regression under general, possibly heavy-tailed, data distributions.

\section{Phase Transitions in Heavy-Tailed Nonparametric Regression}\label{sec: Heavy-tailed Nonparametric Least Squares Regression}

This section studies the phase transitions in heavy-tailed nonparametric regression. In Section~\ref{sec: least squares regression}, we present new convergence guarantees for nonparametric least squares estimators under very mild regularity assumptions (Theorem~\ref{theorem: least squares estimation convergence rate with Linfty covering entropy}).
Section~\ref{sec: General Empirical Risk Minimization} further extends the theoretical analysis: Theorem~\ref{theorem: general ERM convergence rate with Linfty covering entropy} establishes parallel guarantees for a broad class of nonparametric generalized linear models (GLMs), and more general empirical risk minimization (ERM) problems.
Section~\ref{sec: Regression with Set-Structured Function Classes} studies the set-structured least squares regression and gives its convergence rate.

\subsection{Main Results on Nonparametric Regression within a Non-Donsker Class with Infinite Error Variance}\label{sec: least squares regression}

We now present our main theoretical contributions for least squares regression under heavy-tailed and heteroscedastic noises with both Donsker and non-Donsker function classes. These results leverage the newly developed empirical process theory tools in Section \ref{sec: The New Empirical Processes} and are established under minimal regularity conditions. Specifically, we consider the data-generating process in Equation \eqref{eq: definition of least squares regression} and make the following assumptions.

\begin{assumption}[Distributional assumption]\label{assumption: moment condition of least squares regression}
     For the model in Equation \eqref{eq: definition of least squares regression}, assume
     \begin{enumerate}
         \item \textbf{Heavy-tailed noise:} $\EE[|\xi_i|^m]<\infty$ for some $m>1$.
         \item \textbf{Heteroscedastic noise:} Denote $\kappa\in(0,1]$, such that $m\land 2=1+\kappa$. Then, $$\|\EE[|\xi_i|^{1+\kappa}|X_i]\|_{L^\infty}<\infty.$$
     \end{enumerate}
\end{assumption}

This assumption accommodates both heavy-tailed and heteroscedastic noise. Although we assume the conditional moment $\EE[|\xi_i|^{1+\kappa} | X_i]$ is uniformly bounded for simplicity, the result presented below can be extended to cases where this conditional moment is an unbounded function of $X_i$, with necessary adjustments in the proof. {A finite first moment of the noise is also necessary for least-squares regression to be well-defined. Without the integrability of $\xi_i$, the population risk may not exist.}

We also assume the following conditions on the regression function $f_0$ and the candidate function class $\Fcal_n$. For simplicity, we assume the functions are uniformly bounded, although this can be relaxed under additional distributional assumptions.

\begin{assumption}[Function class assumption]\label{assumption: least squares regression model}
    Suppose:
    \begin{enumerate}
        \item \label{item: uniformly boundedness assumption in assumption: least squares regression model} $\Fcal_n\cup\{f_0\}$ is uniformly bounded by some $M>0$.

        \item \label{item: uniformly covering entropy assumption in assumption: least squares regression model}
        \textbf{Uniform covering entropy:} There exists $D_{\Fcal_n}\geq 1$ and $\gamma\geq 0$, such that,
        \begin{equation}\label{eq: covering entropy of function class for LSE}
            \log\Ncal(x,\Fcal_n,L^\infty)\leq_{\log} D_{\Fcal_n}\cdot x^{-\gamma},\qquad \text{for all}\quad x>0.
        \end{equation}

        \item \label{item: interpolation condition of function space in assumption: least squares regression model}
        \textbf{Interpolation condition of $\Fcal_n$:}
        For any \(f_n^\ast \in \Fcal_n\) and any subset \(\Gcal \subseteq \Fcal_n - \{f_n^\ast\}\), let \(F_\Gcal\) denote the envelope function of \(\Gcal\). There exists a constant $s \in [0,1]$ such that:
        \begin{equation}\label{eq: local envelope function of LSE}
            \|F_\Gcal\|_{L^\infty}\lesssim \sup_{f\in\Gcal}\|f\|_{L^2(P)}^s.
        \end{equation}
    \end{enumerate}
\end{assumption}

Assumption~\ref{item: uniformly covering entropy assumption in assumption: least squares regression model} in \ref{assumption: least squares regression model} accommodates a wide range of function classes, including linear functions, polynomials, wavelets, reproducing kernel Hilbert spaces (RKHS), deep and shallow neural networks, Sobolev/Hölder spaces, isotonic functions, and convex functions.
In \eqref{eq: covering entropy of function class for LSE}, the complexity of $\Fcal_n$ is characterized by two parameters, $D_{\Fcal_n}$ and $\gamma$:
\begin{enumerate}
    \item $D_{\Fcal_n} > 0$ serves as an ``effective dimension'' of $\Fcal_n$, influencing both approximation and statistical errors. We define the ``effective sample size'' as $\tilde{n} = n / D_{\Fcal_n}$.
    \item A larger $\gamma$ indicates a more complex $\Fcal_n$, {and leads to a Donsker/non-Donsker dichotomy. A function class $\Fcal$ is called \textit{Donsker} if $\GG_n:=\sqrt{n}(\PP_n-P)$ converges in distribution to a tight Borel measurable limit $\GG$ as an element of $\ell^\infty(\Fcal)$. A standard sufficient condition for this is $\gamma<2$~\cite{vaart2023empirical}. Hence, in what follows,} we refer to $\Fcal_n$ as a non-Donsker class when $\gamma \geq 2$ and a Donsker class when $\gamma < 2$.
\end{enumerate}

Unlike prior works that typically restrict attention to Donsker classes ($\gamma < 2$), this assumption accommodates non-Donsker classes as well. For instance, certain RKHSs are non-Donsker~\cite{yang2020function}, and convex regression becomes non-Donsker when $d \geq 5$ \cite{kur2024convex}.

The interpolation condition in \eqref{eq: local envelope function of LSE}, which generalizes classical smoothness assumptions, has long been used to sharpen convergence rates of M-estimators \cite{gine1983central, chen1998sieve}, and has recently also been considered for NPLSEs in \cite{kuchibhotla2022least}. A larger $s$ reflects smoother functions in $\Fcal_n$ and leads to faster convergence rates.

Moreover, the interpolation exponent $s$ can often be computed explicitly using established analytic tools, such as the Gagliardo–Nirenberg inequality \cite{gagliardo1959ulteriori, nirenberg1966extended} and the Brezis–Mironescu inequality \cite{brezis2018gagliardo}, provided that $\Fcal_n$ exhibits some smoothness structure, say Hölder or Sobolev smoothness. Even in the absence of such smoothness information, the interpolation condition still holds with $s = 0$ as a consequence of the uniform boundedness of $\Fcal_n$.

\begin{table}[ht]
\centering
\caption{Covering entropy exponents $\gamma$ for several common function classes.}
\begin{tabular}{cccc}
\hline
$\Fcal_n$ & $\gamma$ & Reference \\ \hline

$\alpha$-Hölder  & $d/\alpha$ & Theorem 2.7.1 in \cite{vaart2023empirical}\\

Bounded, Lipschitz convex & $d/2$ & Corollary 2.7.15 in \cite{vaart2023empirical} \\

ReLU Networks & $0$ & Lemma 5 in \cite{schmidt2020nonparametric} \\
RKHS with exponential decay kernel & $0$ & Lemma D.2 in \cite{yang2020function} \\

RKHS with polynomial decay kernel & $>0$ & Lemma D.2 in \cite{yang2020function} \\ \hline
\end{tabular}
\label{table: Covering entropy exponents gamma for several common function classes}
\end{table}

\begin{table}[ht]
\centering
\caption{Interpolation exponents $s$ for several common function classes (Table 1 in \cite{kuchibhotla2022least})}
\begin{tabular}{ccc}
\hline
Domain  & $\Fcal_n$  & $s$  \\ \hline
$[0,1]^d$ & $\alpha$-Hölder & $(2\alpha)/(2\alpha+d)$    \\
$[0,1]^d$ & $\alpha$-Sobolev & $(2\alpha-1)/(2\alpha+d-1)$\\
$[0,1]$   & Uniformly Lipschitz & $2/3$  \\ \hline
\end{tabular}
\label{table: Interpolation exponents s for several common function classes}
\end{table}

We summarize typical values of $(\gamma, s)$ for common function classes in Tables~\ref{table: Covering entropy exponents gamma for several common function classes} and \ref{table: Interpolation exponents s for several common function classes}, as reported in the existing literature. Our assumptions encompass the following canonical cases:
\begin{itemize}
    \item \textbf{Linear models (OLS):} $\Fcal_n = \{x \mapsto \beta^\top x : \beta \in \Omega \subset \RR^d\}$, with $D_{\Fcal_n}=d$, $\gamma=0$, and $s=1$ under mild assumptions on the covariates $X_i$.
    \item \textbf{Sparse linear class (best subset selection):} $\Fcal_n = \{x \mapsto \beta^\top x : \beta \in \Omega \subset \RR^d, \|\beta\|_0\leq \sfrak\ll d\}$, with $D_{\Fcal_n}=\sfrak\log d$, $\gamma=0$, and $s=1$ under mild assumptions on $X_i$.
    \item \textbf{Hölder class:} $\Fcal_n=\Hcal^{\alpha}([0,1]^d)$, with $\gamma=d/\alpha$ and $s=\frac{2\alpha}{2\alpha+d}$.
    \item \textbf{ReLU networks:} depth $D$ and width $W$, with $D_{\Fcal_n}\asymp_{\log}(DW)^2$, $\gamma=0$, and $s=0$.
\end{itemize}

We now present the main result for NPLSE under these assumptions:

\begin{theorem}\label{theorem: least squares estimation convergence rate with Linfty covering entropy}
    Under Assumptions \ref{assumption: moment condition of least squares regression} and \ref{assumption: least squares regression model}, the NPLSE in Equation \eqref{eq: definition of NPLSE} satisfies
    \begin{equation}\label{eq: estimation error of NPLSE, theorem: least squares estimation convergence rate with Linfty covering entropy}
        \EE\|\hat{f}_n-f_0\|_{L^2(P)}\lesssim_{\log n}\tilde{n}^{-(\frac{1}{2+\gamma}\land\frac{1}{2\gamma})} + \tilde{n}^{-\big(\frac{2 - s}{1 - 1/m} + s\gamma\big)^{-1}}+\inf_{f\in\Fcal_n}\|f-f_0\|_{L^2(P)}.
    \end{equation}
\end{theorem}

As two special cases: when $\EE[|\xi_i|^m]<\infty$ and $\|\EE[|\xi_i|^{m\land 2}|X_i]\|_{L^\infty}<\infty$ for some $m> 1$,
\begin{itemize}
    \item For the canonical linear model, Theorem~\ref{theorem: least squares estimation convergence rate with Linfty covering entropy} recovers the classical rate
    \[
        \EE\|\hat f_n-f_0\|_{L^2(P)}\lesssim_{\log}\Bigl(\frac{d}{n}\Bigr)^{\frac{1}{2}\land (1-\frac{1}{m})}+\inf_{f\in\Fcal_n}\|f-f_0\|_{L^2(P)}.
    \]
    \item For best subset selection with sparsity $\sfrak\ll d$, it yields the high-dimensional rate
    \[
        \EE\|\hat f_n-f_0\|_{L^2(P)}\lesssim_{\log} \Bigl(\frac{\sfrak\cdot\log d}{n}\Bigr)^{\frac{1}{2}\land (1-\frac{1}{m})}+\inf_{f\in\Fcal_n}\|f-f_0\|_{L^2(P)}.
    \]
\end{itemize}

{
Theorem~\ref{theorem: least squares estimation convergence rate with Linfty covering entropy} is not only minimax optimal for linear models in these two special cases~\cite{sun2020adaptive}, but it also matches known lower bounds for more general function classes. While a unified sharpness characterization is delicate given the interplay between heavy-tailed noise, non-Donsker classes, and the interpolation exponent $s$, this convergence rate matches the lower bound established in~\cite{han2019convergence} when $m\in(1,\infty)$, $\gamma\in(0,2)$, and $s=0$. These results suggest that the phase transitions identified here characterize the fundamental limits of NPLSEs under heavy-tailed distributions.
}

Before comparing our result with existing literature, we first elaborate on the three components of the upper bound in \eqref{eq: estimation error of NPLSE, theorem: least squares estimation convergence rate with Linfty covering entropy}:
\begin{itemize}
    \item \textbf{Statistical error from function complexity.} The first term $\tilde{n}^{-(\frac{1}{2+\gamma}\land \frac{1}{2\gamma})}$ reflects the complexity of $\Fcal_n$. In fact, this is the sharp convergence rate of NPLSE under the Gaussian noise assumption \cite{kur2020suboptimality}. In particular, when $\gamma < 2$, the rate $\tilde{n}^{-1/(2+\gamma)}$ is known as the Donsker rate; when $\gamma \geq 2$, the non-Donsker rate $\tilde{n}^{-1/(2\gamma)}$ characterizes the non-Donsker regime.

    \item \textbf{Statistical error due to heavy-tailed noise.} The second term, $\tilde{n}^{-\big(\frac{2 - s}{1 - 1/m} + s\gamma\big)^{-1}}$, quantifies the effect of heavy-tailed noise, by capturing the interaction between the noise tail thickness $m$ and the function complexity $\gamma$ through the interpolation exponent $s$. A larger $s$ (e.g., function class $\Fcal_n$ is smoother) yields a faster convergence rate.
    Meanwhile, when $m = \infty$, this term is negligible compared to the first term regardless of $s$. More generally, it becomes negligible compared to the first term as long as
    \begin{equation}\label{eq: phase transition between heavy-tailed and donsker/non-donsker rate}
        m \geq \begin{cases} \frac{2 + (1 - s)\gamma}{s + (1 - s)\gamma} & \text{if } \gamma \in (0,2), \\
        \frac{\gamma}{\gamma - 1} & \text{if } \gamma \geq 2. \end{cases}
    \end{equation}
    This condition delineates a \emph{phase transition boundary}:     the NPLSE exhibits the same convergence behavior as under Gaussian noise, despite the presence of heavy tails.

    \item \textbf{Approximation error.} The final term, $\inf_{f\in\Fcal}\|f-f_0\|_{L^2(P)}$, quantifies the expressive power of $\Fcal_n$.
    Despite being independent of sample size, approximation error is connected to statistical error through the complexity of $\Fcal_n$. Typically, choosing a $\Fcal_n$ with larger $D_{\Fcal_n}$ and $\gamma$ reduces the approximation error but increases the statistical error. A trade-off between these two sources of error is necessary for minimizing the overall estimation error.
\end{itemize}

\begin{remark}[Comparison of Theorem \ref{theorem: least squares estimation convergence rate with Linfty covering entropy} to existing works]\label{remark: comparison to existing works on convergence rates of NPLSEs}
    Theorem~\ref{theorem: least squares estimation convergence rate with Linfty covering entropy} recovers the convergence rates of NPLSEs previously established in Theorem~4.1 of \cite{kuchibhotla2022least}. While both results accommodate heavy-tailed and heteroscedastic noise, Theorem~\ref{theorem: least squares estimation convergence rate with Linfty covering entropy} significantly relaxes the model assumptions. Specifically, it allows the noises to have infinite variance (i.e., $m \in (1,2)$), $\Fcal_n$ to be a non-Donsker class (i.e., $\gamma \geq 2$), and incorporates the approximation error $\inf_{f \in \Fcal_n}\|f - f_0\|_{L^2(P)}$.
    Prior to \cite{kuchibhotla2022least}, \cite{han2019convergence} derived convergence rates under a stronger independence assumption between $X_i$ and $\xi_i$. We mark that Theorem \ref{theorem: least squares estimation convergence rate with Linfty covering entropy} matches their rates (cf. Equation (3.1) in \cite{han2019convergence}) by taking $s=0$. A comparison of the phase transition regimes across different results is illustrated in Figure~\ref{fig: phase transition of least squares regression}, highlighting the broader scope of our result.
\end{remark}

\subsection{Nonparametric Generalized Linear Models}\label{sec: General Empirical Risk Minimization}

The previous subsection demonstrates how heavy-tailed noise and function class complexity jointly determine the convergence rate of NPLSEs. It is important to emphasize that the foundation of its theoretical analysis, the empirical process tools developed in Section~\ref{sec: The New Empirical Processes}, extends far beyond least squares regression and applies broadly to statistical learning models, particularly nonparametric generalized linear models (NPGLMs). By contrast, prior works such as \cite{han2019convergence} and \cite{kuchibhotla2022least} rely on multiplier processes specifically tailored to least squares regression, which restricts their generality.

To generalize our results, we introduce the following assumption:

\begin{assumption}\label{assumption: regularity conditions for general ERM}
    Let $\ell(f;x,y)=:\ell(f)$ be a convex loss function. Assume that:
    \begin{enumerate}
        \item There exists a measurable function $U(x,y)\geq 0$ such that for all $f,g\in\Fcal_n$,
        \begin{equation}\label{eq: definition of pseudo-Lipschitz loss function, assumption: regularity conditions for general ERM}
            |\ell(f;x,y)-\ell(g;x,y)|\leq U(x,y)\cdot |f(x)-g(x)|.
        \end{equation}

        \item For some $m>1$, let $m\land 2=1+\kappa$ so that $\kappa\in(0,1]$. Assume $U(X,Y)\in L^m(P)$ and $\|\EE[U(X,Y)^{1+\kappa}|X]\|_{L^\infty}<\infty$.

        \item For target function $f_0$, as the minimizer of $f\mapsto \EE_P[\ell(f)]$, the excess risk satisfies a quadratic condition:
        \begin{equation}\label{eq: quadratic loss function, assumption: regularity conditions for general ERM}
            \|f-f_0\|_{L^2(P)}^2\lesssim \EE_P[\ell(f)]-\EE_P[\ell(f_0)]\lesssim \|f-f_0\|_{L^2(P)}^2,\quad\text{for all}\quad f\in\Fcal_n.
        \end{equation}
    \end{enumerate}
\end{assumption}

This assumption generalizes Assumption \ref{assumption: moment condition of least squares regression} and is satisfied by a variety of models:
\begin{itemize}
    \item \textbf{Least squares regression.} With $\ell(f; x, y) = (y - f(x))^2$ and $\Fcal_n$ uniformly bounded by $M$, one can take \( U(x, y) = 2|y| + 2M\leq 2|\xi|+4M \). The integrability conditions on $U$ then reduce to moment conditions on $\xi$, as in Assumption \ref{assumption: moment condition of least squares regression}.

    \item \textbf{Generalized linear models (GLMs).} When $\ell(f;x,y) = -y\cdot f(x) + \psi(f(x))$ represents the negative log-likelihood function of a GLM for some convex $\psi$:
    \begin{itemize}
        \item \textbf{Logistic regression:} $\psi(x) = \log(1 + e^x)$ gives $U(x, y) \equiv 1$.
        \item \textbf{Poisson regression:} $\psi(x) = e^x$ allows $U(x, y) = e^M + y$ if $\Fcal_n$ is uniformly bounded by $M$, a common regularity condition in the literature (cf. Assumption 3(ii) in \cite{tian2023transfer}).
    \end{itemize}
\end{itemize}

The quadratic excess risk condition \eqref{eq: quadratic loss function, assumption: regularity conditions for general ERM} also holds for NPGLMs with mild assumptions:
\begin{proposition}\label{prop: quadratic stability link function for NPGLM}
    Let $\psi\in C^2(\RR)$ be a convex function, and define the loss function of a nonparametric generalized linear model as $\ell(f):=\ell(f;\xbf,y)=-y\cdot f(\xbf)+\psi(f(\xbf)).$ Suppose the conditional expectation satisfies $\EE[y|\xbf]=\psi^\prime(f_0(\xbf))$ for some target function $f_0$.

    For a function class $\Fcal$, define the convex hull evaluated on the support of $P$ as
    $$\overline\Fcal(P)=\{af(\xbf)+(1-a)f_0(\xbf):a\in[0,1],\xbf\in \supp(P), f\in\Fcal\}.$$
    Then the excess population risk satisfies the following two-sided bound:
    $$\frac{1}{2}\min_{u\in\overline\Fcal(P)}\psi^\pprime(u)\cdot \|f-f_0\|_{L^2(P)}^2\leq \EE_{P}[\ell(f)]-\EE_{P}[\ell(f_0)]\leq \frac{1}{2}\|\psi^\pprime\|_{L^\infty(\overline\Fcal(P))}\cdot\|f-f_0\|_{L^2(P)}^2.$$
\end{proposition}

Robust regression models, including Huber and quantile regression, also satisfy \eqref{eq: definition of pseudo-Lipschitz loss function, assumption: regularity conditions for general ERM}, since these loss functions are Lipschitz. While they may not fulfill the quadratic growth condition in \eqref{eq: quadratic loss function, assumption: regularity conditions for general ERM}, our framework can accommodate them with minor modifications by leveraging their specific properties. With the ReLU network as a concrete example, these extensions are developed in detail in Sections~\ref{sec: our results, sec: Robustness of Deep Huber Regression} and~\ref{sec: Robustness of Quantile Regression}, respectively.

We are now ready to show that the convergence rate established in Theorem~\ref{theorem: least squares estimation convergence rate with Linfty covering entropy} for the least squares loss is also valid for a broader class of models satisfying Assumption~\ref{assumption: regularity conditions for general ERM}:

\begin{theorem}\label{theorem: general ERM convergence rate with Linfty covering entropy}
    Under Assumptions \ref{assumption: least squares regression model} and \ref{assumption: regularity conditions for general ERM}, the empirical risk minimizer $\hat{f}_n\in\argmin_{f\in\Fcal_n}\sum_{i=1}^n \ell(f;X_i,Y_i)$     achieves the same estimation error decomposition as in \eqref{eq: estimation error of NPLSE, theorem: least squares estimation convergence rate with Linfty covering entropy}.
\end{theorem}

The proof of Theorem~\ref{theorem: general ERM convergence rate with Linfty covering entropy} proceeds identically to that of Theorem~\ref{theorem: least squares estimation convergence rate with Linfty covering entropy}. This demonstrates that the newly developed empirical process framework, developed in Section~\ref{sec: The New Empirical Processes}, provides a unified theoretical foundation for a broad class of statistical learning models.

\subsection{Set-Structured Least Squares Regression}\label{sec: Regression with Set-Structured Function Classes}

This subsection discusses least squares regression under ``set-structured’’ function classes, a special but important setting that enables minimax-optimal convergence rates in the non-Donsker regime, as introduced in \cite{han2021set}. Building on and extending the theoretical developments in \cite{han2021set}, we use this setting to further elaborate on the broad applicability of the newly developed empirical process framework.

{ As shown in Section~\ref{sec: least squares regression}, the nonparametric least-squares estimator (NPLSE) may converge at the non-Donsker rate $\tilde{n}^{-1/(2\gamma)}$ when the function class $\Fcal_n$ is non-Donsker (i.e., $\gamma \geq 2$), which is strictly worse than the Donsker rate $\tilde{n}^{-1/(2+\gamma)}$. This sub-optimality was analyzed by \cite{birge1993rates, birge2006model}, who constructed a ``pathological'' function class to demonstrate that this slower rate is inevitable in certain scenarios. Recently, \cite{kur2020suboptimality} confirmed the sub-optimality of NPLSEs for general non-Donsker function classes, which was attributed primarily to the large bias of the estimator rather than its variance \cite{kur2024variance}.

However, the non-Donsker rate is not sharp for all non-Donsker classes.  \cite{yang1999information, yang2001nonparametric} showed that for any function class $\Fcal_n \subseteq L^\infty$, the minimax rate $\varepsilon_n$ for regression with finite-variance noise satisfies:
\(
    \inf_{\hat{f}_n} \sup_{f_0 \in \Fcal_n} \EE\|\hat{f}_n - f_0\|_{L^2(P)} \asymp \varepsilon_n,
\)
where $\log\Ncal(\varepsilon_n, \Fcal_n, L^2(P)) \asymp n \varepsilon_n^2.$ Under the covering entropy condition that \(\log\Ncal(x, \Fcal_n, L^2(P)) \asymp D_{\Fcal_n}\, x^{-\gamma}\), the minimax rate coincides with the Donsker rate $\varepsilon_n \asymp \tilde{n}^{-\frac{1}{2+\gamma}}$, regardless of the value of $\gamma$.

A prominent example of a non-Donsker class that attains this minimax rate is multiple isotonic regression, which is defined as $\{f:[0,1]^d\to\RR, f(\xbf)\leq f(\ybf)\text{ for any }\xbf\leq \ybf\}$. Here $\xbf\leq \ybf$ means that $x_i\leq y_i$ for all $i\in[d]$, $\xbf=(x_1,\cdots,x_d)$ and $\ybf=(y_1,\cdots,y_d)$. Specifically, its $L^2$ covering entropy has exponent $\gamma = 2(d-1)$ for $d\geq 3$~\cite{gao2007entropy}, while its $L^2$ estimation error scales as $n^{-\frac{1}{2d}}$~\cite{han2019isotonic}, which is minimax-optimal and strictly faster than the non-Donsker rate. Similar phenomena also appear in log-concave density estimation~\cite{kur2019optimality,carpenter2018near}.}

These cases have been systematically studied by Han~\cite{han2021set}. In particular, \cite{han2021set} demonstrated that when function classes possess a \textit{set structure}, the estimator can achieve the minimax rate up to logarithmic factors, regardless of $\gamma$. The concept of set structure arises from the observation that, for an indicator class $\Fcal_\Cscr:=\{\II_C:C\in\Cscr\}$, where $\Cscr$ is a collection of measurable sets, the $L^1(P)$ norm of $\II_C$ is the same as its squared $L^2(P)$ norm. As a result, $$\log\Ncal(h,\Fcal_\Cscr, L^1(P))=\log\Ncal(\sqrt{h},\Fcal_\Cscr, L^2(P))\quad\text{for all }h>0,$$ meaning that $\Fcal_\Cscr$ permits a significantly smaller subset to cover the entire function class in the weaker $L^1(P)$ norm, as opposed to the $L^2(P)$ norm. Building on this insight, the minimax optimality of several shape-constrained estimation models has been validated in \cite{han2021set}, including isotonic regression and $s$-concave density estimation.

Using the newly developed empirical process results, we can generalize the conclusions in \cite{han2021set} to broader settings from a technical perspective. For simplicity, we focus here on asymptotic results under light-tailed data distributions. Extension with milder regularity conditions and other models beyond least squares regression are left to interested readers.

\begin{theorem}\label{theorem: indicator regression convergence rate with L2 covering entropy}
    Consider the nonparametric regression in Equation \eqref{eq: definition of least squares regression}. Suppose that
    \begin{enumerate}
        \item \textbf{Envelope function.} For some $C,\eta\geq 0$, $|f(x)|\leq C\cdot \inner{x}^\eta$ for all $x$ and all $f\in \Fcal_n\cup\{f_0\}$.
        \item \textbf{Light-tailed distribution.} Both $\inner{X}^\eta$ and noise $\xi$ are sub-Weibull random variables.
        \item \textbf{Set-structured function class.} For some $\gamma\geq 0$ and $D_{\Fcal_n}\geq 1$, for all $0<x\le1$ and $k=1,2$, $\EE[\log2\Ncal(x,\Fcal_n,L^{k}(\PP_{n}))]\leq D_{\Fcal_n}\cdot x^{-\frac{k\cdot \gamma}{2}}.$
    \end{enumerate}

    Then, the NPLSE in Equation \eqref{eq: definition of NPLSE} achieves the following rate for all $\gamma>0$: $$\|\hat{f}_n-f_0\|_{L^2(P)}=_{\log n}\Ocal_{\PP}(\tilde{n}^{-\frac{1}{2+\gamma}})+\inf_{f\in\Fcal_n}\|f-f_0\|_{L^2(P)}.$$
\end{theorem}

Here, we have established a rate bound for least squares regression for a set-structured function class. Unlike \cite{han2021set}, which assumed uniform boundedness of the function class \(\Fcal_n\) and Gaussian noise, our results relax these assumptions. By accounting for approximation error, our results allow for the statistical-approximation error trade-off. Moreover, with additional assumptions, our framework can accommodate unbounded function classes beyond polynomial envelope functions and heavy-tailed noise distributions.

The proof uses the expected $L^2(\PP_n)$ covering entropy for $0<\gamma<2$ and the expected $L^1(\PP_n)$ covering entropy for $\gamma\ge2$.

\section{Conclusion}\label{conclusion}

We have systematically investigated how heavy-tailedness and function class complexity affect several statistical learning procedures, including nonparametric least squares regression, deep Huber regression, deep quantile regression, and other general empirical risk minimization. By relaxing restrictive regularity assumptions commonly imposed in the literature, we extend theoretical guarantees to more general and realistic settings. Our refined analysis builds upon a suite of new results in empirical process theory. Specifically, we relax classical assumptions, enabling our framework to accommodate heavy-tailed distributions, non-square-integrable and non-Donsker classes, as well as approximation errors.

\stopcontents[main]
\printbibliography

\newpage
\appendix

\setcounter{page}{1}
\setcounter{section}{0}
\setcounter{theorem}{0}
\setcounter{equation}{0}
\setcounter{table}{0}
\setcounter{figure}{0}
\renewcommand\thesection{S\arabic{section}}
\renewcommand\thetheorem{S\arabic{theorem}}
\renewcommand\theproposition{S\arabic{proposition}}
\renewcommand\theequation{S\arabic{equation}}
\renewcommand\thefigure{S\arabic{figure}}
\renewcommand\thetable{S\arabic{table}}

\counterwithin{figure}{section}
\counterwithin{theorem}{section}
\counterwithin{table}{section}

\title{Supplementary Material to ``\worktitle''}
\author{\large Yizhe Ding, Runze Li and Lingzhou Xue}
\date{Department of Statistics, The Pennsylvania State University}

\maketitle

This supplementary material includes additional technical comments and mathematical proofs. Section \ref{sec: Convergence Rates of Sieved M-estimators} serves as a concise yet comprehensive tutorial, guiding readers on how to apply the newly developed user-friendly empirical process theory in statistical learning problems.
Section~\ref{sec: Discussion of Theorem~theorem: convergence of EP with L^1 integrable functions and Corollary~corollary: convergence of EP with L^1 integrable functions} provides additional discussion of Theorem~\ref{theorem: convergence of EP with L^1 integrable functions} and Corollary~\ref{corollary: convergence of EP with L^1 integrable functions}.
Section~\ref{sec: Discussion of Theorem~thm: convergence rate of Huber regression} further discusses incorporating optimization error and the technical advancements into Theorem~\ref{thm: convergence rate of Huber regression}. The remaining sections consist of the mathematical proofs of the theoretical results stated in this article.
\newline

\startcontents[supp]
\printcontents[supp]{}{1}{}

\section{Convergence Rates of M-estimators}\label{sec: Convergence Rates of Sieved M-estimators}

In this section, we discuss the \textit{one-shot localization} method for deriving convergence rates of $M$-estimators~\cite{van1987new, van2002m}, which underlies our analysis of both asymptotic and non-asymptotic estimation errors for empirical risk minimizers (ERMs). Section~\ref{sec: Convergence Rates of Sieved and Penalized M-estimators} introduces Theorem~\ref{thm: Estimation Error of Sieved M-estimators} and Theorem~\ref{thm: Estimation Error of penalized M-estimators}, which relate the estimation errors of sieved and penalized $M$-estimators to their corresponding empirical processes. Section~\ref{sec: Application to ERMs with Empirical Process Theory} illustrates how to apply these results to general ERM tasks. Section~\ref{sec: Different Modes of Convergence Rates} discusses how to derive bounded-in-probability and convergence-in-expectation rates as special cases of the non-asymptotic bounds.

\subsection{Convergence Rates of Sieved and Penalized M-estimators}\label{sec: Convergence Rates of Sieved and Penalized M-estimators}

Compared to other methods, \textit{one-shot localization} imposes weaker assumptions, making it particularly suited for general M-estimation and modern statistical learning. For instance, the \textit{peeling argument} (cf. Theorem 3.4.1 in \cite{vaart2023empirical}) requires the population loss function to be locally quadratic at the true parameter. However, such quadratic behavior is often unavailable. Moreover, one-shot localization naturally accommodates approximation errors, facilitating an explicit trade-off analysis between statistical and approximation errors.

The effectiveness of one-shot localization primarily relies on the convexity of the empirical loss function, a mild assumption that holds for many models, such as least squares, logistic regression, Huber regression, quantile regression, and optimal transport.

We now present the method under minimal assumptions, followed by discussions of its assumptions, implications, and applications.

\begin{theorem}[Estimation error of sieved M-estimators]\label{thm: Estimation Error of Sieved M-estimators}
    Let $(S,\|\cdot\|)$ be a normed linear space, and let $\Theta_n\subseteq S$ be a sieved parameter space, which is not necessarily a linear subspace or a convex set. Define the empirical and population loss functions $L_n$ and $L$ on $S$, with their respective minimizers: $$\hat\theta_n\in\argmin_{\theta\in\Theta_n}L_n(\theta);\quad \theta_0\in\argmin_{\theta\in \Theta_n\cup\{\theta_0\}}L(\theta).$$
    For some deterministic point $\theta_n^\ast\in\Theta_n$, \footnote{Common choices include $\theta_n^\ast\in\argmin_{\theta\in\Theta_n}L(\theta)$ and $\theta_n^\ast\in\argmin_{\theta\in\Theta_n}\|\theta-\theta_0\|$.} define the extended sieved parameter space as $\overline\Theta_n=\{t\theta+(1-t)\theta_n^\ast:t\in[0,1], \theta\in \Theta_n\}$.
    Assume the following:
    \begin{enumerate}
        \item \textbf{Convexity:} The empirical loss function $\theta\mapsto L_n(\theta)$ is convex on $S$.
        \item \textbf{Stability link:} There is a function $\wfrak_n:\RR_{\geq 0}\to\RR_{\geq 0}$ such that, for any $\theta\in S$,
        \begin{equation}\label{eq: definition of stability link in thm: Estimation Error of Sieved M-estimators}
            \wfrak_n(\|\theta-\theta_0\|)\leq L(\theta)-L(\theta_0).
        \end{equation}
        Moreover, its inverse function $\wfrak_n^{-1}$ is non-decreasing and sub-additive, i.e.
        \begin{equation}\label{eq: sub-additive of stability link in thm: Estimation Error of Sieved M-estimators}
            \wfrak_n^{-1}(a+b)\leq \wfrak_n^{-1}(a)+\wfrak_n^{-1}(b)\quad \text{for all }a,b\geq 0.
        \end{equation}
        \item \textbf{Empirical process:} There is a function $\phi_n$, such that for any $c>0$,
        \begin{equation}\label{eq: EP assumption, thm: Estimation Error of Sieved M-estimators}
            \PP\Bigl(\sup_{\theta\in\overline\Theta_n:\|\theta-\theta_n^\ast\|\leq c}\Bigl|L(\theta)-L_n(\theta)-L(\theta_n^\ast)+L_n(\theta_n^\ast)\Bigr|\geq \phi_n(c,\delta_n)\Bigr)\leq\delta_n.
        \end{equation}
        Moreover, $c\mapsto\wfrak_n^{-1}\circ\phi_n(c,\delta_n)$ is a non-decreasing and a sub-linear function: there is some $\tau_0>0$ such that
        $$\begin{aligned}
            \wfrak_n^{-1} \circ \phi_n(c, \delta_n)
                \begin{cases}
                \leq c/4, & \text{if } c \geq \tau_0, \\
                \geq c/4, & \text{if } c \leq \tau_0.
                \end{cases}
        \end{aligned}$$
    \end{enumerate}
    Then, with probability at least $1-\delta_n$, the following bound holds:
    \begin{equation}\label{eq: decomposition of estimation error, thm: Estimation Error of Sieved M-estimators}
        \|\hat\theta_n-\theta_0\|\leq \underbrace{\inf\Bigl\{\tau>0:\wfrak_n^{-1}\circ\phi_n(\tau,\delta_n)\leq\frac{\tau}{4}\Bigr\}}_\text{statistical error}+\underbrace{9\wfrak_n^{-1}\bigl(L(\theta_n^\ast)-L(\theta_0)\bigr)}_\text{approximation error}.
    \end{equation}
\end{theorem}

{

\begin{remark}[Comparison with Theorem 3.4.1 of \cite{vaart2023empirical}]
Theorem~\ref{thm: Estimation Error of Sieved M-estimators} derives the nonasymptotic convergence rate of a sieved $M$-estimator from three core components: the convexity of~$L_n$, a stability link~$\wfrak_n$, and the localized high-probability empirical-process bound~\eqref{eq: EP assumption, thm: Estimation Error of Sieved M-estimators}. In parallel with the classical formulation of Theorem~3.4.1 in~\cite{vaart2023empirical}, we define
\[
    \rho_n(\delta_n):=\inf\Bigl\{c>0:\wfrak_n^{-1}\circ\phi_n(c,\delta_n)\leq c/4\Bigr\},
    \qquad
    a_n:=\wfrak_n^{-1}\bigl(L(\theta_n^\ast)-L(\theta_0)\bigr).
\]
Then, Theorem~\ref{thm: Estimation Error of Sieved M-estimators} guarantees that, with probability at least $1-\delta_n$,
\[
    \|\hat\theta_n-\theta_0\|\leq \rho_n(\delta_n)+9a_n.
\]
Here $\phi_n(c,\delta_n)$ is a high-probability modulus of the centered loss process localized to a ball of radius~$c$, and its balance with the stability link $\wfrak_n$ determines the statistical radius~$\rho_n(\delta_n)$, playing a role precisely analogous to the modulus~$\phi_n$ in Theorem~3.4.1 of~\cite{vaart2023empirical}. A smaller sieve typically reduces the localized empirical-process modulus and hence the statistical error, but increases the sieve approximation error~$a_n$ by limiting the approximation accuracy for~$\theta_0$. Balancing these two competing terms determines the sieve choice.

In practice, applying Theorem~\ref{thm: Estimation Error of Sieved M-estimators} reduces to verifying the localized empirical-process control in~\eqref{eq: EP assumption, thm: Estimation Error of Sieved M-estimators}. For ERMs, the maximal inequalities established in Theorems~\ref{theorem: convergence of EP with L^1 integrable functions} and~\ref{theorem: convergence of EP with L^infty integrable functions} provide explicit construction for the required function~$\phi_n$. Furthermore, this unifying rate framework extends naturally to penalized estimators via Theorem~\ref{thm: Estimation Error of penalized M-estimators} and to approximate empirical minimizers with non-negligible optimization error via Theorem~\ref{thm: Estimation Error of Sieved M-estimators with optimization error}.

In summary, Theorem~\ref{thm: Estimation Error of Sieved M-estimators} plays a role analogous to Theorem~3.4.1 of~\cite{vaart2023empirical}, as both convert local fluctuations of a centered empirical process into an estimator's convergence rate. However, it departs from the classical formulation in Theorem~3.4.1 of~\cite{vaart2023empirical} in three key respects: (i) it provides nonasymptotic finite-sample probability bounds rather than asymptotic convergence results; (ii) it replaces the quadratic relationship between population risk and the estimation-error criterion with a general stability link; and (iii) it explicitly retains and quantifies the sieve approximation error directly through excess population loss.
\end{remark}

}

The following remarks further clarify the assumptions on the stability link, the extended parameter space $\overline\Theta_n$, and the sub-linearity of the function $c \mapsto \wfrak_n^{-1} \circ \phi_n(c, \delta_n)$. Section~\ref{sec: Different Modes of Convergence Rates} subsequently discusses alternative modes of convergence beyond the nonasymptotic formulation above.

\begin{remark}[Stability link $\wfrak_n$]
    Instead of restricting to quadratic population loss functions, we allow $\wfrak_n$ to be a general function linking the evaluation criterion $\|\theta - \theta_0\|$ and the excess population loss $L(\theta) - L(\theta_0)$. In many cases, such as least squares regression and nonparametric generalized linear models (see Proposition \ref{prop: quadratic stability link function for NPGLM}), we can take $\wfrak_n(x) \asymp x^2$. In such settings, the non-decreasing and sub-additive properties of $\wfrak_n^{-1}$ follow naturally from the polynomial structure of $\wfrak_n$.
\end{remark}

\begin{remark}[Extended parameter space $\overline\Theta_n$]
    The empirical process bound in Theorem \ref{thm: Estimation Error of Sieved M-estimators} (cf. Equation \eqref{eq: EP assumption, thm: Estimation Error of Sieved M-estimators}) is evaluated over $\overline{\Theta}_n$, not just $\Theta_n$. A natural concern is whether this expansion significantly increases the covering entropy. Proposition \ref{prop: covering number of extended parameter space} addresses this by showing that the covering entropy of $\overline{\Theta}_n$ grows only logarithmically more than that of $\Theta_n$, which introduces negligible additional complexity.
\end{remark}

\begin{proposition}\label{prop: covering number of extended parameter space}
    Let $\Theta_n$ be a parameter space equipped with a norm $\vertiii{\cdot}$, and define the extended parameter space $\overline\Theta_n=\{t\cdot\theta+(1-t)\theta_n^\ast:t\in[0,1], \theta\in \Theta_n\}$ for some $\theta_n^\ast$.
    \begin{enumerate}
        \item If $\Theta_n$ is bounded, i.e., there is some $R>0$, such that $\vertiii{\theta}\leq R$ for all $\theta\in\Theta_n$, then for any $h>0$, $$\log\Ncal(h,\overline\Theta_n,\vertiii{\cdot})\leq \log\Ncal(h/2,\Theta_n,\vertiii{\cdot})+\log\Ncal(h/(4R), [0,1], |\cdot|).$$
        \item For unbounded $\Theta_n$, define localized set: for some $c>0$, $\overline\Theta_{n,c}:=\{\theta\in\overline\Theta_n:\vertiii{\theta-\theta_n^\ast}\leq c\}$. Then, for any $h>0$, $$\log\Ncal(h,\overline\Theta_{n,c},\vertiii{\cdot})\leq \log\Ncal(h/4,\Theta_n,\vertiii{\cdot})+\log\Ncal(h/(2c), [0,1], |\cdot|).$$
    \end{enumerate}
\end{proposition}

\begin{remark}[Sub-linearity of $c \mapsto \wfrak_n^{-1} \circ \phi_n(c, \delta_n)$]
    This sub-linearity assumption is essential for expressing the estimation error as a combination of statistical and approximation errors, as captured in Equation \eqref{eq: decomposition of estimation error, thm: Estimation Error of Sieved M-estimators}. In practice, this condition is typically satisfied for two reasons: (1) $c\mapsto \phi_n(c, \delta_n)$ often scales like $c^k$ for $k \in [0,1]$; (2) $\wfrak_n$ is usually a polynomial function (e.g., $\wfrak_n(x) = x^2$).
\end{remark}

{
While Theorem~\ref{thm: Estimation Error of Sieved M-estimators} targets sieved M-estimators, penalized M-estimators also merit attention due to their broad use in statistics and machine learning. By augmenting the empirical risk with a penalty term, penalized M-estimators can exhibit desirable structural properties, such as sparsity under Lasso or folded concave penalties~\cite{fan2001variable, fan2014strong}. Accommodating penalized estimators may also facilitate the analysis of over-parameterized models, a phenomenon that is highly relevant for neural network estimators in practice. In what follows, we present a variant of Theorem~\ref{thm: Estimation Error of Sieved M-estimators} that establishes estimation error bounds for penalized M-estimators.

\begin{theorem}[Estimation error of penalized M-estimators]\label{thm: Estimation Error of penalized M-estimators}
    Let $p_{\lambda_n}:\Theta_n\to\RR$ be a penalty function with tuning parameter $\lambda_n$, and define the penalized empirical risk minimizer by
    \[
        \hat\theta_n\in\argmin_{\theta\in\Theta_n}\, L_n(\theta) + p_{\lambda_n}(\theta).
    \]
    Under the same conditions as in Theorem~\ref{thm: Estimation Error of Sieved M-estimators}, with probability at least $1-\delta_n$, the following bound holds:
    \begin{equation}\label{eq: decomposition of estimation error, thm: Estimation Error of penalized M-estimators}
        \begin{aligned}
            \|\hat\theta_n-\theta_0\|
            \leq& \underbrace{\inf\Bigl\{\tau>0:\wfrak_n^{-1}\circ\phi_n(\tau,\delta_n)\leq\frac{\tau}{4}\Bigr\}}_\text{statistical error}+\underbrace{9\wfrak_n^{-1}\bigl(L(\theta_n^\ast)-L(\theta_0)\bigr)}_\text{approximation error}\\
            &\qquad + \underbrace{4\wfrak_n^{-1}\Bigl(\bigl|p_{\lambda_n}(\theta_n^\ast) - p_{\lambda_n}(\hat\theta_n)\bigr|\Bigr)}_\text{penalty}.
        \end{aligned}
    \end{equation}
\end{theorem}

In this decomposition, the first two terms coincide with those in the unpenalized case (see Theorem~\ref{thm: Estimation Error of Sieved M-estimators}), while the third term captures the specific contribution of the penalty. When the relationship between $p_{\lambda_n}(\theta_n^\ast)$ and $p_{\lambda_n}(\hat\theta_n)$ is not directly tractable, one may bound it by
\[
    4\wfrak_n^{-1}\Bigl(\bigl|p_{\lambda_n}(\theta_n^\ast) - p_{\lambda_n}(\hat\theta_n)\bigr|\Bigr)\leq 4\wfrak_n^{-1}\Bigl(2\sup_{\theta\in\Theta_n}|p_{\lambda_n}(\theta)\bigr|\Bigr).
\]
This bound also makes explicit how the tuning parameter $\lambda_n$ can affect the overall estimation error, thereby providing guidance on the appropriate scaling of $\lambda_n$.

Moreover, although optimizing deep neural networks is typically a non-convex problem in parameter space,  Theorem~\ref{thm: Estimation Error of penalized M-estimators} remains valid because the analysis relies on the convexity of the loss functional in function space. In particular, let $S=L^2(P)$ be the ambient normed linear space and let $\Theta_n=\{f_{\theta}:\theta\}\subseteq S$ denote the class of neural networks parameterized by $\theta$. Then, we know that common empirical risks, such as least squares regression, Huber regression, and cross-entropy loss, are convex functionals $f\mapsto L_n(f)$ on $f\in S=L^2(P)$. The penalty $p_{\lambda_n}(f_\theta)$ is defined only on the network class $\Theta_n$. One may therefore use ridge, Lasso, or folded concave penalties on the parameter $\theta$ of $f_\theta$. No convexity is required for either $f_\theta\mapsto p_{\lambda_n}(f_\theta)$ or $\theta\mapsto p_{\lambda_n}(f_\theta)$.

Consequently, Theorem~\ref{thm: Estimation Error of penalized M-estimators} provides a general route for establishing estimation error bounds for penalized neural network estimators, bridging the gap between our theoretical results and common practical implementations.
}

\subsection{Application to ERMs with Empirical Process Theory}\label{sec: Application to ERMs with Empirical Process Theory}

Theorem \ref{thm: Estimation Error of Sieved M-estimators} provided a general framework for analyzing the convergence rates of sieved M-estimators. In many practical applications, such M-estimators arise naturally as \textit{empirical risk minimizers} (ERMs).

Suppose data $\zbf \sim P$ and consider a function class $\Fcal_n$. Let $\ell(f;\zbf) =: \ell(f)$ denote the loss function for $f \in \Fcal_n\cup\{f_0\}$. Given the empirical measure $\PP_n$ consisting of i.i.d. observations from $P$, define the empirical and population loss functions as
$$
L_n(f) = \PP_n \ell(f), \qquad L(f) = P \ell(f).
$$
Let $f_0$ denote the target function, defined as the minimizer of the population loss $L(f)$ over $\Fcal_n \cup \{f_0\}$. Define the extended function class $\overline{\Fcal}_n$ associated with $\Fcal_n$ as in Theorem \ref{thm: Estimation Error of Sieved M-estimators}.

Let $f_n^\ast \in \Fcal_n$ be a fixed ``approximator'' to $f_0$, and for any $c > 0$, define the localized loss function class:
$$\Lscr_c=\{\ell(f)-\ell(f_n^\ast):f\in\overline\Fcal_n, \|f-f_n^\ast\|\leq c\}.$$
Then, the empirical process term in Equation \eqref{eq: EP assumption, thm: Estimation Error of Sieved M-estimators} can be rewritten as an empirical process:
$$\sup_{f\in\overline\Fcal_n:\|f-f_n^\ast\|\leq c}\bigl|(\PP_n-P)(\ell(f)-\ell(f_n^\ast))\bigr|=:\|\PP_n-P\|_{\Lscr_c}.$$

To obtain the localized uniform bound $\phi_n$, we invoke the maximal inequalities for expected empirical processes stated in Theorems \ref{theorem: convergence of EP with L^1 integrable functions} and \ref{theorem: convergence of EP with L^infty integrable functions}. The computation of $\phi_n$ typically proceeds through the following three steps:
\begin{enumerate}
    \item \textbf{Envelope function.} Identify the envelope function $LF_c$ of the class $\Lscr_c$.
    \item \textbf{Uniform norm bound.} Establish a uniform bound $\sigma$ such that $\sigma \geq \sup_{\ell \in \Lscr_c} \|\ell\|_{L^{1+\kappa}(P)}$, typically as a function of $c$.
    \item \textbf{Covering entropy.} Bound the covering entropy of $\Lscr_c$, which is generally controlled by that of $\overline{\Fcal}_n$. Meanwhile, recall that by Proposition \ref{prop: covering number of extended parameter space}, the covering entropy of $\overline{\Fcal}_n$ typically differs from that of $\Fcal_n$ only by a logarithmic factor.
\end{enumerate}

Once these components are determined, the bound $\phi_n$ follows directly from the newly developed maximal inequalities, with some basic calculus. To further streamline this process, Corollaries \ref{proposition: convergence of EP with L^1 integrable functions} and \ref{proposition: convergence of EP with L^infty integrable functions} provide ready-to-use expressions in terms of general entropy conditions.

Finally, denote $\tilde{n}=n/D_\Fcal$, and the estimation error bound in Equation \eqref{eq: decomposition of estimation error, thm: Estimation Error of Sieved M-estimators} is obtained by solving the inequality $\wfrak_n^{-1} \circ \phi_n(c, \delta_n) \lesssim c$.

\begin{proposition}\label{proposition: convergence of EP with L^1 integrable functions}
    Under the same assumption of Theorem \ref{theorem: convergence of EP with L^1 integrable functions}, suppose in addition that
    \begin{enumerate}
        \item $\sigma\leq2^{-2}$ is small enough.
        \item The envelope function $F\in L^m(P)$ for some $m> 1$.
        \item There are some $D_\Fcal,U_\Fcal\geq 1$, $\gamma\geq 0$ and $\gamma^\prime\geq 1$, such that for any $h>0$, $$\EE[\log\Ncal(h,\Fcal,L^{1+\kappa}(\PP_n))]\leq D_\Fcal\cdot h^{-\gamma}\cdot\log_+(U_\Fcal/h)^{\gamma^\prime}.$$
    \end{enumerate}
When $\gamma\geq 1+\kappa$, we require $\bigl(\frac{M^{1-\kappa}}{\tilde{n}}\bigr)^{\frac{1}{\gamma+1-\kappa}}\leq \sigma/8$. Then, for some suppressed constant depending on $\kappa$, $\gamma$ and $\gamma^\prime$, we have
    $$\begin{aligned}
        &\EE^\ast\|\PP_n-P\|_\Fcal\\
        \lesssim&
        \left\{\begin{aligned}
            & \tilde{n}^{-\frac{1}{2}}\sigma^{1-\frac{\gamma}{2}}\log_+(2U_\Fcal/\sigma)^{\frac{\gamma^\prime}{2}} + \|F\|_{L^m(P)}\cdot \tilde{n}^{-\frac{m-1}{m}}\sigma^{-\gamma\frac{m-1}{m}}\log_+(2U_\Fcal/\sigma)^{\gamma^\prime\frac{m-1}{m}}\\
            &\qquad \text{when}\quad \kappa=1,\gamma\in[0,2);\\
            & \tilde{n}^{-\frac{1}{\gamma}}\log_+(U_\Fcal\cdot \tilde{n}^{\frac{\gamma}{2}})^{\frac{\gamma^\prime}{2}}+ \|F\|_{L^m(P)}\cdot \tilde{n}^{-\frac{m-1}{m}}\sigma^{-\gamma\frac{m-1}{m}}\log_+(2U_\Fcal/\sigma)^{\gamma^\prime\frac{m-1}{m}}\\
            &\qquad \text{when}\quad \kappa=1,\gamma>2;\\
            & \sqrt\frac{M^{1-\kappa}}{\tilde{n}}\sigma^{\frac{\kappa+1-\gamma}{2}}\log_+(2U_\Fcal/\sigma)^{\frac{\gamma^\prime}{2}} + \frac{ M}{\tilde{n}} \sigma^{-\gamma}\cdot \log_+(2U_\Fcal/\sigma)^{\gamma^\prime} + \frac{\|F\|_{L^m(P)}^m}{M^{m-1}}\\
            & \qquad \text{when}\quad \kappa<1, \gamma\in[0,1+\kappa);\\
            & \bigl(\frac{M^{1-\kappa}}{\tilde{n}}\bigr)^{\frac{1}{\gamma+1-\kappa}} \log_+\Bigl(U_\Fcal\cdot \bigl(\frac{\tilde{n}}{M^{1-\kappa}}\bigr)^{\frac{1}{\gamma+1-\kappa}}\Bigr)^{\frac{\gamma^\prime}{2}} + \frac{ M}{\tilde{n}} \sigma^{-\gamma}\cdot \log_+(2U_\Fcal/\sigma)^{\gamma^\prime} + \frac{\|F\|_{L^m(P)}^m}{M^{m-1}}\\
            & \qquad \text{when}\quad \kappa<1, \gamma> 1+\kappa.\\
        \end{aligned}\right.
    \end{aligned}$$
\end{proposition}

\begin{proposition}\label{proposition: convergence of EP with L^infty integrable functions}
    Under the same assumption of Theorem \ref{theorem: convergence of EP with L^infty integrable functions}, suppose in addition that there are some $D_\Fcal,U_\Fcal\geq 1$, $\gamma\geq 0$ and $\gamma^\prime\geq 1$, such that for any $h>0$, $$\log\Ncal(h,\Fcal,L^{\infty}(w))\leq D_\Fcal\cdot h^{-\gamma}\cdot\log_+(U_\Fcal/h)^{\gamma^\prime}.$$
    Assume $n\ge D_\Fcal$ and $\|F\|_{L^\infty(w)}\le R_0$ for a fixed $R_0$. Take the radius bound $\sigma\le \|1/w\|_{L^m(P)}\|F\|_{L^\infty(w)}$, as is always possible. Write $\tilde n=n/D_\Fcal$ and, for $\gamma>0$, set $K_n=(2+\log\tilde n)^{1+\gamma^\prime/\gamma}$.
    Then for some suppressed constant depending on $m$, $\gamma$, $\gamma^\prime$, $R_0$ and $\|1/w\|_{L^m(P)}$, we have
    $$\begin{aligned}
        &\Bigl\|\|\PP_n-P\|_{\Fcal}\Bigr\|_{L^m(P)}\\
        \lesssim&
        \left\{\begin{aligned}
            & \tilde{n}^{-\frac{1}{2}}\sigma^{1-\frac{\gamma}{2}}\log_+\Bigl(\tfrac{U_\Fcal \|1/w\|_{L^m(P)}}{2\sigma}\Bigr)^{\frac{\gamma^\prime}{2}} + \tilde{n}^{\frac{1}{m}-1}\|F\|_{L^\infty(w)}^{1-(1-\frac{1}{m})\gamma}\cdot \log_+\Bigl(\tfrac{U_\Fcal}{2\|F\|_{L^\infty(w)}}\Bigr)^{(1-\frac{1}{m})\gamma^\prime}\\
            &\qquad \text{when}\quad m\geq 2, \gamma(1-\frac{1}{m})<1 \\
            & \tilde{n}^{-\frac{1}{2}}\sigma^{1-\frac{\gamma}{2}}\log_+\Bigl(\frac{U_\Fcal \|1/w\|_{L^m(P)}}{2\sigma}\Bigr)^{\frac{\gamma^\prime}{2}} + K_n\tilde{n}^{-\frac{1}{\gamma}} \log_+\Bigl(\frac{U_\Fcal}{2\|F\|_{L^\infty(w)}}\Bigr)^{\frac{\gamma^\prime}{\gamma}} \\
            &\qquad \text{when}\quad m\geq 2, \gamma(1-\frac{1}{m})\geq1, \gamma<2 \\
            & \tilde{n}^{-\frac{\kappa}{1+\kappa}}\|F\|_{L^\infty(w)}^{1-\frac{\kappa}{1+\kappa}\gamma}\log_+\Bigl(\frac{U_\Fcal}{2\|F\|_{L^\infty}(w)}\Bigr)^{\frac{\kappa}{1+\kappa}\gamma^\prime} \qquad\text{when}\quad m\in(1,2), \frac{\kappa}{1+\kappa}\gamma<1\\
            & K_n\tilde{n}^{-\frac{1}{\gamma}}\log_+\Bigl(\frac{U_\Fcal \|1/w\|_{L^m(P)}}{2\sigma}\Bigr)^{\frac{\gamma^\prime}{\gamma}}\qquad\text{when}\quad m\geq 2, \gamma\geq 2\\
            & K_n\tilde{n}^{-\frac{1}{\gamma}} \log_+\Bigl(\frac{U_\Fcal}{2\|F\|_{L^\infty}(w)}\Bigr)^{\frac{\gamma^\prime}{\gamma}} \qquad\text{when}\quad m\in(1,2), \frac{\kappa}{1+\kappa}\gamma\geq 1.
        \end{aligned}\right.
    \end{aligned}$$
\end{proposition}

\subsection{Different Modes of Convergence Rates}\label{sec: Different Modes of Convergence Rates}

In Theorem \ref{thm: Estimation Error of Sieved M-estimators}, we have established a non-asymptotic convergence rate for $\hat\theta_n$, based on a high-probability bound $\phi_n(\cdot,\cdot)$ for the empirical process, as specified in Equation \eqref{eq: EP assumption, thm: Estimation Error of Sieved M-estimators}. However, the new maximal inequalities primarily provide upper bounds on the expected empirical process.

In many scenarios, obtaining such a rate typically requires stronger assumptions and significantly more technical effort to derive the high-probability bound $\phi_n$. Besides, a weaker guarantee may suffice for both subsequent theoretical analysis and downstream applications. Such results are also more accessible under milder conditions and are often easier to analyze.

In this section, we bridge the gap between non-asymptotic requirements and the expected empirical process upper bound by developing two alternative modes of convergence: an asymptotic bounded-in-probability result and a convergence-in-expectation result.

\subsubsection{Bounded in Probability Rate}\label{sec: Bounded in Probability Rate}

Suppose we have established the following bound on the expected empirical process: for any $c > 0$,
\begin{equation}\label{eq: expected empirical process control in sec: Convergence Rates of Sieved M-estimators}
    \EE^\ast\Bigl[\sup_{\theta\in\Theta_n:\|\theta-\theta_n^\ast\|\leq c}\Bigl|L(\theta)-L_n(\theta)-L(\theta_n^\ast)+L_n(\theta_n^\ast)\Bigr|\Bigr]\leq \phi_n(c).
\end{equation}
Then, by Markov’s inequality, for any $\delta \in (0,1)$, with probability at least $1 - \delta$, Equation \eqref{eq: EP assumption, thm: Estimation Error of Sieved M-estimators} holds with
$$\sup_{\theta\in\overline\Theta_n:\|\theta-\theta_n^\ast\|\leq c}\Bigl|L(\theta)-L_n(\theta)-L(\theta_n^\ast)+L_n(\theta_n^\ast)\Bigr|\leq \frac{\phi_n(c)}{\delta}\quad\text{ for any }c>0.$$
As a result, it follows that, for any $\delta\in(0,1)$,
$$\PP\Bigl(\|\hat\theta_n-\theta_0\|\leq \inf\bigl\{\tau>0:\wfrak_n^{-1}\bigl(\phi_n(\tau)/\delta\bigr)\leq\frac{\tau}{4}\bigr\}+9\wfrak_n^{-1}\bigl(L(\theta_n^\ast)-L(\theta_0)\bigr)\Bigr)\geq 1-\delta.$$

If $\wfrak_n(\cdot)$ is a polynomial function, we can further simplify it using homogeneity of $\wfrak_n^{-1}$: $$\wfrak_n^{-1}({\phi_n(\tau)}/{\delta})=\frac{\wfrak_n^{-1}\circ\phi_n(\tau)}{\wfrak_n^{-1}(\delta)}.$$
As $\tau\mapsto \wfrak_n^{-1}\circ\phi_n(\tau)$ is sub-linear, this inequality may imply the existence of a function $C(\delta)$, such that with probability at least $1-\delta$,
$$\begin{aligned}
    \|\hat\theta_n-\theta_0\|
    \leq& \inf\Bigl\{\tau>0:\wfrak_n^{-1}\circ\phi_n(\tau,\delta_n)\leq\frac{\tau}{4}\Bigr\}+9\wfrak_n^{-1}\bigl(L(\theta_n^\ast)-L(\theta_0)\bigr)\\
    \leq& C(\delta)\cdot\inf\Bigl\{\tau>0: \wfrak_n^{-1}\circ\phi_n(\tau)\leq\tau\Bigr\}+9\wfrak_n^{-1}\bigl(L(\theta_n^\ast)-L(\theta_0)\bigr).\\
\end{aligned}$$

This establishes the bounded-in-probability convergence rate, under the expected empirical process control condition (cf. Equation \eqref{eq: expected empirical process control in sec: Convergence Rates of Sieved M-estimators}):
\begin{equation}\label{eq: asymptotic conv rate one-shot localization}
    \|\hat\theta_n-\theta_0\|=\Ocal_{\PP}\Bigl(\inf\Bigl\{\tau>0: \wfrak_n^{-1}\circ\phi_n(\tau)\leq{\tau}\Bigr\}\Bigr)+\wfrak_n^{-1}\bigl(L(\theta_n^\ast)-L(\theta_0)\bigr).
\end{equation}

\subsubsection{Convergence in Expectation Rate}

For simplicity, we focus on empirical risk minimizers (ERMs) and adopt the notations from Section \ref{sec: Application to ERMs with Empirical Process Theory}. Lemma \ref{lemma: moment bound of empirical process to its expectation} states that if the envelope function $LF_c \in L^m(P)$ for some $m > 1$, then for any $\delta \in(0,1)$, we have
$$\begin{aligned}
    &\PP\Bigl(\|\PP_n-P\|_{\Lscr_c}\gtrsim \delta^{-\frac{1}{m}}\cdot \underbrace{\bigl(\EE^\ast\|\PP_n-P\|_{\Lscr_c}+n^{\frac{1}{m}-1}\|LF_c\|_{L^m(P)}\bigr)}_{=:\phi_n(c)}\Bigr)\\
    \leq& \PP\Bigl(\|\PP_n-P\|_{\Lscr_c}\geq \frac{\|\|\PP_n-P\|_{\Lscr_c}\|_{L^m(P)}}{\delta^{1/m}}\Bigr)\leq\delta.
\end{aligned}$$

Then, by applying Theorem \ref{thm: Estimation Error of Sieved M-estimators}, the convergence rate of the ERM, denoted by $\sigma_n(\delta)$, is obtained by solving
$\wfrak_n^{-1} \Bigl(\delta^{-\frac{1}{m}}\cdot \phi_n(\tau)\Bigr) \leq \frac{\tau}{4}.$ Taking $\wfrak_n(x) = x^2$, we expect the solution to have the form $\sigma_n(\delta) \asymp \delta^{-a} \tilde{n}^{-b}$ for some $a < 1$ and $b > 0$. Therefore, for any $\delta \in (0,1)$,
$$\PP\Bigl(\|\hat{f}_n-f_0\|\gtrsim \delta^{-a}\tilde{n}^{-b}+\inf_{f\in\Fcal}\|f-f_0\|\Bigr)\leq \delta.$$
Then, Lemma \ref{lemma: connection between high-probability bound and convergence in mean rate} yields the convergence-in-expectation result:
$$\EE\|\hat{f}_n-f_0\|\lesssim \tilde{n}^{-b}+\inf_{f\in\Fcal}\|f-f_0\|.$$

We have seen that the statistical error in this convergence-in-expectation rate arises from solving the inequality $\wfrak_n^{-1} \circ \phi_n(\tau, \delta) \lesssim \tau$. For convenience, the following result expresses the rate exponent $b$ using the weighted $L^\infty(w)$ covering entropy:

\begin{proposition}\label{prop: convergence rate expression in terms of weighted covering entropy}
    Under the same assumption of Proposition \ref{proposition: convergence of EP with L^infty integrable functions}, suppose in addition that for some $s\in[0,1]$, $\|F\|_{L^{\infty}(w)}\lesssim \sigma^s$. Let $\phi_n(\sigma)$ be the upper bound of $\|\|\PP_n-P\|_\Fcal\|_{L^m(P)}$ given in Proposition \ref{proposition: convergence of EP with L^infty integrable functions}, and define $\sigma_n(\delta)=\inf\{\sigma>0: \delta^{-\frac{1}{m}}\cdot \phi_n(\sigma)\leq \sigma^2\}$.
    Then, $$\sigma_n(\delta)\lesssim_{\log n}\delta^{-\frac{1}{m}}\cdot\Bigl(\tilde{n}^{-(\frac{1}{2+\gamma}\land\frac{1}{2\gamma})} + \tilde{n}^{-\frac{1}{\frac{2-s}{1-1/m}+s\gamma}}\Bigr).$$
\end{proposition}

\begin{remark}
    The upper bound of $\sigma_n(\delta)$ comprises two components: $$\tilde{n}^{-(\frac{1}{2+\gamma}\land\frac{1}{2\gamma})} \qquad\text{and}\qquad \tilde{n}^{-\frac{1}{\frac{2-s}{1-1/m}+s\gamma}}.$$
    When $m\to\infty$, corresponding to the case where $P$ is light-tailed, the first term dominates. This term characterizes the phase transition between Donsker and non-Donsker function classes.
    In contrast, when $m$ is small, the second term may dominate. This captures the influence of heavy-tailed distributions on empirical processes and the resulting convergence rate of ERMs.

    Importantly, the introduction of the parameter $s$ serves to mitigate the adverse effects of heavy-tailedness by incorporating smoothness information about the envelope function.
\end{remark}

\begin{lemma}\label{lemma: moment bound of empirical process to its expectation}
    Let $\Fcal$ be a class of measurable functions with measurable envelope function $F$, $\PP_n$ be an empirical process made up of i.i.d. observations from probability measure $P$. Then, for some suppressed constant depending on $m$, $$\Bigl\|\|\PP_n-P\|_{\Fcal}\Bigr\|_{L^m(P)}\lesssim \EE^\ast\|\PP_n-P\|_\Fcal+n^{\frac{1}{m}-1}\|F\|_{L^m(P)},\quad\text{for all }m\geq 1.$$
\end{lemma}
\begin{proof}
    This is essentially Theorem 2.14.23 in \cite{vaart2023empirical}. Although it only states the results for $m\geq 2$, it holds for $m\geq 1$ as shown in its proof.
\end{proof}

\begin{lemma}\label{lemma: connection between high-probability bound and convergence in mean rate}
    Let $X$ be a non-negative random variable. Suppose that for some constants $A,B\geq 0$ and $a\in(0,1)$, the following holds for all $\delta\in(0,1)$:
    $$\PP(X\geq A\delta^{-a}+B)\leq \delta.$$ Then, $\EE[X]\leq \frac{1}{1-a}A+B$.
\end{lemma}
\begin{proof}
    Using the tail bound,
    $$\begin{aligned}
        \EE[X]\leq& \int_0^\infty \PP(X\geq x)\, \dd x \leq A+B+\int_{A+B}^\infty \PP(X\geq x)\, \dd x\\
        =& A+B+\int_0^1 \PP(X\geq A\delta^{-a}+B)\cdot a A \delta^{-a-1}\, \dd\delta\\
        \leq& A+B+\int_0^1 \delta \cdot a A \delta^{-a-1}\, \dd\delta\\
        =& \frac{1}{1-a}A+B.
    \end{aligned}$$
\end{proof}

{
\section{Discussion of Theorem~\ref{theorem: convergence of EP with L^1 integrable functions} and Corollary~\ref{corollary: convergence of EP with L^1 integrable functions}}\label{sec: Discussion of Theorem~theorem: convergence of EP with L^1 integrable functions and Corollary~corollary: convergence of EP with L^1 integrable functions}

This section provides additional discussion of Theorem~\ref{theorem: convergence of EP with L^1 integrable functions} and Corollary~\ref{corollary: convergence of EP with L^1 integrable functions}.
Section~\ref{sec: Relationship between Corollary~corollary: convergence of EP with L^1 integrable functions and Classical L2 Maximal Inequalities} presents a direct comparison between Corollary~\ref{corollary: convergence of EP with L^1 integrable functions}, the $\kappa=1$ specialization of Theorem~\ref{theorem: convergence of EP with L^1 integrable functions}, and the classical $L^2$ maximal inequality employed in~\cite{fan2024noise} (their Lemma~A.8).
Section~\ref{sec: Explanations to the Two Terms in Corollary~corollary: convergence of EP with L^1 integrable functions} discusses the roles of the two terms in Corollary~\ref{corollary: convergence of EP with L^1 integrable functions}.
Section~\ref{sec: Major Technical Differences from the Classical L2 Approach} summarizes the major technical differences between the proof of Theorem~\ref{theorem: convergence of EP with L^1 integrable functions} and the classical approach.
Finally, Section~\ref{sec: Extension of Theorem~theorem: convergence of EP with L^1 integrable functions to the Analysis of Weak Conditional Fréchet Mean} outlines how to combine Theorem~\ref{theorem: convergence of EP with L^1 integrable functions} with Proposition~\ref{proposition: average covering entropy domminated for parametrized class} to establish convergence rates for weak conditional Fréchet mean estimators in a general metric space.

\subsection{Relationship between Corollary~\ref{corollary: convergence of EP with L^1 integrable functions} and Classical \texorpdfstring{$L^{2}$}{L2} Maximal Inequalities}\label{sec: Relationship between Corollary~corollary: convergence of EP with L^1 integrable functions and Classical L2 Maximal Inequalities}

Theorem~\ref{theorem: convergence of EP with L^1 integrable functions}, along with Corollary~\ref{corollary: convergence of EP with L^1 integrable functions}, broadens the scope of existing maximal inequalities and provides a more user-friendly tool for evaluating complex models like Deep Neural Networks (DNNs). For comparison, the maximal inequality used in \cite{fan2024noise} (see their Lemma~A.8, rooted in Theorem 5.2 of \cite{chernozhukov2014Gaussian}) can be reformulated in our notation as
\[
    \EE^\ast\|\PP_n-P\|_{\Fcal}\lesssim \frac{1}{\sqrt{n}} J(\sigma, \Fcal, F) + \frac{\|\overline{F}\|_{L^2(P)}}{\sigma^2\, n} J(\sigma, \Fcal, F)^2,
\]
where $\overline{F}=\|F\|_{L^\infty(\PP_n)}$ and
\[
    J(\sigma, \Fcal, F)=\int_0^{r} \sup_{Q_n} \sqrt{1+\log\Ncal(x, \Fcal, L^2(Q_n))}\, \dd x,
\]
with the supremum taken over all $n$-discrete probability measures.

A direct comparison reveals that our framework introduces several fundamental distinctions in how complexity and tail behavior are addressed:
\begin{itemize}
    \item The classical $L^2$ maximal inequality is stated in terms of uniform covering entropy, which is a stronger requirement than the expected covering entropy utilized in Corollary~\ref{corollary: convergence of EP with L^1 integrable functions}. The latter is also convenient for dealing with DNN classes: as their complexity is often quantified via pseudo-dimension~\cite{karpinski1997polynomial, bartlett2019nearly}, which naturally yields expected random covering entropy controls~\cite{anthony2009neural}.
    \item Lemma~A.8 in \cite{fan2024noise} involves the factor $\|\overline{F}\|_{L^2(P)}$, whose evaluation can be cumbersome. In contrast, the second term in Corollary~\ref{corollary: convergence of EP with L^1 integrable functions} captures tail thickness in a manner that is typically more straightforward to compute.
    \item Our result does not require the envelope $F$ to be square-integrable and accommodates non-Donsker classes. These features are crucial in our analysis of robust deep ReLU estimators in Section~\ref{sec: Robustness of Deep Huber Regression} and instrumental in establishing a rate bound for the set-structured estimators discussed in Section~\ref{sec: Regression with Set-Structured Function Classes}.
\end{itemize}

\subsection{Explanations to the Two Terms in Corollary~\ref{corollary: convergence of EP with L^1 integrable functions}}\label{sec: Explanations to the Two Terms in Corollary~corollary: convergence of EP with L^1 integrable functions}

The two terms in Corollary~\ref{corollary: convergence of EP with L^1 integrable functions} reflect the complexity of $\Fcal$ and the tail behavior of the underlying distribution:

\begin{itemize}
    \item As a complexity measure, the first term involves the covering integral
    \[
        \int_{\epsilon/8}^{\sigma/2}
        \sqrt{\EE[\log(2\Ncal(x,\Fcal,L^2(\PP_n)))]}\,\dd x.
    \]
    Our result extends the results of Theorem~2.1(1) in \cite{han2021set}, established under a uniform boundedness assumption on $\Fcal$, to unbounded function classes.     Rooted in the $\epsilon$-separation technique \cite{bousquet2002concentration, von2004distance}, the first term also includes an optimization over $\epsilon$ to handle non-Donsker classes, where the covering integral may be non-integrable at the origin. If the entropy integral is finite at zero, one may let $\epsilon\downarrow0$; otherwise, a positive cutoff keeps the integral away from its singularity. Our formulation thus integrates the classical Donsker regimes and non-Donsker, heavy-tailed settings.

    \item The optimization over $M$ in the second term characterizes the interplay between the complexity of $\Fcal$ through $\EE[\log(2\Ncal(\sigma/4,\Fcal,L^2(\PP_n)))]$ and the tail behavior encoded by the envelope $F$. In this corollary, $\kappa=1$, so only the second term depends on $M$. In Theorem~\ref{theorem: convergence of EP with L^1 integrable functions}, when $\kappa<1$, the factor $M^{(1-\kappa)/2}$ also multiplies the entropy-integral terms. The same $M$ must therefore be used throughout when optimizing the overall bound. This coupling reflects a fundamental trade-off. In light-tailed settings, the tail contribution $\EE[F\,\II(F>M)]$ decays exponentially fast as $M\to\infty$. Consequently, optimizing over $M$ typically inflates the complexity factor, $\frac{1}{n}\EE[\log(2\Ncal(\sigma/4,\Fcal,L^2(\PP_n)))]$, by at most a logarithmic factor. In contrast, under heavy-tailed settings, the tail contribution can decay much more slowly, and the optimized second term may dominate the first term. This highlights a key phase transition: sufficiently heavy tails can effectively override function class complexity, shifting the scale of the expected empirical process from a complexity-dominated regime to a tail-dominated one.
\end{itemize}

\subsection{Major Technical Differences from the Classical \texorpdfstring{$L^{2}$}{L2} Approach}\label{sec: Major Technical Differences from the Classical L2 Approach}

The weaker conditions considered in this work necessitate a proof strategy for Theorem~\ref{theorem: convergence of EP with L^1 integrable functions} that differs fundamentally from existing $L^2$-based approaches. For classical maximal inequalities with an $L^2$-integrable $\Fcal$, Dudley’s inequality~\cite{dudley1967sizes} is the primary tool for bounding the empirical process, for example, in the proof of Theorem 5.2 in~\cite{chernozhukov2014Gaussian} (used as Lemma A.8 in~\cite{fan2024noise}). The standard proof typically involves conditioning on the empirical measure $\mathbb{P}_n$ to apply Dudley’s inequality directly, followed by controlling the envelope moments using the standard Hoffmann--Jørgensen moment inequality~\cite{ledoux1991probability}.

More specifically, the proof from the classical $L^2$ approach can be viewed as two steps:
    \begin{enumerate}
        \item Conditioning on the empirical measure $\PP_n$, square-integrability of $\Fcal$ allows one to apply Dudley’s inequality directly to obtain the following control on the conditional expected empirical process:
        \begin{equation}\label{eq: classical Dudley's inequality}
            \EE^\ast[\|\PP_n-P\|_{\Fcal}|\PP_n]\lesssim \frac{1}{\sqrt{n}}\int_{0}^{\sigma_n} \sqrt{1+\log\Ncal(x,\Fcal, L^2(\PP_n))}, \dd x,
        \end{equation}
        where
        \[
            \sigma_n^2 = \sup_{f\in\Fcal} \|f\|_{L^2(\PP_n)}^2.
        \]

        \item One then controls $\sigma_n$ using the standard Hoffmann--Jørgensen moment inequality~\cite{ledoux1991probability}, which extends the classical Khinchin--Kahane moment inequality to maxima of collections of random variables.
    \end{enumerate}
However, this conditioning and the subsequent covering integral control are no longer available in our setting once $\Fcal$ is not $L^2(P)$-integrable. Moreover, although bracketing entropy is often used for maximal inequalities (see, e.g., Theorem 2.1 in~\cite{han2021set}), we instead utilize the expected covering entropy. This choice is more aligned with the complexity of DNNs.

As a result, Theorem~\ref{theorem: convergence of EP with L^1 integrable functions} serves as a Dudley-type inequality that unifies both the $L^2$ and non-$L^2$ regimes. It is established using standard high-dimensional probability tools to accommodate the unique tail behaviors and entropy requirements of non-Donsker function classes. Specifically, we have proved Theorem~\ref{theorem: convergence of EP with L^1 integrable functions} from scratch using high-dimensional probability tools, including a truncation argument, Hoeffding’s inequality, and a chaining construction, across the three steps summarized in Section~\ref{sec: Technical Proofs for Section sec: Upper Bounds of Expected Empirical Processes}.

\subsection{Extension of Theorem~\ref{theorem: convergence of EP with L^1 integrable functions} to the Analysis of Weak Conditional Fréchet Mean}\label{sec: Extension of Theorem~theorem: convergence of EP with L^1 integrable functions to the Analysis of Weak Conditional Fréchet Mean}

As this work primarily focuses on nonparametric function estimators in $L^{1+\kappa}$-normed linear spaces, Theorem~\ref{theorem: convergence of EP with L^1 integrable functions} uses expected $L^{1+\kappa}(\PP_n)$ covering entropy to quantify the complexity of the function class $\Fcal$. However, Theorem~\ref{theorem: convergence of EP with L^1 integrable functions} can also be used to analyze more general estimators. In particular, by leveraging Proposition~\ref{proposition: average covering entropy domminated for parametrized class} as a bridge, it can help establish convergence rates of global Fréchet estimators for random objects in a metric space~\cite{petersen2019frechet}.  The following points describe the scope of the application:
\begin{itemize}
    \item First, the tail condition concerns the random covariate $X$, whereas uniformity is over prediction points $x$ in a fixed ball.

    \item Second, our results accommodate both Donsker and non-Donsker function classes. By comparison, Assumptions (P1) and (U1) in \cite{petersen2019frechet} restrict the Fr\'echet regression function class to the Donsker regime.

    \item Third, while \cite{petersen2019frechet} targets Fr\'echet regression specifically, the empirical process tools developed in this work apply to a wider range of estimation problems and loss functions beyond Fr\'echet regression.
\end{itemize}

In the following, using Theorem~\ref{theorem: convergence of EP with L^1 integrable functions} and Proposition~\ref{proposition: average covering entropy domminated for parametrized class} as technical tools, we extend the convergence analysis of global Fréchet regression in~\cite{petersen2019frechet} from a technical perspective.

Before proceeding, we review the notation for the global Fréchet regression model. For a fixed $x\in \RR^p$, the target is the weak conditional Fréchet function $m_{\oplus}(x)\in\Omega$ as defined in \cite{petersen2019frechet,bhattacharjee2025nonlinear}, where $\Omega$ is a general metric space endowed with metric $d$, defined as a minimizer of the population risk
\[
    \omega\mapsto M(\omega, x):=\EE[s(X,x)d^2(Y,\omega)],
\]
where $s$ is a weight function. Specifically, letting $\mu\in\RR^p$ and $\Sigma\in \RR^{p\times p}$ denote the mean vector and covariance matrix of random vector $X$, respectively, the weight function $s$ is defined as
\[
    s(z,x) = 1+(z-\mu)^\top \Sigma^{-1}(x-\mu).
\]
Since the distribution of $(X,Y)$ is unknown and we only observe i.i.d.\ samples $\{(X_i, Y_i)\}_{i=1}^n$, the empirical risk is defined by
\[
    \omega\mapsto M_n(\omega, x):=\frac{1}{n}\sum_{i=1}^n \hat{s}_n(X_i,x)d^2(Y_i,\omega),
\]
where $\hat{s}_n$ is an empirical weight function used to estimate $s$, defined using the empirical mean $\hat\mu_n$ and empirical covariance $\hat\Sigma_n$:
\[
    \hat{s}_n(X_i,x) = 1+(X_i-\hat\mu_n)^\top \hat\Sigma_n^{-1}(x-\hat\mu_n).
\]
The global Fréchet estimator $\hat{m}_{\oplus}(x)$ is defined as an empirical risk minimizer:
\[
    \hat{m}_{\oplus}(x)\in\argmin_{\omega\in\Omega}\, M_n(\omega, x).
\]
We now utilize Theorem~\ref{theorem: convergence of EP with L^1 integrable functions} and Proposition~\ref{proposition: average covering entropy domminated for parametrized class} to establish a convergence rate for $\hat{m}_{\oplus}(x)$ in terms of
\[
    \sup_{\|x\|_2\leq B} d\bigl(\hat{m}_{\oplus}(x), m_{\oplus}(x)\bigr).
\]

\begin{theorem}\label{thm: convergence rate of global frechet regression}
    Fix $p$ and $B<\infty$, and suppose that $\Sigma$ is nonsingular. Assume the following conditions:
    \begin{itemize}
        \item Almost surely, for all $\|x\|_2\leq B$, the objects $m_\oplus(x)$ and $\hat{m}_\oplus(x)$ exist and are unique. Additionally, for any $\varepsilon>0$,
        \begin{equation}\label{eq: assumption identifiable in thm: convergence rate of global frechet regression}
            \inf_{\|x\|_2\leq B} \inf_{d(\omega,m_\oplus(x))>\varepsilon} \bigl[M(\omega, x) - M(m_\oplus(x), x)\bigr]>0,
        \end{equation}
        and there exists $\zeta = \zeta(\varepsilon)>0$ such that
        \begin{equation}\label{eq: assumption consistent in thm: convergence rate of global frechet regression}
            \PP\Bigl(\inf_{\|x\|_2\leq B} \inf_{d(\omega,m_\oplus(x))>\varepsilon} \bigl[M(\omega, x) - M(\hat{m}_\oplus(x), x)\bigr]\geq \zeta\Bigr) \to 1.
        \end{equation}
        \item There exists $\tau>0$, $D>0$, and $\alpha>1$, possibly depending on $B$, such that
        \begin{equation}\label{eq: assumption stable link in thm: convergence rate of global frechet regression}
            \inf_{\|x\|_2\leq B}\inf_{d(\omega,m_\oplus(x))<\tau}\Bigl(M(\omega,x)-M(m_\oplus(x), x) - Dd(\omega, m_\oplus(x))^\alpha\Bigr)\geq 0.
        \end{equation}
        \item Let $\Ncal(\epsilon, \Omega, d)$ be the covering number of $\Omega$ using balls of size $\epsilon$. For some $D_\Omega>1$ and $\gamma\geq 0$, $\Ncal(\epsilon, \Omega, d)$ satisfies that
        \begin{equation}\label{eq: assumption convering entropy in thm: convergence rate of global frechet regression}
            \log\Ncal(x, \Omega, d)\leq D_{\Omega}\, x^{-\gamma}.
        \end{equation}
        \item For $(X,Y)\sim P$, both $X$ and $d(Y,\omega_\oplus)$ follow sub-Weibull distributions, where $\omega_\oplus$ is any fixed point of $\Omega$.
    \end{itemize}
    Then, with $\tilde{n} = n/ D_{\Omega}$, the estimator $\hat{m}_{\oplus}(x)$ in the global Fréchet regression \cite{petersen2019frechet} satisfies:
    \[
        \sup_{\|x\|_2\leq B}d\bigl(\hat{m}_{\oplus}(x), m_{\oplus}(x)\bigr) \lesssim_{\log} \Ocal_{\PP}\Bigl(\tilde{n}^{-\frac{1}{(\gamma\lor 2)\alpha + (\gamma-2)_+}} + \tilde{n}^{-\frac{1}{\alpha+\gamma}}\Bigr).
    \]
\end{theorem}

In establishing this result, we adopt the  assumptions~\eqref{eq: assumption identifiable in thm: convergence rate of global frechet regression}, \eqref{eq: assumption consistent in thm: convergence rate of global frechet regression}, and \eqref{eq: assumption stable link in thm: convergence rate of global frechet regression} from \cite{petersen2019frechet} to align with the framework of global Fréchet regression. This ensures that standard $M$-estimation arguments (e.g., those detailed in Chapter 3.2 of~\cite{vaart2023empirical}) remain applicable.

We allow $X$ to have unbounded support under a sub-Weibull tail condition and establish the rate uniformly over prediction points $\|x\|_2\le B$. Furthermore, our framework accommodates both Donsker ($\gamma < 2$) and non-Donsker regimes ($\gamma \geq 2$). Thus, Theorem \ref{thm: convergence rate of global frechet regression}  broadens the applicability of the global Fréchet regression framework.

\begin{proof}[Proof of Theorem~\ref{thm: convergence rate of global frechet regression}]
    As in the proof of Theorem~2 in~\cite{petersen2019frechet}, assumptions~\eqref{eq: assumption identifiable in thm: convergence rate of global frechet regression}, \eqref{eq: assumption consistent in thm: convergence rate of global frechet regression}, and \eqref{eq: assumption stable link in thm: convergence rate of global frechet regression} ensure that standard $M$-estimation arguments (e.g., Chapter 3.2 of~\cite{vaart2023empirical}) apply. Consequently, to establish the convergence rate of $\hat{m}_{\oplus}(x)$, it remains to derive an upper bound for the following empirical process.

    In particular, for $\delta>0$ and each $\|x\|_2\le B$, write $\omega_\oplus=m_\oplus(x)$ and let $B_\delta(m_\oplus(x))\subseteq \Omega$ denote the ball of radius $\delta$ centered at $m_\oplus(x)$:
    \[\begin{aligned}
        &\sup_{\omega\in B_\delta(m_\oplus(x))}|M(\omega, x) - M(\omega_\oplus, x) - M_n(\omega, x) + M_n(\omega_\oplus, x)|\\
        =&\sup_{\omega\in B_\delta(m_\oplus(x))}\Bigl|Ps(X,x)\bigl(d^2(Y,\omega) - d^2(Y,\omega_\oplus)\bigr) - \PP_n\hat{s}_n(X,x)\bigl(d^2(Y,\omega) - d^2(Y,\omega_\oplus)\bigr)\Bigr|\\
        \leq& \sup_{\omega\in B_\delta(m_\oplus(x))}\Bigl|(P-\PP_n)s(X,x)\bigl(d^2(Y,\omega) - d^2(Y,\omega_\oplus)\bigr)\Bigr| \\
        &\qquad + \sup_{\omega\in B_\delta(m_\oplus(x))}\Bigl|\PP_n(s(X,x) - \hat{s}_n(X,x))\bigl(d^2(Y,\omega) - d^2(Y,\omega_\oplus)\bigr) \Bigr|,\\
    \end{aligned}\]
    where the second term can be bounded by
    \[\begin{aligned}
        &\sup_{\omega\in B_\delta(m_\oplus(x))}\Bigl|\PP_n(s(X,x) - \hat{s}_n(X,x))\bigl(d^2(Y,\omega) - d^2(Y,\omega_\oplus)\bigr) \Bigr|\\
        \leq& \max_{i\in[n]}\sup_{\omega\in B_\delta(m_\oplus(x))}\Bigl|(s(X_i,x) - \hat{s}_n(X_i,x))\bigl(d^2(Y_i,\omega) - d^2(Y_i,\omega_\oplus)\bigr) \Bigr|\\
        &\le 2\operatorname{diam}(\Omega)\,\delta
          \max_{i\in[n]}\sup_{\|x'\|_2\le B}|s(X_i,x')-\hat s_n(X_i,x')|.
    \end{aligned}\]
    Condition~\eqref{eq: assumption convering entropy in thm: convergence rate of global frechet regression} implies that $\Omega$ is totally bounded, so $\operatorname{diam}(\Omega)<\infty$.
    Expanding the affine weights on $\|x'\|_2\le B$, the maximum in the last line is bounded, with probability tending to one, by a constant times
    \[
       (1+\max_{i\in[n]}\|X_i\|_2)
       \bigl(\|\hat\mu_n-\mu\|_2+\|\hat\Sigma_n^{-1}-\Sigma^{-1}\|_{\rm op}\bigr).
    \]
    The first factor is poly-logarithmic with high probability; the second is $\Ocal_\PP(n^{-1/2})$ since $p$ is fixed and $X$ has finite fourth moments.
    Therefore, up to poly-logarithmic factors, the second term is of order $\Ocal_{\PP}(\delta/\sqrt{n})$ uniformly in $\|x\|_2\le B$.

    For the first term, define the function class
    \[
        \Fcal_x:=\bigl\{(X,Y)\mapsto \underbrace{s(X,x)\bigl(d^2(Y,\omega) - d^2(Y,\omega_{\oplus})\bigr)}_{=:f_x(X,Y;\Omega)}: \omega\in B_{\delta}(m_\oplus(x))\bigr\}.
    \]
    Then it can be rewritten as a standard empirical process and controlled by Theorem~\ref{theorem: convergence of EP with L^1 integrable functions}:
    \[
        \sup_{\omega\in B_\delta(m_\oplus(x))}\Bigl|(P-\PP_n)s(X,x)\bigl(d^2(Y,\omega) - d^2(Y,\omega_\oplus)\bigr)\Bigr| = \|\PP_n-P\|_{\Fcal_x}.
    \]

To bound the first term simultaneously in $x$, write
$\Delta_{\omega,v}(Y)=d^2(Y,\omega)-d^2(Y,v)$ and take
\[
 \Fcal_\delta=\{h(X)\Delta_{\omega,v}(Y):d(\omega,v)\le\delta,\quad
 h(X)\in\{1,X^{(1)}-\mu^{(1)},\ldots,X^{(p)}-\mu^{(p)}\}\}.
\]
Expanding $s(X,x)$ in these $p+1$ coordinates gives
\[
 \sup_{\|x\|_2\le B}\|\PP_n-P\|_{\Fcal_x}
 \le \Big(1+\sup_{\|x\|_2\le B}\|\Sigma^{-1}(x-\mu)\|_1\Big)
       \|\PP_n-P\|_{\Fcal_\delta}.
\]
An envelope is $F_\delta=2\operatorname{diam}(\Omega)\delta(1+\|X-\mu\|_2)$.
Thus $\sup_{f\in\Fcal_\delta}\|f\|_{L^2(P)}\le C\delta=:\sigma$, with $C$ independent of $x$ and $\delta$.

For $0<\sigma<1/4$, Corollary~\ref{corollary: convergence of EP with L^1 integrable functions} gives
\[
\begin{aligned}
 \EE^\ast\|\PP_n-P\|_{\Fcal_\delta}
 &\lesssim \inf_{0<\epsilon<\sigma}\left\{\epsilon+
 \frac1{\sqrt n}\int_{\epsilon/8}^{\sigma/2}
 \sqrt{\EE\log\big(2\Ncal(u,\Fcal_\delta,L^2(\PP_n))\big)}\,\dd u\right\}\\
 &\quad+\frac M n\EE\log\big(2\Ncal(\sigma/4,\Fcal_\delta,L^2(\PP_n))\big)
 +\EE[F_\delta\II(F_\delta>M)].
\end{aligned}
\]
For $\delta\le1$, the sub-Weibull tail of $X$ makes the last term $O(n^{-2})$ for $M$ a sufficiently large power of $\log(en)$ times a constant.

For each coordinate weight $h$, the triangle inequality gives
\[
 |h(X)(\Delta_{\omega,v}-\Delta_{\omega',v'})|
 \le 2\operatorname{diam}(\Omega)(1+\|X-\mu\|_2)
       \{d(\omega,\omega')+d(v,v')\}.
\]
Equip the localized pairs $(\omega,v)$ with the sum metric. Two $u/4$-nets of $\Omega$, with one representative per occupied product cell, give at most $\Ncal(u/4,\Omega,d)^2$ balls of radius $u$.
The Lipschitz factor $G=2\operatorname{diam}(\Omega)(1+\|X-\mu\|_2)$ satisfies $\sup_n\EE\|G\|_{L^2(\PP_n)}^q\le P G^q<\infty$ for $q>\max\{2,\gamma\}$.
Thus Proposition~\ref{proposition: average covering entropy domminated for parametrized class} applies to each $h$ (its tail condition follows by Markov's inequality). Taking the union over the $p+1$ weights gives, for $0<u\le1$,
\[
 \EE\log\big(2\Ncal(u,\Fcal_\delta,L^2(\PP_n))\big)
 \le \log(2(p+1))+C D_\Omega u^{-\gamma}
 \lesssim D_\Omega u^{-\gamma}.
\]
    Consequently, letting $\tilde{n} = n/ D_{\Omega}$, we obtain
    \[\begin{aligned}
        &\EE^\ast \|\PP_n-P\|_{\Fcal_\delta}\\
        &\lesssim_{\log}\inf_{0<\epsilon<\sigma}\Bigl(\epsilon+\frac1{\sqrt{\tilde n}}\int_{\epsilon/8}^{C\delta/2}u^{-\gamma/2}\,\dd u\Bigr)+\frac{\delta^{-\gamma}}{\tilde n}.
    \end{aligned}\]
    For any fixed sufficiently small $\delta>0$ and sufficiently large $\tilde n$, letting $\epsilon\downarrow0$ when $\gamma<2$, $\epsilon = \tilde{n}^{-\frac{1}{\gamma}}$ when $\gamma\geq 2$ gives
    \[\begin{aligned}
        \EE^\ast \|\PP_n-P\|_{\Fcal_\delta}
        \lesssim_{\log}&
        \left\{\begin{aligned}
            &\frac{\delta^{1-\frac{\gamma}{2}}}{\sqrt{\tilde{n}}} + \frac{\delta^{-\gamma}}{\tilde{n}}\qquad & \gamma<2;\\
            & \tilde{n}^{-\frac{1}{\gamma}} + \frac{\delta^{-\gamma}}{\tilde{n}}\qquad & \gamma\geq 2.\\
        \end{aligned}\right.
    \end{aligned}\]
    Thus,
    \[\begin{aligned}
        &\sup_{\|x\|_2\le B}\,\sup_{\omega\in B_\delta(m_\oplus(x))}|M(\omega, x) - M(\omega_\oplus, x) - M_n(\omega, x) + M_n(\omega_\oplus, x)|\\
        &\lesssim_{\Ocal_\PP, \log}
        \left\{\begin{aligned}
            &\frac{\delta^{1-\frac{\gamma}{2}}}{\sqrt{\tilde{n}}} + \frac{\delta^{-\gamma}}{\tilde{n}} + \frac{\delta}{\sqrt{n}} \qquad & \gamma<2;\\
            & \tilde{n}^{-\frac{1}{\gamma}} + \frac{\delta^{-\gamma}}{\tilde{n}} + \frac{\delta}{\sqrt{n}}\qquad & \gamma\geq 2.\\
        \end{aligned}\right.\\
        &=:\phi_n(\delta).
    \end{aligned}\]

Fix $0<\delta<\min\{\tau,1,1/(4C)\}$. By~\eqref{eq: assumption consistent in thm: convergence rate of global frechet regression}, with probability tending to one all $\hat m_\oplus(x)$ lie in their respective $\delta$-balls (otherwise take $\omega=\hat m_\oplus(x)$ in that assumption). Uniform curvature and empirical minimality then give
\[
\begin{aligned}
 D\sup_{\|x\|_2\le B} d(\hat m_\oplus(x),m_\oplus(x))^\alpha
 &\le \sup_{\substack{\|x\|_2\le B\\ d(\omega,m_\oplus(x))\le\delta}}
 |(M-M_n)(\omega,x)-(M-M_n)(m_\oplus(x),x)|\\
 &\lesssim_{\Ocal_\PP,\log}\tilde n^{-1/(\gamma\lor2)} .
\end{aligned}
\]
Since $(\gamma\lor2)\alpha\le(\gamma\lor2)\alpha+(\gamma-2)_+$, this implies the stated bound:
    \[
        \sup_{\|x\|_2\leq B}d\bigl(\hat{m}_{\oplus}(x), m_{\oplus}(x)\bigr) \lesssim_{\log} \Ocal_{\PP}\Bigl(\tilde{n}^{-\frac{1}{(\gamma\lor 2)\alpha + (\gamma-2)_+}} + \tilde{n}^{-\frac{1}{\alpha+\gamma}}\Bigr).
    \]
\end{proof}

\section{Discussion of Theorem~\ref{thm: convergence rate of Huber regression}}\label{sec: Discussion of Theorem~thm: convergence rate of Huber regression}

This section provides additional discussion of Theorem~\ref{thm: convergence rate of Huber regression}. Section~\ref{sec: Incorporating optimization error} discusses how to incorporate optimization error into the estimation error bound established in Theorem~\ref{thm: convergence rate of Huber regression}. Section~\ref{sec: Technical Advancements in Theorem~thm: convergence rate of Huber regression} presents two technical difficulties we address in establishing a uniform sub-Gaussian concentration bound relative to~\cite{fan2024noise}.

\subsection{Incorporating optimization error}\label{sec: Incorporating optimization error}

In the practical training of deep neural networks, an algorithm typically returns an approximate minimizer rather than the exact global empirical risk minimizer $\hat f_n(\tau)$. Suppose $\tilde f_n\in\Fcal_n$ satisfies
\[
    \frac{1}{n}\sum_{i=1}^n \ell_\tau\bigl(Y_i-\tilde f_n(X_i)\bigr)
    \leq \inf_{f\in\Fcal_n}\frac{1}{n}\sum_{i=1}^n \ell_\tau\bigl(Y_i-f(X_i)\bigr)+\varepsilon_n,
    \qquad \varepsilon_n\geq0.
\]
{
Let $\Ecal_n(\delta)$ denote the right-hand side of the estimation bound of deep Huber regression estimator in \eqref{eq: main result in thm: convergence rate of Huber regression}. Proposition~\ref{prop: convergence rate of Huber estimator}, whose proof is provided in Section~\ref{sec: Properties of Huber Regression}, establishes that, for every $\delta\in(0,1)$,
\begin{equation}\label{eq: conv rate of approximate minimizer}
    \PP\left(\|\tilde f_n-f_0\|_{L^2(P)}\gtrsim \Ecal_n(\delta)+\sqrt{\varepsilon_n}\right)\leq\delta.
\end{equation}
Consequently, the optimization error is negligible relative to the statistical estimation error whenever $\varepsilon_n\ll \Ecal_n(\delta)^2$.
}

{
Notably, this extension does not require $\tilde{f}_n$ to be close to $\hat{f}_n(\tau)$ in the function space; it only requires their empirical risks to be close. This result is highly relevant for DNN estimators in practice, where over-parameterized models and modern gradient-based methods often achieve exceptionally small training losses ($\varepsilon_n\approx 0$), effectively making the optimization error negligible relative to the statistical estimation error.
}

{

While \eqref{eq: conv rate of approximate minimizer} extends Theorem~\ref{thm: convergence rate of Huber regression} to approximate empirical minimizers in Huber regression, the following theorem provides the corresponding extension for general sieved $M$-estimators and quantifies the additional error due to approximate optimization.

\begin{theorem}[Estimation error of approximate sieved $M$-estimators]\label{thm: Estimation Error of Sieved M-estimators with optimization error}
    Under the assumptions and with the notation of Theorem~\ref{thm: Estimation Error of Sieved M-estimators}, suppose instead that the sieved M-estimator is an approximate minimizer satisfying
    \[
        \inf_{\theta\in\Theta_n} L_n(\theta)\leq L_n(\hat\theta_n)\leq \inf_{\theta\in\Theta_n} L_n(\theta) + \varepsilon_n
    \]
    for some optimization error $\varepsilon_n\geq0$. Then, with probability at least $1-\delta_n$,
    \begin{equation}\label{eq: decomposition of estimation error, thm: Estimation Error of Sieved M-estimators with optimization error}
        \begin{aligned}
            \|\hat\theta_n-\theta_0\|
            \leq& \underbrace{\inf\Bigl\{\tau>0:\wfrak_n^{-1}\circ\phi_n(\tau,\delta_n)\leq\frac{\tau}{4}\Bigr\}}_\text{statistical error}\\
            &\quad +\underbrace{9\wfrak_n^{-1}\bigl(L(\theta_n^\ast)-L(\theta_0)\bigr)}_\text{approximation error} +\underbrace{4\wfrak_n^{-1}(\varepsilon_n)}_\text{optimization error}
        \end{aligned}
    \end{equation}
\end{theorem}

\begin{proof}    The proof follows that of Theorem~\ref{thm: Estimation Error of Sieved M-estimators}; we record the modifications induced by the optimization error. Since
    \[
        L_n(\hat\theta_n)\leq \inf_{\theta\in\Theta_n} L_n(\theta) + \varepsilon_n,
    \]
    and $\theta_n^\ast\in\Theta_n$, convexity of $L_n$ implies that, for every $t\in[0,1]$,
    \[
        L_n(\hat\theta_{n,t})\leq t\cdot L_n(\hat\theta_n)+(1-t)\cdot L_n(\theta_n^\ast)\leq L_n(\theta_n^\ast) + \varepsilon_n.
    \]
    Consequently, the oracle inequality in \eqref{eq: original oracle inequality in thm: Estimation Error of Sieved M-estimators} becomes
    \[\begin{aligned}
        L(\hat\theta_{n,t})-L(\theta_0)
        \leq& L(\hat\theta_{n,t})-L_n(\hat\theta_{n,t})+L_n(\theta_n^\ast)-L(\theta_0)+ \varepsilon_n\\
        \leq&\sup_{\theta\in\overline\Theta_n:\|\theta-\theta_n^\ast\|\leq \tau_n}\Bigl|L(\theta)-L_n(\theta)+L_n(\theta_n^\ast)-L(\theta_n^\ast)\Bigr|\\
        &\quad +L(\theta_n^\ast)-L(\theta_0)+ \varepsilon_n
    \end{aligned}\]
    Hence, with probability at least $1-\delta_n$,
    \[\begin{aligned}
        \|\hat\theta_{n,t}-\theta_0\|
        \leq&  \wfrak_n^{-1}\Bigl(\phi_n(\tau_n,\delta_n)+L(\theta_n^\ast)-L(\theta_0)+ \varepsilon_n\Bigr)\\
        \leq& \wfrak_n^{-1}\bigl(\phi_n(\tau_n,\delta_n)\bigr)+\wfrak_n^{-1}\bigl(L(\theta_n^\ast)-L(\theta_0)\bigr) + \wfrak_n^{-1}(\varepsilon_n).
    \end{aligned}\]
    The triangle inequality then gives
    \[\begin{aligned}
        \|\hat\theta_{n,t}-\theta_n^\ast\|
        \leq& \wfrak_n^{-1}\bigl(\phi_n(\tau_n,\delta_n)\bigr)+2\wfrak_n^{-1}\bigl(L(\theta_n^\ast)-L(\theta_0)\bigr) + \wfrak_n^{-1}(\varepsilon_n).
    \end{aligned}\]
    The remainder of the proof proceeds as in Theorem~\ref{thm: Estimation Error of Sieved M-estimators}, with $B$ replaced by
    \[
        B:=2\wfrak_n^{-1}\bigl(L(\theta_n^\ast)-L(\theta_0)\bigr) + \wfrak_n^{-1}(\varepsilon_n),
    \]
    which yields
    $$\begin{aligned}
        \|\hat\theta_n-\theta_n^\ast\|
        \leq& 4 B + \inf\{\tau>0:A(\tau)\leq\tau/4\} \\
        \leq& \inf\{\tau>0:\wfrak_n^{-1}\circ\phi_n(\tau,\delta_n)\leq\tau/4\}\\
        &\quad + 8\wfrak_n^{-1}\bigl(L(\theta_n^\ast)-L(\theta_0)\bigr) + 4 \wfrak_n^{-1}(\varepsilon_n).
    \end{aligned}$$
    Finally, the triangle inequality
    \[
        \|\hat\theta_n-\theta_0\| \leq \|\hat\theta_n-\theta_n^\ast\| + \|\theta_n^\ast- \theta_0\|
    \]
    together with $\|\theta_n^\ast- \theta_0\|\leq \wfrak_n^{-1}\bigl(L(\theta_n^\ast)-L(\theta_0)\bigr)$ proves the result.
\end{proof}

}

\subsection{Technical Advancements in Theorem~\ref{thm: convergence rate of Huber regression}}\label{sec: Technical Advancements in Theorem~thm: convergence rate of Huber regression}

As noted in {Section \ref{sec: Introduction}}, the methodology in Fan et al. (2024)~\cite{fan2024noise} relies on two technical constraints that prevent them from obtaining a unified sub-Gaussian concentration bound across all regimes of the Huber parameter $\tau$ and the noise moment. These limitations correspond directly to the methodological advancements introduced in this work to eliminate the polynomial deviation term:
\begin{itemize}
    \item The maximal inequality in Lemma~A.8 of~\cite{fan2024noise} is restricted for square-integrable classes and fails to cover regimes where the noise may have infinite variance. In contrast, our Theorem~\ref{theorem: convergence of EP with L^1 integrable functions} accommodates non-square-integrable function classes. Moreover, our bound is driven by the expected random covering entropy $\EE\bigl[\log \Ncal\bigl(x,\Fcal,L^{1+\kappa}(\PP_n)\bigr)\bigr]$, which is strictly weaker than the uniform entropy requirement $\sup_{Q_n}\log \Ncal\bigl(x,\Fcal,L^2(Q_n)\bigr)$ used in Lemma~A.8 of~\cite{fan2024noise}, where the supremum is taken over all $n$-discrete probability measures.

    \item Fan et al. (2024)~\cite{fan2024noise} utilize a technical device (their Lemma A.1) to convert empirical process bounds into estimation error bounds specifically for Huber regression. This specific step is where the additional polynomial deviation term is introduced into their analysis. In contrast, our proof exploits the convexity of the loss to establish a one-shot localization result (Theorem~\ref{thm: Estimation Error of Sieved M-estimators} in Section~\ref{sec: Convergence Rates of Sieved and Penalized M-estimators}). This approach directly yields an estimation error bound without producing polynomial tails, leading to a cleaner error decomposition that applies to a broader class of $M$-estimators beyond Huber regression. Crucially, under this reduction, the nonasymptotic deviation of the deep Huber estimator is governed by the deviation behavior of the associated empirical process term, which is sub-Gaussian due to the robust nature of the Huber loss.
\end{itemize}
}

\section{Proof of Results in Section \ref{sec: Convergence Rates of Sieved M-estimators}}

\subsection{Calculus Inequalities}

\begin{lemma}\label{lemma: integral of covering entropy with logarithm}
    For $c<e^{-1}$, $a\in[0,1)$ and $b\geq 0$,
    $$\int_0^c x^{-a}\log(1/x)^b\, \dd x\leq \inf_{\substack{k_1>1\text{ and }k_2>0: \\ k_1(1-a-k_2b)\geq 1-a \\ a+k_2 b<1}}\Bigl(\frac{1/k_2^b}{1-a-k_2b}+ \frac{k_1^b}{1-a}\Bigr)\cdot c^{1-a}\log(1/c)^b.$$

    Consequently, if $c<e^{-1} U$ for some $U>0$, $$\int_0^c x^{-a}\log(U/x)^b\, \dd x\leq \inf_{\substack{k_1>1\text{ and }k_2>0: \\ k_1(1-a-k_2b)\geq 1-a \\ a+k_2 b<1}}\Bigl(\frac{1/k_2^b}{1-a-k_2b}+ \frac{k_1^b}{1-a}\Bigr)\cdot c^{1-a}\log(U/c)^b.$$
\end{lemma}
\begin{proof}
    For any $k_1>1$ and $k_2>0$ such that $k_1(1-a-k_2b)\geq 1-a$ and $a+k_2 b<1$,
    $$\begin{aligned}
        \int_0^c x^{-a}\log(1/x)^b\, \dd x
        =& \int_0^{c^{k_1}} x^{-a}\log(1/x)^b\, \dd x + \int_{c^{k_1}}^c x^{-a}\log(1/x)^b\, \dd x\\
        =& \int_0^{c^{k_1}} x^{-a}\Bigl(\frac{1}{k_2}\log\frac{1}{x^{k_2}}\Bigr)^b\, \dd x + \int_{c^{k_1}}^c x^{-a}\log(1/x)^b\, \dd x\\
        \leq& \int_0^{c^{k_1}} x^{-a} \frac{1}{k_2^b}\cdot \frac{1}{x^{k_2\cdot b}}\, \dd x +\log(1/c^{k_1})^b \int_{0}^c x^{-a}\, \dd x\\
        =& \frac{1}{k_2^b}\cdot\frac{1}{1-a-k_2b}(c^{k_1})^{1-a-k_2b}+k_1^b \log(1/c)^b\frac{1}{1-a}c^{1-a}\\
        =& \frac{1/k_2^b}{1-a-k_2b}\cdot c^{k_1(1-a-k_2b)}+ \frac{k_1^b}{1-a}\cdot c^{1-a}\log(1/c)^b\\
        \leq& \Bigl(\frac{1/k_2^b}{1-a-k_2b}+ \frac{k_1^b}{1-a}\Bigr)\cdot c^{1-a}\log(1/c)^b\\
    \end{aligned}$$
\end{proof}

\begin{lemma}\label{lemma: convergence rate involving logarithm}
    For $a\geq0$ and $b,c,U>0$. Then for sufficiently large $n$, the inequality
    $$\delta_n^{-b}\cdot\log(U/\delta_n^c)^a\lesssim n$$
    is solved by
    $$\delta_n\gtrsim n^{-\frac{1}{b}}\cdot\log (n^c\cdot U^{b})^{a/b}.$$
\end{lemma}
\begin{proof}
    We begin by considering the simplified case $U = 1$. Set $$\delta_n=n^{-1/b}\cdot (\log n)^{a/b}.$$
    Then,
    $$\begin{aligned}
        \log (1/\delta_n)^a\cdot \delta_n ^{-b}
        =&\bigl(\frac{1}{b}\log n-\frac{a}{b}\log\log n\bigr)^a \cdot n \cdot (\log n)^{-a}\\
        =&\Bigl(\frac{1}{b}-\frac{a}{b}\frac{\log\log n}{\log n}\Bigr)^a\cdot n\\
        \leq& \frac{1}{b^a} n\quad\text{when }n\geq 3.
    \end{aligned}$$
    Thus, $\delta_n\gtrsim n^{-1/b}\cdot (\log n)^{a/b}$ solves the inequality $\log (1/\delta_n)^a\cdot \delta_n ^{-b}\lesssim n$ when $n\geq 3$.

    For general $U>0$, rewrite the inequality $\delta_n^{-b}\cdot\log(U/\delta_n^c)^a\lesssim n$ as
    $$\frac{n\cdot U^{b/c}}{c^a}\gtrsim (\delta_n/U^{1/c})^{-b}\cdot\log \bigl((\delta_n/U^{1/c})^{-1}\bigr)^{a}.$$
    By applying the previous result, we find, for large $n$, this inequality can be solved by $$\delta_n\gtrsim n^{-\frac{1}{b}}\cdot(\log (n\cdot U^{b/c}))^{a/b}\asymp n^{-\frac{1}{b}}\cdot\log (n^c\cdot U^{b})^{a/b}$$
\end{proof}

\begin{lemma}\label{lemma: minimal value between two terms}
    For any $A,B,a,b>0$, $$\inf_{M>0} A\cdot M^a+B/M^b\leq 2A^{\frac{b}{a+b}}\cdot B^{\frac{a}{a+b}}.$$
\end{lemma}
\begin{proof}
    Note that by taking $M=\Bigl(\frac{B\cdot b}{A\cdot a}\Bigr)^{\frac{1}{a+b}}$, we have $$A\cdot M^a+B/M^b\leq A^{\frac{b}{a+b}}\cdot B^{\frac{a}{a+b}}\cdot \Bigl((b/a)^{\frac{a}{a+b}}+(a/b)^{\frac{b}{a+b}}\Bigr)\leq 2A^{\frac{b}{a+b}}\cdot B^{\frac{a}{a+b}}.$$
\end{proof}

\subsection{Proof of Theorem \ref{thm: Estimation Error of Sieved M-estimators}}\label{sec: Proof of Theorem Estimation Error of Sieved M-estimators}

Let $\hat\theta_{n,t}=t\hat\theta_n+(1-t)\theta_n^\ast$, where $t=\frac{\tau_n}{\tau_n+\|\hat\theta_n-\theta_n^\ast\|}$ and $\tau_n>0$ is to be determined later. By the convexity of the empirical loss function $L_n$,
\begin{equation}\label{eq: result of convexity of empirical loss function, thm: Estimation Error of Sieved M-estimators}
    L_n(\hat\theta_{n,t})\leq t\cdot L_n(\hat\theta_n)+(1-t)\cdot L_n(\theta_n^\ast)\leq L_n(\theta_n^\ast).
\end{equation}
Since $\|\hat\theta_{n,t}-\theta_n^\ast\|=t\cdot\|\hat\theta_n-\theta_n^\ast\|\leq \tau_n$, we have:
\begin{equation}\label{eq: original oracle inequality in thm: Estimation Error of Sieved M-estimators}
    \begin{aligned}
        L(\hat\theta_{n,t})-L(\theta_0)
        =& L(\hat\theta_{n,t})-L_n(\hat\theta_{n,t})+L_n(\hat\theta_{n,t})-L(\theta_0)\\
        \leq& L(\hat\theta_{n,t})-L_n(\hat\theta_{n,t})+L_n(\theta_n^\ast)-L(\theta_0)\\
        =& L(\hat\theta_{n,t})-L_n(\hat\theta_{n,t})+L_n(\theta_n^\ast)-L(\theta_n^\ast)+L(\theta_n^\ast)-L(\theta_0)\\
        \leq&\sup_{\theta\in\overline\Theta_n:\|\theta-\theta_n^\ast\|\leq \tau_n}\Bigl|L(\theta)-L_n(\theta)+L_n(\theta_n^\ast)-L(\theta_n^\ast)\Bigr|+L(\theta_n^\ast)-L(\theta_0),
    \end{aligned}
\end{equation}
where the first inequality is due to Equation \eqref{eq: result of convexity of empirical loss function, thm: Estimation Error of Sieved M-estimators}.

Thus, with probability at least $1-\delta_n$, we obtain:
\begin{equation}\label{eq: upper bound of estimated parameter, first part, thm: Estimation Error of Sieved M-estimators}
    \begin{aligned}
        \|\hat\theta_{n,t}-\theta_0\|
        \leq& \wfrak_n^{-1}(L(\hat\theta_{n,t})-L(\theta_0))\leq \wfrak_n^{-1}\Bigl(\phi_n(\tau_n,\delta_n)+L(\theta_n^\ast)-L(\theta_0)\Bigr)\\
        \leq& \wfrak_n^{-1}\bigl(\phi_n(\tau_n,\delta_n)\bigr)+\wfrak_n^{-1}\bigl(L(\theta_n^\ast)-L(\theta_0)\bigr)\\
    \end{aligned}
\end{equation}
where the first inequality is due to the definition of stability link $\wfrak_n$ in Equation \eqref{eq: definition of stability link in thm: Estimation Error of Sieved M-estimators}, the second inequality is due to Equation \eqref{eq: original oracle inequality in thm: Estimation Error of Sieved M-estimators}, and the third inequality is due to the sub-additivity of $\wfrak_n^{-1}$, cf. Equation \eqref{eq: sub-additive of stability link in thm: Estimation Error of Sieved M-estimators}.

By triangle inequality, Equations \eqref{eq: upper bound of estimated parameter, first part, thm: Estimation Error of Sieved M-estimators} and \eqref{eq: definition of stability link in thm: Estimation Error of Sieved M-estimators}, we further have
$$\begin{aligned}
    \|\hat\theta_{n,t}-\theta_n^\ast\|
    \leq& \|\hat\theta_{n,t}-\theta_0\| + \|\theta_n^\ast-\theta_0\|\\
    \leq& \wfrak_n^{-1}\bigl(\phi_n(\tau_n,\delta_n)\bigr)+2\wfrak_n^{-1}\bigl(L(\theta_n^\ast)-L(\theta_0)\bigr)
\end{aligned}$$
Note that, if the right-hand side is smaller than $\tau_n/2$, then $$\frac{\tau_n\cdot\|\hat\theta_{n}-\theta_n^\ast\|}{\tau_n+\|\hat\theta_{n}-\theta_n^\ast\|}=\|\hat\theta_{n,t}-\theta_n^\ast\|\leq \frac{\tau_n}{2}\quad\text{implies}\quad \|\hat\theta_n-\theta_n^\ast\|\leq \tau_n.$$ Hence, $\tau_n$ represents the estimation error of $\hat\theta_n$.

Define:
$$A(\tau):=\wfrak_n^{-1} \circ \phi_n(\tau, \delta_n),\quad B:=2\wfrak_n^{-1}\bigl(L(\theta_n^\ast)-L(\theta_0)\bigr).$$
Then, with probability at least $1-\delta_n$, we deduce:
\begin{equation}\label{eq: upper bound of estimation error thm: Estimation Error of Sieved M-estimators}
    \|\hat\theta_n-\theta_n^\ast\|\leq \inf\Bigl\{\tau>0: A(\tau)+B\leq\frac{\tau}{2}\Bigr\}.
\end{equation}

Now, let us consider the following two scenarios. Let $\tau_1^\ast=4B\geq 0$.
\begin{enumerate}
    \item If $A(\tau_1^\ast)\leq B$, then $A(\tau_1^\ast) + B \leq \tau_1^\ast / 2$, so $\|\hat\theta_n - \theta_n^\ast\| \leq \tau_1^\ast$.

    \item Otherwise, suppose $A(\tau_1^\ast) > B = \tau_1^\ast / 4$. Let
        $$
        \tau_2^\ast := \inf\left\{\tau > 0 : A(\tau) \leq \tau / 4\right\}.
        $$
        By the assumption on $\wfrak_n^{-1}\circ\phi$, $\tau_2^\ast \geq \tau_0 \geq \tau_1^\ast$, and since $A(\cdot)$ is non-decreasing,
        $$
        B \leq A(\tau_1^\ast) \leq A(\tau_2^\ast) \leq \frac{\tau_2^\ast}{4}.
        $$
        Therefore, $A(\tau_2^\ast) + B \leq \tau_2^\ast / 2$, and so $\|\hat\theta_n - \theta_n^\ast\| \leq \tau_2^\ast$.
\end{enumerate}

In either case,
$$\begin{aligned}
    \|\hat\theta_n-\theta_n^\ast\|\leq& \tau_1^\ast+\tau_2^\ast\\
    \leq& 4 B + \inf\{\tau>0:A(\tau)\leq\tau/4\} \\
    \leq& \inf\{\tau>0:\wfrak_n^{-1}\circ\phi_n(\tau,\delta_n)\leq\tau/4\}+ 8\wfrak_n^{-1}\bigl(L(\theta_n^\ast)-L(\theta_0)\bigr).
\end{aligned}$$

\subsection{Proof of Theorem~\ref{thm: Estimation Error of penalized M-estimators}}
{
As in the proof of Theorem~\ref{thm: Estimation Error of Sieved M-estimators}, define
\[
    \hat\theta_{n,t}=t\hat\theta_n+(1-t)\theta_n^\ast,
\]
where $t=\frac{\tau_n}{\tau_n+\|\hat\theta_n-\theta_n^\ast\|}$ and $\tau_n>0$ is to be determined later.

Since $\hat\theta_n$ is a penalized empirical risk minimizer,
\[
    L_n(\hat\theta_n) + p_{\lambda_n}(\hat\theta_n)\leq L_n(\theta_n^\ast) + p_{\lambda_n}(\theta_n^\ast),
\]
equivalently,
\[
    L_n(\hat\theta_n) \leq L_n(\theta_n^\ast) + p_{\lambda_n}(\theta_n^\ast) - p_{\lambda_n}(\hat\theta_n).
\]
By convexity of $L_n$,
\[
    L_n(\hat\theta_{n,t})\leq t\cdot L_n(\hat\theta_n)+(1-t)\cdot L_n(\theta_n^\ast)\leq L_n(\theta_n^\ast) + t\cdot\Bigl(p_{\lambda_n}(\theta_n^\ast) - p_{\lambda_n}(\hat\theta_n)\Bigr).
\]
Moreover,
\[
    \|\hat\theta_{n,t}-\theta_n^\ast\|=t\cdot\|\hat\theta_n-\theta_n^\ast\|\leq \tau_n,
\]
Therefore,
\[\begin{aligned}
    &L(\hat\theta_{n,t})-L(\theta_0)\\
    =& L(\hat\theta_{n,t})-L_n(\hat\theta_{n,t})+L_n(\hat\theta_{n,t})-L(\theta_0)\\
    \leq& L(\hat\theta_{n,t})-L_n(\hat\theta_{n,t})+L_n(\theta_n^\ast)-L(\theta_0) + t\cdot\Bigl(p_{\lambda_n}(\theta_n^\ast) - p_{\lambda_n}(\hat\theta_n)\Bigr)\\
    =& L(\hat\theta_{n,t})-L_n(\hat\theta_{n,t})+L_n(\theta_n^\ast)-L(\theta_n^\ast)+L(\theta_n^\ast)-L(\theta_0)+ t\cdot\Bigl(p_{\lambda_n}(\theta_n^\ast) - p_{\lambda_n}(\hat\theta_n)\Bigr)\\
    \leq&\sup_{\theta\in\overline\Theta_n:\|\theta-\theta_n^\ast\|\leq \tau_n}\Bigl|L(\theta)-L_n(\theta)+L_n(\theta_n^\ast)-L(\theta_n^\ast)\Bigr|\\
    &\qquad\qquad +L(\theta_n^\ast)-L(\theta_0)+ \bigl|p_{\lambda_n}(\theta_n^\ast) - p_{\lambda_n}(\hat\theta_n)\bigr|.
\end{aligned}\]
Thus, with probability at least $1-\delta_n$, we obtain:
\[\begin{aligned}
    \|\hat\theta_{n,t}-\theta_0\|
    \leq& \wfrak_n^{-1}(L(\hat\theta_{n,t})-L(\theta_0))\\
    \leq& \wfrak_n^{-1}\bigl(\phi_n(\tau_n,\delta_n)\bigr)+\wfrak_n^{-1}\bigl(L(\theta_n^\ast)-L(\theta_0)\bigr) + \wfrak_n^{-1}\Bigl(\bigl|p_{\lambda_n}(\theta_n^\ast) - p_{\lambda_n}(\hat\theta_n)\bigr|\Bigr)\\
\end{aligned}\]

By the triangle inequality together with \eqref{eq: definition of stability link in thm: Estimation Error of Sieved M-estimators}, we further have
$$\begin{aligned}
    \|\hat\theta_{n,t}-\theta_n^\ast\|
    \leq& \|\hat\theta_{n,t}-\theta_0\| + \|\theta_n^\ast-\theta_0\|\\
    \leq& \wfrak_n^{-1}\bigl(\phi_n(\tau_n,\delta_n)\bigr)+2\wfrak_n^{-1}\bigl(L(\theta_n^\ast)-L(\theta_0)\bigr) + \wfrak_n^{-1}\Bigl(\bigl|p_{\lambda_n}(\theta_n^\ast) - p_{\lambda_n}(\hat\theta_n)\bigr|\Bigr).
\end{aligned}$$

The remainder of the proof follows the same argument as in Theorem~\ref{thm: Estimation Error of Sieved M-estimators}.
}

\subsection{Proof of Proposition \ref{prop: covering number of extended parameter space}}

We first consider the case that $\Theta_n$ is bounded. For any $h>0$, let $\{\theta_1,\cdots,\theta_N\}$ be an $h/2$-covering set of $\Theta_n$, $\{t_1,\cdots,t_M\}$ be an $h/(4R)$-covering set of $[0,1]$. Then for any $t\cdot\theta+(1-t)\theta_n^\ast\in\overline\Theta_n$, there are $\theta_i$ and $t_j$ such that $\vertiii{\theta_i-\theta}\leq h/2$ and $\vertiii{t_j-t}\leq h/(4R)$. Thus, for any $t\theta+(1-t)\theta_n^\ast\in\overline\Theta_n$ for some $t\in[0,1]$ and $\theta\in\Theta_n$,
$$\begin{aligned}
    &\vertiii{(t\cdot\theta+(1-t)\theta_n^\ast)-(t_j\cdot\theta_i+(1-t_j)\theta_n^\ast)}\\
    =& \vertiii{t\cdot(\theta-\theta_n^\ast)-t_j\cdot(\theta_i-\theta_n^\ast)}\\
    =& \vertiii{t\cdot(\theta-\theta_n^\ast)-t\cdot(\theta_i-\theta_n^\ast)+t\cdot(\theta_i-\theta_n^\ast)-t_j\cdot(\theta_i-\theta_n^\ast)}\\
    \leq& \vertiii{t\cdot(\theta-\theta_n^\ast)-t\cdot(\theta_i-\theta_n^\ast)}+\vertiii{t\cdot(\theta_i-\theta_n^\ast)-t_j\cdot(\theta_i-\theta_n^\ast)}\\
    =& t\cdot \vertiii{\theta-\theta_i}+|t-t_j|\cdot\vertiii{\theta_i-\theta_n^\ast}\\
    \leq& 1\cdot \frac{h}{2}+\frac{h}{4R}\cdot (2R)=h.
\end{aligned}$$
This result implies that $$\Ncal(h,\overline\Theta_n,\vertiii{\cdot})\leq \Ncal(h/2,\Theta_n,\vertiii{\cdot})\cdot\Ncal(h/(4R), [0,1], |\cdot|).$$

Similarly, when $\Theta_n$ is unbounded, let $\{\theta_1,\cdots,\theta_N\}$ be an $h/2$-covering set of $\Theta_{n,c}:=\{\theta\in\Theta_n:\vertiii{\theta-\theta_n^\ast}\leq c\}$, $\{t_1,\cdots,t_M\}$ be an $h/(2c)$-covering set of $[0,1]$. Thus, for any $t\theta+(1-t)\theta_n^\ast\in\overline\Theta_n$ for some $t\in[0,1]$ and $\theta\in\Theta_n$,
$$\begin{aligned}
    &\vertiii{(t\cdot\theta+(1-t)\theta_n^\ast)-(t_j\cdot\theta_i+(1-t_j)\theta_n^\ast)}\\
    \leq& t\cdot \vertiii{\theta-\theta_i}+|t-t_j|\cdot\vertiii{\theta_i-\theta_n^\ast}\\
    \leq& 1\cdot \frac{h}{2}+\frac{h}{2c}\cdot c=h,
\end{aligned}$$
which implies $$\Ncal(h,\overline\Theta_{n,c},\vertiii{\cdot})\leq \Ncal(h/2,\Theta_{n,c},\vertiii{\cdot})\cdot\Ncal(h/(2c), [0,1], |\cdot|).$$ Lastly, $\Theta_{n,c}\subseteq\Theta_n$ implies $$\Ncal(h/2,\Theta_{n,c},\vertiii{\cdot})\leq \Ncal(h/4,\Theta_{n},\vertiii{\cdot}).$$

\subsection{Proof of Proposition \ref{proposition: convergence of EP with L^1 integrable functions}}

    We start from Equation \eqref{eq: final upper bound of theorem: convergence of EP with L^1 integrable functions}: for any $M>0$, if $\sigma\leq2^{-2}$,
    \begin{equation}\label{eq: final result from theorem: convergence of EP with L^1 integrable functions, corollary: convergence of EP with L^1 integrable functions}
        \begin{aligned}
            &\EE^\ast\|\PP_n-P\|_\Fcal\\
            &\leq 52.8172\cdot \sqrt\frac{M^{1-\kappa}}{n}\cdot \sqrt{\lceil\log_2(1/\sigma)\rceil} \cdot \sigma^\frac{1+\kappa}{2}\\
            &+\inf_{\epsilon\in (0, \sigma )}\Bigl[4\epsilon +60.3400\sqrt\frac{M^{1-\kappa}}{n}\int_{\epsilon/8}^{\frac{\sigma}{2} } \sqrt{\frac{\EE[\log \Ncal(x,\Fcal,L^{1+\kappa}(\PP_n))]}{x^{1-\kappa}}}\, \dd x\Bigr]\\
            &+ \frac{6\cdot M^\frac{1-\kappa}{2}}{\sqrt{n}} \sigma^\frac{1+\kappa}{2}\cdot  \sqrt{ \EE[\log(2\Ncal(\sigma/4 , \Fcal, L^{1+\kappa}(\PP_n)))]}\\
            &+\frac{32M}{n} \EE[\log(2\Ncal(\sigma/4 , \Fcal, L^{1+\kappa}(\PP_n)))] + 2\EE[F\cdot \II(F> M)].
        \end{aligned}
    \end{equation}

    Note that,
    $$\int_{\sigma/8}^{\sigma/2}x^{\frac{\kappa-1}{2}}\, \dd x=\frac{2}{\kappa+1}\Bigl((1/2)^{\frac{\kappa+1}{2}}-(1/8)^{\frac{\kappa+1}{2}}\Bigr)\sigma^{\frac{\kappa+1}{2}}\geq \frac{3}{8}\sigma^{\frac{\kappa+1}{2}}.$$ Thus, we conclude that
    \begin{equation}\label{eq: upper bound of the first term in corollary: convergence of EP with L^1 integrable functions}
        \begin{aligned}
            \sqrt{\lceil\log_2(1/\sigma)\rceil} \cdot \sigma^\frac{1+\kappa}{2}
            \leq& \frac{8}{3}\sqrt{\lceil\log_2(1/\sigma)\rceil}\int_{\sigma/8}^{\sigma/2}x^{\frac{\kappa-1}{2}}\, \dd x
            \leq \frac{8}{3}\sqrt{\lceil\log_2(1/\sigma)\rceil}\int_{\epsilon/8}^{\sigma/2}x^{\frac{\kappa-1}{2}}\, \dd x\\
            \leq& \frac{4\sqrt2}{\sqrt3}\sqrt{\log_2(e)} \sqrt{\log(2/\sigma)}\int_{\epsilon/8}^{\sigma/2}x^{\frac{\kappa-1}{2}}\, \dd x\\
        \end{aligned}
    \end{equation}
    where the last inequality is because, for $\sigma\leq 2^{-2}$, $$\lceil\log_2(1/\sigma)\rceil\leq 1+\log_2(1/\sigma)\leq \frac{3}{2}\log_2(1/\sigma)=\frac{3}{2}\log_2(e)\log(1/\sigma).$$
Write $B(h)=D_\Fcal h^{-\gamma}[\log_+(U_\Fcal/h)]^{\gamma^\prime}$ for the assumed entropy upper bound. For $0<h\leq1/4$, we have $B(h)\geq1$. Meanwhile, as the covering entropy $x\mapsto \EE[\log\Ncal(x,\Fcal,L^{1+\kappa}(\PP_n))]$ is a non-increasing function,
    \begin{equation}\label{eq: upper bound of the third term in corollary: convergence of EP with L^1 integrable functions}
        \begin{aligned}
            &\sigma^\frac{1+\kappa}{2}\cdot  \sqrt{ \EE[\log(2\Ncal(\sigma/4 , \Fcal, L^{1+\kappa}(\PP_n)))]}\\
            &\leq 8\int_{\epsilon/8}^{\sigma/2}\sqrt{\frac{\EE[\log2\Ncal(x,\Fcal,L^{1+\kappa}(\PP_n))]}{x^{1-\kappa}}}\, \dd x\\
            &\lesssim \int_{\epsilon/8}^{\sigma/2}\sqrt{\frac{B(x)}{x^{1-\kappa}}}\, \dd x.\\
        \end{aligned}
    \end{equation}
    The first inequality uses $[\sigma/8,\sigma/4]\subseteq[\epsilon/8,\sigma/2]$ and $x^{(\kappa-1)/2}\geq\sigma^{(\kappa-1)/2}$ on that subinterval. The second uses the following bound, valid even for a singleton class: for $0<h\leq1/4$,
    \begin{equation}\label{eq: covering entropy inequality in proposition: convergence of EP with L^1 integrable functions}
        \begin{aligned}
            \EE[\log2\Ncal(h,\Fcal,L^{1+\kappa}(\PP_n))]
            &=\log2+\EE[\log\Ncal(h,\Fcal,L^{1+\kappa}(\PP_n))]\\
            \leq& (1+\log 2)B(h).
        \end{aligned}
    \end{equation}

    Plugging Equations \eqref{eq: upper bound of the first term in corollary: convergence of EP with L^1 integrable functions} and \eqref{eq: upper bound of the third term in corollary: convergence of EP with L^1 integrable functions} into Equation \eqref{eq: final result from theorem: convergence of EP with L^1 integrable functions, corollary: convergence of EP with L^1 integrable functions} gives, for any $M>0$ and $0<\epsilon<\sigma$,
    \[\begin{aligned}
        &\EE^\ast\|\PP_n-P\|_\Fcal\\
        \lesssim& \sqrt\frac{M^{1-\kappa}}{n}\cdot \sqrt{\log(2/\sigma)}\int_{\epsilon/8}^{\sigma/2}x^{\frac{\kappa-1}{2}}\, \dd x\\
        &+\Bigl[2\epsilon +\frac{64\sqrt{2}}{3}\sqrt\frac{M^{1-\kappa}}{n}\int_{\epsilon/8}^{\frac{\sigma}{2} } \sqrt{\frac{B(x)}{x^{1-\kappa}}}\, \dd x\Bigr]\\
        &+\frac{2 M}{3n} \EE[\log(2\Ncal(\sigma/4 , \Fcal, L^{1+\kappa}(\PP_n)))] + 2\EE[F\cdot \II(F> M)].
    \end{aligned}\]
    For this upper bound, Markov's inequality yields
    \[
        \EE[F\cdot \II(F> M)]\leq \frac{\|F\|_{L^m(P)}^m}{M^{m-1}}.
    \]
    Thus, combining it with \eqref{eq: covering entropy inequality in proposition: convergence of EP with L^1 integrable functions}, the expected empirical process $\EE^\ast\|\PP_n-P\|_\Fcal$ can be further bounded by
    $$\begin{aligned}
        &\EE^\ast\|\PP_n-P\|_\Fcal\\
        \lesssim& \sqrt\frac{M^{1-\kappa}}{n}\cdot \sqrt{\log(2/\sigma)}\int_{\epsilon/8}^{\sigma/2}x^{\frac{\kappa-1}{2}}\, \dd x\\
        &+\Bigl[2\epsilon +\frac{64\sqrt{2}}{3}\sqrt\frac{M^{1-\kappa}}{n}\int_{\epsilon/8}^{\frac{\sigma}{2} } \sqrt{\frac{B(x)}{x^{1-\kappa}}}\, \dd x\Bigr]\\
        &+\frac{4 M}{3n} B(\sigma/4) + 2\frac{\|F\|_{L^m(P)}^m}{M^{m-1}}.
    \end{aligned}$$

    For $\sqrt{\log(2/\sigma)}$ in the first term in the RHS, the assumption on the upper bound of expected random covering entropy implies that,
    $$\begin{aligned}
        \sqrt{\log(2/\sigma)}\int_{\epsilon/8}^{\sigma/2}x^{\frac{\kappa-1}{2}}\, \dd x
        \leq& \int_{\epsilon/8}^{\sigma/2}\sqrt{\frac{D_\Fcal\cdot x^{-\gamma}\cdot \log_+(U_\Fcal/x)^{\gamma^\prime}}{x^{1-\kappa}}}\, \dd x.
    \end{aligned}$$
    Hence, we further have, for any $0<\epsilon\leq\sigma/8$ and $M>0$,
    \begin{equation}\label{eq: refined upper bound for EP in corollary: convergence of EP with L^1 integrable functions}
        \begin{aligned}
            &\EE^\ast\|\PP_n-P\|_\Fcal\\
            \lesssim& 16\epsilon +\sqrt\frac{M^{1-\kappa}}{\tilde{n}}\int_{\epsilon}^{\frac{\sigma}{2} } \sqrt{\frac{x^{-\gamma}\cdot \log_+(U_\Fcal/x)^{\gamma^\prime}}{x^{1-\kappa}}}\, \dd x\\
            &+\frac{4 M}{3\tilde{n}} (\sigma/2)^{-\gamma}\cdot \log_+(2U_\Fcal/\sigma)^{\gamma^\prime} + 2\frac{\|F\|_{L^m(P)}^m}{M^{m-1}}\\
            \lesssim& 16\epsilon + \sqrt\frac{M^{1-\kappa}}{\tilde{n}}\int_{\epsilon}^{\frac{\sigma}{2} } x^{\frac{\kappa-1-\gamma}{2}}\log_+(U_\Fcal/x)^{\frac{\gamma^\prime}{2}}\, \dd x \\
            &+\frac{4 M}{3\tilde{n}} (\sigma/2)^{-\gamma}\cdot \log_+(2U_\Fcal/\sigma)^{\gamma^\prime} + 2\frac{\|F\|_{L^m(P)}^m}{M^{m-1}}.\\
        \end{aligned}
    \end{equation}
    Here $B(\sigma/4)\lesssim D_\Fcal\sigma^{-\gamma}[\log_+(2U_\Fcal/\sigma)]^{\gamma^\prime}$ by the explicit form of $B$.

    Now, we consider the following cases to simplify Equation \eqref{eq: refined upper bound for EP in corollary: convergence of EP with L^1 integrable functions}.
    \begin{enumerate}
        \item When $\kappa=1$ and $\gamma\in[0,2)$, taking $\epsilon\to0$ for Equation \eqref{eq: refined upper bound for EP in corollary: convergence of EP with L^1 integrable functions} gives, for some suppressed constant depending on $\gamma$ and $\gamma^\prime$,
        $$\begin{aligned}
            &\EE^\ast\|\PP_n-P\|_\Fcal\\
            \lesssim& \frac{1}{\sqrt{\tilde{n}}}\int_{0}^{\frac{\sigma}{2} } x^{-\frac{\gamma}{2}}\log_+(U_\Fcal/x)^{\frac{\gamma^\prime}{2}}\, \dd x \\
            &+\inf_{M>0}\Bigl(\frac{4 M}{3\tilde{n}} (\sigma/2)^{-\gamma}\cdot \log_+(2U_\Fcal/\sigma)^{\gamma^\prime} + 2\frac{\|F\|_{L^m(P)}^m}{M^{m-1}}\Bigr)\\
            \lesssim& \tilde{n}^{-\frac{1}{2}}\sigma^{1-\frac{\gamma}{2}}\log_+(2U_\Fcal/\sigma)^{\frac{\gamma^\prime}{2}} + \inf_{M>0}\Bigl(\frac{ M}{\tilde{n}} \sigma^{-\gamma}\cdot \log_+(2U_\Fcal/\sigma)^{\gamma^\prime} + \frac{\|F\|_{L^m(P)}^m}{M^{m-1}}\Bigr)\\
            \lesssim& \tilde{n}^{-\frac{1}{2}}\sigma^{1-\frac{\gamma}{2}}\log_+(2U_\Fcal/\sigma)^{\frac{\gamma^\prime}{2}} + \|F\|_{L^m(P)}\cdot \tilde{n}^{-\frac{m-1}{m}}\sigma^{-\gamma\frac{m-1}{m}}\log_+(2U_\Fcal/\sigma)^{\gamma^\prime\frac{m-1}{m}}.
        \end{aligned}$$
        where the second inequality is due to Lemma \ref{lemma: integral of covering entropy with logarithm}, the third one is due to Lemma \ref{lemma: minimal value between two terms}.

        \item When $\kappa=1$ and $\gamma>2$, taking $\epsilon=\tilde{n}^{-\frac{1}{\gamma}}$ for Equation \eqref{eq: refined upper bound for EP in corollary: convergence of EP with L^1 integrable functions} gives, for some suppressed constant depending on $\gamma$ and $\gamma^\prime$,
        $$\begin{aligned}
            &\EE^\ast\|\PP_n-P\|_\Fcal\\
            \lesssim& 16\tilde{n}^{-\frac{1}{\gamma}} + \frac{1}{\sqrt{\tilde{n}}}\int_{\tilde{n}^{-\frac{1}{\gamma}}}^{\frac{\sigma}{2} } x^{-\frac{\gamma}{2}}\log_+(U_\Fcal/x)^{\frac{\gamma^\prime}{2}}\, \dd x \\
            &\qquad+\inf_{M>0}\Bigl(\frac{4 M}{3\tilde{n}} (\sigma/2)^{-\gamma}\cdot \log_+(2U_\Fcal/\sigma)^{\gamma^\prime} + 2\frac{\|F\|_{L^m(P)}^m}{M^{m-1}}\Bigr)\\
            \lesssim& \tilde{n}^{-\frac{1}{\gamma}}+\frac{1}{\sqrt{\tilde{n}}}\int_{\tilde{n}^{-\frac{1}{\gamma}}}^{\infty} x^{-\frac{\gamma}{2}}\, \dd x \cdot \log_+(U_\Fcal\cdot \tilde{n}^{\frac{\gamma}{2}})^{\frac{\gamma^\prime}{2}}\\
            &\qquad + \|F\|_{L^m(P)}\cdot \tilde{n}^{-\frac{m-1}{m}}\sigma^{-\gamma\frac{m-1}{m}}\log_+(2U_\Fcal/\sigma)^{\gamma^\prime\frac{m-1}{m}}\\
            \lesssim& \tilde{n}^{-\frac{1}{\gamma}}\log_+(U_\Fcal\cdot \tilde{n}^{\frac{\gamma}{2}})^{\frac{\gamma^\prime}{2}}+ \|F\|_{L^m(P)}\cdot \tilde{n}^{-\frac{m-1}{m}}\sigma^{-\gamma\frac{m-1}{m}}\log_+(2U_\Fcal/\sigma)^{\gamma^\prime\frac{m-1}{m}}.\\
        \end{aligned}$$

        \item When $\kappa<1$ and $\gamma\in[0,1+\kappa)$, taking $\epsilon\to0$ for Equation \eqref{eq: refined upper bound for EP in corollary: convergence of EP with L^1 integrable functions} gives, for some suppressed constant depending on $\kappa$, $\gamma$ and $\gamma^\prime$,
        $$\begin{aligned}
            &\EE^\ast\|\PP_n-P\|_\Fcal\\
            \lesssim& \sqrt\frac{M^{1-\kappa}}{\tilde{n}}\int_{0}^{\frac{\sigma}{2} } x^{\frac{\kappa-1-\gamma}{2}}\log_+(U_\Fcal/x)^{\frac{\gamma^\prime}{2}}\, \dd x \\
            &+\frac{4 M}{3\tilde{n}} (\sigma/2)^{-\gamma}\cdot \log_+(2U_\Fcal/\sigma)^{\gamma^\prime} + 2\frac{\|F\|_{L^m(P)}^m}{M^{m-1}}\\
            \lesssim& \sqrt\frac{M^{1-\kappa}}{\tilde{n}}\sigma^{\frac{\kappa+1-\gamma}{2}}\log_+(2U_\Fcal/\sigma)^{\frac{\gamma^\prime}{2}} + \frac{ M}{\tilde{n}} \sigma^{-\gamma}\cdot \log_+(2U_\Fcal/\sigma)^{\gamma^\prime} + \frac{\|F\|_{L^m(P)}^m}{M^{m-1}}.
        \end{aligned}$$

        \item When $\kappa<1$ and $\gamma>1+\kappa$, taking $\epsilon=\bigl(\frac{M^{1-\kappa}}{\tilde{n}}\bigr)^{\frac{1}{\gamma+1-\kappa}}$ for Equation \eqref{eq: refined upper bound for EP in corollary: convergence of EP with L^1 integrable functions} gives, for some suppressed constant depending on $\kappa$, $\gamma$ and $\gamma^\prime$,
        $$\begin{aligned}
            &\EE^\ast\|\PP_n-P\|_\Fcal\\
            \lesssim& \bigl(\tfrac{M^{1-\kappa}}{\tilde{n}}\bigr)^{\frac{1}{\gamma+1-\kappa}} + \sqrt{\tfrac{M^{1-\kappa}}{\tilde{n}}}\int_{\epsilon}^{\infty} x^{\frac{\kappa-1-\gamma}{2}}\, \dd x \cdot \log_+(U_\Fcal/\epsilon)^{\frac{\gamma^\prime}{2}}\\
            &\qquad\qquad+ \frac{ M}{\tilde{n}} \sigma^{-\gamma}\cdot \log_+(2U_\Fcal/\sigma)^{\gamma^\prime} + \frac{\|F\|_{L^m(P)}^m}{M^{m-1}}\\
            \lesssim& \bigl(\frac{M^{1-\kappa}}{\tilde{n}}\bigr)^{\frac{1}{\gamma+1-\kappa}} \log_+\Bigl(U_\Fcal\cdot \bigl(\frac{\tilde{n}}{M^{1-\kappa}}\bigr)^{\frac{1}{\gamma+1-\kappa}}\Bigr)^{\frac{\gamma^\prime}{2}} \\
            &\qquad\qquad+ \frac{ M}{\tilde{n}} \sigma^{-\gamma}\cdot \log_+(2U_\Fcal/\sigma)^{\gamma^\prime} + \frac{\|F\|_{L^m(P)}^m}{M^{m-1}}.\\
        \end{aligned}$$
    \end{enumerate}

    At the boundary $\gamma=1+\kappa$, set $p=\gamma^\prime/2$ and
    $r=(M^{1-\kappa}/\tilde n)^{1/2}$. The cutoff condition gives $r\leq\sigma/8$.
    Take $\epsilon=r$ in \eqref{eq: refined upper bound for EP in corollary: convergence of EP with L^1 integrable functions}.
    The power of $x$ is now $-1$, so we keep the finite upper limit. Since
    $U_\Fcal\geq1$ and $\sigma\leq1/4$, the substitution $t=\log(U_\Fcal/x)$ gives
    \[
        \int_r^{\sigma/2}\frac{[\log_+(U_\Fcal/x)]^p}{x}\,\dd x
        =\frac{[\log(U_\Fcal/r)]^{p+1}-[\log(2U_\Fcal/\sigma)]^{p+1}}{p+1}.
    \]
    Thus the cutoff and integral together contribute at most a constant times
    $r[\log_+(U_\Fcal/r)]^{p+1}$. When $\kappa=1$, $r=\tilde n^{-1/2}$ is independent of $M$, so the same tail optimization applies.
    When $\kappa<1$, the bound holds for each admissible $M$; any optimization over $M$
    must use that same value in all three terms.

\subsection{Proof of Proposition \ref{proposition: convergence of EP with L^infty integrable functions}}

    By Equation \eqref{eq: expected empirical process, theorem: convergence of EP with L^infty integrable functions} and Lemma \ref{lemma: moment bound of empirical process to its expectation}, for some suppressed constant depending on $m$ and $\kappa$,
    $$\begin{aligned}
        &\Bigl\|\|\PP_n-P\|_{\Fcal}\Bigr\|_{L^m(P)}\lesssim \EE^\ast\|\PP_n-P\|_\Fcal+n^{\frac{1}{m}-1}\|F\|_{L^m(P)}\\
        \lesssim& \frac{\|1/w\|_{L^{1+\kappa}(P)}^{\frac{1-\kappa}{2}}}{n^{\frac{\kappa}{1+\kappa}}}\cdot  \Bigl(\int_{0}^{2\sigma}1+\log\Ncal(x,\Fcal,L^{1+\kappa}(P))^\frac{\kappa}{1+\kappa}\, \dd x\Bigr)^{\frac{1+\kappa}{2}}\\
        &\qquad\cdot\Bigl(\int_{0}^{2 \|F\|_{L^\infty(w)}} 1+\log\Ncal(x,\Fcal,L^{\infty}(w))^\frac{\kappa}{1+\kappa}\, \dd x\Bigr)^{\frac{1-\kappa}{2}}\\
        &+ \frac{\|1/w\|_{L^m(P)}}{n^{1-\frac{1}{m}}}\cdot\Bigl(\int_{0}^{2 \|F\|_{L^\infty(w)}} 1+\log\Ncal(x,\Fcal,L^{\infty}(w))^{1-\frac{1}{m}}\, \dd x\Bigr)\\
        &+ n^{\frac{1}{m}-1}\|F\|_{L^m(P)}\\
        \lesssim& \frac{\|1/w\|_{L^{1+\kappa}(P)}^{\frac{1-\kappa}{2}}}{n^{\frac{\kappa}{1+\kappa}}}\cdot  \Bigl(\int_{0}^{2\sigma}1+\log\Ncal(x,\Fcal,L^{1+\kappa}(P))^\frac{\kappa}{1+\kappa}\, \dd x\Bigr)^{\frac{1+\kappa}{2}}\\
        &\qquad\cdot\Bigl(\int_{0}^{2 \|F\|_{L^\infty(w)}} 1+\log\Ncal(x,\Fcal,L^{\infty}(w))^\frac{\kappa}{1+\kappa}\, \dd x\Bigr)^{\frac{1-\kappa}{2}}\\
        &+ \frac{\|1/w\|_{L^m(P)}}{n^{1-\frac{1}{m}}}\cdot\Bigl(\int_{0}^{2 \|F\|_{L^\infty(w)}} 1+\log\Ncal(x,\Fcal,L^{\infty}(w))^{1-\frac{1}{m}}\, \dd x\Bigr)\\
    \end{aligned}$$
    where the second inequality is due to the fact that $$\|F\|_{L^m(P)}\leq \|F\|_{L^\infty(w)}\cdot\|1/w\|_{L^m(P)}\leq \|1/w\|_{L^m(P)}\int_0^{\|F\|_{L^\infty(w)}}\, \dd x.$$

    Meanwhile, for any $f,g\in\Fcal$, the inequality $$\|f-g\|_{L^{1+\kappa}(P)}\leq \|f-g\|_{L^\infty(w)}\|1/w\|_{L^m(P)}$$ implies that, for any $x>0$, $$\log\Ncal(x,\Fcal,L^{1+\kappa}(P))\leq \log\Ncal(x/\|1/w\|_{L^m(P)}, \Fcal, L^\infty(w)).$$
    Consequently,
    $$\begin{aligned}
        &\Bigl\|\|\PP_n-P\|_{\Fcal}\Bigr\|_{L^m(P)}\\
        \lesssim& \frac{\|1/w\|_{L^{1+\kappa}(P)}^{\frac{1-\kappa}{2}}}{n^{\frac{\kappa}{1+\kappa}}}\cdot  \Bigl(\int_{0}^{2\sigma}1+\log\Ncal(x/\|1/w\|_{L^m(P)}, \Fcal, L^\infty(w))^\frac{\kappa}{1+\kappa}\, \dd x\Bigr)^{\frac{1+\kappa}{2}}\\
        &\qquad\cdot\Bigl(\int_{0}^{2 \|F\|_{L^\infty(w)}} 1+\log\Ncal(x,\Fcal,L^{\infty}(w))^\frac{\kappa}{1+\kappa}\, \dd x\Bigr)^{\frac{1-\kappa}{2}}\\
        &+ \frac{\|1/w\|_{L^m(P)}}{n^{1-\frac{1}{m}}}\cdot\Bigl(\int_{0}^{2 \|F\|_{L^\infty(w)}} 1+\log\Ncal(x,\Fcal,L^{\infty}(w))^{1-\frac{1}{m}}\, \dd x\Bigr)\\
        \lesssim& \frac{\|1/w\|_{L^{m}(P)} }{n^{\frac{\kappa}{1+\kappa}}}\cdot  \Bigl(\int_{0}^{2\sigma/\|1/w\|_{L^m(P)}}1+\log\Ncal(x, \Fcal, L^\infty(w))^\frac{\kappa}{1+\kappa}\, \dd x\Bigr)^{\frac{1+\kappa}{2}}\\
        &\qquad\cdot\Bigl(\int_{0}^{2 \|F\|_{L^\infty(w)}} 1+\log\Ncal(x,\Fcal,L^{\infty}(w))^\frac{\kappa}{1+\kappa}\, \dd x\Bigr)^{\frac{1-\kappa}{2}}\\
        &+ \frac{\|1/w\|_{L^m(P)}}{n^{1-\frac{1}{m}}}\cdot\Bigl(\int_{0}^{2 \|F\|_{L^\infty(w)}} 1+\log\Ncal(x,\Fcal,L^{\infty}(w))^{1-\frac{1}{m}}\, \dd x\Bigr).\\
    \end{aligned}$$

    Write $W=\|1/w\|_{L^m(P)}$, $R=\|F\|_{L^\infty(w)}$ and $q=1-1/m$; the case $R=0$ is trivial. Both weighted radii $x=2\sigma/W$ and $x=2R$ are at most $2R_0$. For $0<u\le 2R_0$, the assumed entropy upper bound $D_\Fcal u^{-\gamma}\log_+(U_\Fcal/u)^{\gamma^\prime}$ is at least $(1\lor 2R_0)^{-\gamma}$. Its $a$th power therefore bounds $1+[\log\Ncal(u,\Fcal,L^\infty(w))]^a$ up to a constant uniform in $0<a\le 1$. Thus, it suffices to evaluate the covering integrals. For $0<x\le 2R_0$ and $0<a\le1$, if $a\gamma<1$, by Lemma \ref{lemma: integral of covering entropy with logarithm},
    $$\begin{aligned}
        &n^{-a}\int_0^x 1+\log\Ncal(u,\Fcal, L^\infty(w))^a\, \dd u
        \lesssim \tilde{n}^{-a}\int_0^x u^{-a\gamma}\log_+(U_\Fcal/u)^{a\gamma^\prime}\, \dd u\\
        \lesssim& \tilde{n}^{-a} x^{1-a\gamma}\cdot \log_+(U_\Fcal/x)^{a\gamma^\prime}.
    \end{aligned}$$
    When $q\gamma\ge1$, apply Theorem~\ref{theorem: convergence of EP with L^infty integrable functions} to the expectation at moment order $m_b=(1-b)^{-1}$, where
    \[
        b=\gamma^{-1}-\frac{c_\gamma}{2+\log\tilde n},\qquad
        c_\gamma=
        \begin{cases}
           1/\gamma-1/2,&\gamma<2,\\
           1/(2\gamma),&\gamma\ge2 .
        \end{cases}
    \]
    Here $0<b<1/\gamma\le q$, so $1<m_b\le m$. The radius $\sigma$ remains valid in $L^{m_b\land 2}(P)$, and $\|1/w\|_{L^{m_b}(P)}\le W$. In applying the theorem we use $W$ for the metric comparison above; its constants are uniform in $n$ since $b\ge1/(2\gamma)$ and thus $(m_b\land2)-1$ is bounded away from zero.

    To track the integration constant, put $d=1-b\gamma$ and $p=b\gamma^\prime$. Substituting $u=xe^{-v}$ and then $z=dv$, using $\log_+(U_\Fcal/u)\le\log_+(U_\Fcal/x)(1+v)$, gives
    \[
       \int_0^x u^{-b\gamma}\log_+(U_\Fcal/u)^p\,\dd u
       \le x^d\log_+(U_\Fcal/x)^p d^{-1-p}
              \int_0^\infty e^{-z}(1+z)^{\gamma^\prime}\,\dd z .
    \]
    Since $x\le 2R_0$, $\tilde n^{-b}\le e^{c_\gamma}\tilde n^{-1/\gamma}$ and $d=\gamma c_\gamma/(2+\log\tilde n)$, the normalized integral with exponent $b$, including its constant integrand, is bounded by a constant times
    \[
      K_n\tilde n^{-1/\gamma}\log_+(U_\Fcal/x)^{\gamma^\prime/\gamma}.
    \]
    If $\gamma<2$, then $m_b>2$: the local $L^2(P)$ integral keeps exponent $1/2$ and only the weighted integral uses $b$. If $\gamma\ge2$, then $m_b<2$: both exponents are $b$, and the local integral, after the metric comparison, is bounded by the weighted integral up to $2R$ since $2\sigma/W\le2R$. The two powers in the product sum to one. Finally, applying Lemma~\ref{lemma: moment bound of empirical process to its expectation} at the original order $m$ adds at most $n^{-q}WR\lesssim \tilde n^{-1/\gamma}$ in these branches.

    Using these calculus inequalities, the $L^m(P)$ norm of the empirical process $\|\PP_n-P\|_\Fcal$ can be controlled under the following scenarios:
    \begin{enumerate}
        \item When $m\geq 2$ so that $\kappa=1$ and $\gamma(1-\frac{1}{m})<1$, we have,
        $$\begin{aligned}
            \Bigl\|\|\PP_n-P\|_{\Fcal}\Bigr\|_{L^m(P)}
            \lesssim& \tilde{n}^{-\frac{1}{2}}\sigma^{1-\frac{\gamma}{2}}\log_+\Bigl(\frac{U_\Fcal \|1/w\|_{L^m(P)}}{2\sigma}\Bigr)^{\frac{\gamma^\prime}{2}} \\
            &+ \tilde{n}^{\frac{1}{m}-1}\|F\|_{L^\infty(w)}^{1-(1-\frac{1}{m})\gamma}\cdot \log_+\Bigl(\frac{U_\Fcal}{2\|F\|_{L^\infty(w)}}\Bigr)^{(1-\frac{1}{m})\gamma^\prime}.
        \end{aligned}$$

        \item When $m\geq 2$ so that $\kappa=1$. $\gamma< 2$ and $\gamma(1-\frac{1}{m})\geq 1$, we have,
        $$\begin{aligned}
            &\Bigl\|\|\PP_n-P\|_{\Fcal}\Bigr\|_{L^m(P)}\\
            \lesssim& \tilde{n}^{-\frac{1}{2}}\sigma^{1-\frac{\gamma}{2}}\log_+\Bigl(\frac{U_\Fcal \|1/w\|_{L^m(P)}}{2\sigma}\Bigr)^{\frac{\gamma^\prime}{2}} + K_n\tilde{n}^{-\frac{1}{\gamma}} \log_+\Bigl(\frac{U_\Fcal}{2\|F\|_{L^\infty(w)}}\Bigr)^{\frac{\gamma^\prime}{\gamma}}.
        \end{aligned}$$

        \item When $m\geq 2$ so that $\kappa=1$ and $\gamma\geq  2$, we have,
        $$\begin{aligned}
            &\Bigl\|\|\PP_n-P\|_{\Fcal}\Bigr\|_{L^m(P)}\\
            \lesssim& K_n\tilde{n}^{-\frac{1}{\gamma}}\log_+\Bigl(\frac{U_\Fcal \|1/w\|_{L^m(P)}}{2\sigma}\Bigr)^{\frac{\gamma^\prime}{\gamma}} + K_n\tilde{n}^{-\frac{1}{\gamma}} \log_+\Bigl(\frac{U_\Fcal}{2\|F\|_{L^\infty(w)}}\Bigr)^{\frac{\gamma^\prime}{\gamma}}\\
            \lesssim&K_n\tilde{n}^{-\frac{1}{\gamma}}\log_+\Bigl(\frac{U_\Fcal \|1/w\|_{L^m(P)}}{2\sigma}\Bigr)^{\frac{\gamma^\prime}{\gamma}}.
        \end{aligned}$$
    \end{enumerate}

    When $m=1+\kappa<2$ and $q\gamma<1$, we have
    $$\begin{aligned}
        &\Bigl\|\|\PP_n-P\|_{\Fcal}\Bigr\|_{L^m(P)}\\
        \lesssim& \frac{\|1/w\|_{L^{m}(P)} }{n^{\frac{\kappa}{1+\kappa}}}\cdot  \Bigl(\int_{0}^{2\sigma/\|1/w\|_{L^m(P)}}1+\log\Ncal(x, \Fcal, L^\infty(w))^\frac{\kappa}{1+\kappa}\, \dd x\Bigr)^{\frac{1+\kappa}{2}}\\
        &\qquad\cdot\Bigl(\int_{0}^{2 \|F\|_{L^\infty(w)}} 1+\log\Ncal(x,\Fcal,L^{\infty}(w))^\frac{\kappa}{1+\kappa}\, \dd x\Bigr)^{\frac{1-\kappa}{2}}\\
        &+ \frac{\|1/w\|_{L^m(P)}}{n^{\frac{\kappa}{1+\kappa}}}\cdot\Bigl(\int_{0}^{2 \|F\|_{L^\infty(w)}} 1+\log\Ncal(x,\Fcal,L^{\infty}(w))^{\frac{\kappa}{1+\kappa}}\, \dd x\Bigr)\\
        \lesssim& \frac{\|1/w\|_{L^m(P)}}{n^{\frac{\kappa}{1+\kappa}}}\cdot\Bigl(\int_{0}^{2 \|F\|_{L^\infty(w)}} 1+\log\Ncal(x,\Fcal,L^{\infty}(w))^{\frac{\kappa}{1+\kappa}}\, \dd x\Bigr).
    \end{aligned}$$
    As a result,
    \begin{enumerate}
        \addtocounter{enumi}{3}
        \item When $m\in(1,2)$ and $\frac{\kappa}{1+\kappa}\gamma<1$,
        $$\begin{aligned}
            \Bigl\|\|\PP_n-P\|_{\Fcal}\Bigr\|_{L^m(P)}
            \lesssim \tilde{n}^{-\frac{\kappa}{1+\kappa}}\|F\|_{L^\infty(w)}^{1-\frac{\kappa}{1+\kappa}\gamma}\log_+\Bigl(\frac{U_\Fcal}{2\|F\|_{L^\infty}(w)}\Bigr)^{\frac{\kappa}{1+\kappa}\gamma^\prime}.
        \end{aligned}$$
        \item When $m\in(1,2)$ and $\frac{\kappa}{1+\kappa}\gamma\geq 1$,
        $$\begin{aligned}
            \Bigl\|\|\PP_n-P\|_{\Fcal}\Bigr\|_{L^m(P)}
            \lesssim K_n\tilde{n}^{-\frac{1}{\gamma}} \log_+\Bigl(\frac{U_\Fcal}{2\|F\|_{L^\infty}(w)}\Bigr)^{\frac{\gamma^\prime}{\gamma}}.
        \end{aligned}$$
    \end{enumerate}

\subsection{Proof of Proposition \ref{prop: convergence rate expression in terms of weighted covering entropy}}

    Plugging $\|F\|_{L^{\infty}(w)}\lesssim \sigma^s$ into the result of Proposition \ref{proposition: convergence of EP with L^infty integrable functions} yields
    \[\begin{aligned}
        &\Bigl\|\|\PP_n-P\|_{\Fcal}\Bigr\|_{L^m(P)}\\
        \lesssim_{\log}&\left\{\begin{aligned}
            & \tilde{n}^{-\frac{1}{2}}\sigma^{1-\frac{\gamma}{2}} + \tilde{n}^{\frac{1}{m}-1}\sigma^{s(1-(1-\frac{1}{m})\gamma)}\quad
            \qquad  \text{when}\quad m\geq 2, \gamma(1-\frac{1}{m})<1 \\
            & \tilde{n}^{-\frac{1}{2}}\sigma^{1-\frac{\gamma}{2}} + \tilde{n}^{-\frac{1}{\gamma}}
            \qquad \text{when}\quad m\geq 2, \gamma(1-\frac{1}{m})\geq1, \gamma<2 \\
            & \tilde{n}^{\frac{1}{m}-1}\sigma^{s(1-(1-\frac{1}{m})\gamma)} \qquad \text{when}\quad m\in(1,2), \frac{\kappa}{1+\kappa}\gamma<1\\
            & \tilde{n}^{-\frac{1}{\gamma}}\qquad \text{when}\quad m\geq 2, \gamma\geq 2;\quad\text{or}\quad m\in(1,2), \frac{\kappa}{1+\kappa}\gamma\geq 1
        \end{aligned}\right.\\
    \end{aligned}\]
    Then, define $\phi_{n}^{(1)}(\sigma):=\tilde{n}^{-\frac{1}{2}}\sigma^{1-\frac{\gamma}{2}}$, $\phi_{n}^{(2)}(\sigma):= \tilde{n}^{\frac{1}{m}-1}\sigma^{s(1-(1-\frac{1}{m})\gamma)}$, $\phi_{n}^{(3)}(\sigma):= \tilde{n}^{-\frac{1}{\gamma}}$ and we have
    $$\begin{aligned}
        &\Bigl\|\|\PP_n-P\|_{\Fcal}\Bigr\|_{L^m(P)}\\
        \lesssim_{\log}&\left\{\begin{aligned}
            & \phi_{n}^{(1)}(\sigma) + \phi_{n}^{(2)}(\sigma)\quad
            \qquad  \text{when}\quad m\geq 2, \gamma(1-\frac{1}{m})<1 \\
            & \phi_{n}^{(1)}(\sigma) + \phi_{n}^{(3)}(\sigma)
            \qquad \text{when}\quad m\geq 2, \gamma(1-\frac{1}{m})\geq1, \gamma<2 \\
            & \phi_{n}^{(2)}(\sigma) \qquad \text{when}\quad m\in(1,2), \frac{\kappa}{1+\kappa}\gamma<1\\
            & \phi_{n}^{(3)}(\sigma)\qquad \text{when}\quad m\geq 2, \gamma\geq 2;\quad\text{or}\quad m\in(1,2), \frac{\kappa}{1+\kappa}\gamma\geq 1
        \end{aligned}\right.\\
        =:&\phi_n(\sigma).
    \end{aligned}$$
    Hence, for any $\delta\in(0,1)$, $$\PP\Bigl(\|\PP_n-P\|_{\Fcal}\gtrsim_{\log }\delta^{-\frac{1}{m}}\phi_n(\sigma)\Bigr)\leq\delta.$$

    Define $\sigma_n(\delta)=\inf\{\sigma>0: \delta^{-\frac{1}{m}}\cdot \phi_n(\sigma)\leq \sigma^2\}$, and for $i\in[4]$, denote $$\sigma_n^{(i)}(\delta)=\inf\{\sigma>0:\delta^{-\frac{1}{m}}\cdot \phi_n^{(i)}(\sigma)\asymp \sigma^2\}.$$
    Then,
    $$\sigma_n(\delta)\lesssim\left\{\begin{aligned}
        & \sigma_n^{(1)}(\delta) + \sigma_n^{(2)}(\delta)\quad & \text{when}\quad m\geq 2, \gamma(1-\frac{1}{m})<1 \\
        & \sigma_n^{(1)}(\delta) + \sigma_n^{(3)}(\delta)\quad & \text{when}\quad m\geq 2, \gamma(1-\frac{1}{m})\geq1, \gamma<2 \\
        & \sigma_n^{(2)}(\delta)\quad & \text{when}\quad m\in(1,2), \frac{\kappa}{1+\kappa}\gamma<1 \\
        & \sigma_n^{(3)}(\delta)\quad & \text{when}\quad m\geq 2, \gamma\geq 2;\quad\text{or}\quad m\in(1,2), \frac{\kappa}{1+\kappa}\gamma\geq 1 \\
    \end{aligned}\right.$$

    Furthermore, the solutions to $\sigma_n^{(i)}(\delta)$ for $i\in[4]$ reads:
    $$\begin{aligned}
        \sigma_n^{(1)}\asymp& \delta^{-\frac{2}{(2+\gamma)m}}\cdot \tilde{n}^{-\frac{1}{2+\gamma}}\leq \delta^{-\frac{1}{m}}\cdot \tilde{n}^{-\frac{1}{2+\gamma}}\\
        \sigma_n^{(2)}\asymp&\delta^{-\frac{1}{m(2-s(1-(1-\frac{1}{m})\gamma))}}\cdot \tilde{n}^{-\frac{1-\frac{1}{m}}{2-s(1-(1-\frac{1}{m})\gamma)}}\leq \delta^{-\frac{1}{m}}\tilde{n}^{-\frac{1}{\frac{2-s}{1-1/m}+s\gamma}}\\
        \sigma_n^{(3)}\asymp&\delta^{-\frac{1}{2m}}\cdot\tilde{n}^{-\frac{1}{2\gamma}}\leq \delta^{-\frac{1}{m}}\cdot\tilde{n}^{-\frac{1}{2\gamma}}.\\
    \end{aligned}$$
    As a result, we conclude that
    $$\begin{aligned}
        &\sigma_n(\delta)
        \lesssim\delta^{-\frac{1}{m}}\cdot\\
        &\left\{\begin{aligned}
            & \tilde{n}^{-\frac{1}{2+\gamma}} + \tilde{n}^{-\frac{1}{\frac{2-s}{1-1/m}+s\gamma}} & \text{when}\quad m\geq 2, \gamma(1-\frac{1}{m})<1 \\
            & \tilde{n}^{-\frac{1}{2+\gamma}} + \tilde{n}^{-\frac{1}{2\gamma}}\asymp \tilde{n}^{-\frac{1}{2+\gamma}} & \text{when}\quad m\geq 2, \gamma(1-\frac{1}{m})\geq1, \gamma<2 \\
            & \tilde{n}^{-\frac{1}{\frac{2-s}{1-1/m}+s\gamma}}& \text{when}\quad m\in(1,2), \frac{\kappa}{1+\kappa}\gamma<1 \\
            & \tilde{n}^{-\frac{1}{2\gamma}} & \text{when}\quad m\geq 2, \gamma\geq 2;\quad\text{or}\quad m\in(1,2), \frac{\kappa}{1+\kappa}\gamma\geq 1 \\
        \end{aligned}\right.\\
    \end{aligned}$$

    Note that
    \begin{enumerate}
        \item When $m\geq 2$ and $\gamma\geq 1$ (so the second case holds), we always have $\frac{1}{\frac{2-s}{1-1/m}+s\gamma}\geq \frac{1}{2+\gamma}$. As a result, the first two cases can be combined and written as $$\tilde{n}^{-\frac{1}{2+\gamma}} + \tilde{n}^{-\frac{1}{\frac{2-s}{1-1/m}+s\gamma}} \quad \text{when}\quad m\geq 2\quad\text{and}\quad  \gamma<2.$$
        \item $\frac{1}{\frac{2-s}{1-1/m}+s\gamma}\geq \frac{1}{2\gamma}$ if and only if $\gamma(1-\frac{1}{m})\geq 1$. So the last case can be combined with part of the third case (i.e., when $m\in(1,2)$ and $\gamma\geq 2$) and written as $$\tilde{n}^{-\frac{1}{2\gamma}} + \tilde{n}^{-\frac{1}{\frac{2-s}{1-1/m}+\gamma}} \quad\text{when} \quad\gamma\geq 2.$$
        \item When $m\in(1,2)$ and $\gamma<2$ (so the remaining part of case 3 holds), $\frac{1}{\frac{2-s}{1-1/m}+s\gamma}\leq \frac{1}{2+\gamma}$. Hence, the rate can be written as $$\tilde{n}^{-\frac{1}{2+\gamma}} + \tilde{n}^{-\frac{1}{\frac{2-s}{1-1/m}+s\gamma}}\quad\text{when}\quad m\in(1,2),\gamma<2.$$
    \end{enumerate}
    Therefore,
    $$\begin{aligned}
        \sigma_n(\delta)
        \lesssim&\delta^{-\frac{1}{m}}\cdot\left\{\begin{aligned}
            & \tilde{n}^{-\frac{1}{2+\gamma}} + \tilde{n}^{-\frac{1}{\frac{2-s}{1-1/m}+s\gamma}} & \text{when}\quad m\geq 2, \gamma<2 \\
            & \tilde{n}^{-\frac{1}{2\gamma}} + \tilde{n}^{-\frac{1}{\frac{2-s}{1-1/m}+\gamma}} &\quad\text{when}\quad  \gamma\geq 2\\
            & \tilde{n}^{-\frac{1}{2+\gamma}} + \tilde{n}^{-\frac{1}{\frac{2-s}{1-1/m}+s\gamma}}& \text{when}\quad m\in(1,2), \gamma<2 \\
        \end{aligned}\right.\\
        =& \delta^{-\frac{1}{m}}\cdot\Bigl(\tilde{n}^{-(\frac{1}{2+\gamma}\land\frac{1}{2\gamma})} + \tilde{n}^{-\frac{1}{\frac{2-s}{1-1/m}+s\gamma}}\Bigr)
    \end{aligned}$$

\section{Proof of Results in Section \ref{sec: The New Empirical Processes}}

\subsection{Auxiliary Inequalities}

\begin{lemma}\label{lemma: upper bound of incomplete Gamma integral}
    For any $x>1/2$, the incomplete Gamma function, $\Gamma(\frac{3}{2},x)$, can be bounded by: $$\Gamma(\frac{3}{2},x)=\int_x^\infty \sqrt{x}e^{-x}\, \dd x\leq 2\sqrt{x} e^{-x}.$$
\end{lemma}
\begin{proof}
    Case 3.2.2 with $B=2$ and $a=3/2$ in \cite{natalini2000inequalities} proves this inequality when $x>1$. Then, note that $x\mapsto \Gamma(3/2,x)$ is monotonically decreasing on $x\geq 0$, $x\mapsto 2 \sqrt{x}e^{-x}$ is also decreasing on $x\in[1/2, \infty)$. Furthermore, $2 \sqrt{x}e^{-x}|_{x=1}\approx 0.73679> 0.71009\approx \Gamma(3/2, 1/2)$.
\end{proof}

\begin{lemma}\label{lemma: mgf of sum of random variables, new EP}
    Suppose $X_1,\cdots,X_n$ are independent centered random variables such that for any $i\in[n]$, there exist some $M,\sigma>0$ such that $|X_i|\leq M$ and $\EE|X_i|^{1+\kappa}\leq \sigma^{1+\kappa}$ for some $\kappa\in[0,1]$. Denote $S_n=\sum_{i=1}^n X_i$. Then,
    \begin{equation}
        \EE\exp(\lambda S_n)\leq\exp\Bigl(n\Bigl(\frac{\sigma}{M}\Bigr)^{1+\kappa} (e^{\lambda M}-1-\lambda M)\Bigr).
    \end{equation}
\end{lemma}
\begin{proof}
    Since $X_1,\cdots,X_n$ are independent, $\EE\exp(\lambda S_n)=\prod_{i=1}^n\EE\exp(\lambda X_i)$. Then,
    \begin{equation}
        \begin{aligned}
            &\EE e^{\lambda X_i}
            = 1+\sum_{k=2}^\infty\frac{\lambda^k}{k!}\EE X_i^k \leq \exp\Bigl(\sum_{k=2}^\infty \frac{\lambda^k}{k!}\EE |X_i|^k\Bigr)\\
            \leq& \exp\Bigl(\sum_{k=2}^\infty \frac{\lambda^k}{k!}\EE |X_i|^{1+\kappa}\cdot M^{k-1-\kappa}\Bigr)
            \leq \exp\Bigl(\Bigl(\frac{\sigma}{M}\Bigr)^{1+\kappa}\cdot \sum_{k=2}^\infty \frac{(M\lambda)^k}{k!}\Bigr)\\
            \leq& \exp\Bigl(\Bigl(\frac{\sigma}{M}\Bigr)^{1+\kappa}\cdot (e^{\lambda M}-1-\lambda M)\Bigr)\\
        \end{aligned}
    \end{equation}
\end{proof}

\begin{lemma}[Theorem 3.1.10 (b) in \cite{gine2021mathematical}]\label{lemma: maximal inequality for random variables, new EP}
    Suppose $\EE e^{\lambda X_r}\leq\exp\Bigl(v_r\cdot(e^\lambda-1-\lambda)\Bigr)$ for $0<\lambda\leq 1/3$ for all $r\in[N]$. Let $v=\max_{r\in[N]}v_r$. Then,
    \begin{equation}
        \EE\Bigl[\max_{r\in[N]} X_r\Bigr]\leq\sqrt{2v\cdot\log N}+\frac{1}{3}\log N.
    \end{equation}
\end{lemma}

\begin{corollary}\label{coro: maximal inequality for EP, new EP}
    For empirical measure $\PP_n=\frac{1}{n}\sum_{i=1}^n\delta_{X_i}$ made up of i.i.d. observations from distribution $P$, assume the existence of some $M,\sigma>0$ such that $$\max_{r\in[N]}\|f_r\|_{L^\infty}\leq M,\quad\text{and}\quad \max_{r\in[N]}(\EE|f_r(X_1)|^{1+\kappa})^{\frac{1}{1+\kappa}}\leq \sigma.$$ Then the Rademacher complexity of $\{f_r\}_{r\in[N]}$ can be dominated by
    \begin{equation}
        \EE\Bigl[\max_{r\in[N]}\Bigl|\PP_n^\varepsilon  f_r\Bigr|\Bigr]\leq \frac{M^{\frac{1-\kappa}{2}}}{\sqrt{n}}\sqrt{2\sigma^{1+\kappa}\cdot \log(2N)}+\frac{M}{3n}\log(2N).
    \end{equation}
\end{corollary}
\begin{proof}
    Denote $\varepsilon_i$'s as independent Rademacher random variables and $g_r=\sgn(r)\cdot f_{|r|}$, $r\in\{\pm1, \pm2,\cdots,\pm N\}=:\pm[N]$. So, $$\EE\Bigl[\max_{r\in[N]}\Bigl|\PP_n^\varepsilon  f_r\Bigr|\Bigr]=\EE\Bigl[\max_{r\in\pm[N]}\PP_n^\varepsilon  g_r\Bigr]=\frac{1}{n}\EE\Bigl[\max_{\pm r\in[N]}\sum_{i=1}^n \varepsilon_i g_r(X_i)\Bigr].$$ Then applying Lemma \ref{lemma: mgf of sum of random variables, new EP} gives $$\EE\Bigl[\exp\Bigl(\frac{\lambda}{M}\sum_{i=1}^n \varepsilon_i \cdot g_r(X_i)\Bigr)\Bigr]\leq\exp\Bigl(n\frac{\sigma^{1+\kappa}}{M^{1+\kappa}}(e^\lambda-1-\lambda)\Bigr).$$ Thus, applying Lemma \ref{lemma: maximal inequality for random variables, new EP} gives $$\EE\Bigl[\max_{\pm r\in[N]}\frac{1}{M}\sum_{i=1}^n \varepsilon_i g_r(X_i)\Bigr]\leq \sqrt{\frac{2n}{M^{1+\kappa}}\cdot\sigma^{1+\kappa}\cdot \log(2N)}+\frac{1}{3}\log(2N).$$
\end{proof}

\begin{corollary}\label{coro: maximal inequality for EP, unbounded function}
    Let $\PP_n = \frac{1}{n} \sum_{i=1}^n \delta_{X_i}$ be the empirical measure based on i.i.d. observations from a distribution $P$. Consider a collection of functions $\{f_r\}_{r\in [N]}$, and suppose there exists an envelope function $F \in L^m(P)$ for some $m > 1$. Let $1+\kappa := m \land 2$ so that $\kappa \in (0,1]$, and assume
    $$
    \max_{r\in[N]}\|f_r\|_{L^{1+\kappa}(P)}\leq \sigma.
    $$
    Then the Rademacher complexity of $\{f_r\}_{r\in[N]}$ satisfies
    $$\begin{aligned}
        \EE\Bigl[\max_{r\in[N]}\Bigl|\PP_n^\varepsilon  f_r\Bigr|\Bigr]
        \leq& \sqrt{2} 3^{\frac{1-\kappa}{2(1+\kappa)}}\sigma^{\frac{1+\kappa}{2}}\|F\|_{L^{1+\kappa}(P)}^{\frac{1-\kappa}{2}}\Bigl(\frac{\log(2N)}{n}\Bigr)^{\frac{\kappa}{1+\kappa}}\\
        &\qquad+2\Bigl(\frac{\log(2N)}{3n}\Bigr)^{1-\frac{1}{m}}\|F\|_{L^{m}(P)}.
    \end{aligned}$$
\end{corollary}
\begin{proof}
    For any $M>0$,
    $$\begin{aligned}
        \EE\Bigl[\max_{r\in[N]}\Bigl|\PP_n^\varepsilon  f_r\Bigr|\Bigr]
        \leq&\EE\Bigl[\max_{r\in[N]}\Bigl|\PP_n^\varepsilon  f_r\cdot \II(F\leq M)\Bigr|\Bigr] + \EE\Bigl[\max_{r\in[N]}\Bigl|\PP_n^\varepsilon  f_r\II(F> M)\Bigr|\Bigr]\\
        \leq& \Bigl[\max_{r\in[N]}\Bigl|\PP_n^\varepsilon  f_r\cdot \II(F\leq M)\Bigr|\Bigr] + \EE[\PP_n  F\II(F> M)].\\
    \end{aligned}$$
    For the second term, since $F \in L^m(P)$, applying Markov's inequality gives $$\EE[\PP_n  F\II(F> M)]=\EE[F\II(F>M)]\leq \frac{\|F\|_{L^m(P)}^m}{M^{m-1}}.$$

    Now apply Corollary \ref{coro: maximal inequality for EP, new EP} to the truncated functions in the first term. This yields, for any $M > 0$,
    $$\begin{aligned}
        \EE\Bigl[\max_{r\in[N]}\Bigl|\PP_n^\varepsilon  f_r\Bigr|\Bigr]
        \leq&\sqrt\frac{M^{1-\kappa}}{n}\sqrt{2\sigma^{1+\kappa}\cdot \log(2N)}+\frac{M}{3n}\log(2N) + \frac{\|F\|_{L^m(P)}^m}{M^{m-1}}.
    \end{aligned}$$

    We now analyze two cases.
    \begin{enumerate}
        \item If $m\geq 2$, i.e. $\kappa=1$, applying Lemma \ref{lemma: minimal value between two terms} gives,
        $$\begin{aligned}
            \EE\Bigl[\max_{r\in[N]}\Bigl|\PP_n^\varepsilon  f_r\Bigr|\Bigr]
        \leq&\frac{1}{ \sqrt n}\sqrt{2\sigma^{2}\cdot \log(2N)}+\inf_{M>0}\Bigl(\frac{M}{3n}\log(2N) + \frac{\|F\|_{L^m(P)}^m}{M^{m-1}}\Bigr)\\
        \leq& \frac{1}{ \sqrt n}\sqrt{2\sigma^{2}\cdot \log(2N)}+2\Bigl(\frac{\log(2N)}{3n}\Bigr)^{1-\frac{1}{m}}\|F\|_{L^m(P)}.
        \end{aligned}$$

        \item If $\kappa\in(0,1)$, taking $M=\Bigl(\frac{3n}{\log(2N)}\Bigr)^{\frac{1}{1+\kappa}}\|F\|_{L^{m}(P)}$ gives,
        $$\begin{aligned}
            &\EE\Bigl[\max_{r\in[N]}\Bigl|\PP_n^\varepsilon  f_r\Bigr|\Bigr]\\
            \leq&\Bigl(\frac{3n}{\log(2N)}\Bigr)^{\frac{1-\kappa}{2(1+\kappa)}}\|F\|_{L^{m}(P)}^{\frac{1-\kappa}{2}}\sqrt\frac{2\sigma^{1+\kappa}\cdot \log(2N)}{n}
            +\Bigl(\frac{\log(2N)}{3n}\Bigr)^{\frac{\kappa}{1+\kappa}}\|F\|_{L^{m}(P)} \\
            &\qquad + \Bigl(\frac{3n}{\log(2N)}\Bigr)^{-\frac{\kappa}{1+\kappa}}\|F\|_{L^{m}(P)}\\
            =& \sqrt{2} 3^{\frac{1-\kappa}{2(1+\kappa)}}\sigma^{\frac{1+\kappa}{2}}\|F\|_{L^m(P)}^{\frac{1-\kappa}{2}}\Bigl(\frac{\log(2N)}{n}\Bigr)^{\frac{\kappa}{1+\kappa}}
            +2\Bigl(\frac{\log(2N)}{3n}\Bigr)^{\frac{\kappa}{1+\kappa}}\|F\|_{L^{m}(P)}.
        \end{aligned}$$
    \end{enumerate}

    In both cases, the desired bound follows.
\end{proof}

\begin{lemma}\label{lemma: expectation from concentration inequality}
    Suppose for the nonnegative random variable $X$ and any $t>0$, there are some positive $a,b$ such that $\PP(X\geq a\cdot\sqrt{1+t}+b)\leq 2\exp(-t)$, then $\EE[X]\leq c_0\cdot a+ b$, where
    \begin{equation}\label{eq: value of c0 in Lemma: expectation from concentration inequality}
        c_0= \sqrt{\log(2e)}+2e \int_{\sqrt{\log(2e)}}^\infty \exp(-x^2)\, \dd x\leq 1.618.
    \end{equation}
\end{lemma}
\begin{proof}
    As $\PP(X\geq a\cdot\sqrt{1+t}+b)=\PP(X/a\geq \sqrt{1+t}+b/a)\leq 2\exp(-t)$, consider the change of variable $u= \sqrt{1+t}+b/a$ with $u>1+b/a$. Equivalently, we have $t=(u-b/a)^2-1$. Then we have, $$\PP(X/a\geq u)\leq \begin{cases}
        1\qquad & 0<u\leq 1+b/a\\
        2e\cdot \exp(-(u-b/a)^2) & u>1+b/a.\\
    \end{cases}$$
    Since the probability of an event should be bounded by $1$, we have the following more precise estimation:
    $$\PP(X/a\geq u)\leq \begin{cases}
        1\qquad & 0<u\leq \sqrt{\log(2e)}+b/a\\
        2e\cdot \exp(-(u-b/a)^2) & u>\sqrt{\log(2e)}+b/a,\\
    \end{cases}$$
    which gives,
    $$\begin{aligned}
        \EE[X/a]
        =& \int_0^\infty \PP(X/a\geq u)\, \dd u\\
        \leq& 1\cdot (\sqrt{\log(2e)}+b/a)+2e\int_{\sqrt{\log(2e)}+b/a}^\infty \exp(-(u-b/a)^2)\ \dd u\\
        \leq& \sqrt{\log(2e)}+b/a + 2e\int_{\sqrt{\log(2e)}}^\infty \exp(-x^2)\, \dd x.
    \end{aligned}$$
    Hence, we conclude that
    $$\EE[X]\leq \Bigl(\sqrt{\log(2e)}+2e \int_{\sqrt{\log(2e)}}^\infty \exp(-x^2)\, \dd x\Bigr)a+b.$$
\end{proof}

\subsection{Proof of Theorem \ref{theorem: convergence of EP with L^1 integrable functions}}
    For $s=s_0,s_0+1,s_0+2,\cdots, S$ (where $s_0,S\in\NN\cup\NN+\{\frac{1}{2}\}$ to be fixed), let $\{f_{j,s}\}_{j=1}^{N_s}\subset \Fcal$ be the minimal $2^{-2s}$-covering of $\Fcal$ with norm  $L^{1+\kappa}(\PP_n)$, i.e. $N_s=\Ncal(2^{-2s},\Fcal,L^{1+\kappa}(\PP_n))$. Given $f\in\Fcal$, define its chain as $\{\pi_sf\}_{s=s_0}^S\subseteq\Fcal$, where
    $$\begin{aligned}
        \pi_{S}f=&\argmin\Bigl\{\|f-f_{j,S}\|_{L^{1+\kappa}(\PP_n)}:f_{j,S}\in \{f_{1,S},\cdots, f_{N_S, S}\}\Bigr\}\\
        \pi_{s}f=&\argmin\Bigl\{\|\pi_{s+1}f-f_{j,s}\|_{L^{1+\kappa}(\PP_n)}:f_{j,s}\in \{f_{1,s},\cdots, f_{N_s, s}\}\Bigr\},\qquad s_0\leq s\leq S-1.\\
    \end{aligned}$$
    With the help of this chain, we have the decomposition $f=\sum_{s=s_0}^{S-1} (\pi_{s+1}f-\pi_sf)+(f-\pi_{S}f)+\pi_{s_0}f$.

    Define $\Fcal_{s_0}^M=\{(f-\pi_{s_0}f)\II(F\leq M):f\in\Fcal\}$, $\Pi_{s_0}^M=\{\pi_{s_0}f\II(F\leq M):f\in\Fcal\}$, $\Fcal^{M}=\{f\II(F>M):f\in\Fcal\}$, and $\Fcal_M=\{f\II(F\leq M):f\in\Fcal\}$. Applying the Rademacher symmetrization theorem to the fixed class $\Fcal_M$ (Lemma 2.3.1 in \parencite{vaart2023empirical}) gives
    \begin{equation}\label{eq: first trade-off, convergence of EP with L^1 integrable functions, supp}
        \begin{aligned}
            &\EE^\ast\|\PP_n-P\|_\Fcal\\
            \leq{}&\EE^\ast\Bigl[\|\PP_n-P\|_{\Fcal_M}\Bigr]
            +\EE^\ast\Bigl[\|\PP_n-P\|_{\Fcal^M}\Bigr]\\
            \leq{}&2\EE^\ast\Bigl[\|\PP_n^\varepsilon\|_{\Fcal_M}\Bigr]
            +2\EE\Bigl[F\cdot\II(F>M)\Bigr]\\
        \leq{}&2\EE^\ast\Bigl[\|\PP_n^\varepsilon\|_{\Fcal_{s_0}^M}\Bigr]
        +2\EE^\ast\Bigl[\|\PP_n^\varepsilon\|_{\Pi_{s_0}^M}\Bigr]+2\EE\Bigl[F\cdot\II(F>M)\Bigr]
    \end{aligned}
    \end{equation}

    Let us first estimate $\EE^\ast[\|\PP_n^\varepsilon\|_{\Fcal_{s_0}^M}]$. Conditional on $\PP_n$, Hoeffding's inequality for the Rademacher random variables $\varepsilon_i$ gives, for any $t>0$,
    \[\begin{aligned}
        &\PP_\varepsilon \Bigl(|\PP_n^\varepsilon (\pi_{s+1}f-\pi_s f)\II(F\leq M)|\geq\frac{t}{\sqrt{n}}\Big|\PP_n\Bigr)\\
        \leq& 2\cdot\exp\Bigl(-\frac{t^2}{2\PP_n|(\pi_{s+1}f-\pi_sf)\II(F\leq M)|^2}\Bigr).
    \end{aligned}\]
    Note that in the denominator of the right-hand side, $\PP_n|(\pi_{s+1}f-\pi_sf)|^2$ may diverge as $n\to\infty$, since $f\in\Fcal$ may not be $L^2(P)$ integrable. However, with the existence of $M>0$ such that $|f(X_i)\II(F(X_i)\leq M)|\leq M$ for all $f\in\Fcal$ and all $X_i$'s, the inequality above can be further relaxed to
    \[\begin{aligned}
        &\PP_\varepsilon \Bigl(|\PP_n^\varepsilon (\pi_{s+1}f-\pi_s f)\II(F\leq M)|\geq\frac{t}{\sqrt{n}}\Big|\PP_n\Bigr)\\
        \leq& 2\exp\Bigl(-\frac{t^2}{2(2M)^{1-\kappa}\cdot \|\pi_{s+1}f-\pi_s f\|_{L^{1+\kappa}(\PP_n)}^{1+\kappa}}\Bigr),
    \end{aligned}\]
    which can be rewritten as, for any $t>0$, $$\PP_\varepsilon \Bigl(|\PP_n^\varepsilon (\pi_{s+1}f-\pi_s f)\II(F\leq M)|\geq2\sqrt\frac{M^{1-\kappa}\cdot 2^{-2(1+\kappa)s}\cdot t}{2^\kappa\cdot n}\Big|\PP_n\Bigr)\leq 2\exp(-t).$$

    Despite there can be infinite many $f$'s, there are at most $N_s\cdot N_{s+1}\leq N_{s+1}^2$ different $\pi_{s+1}f-\pi_s f$'s. Applying the union bound gives, for any $t>0$,
    $$\begin{aligned}
        &\PP_\varepsilon \Bigl(\sup_{f\in\Fcal}|\PP_n^\varepsilon (\pi_{s+1}f-\pi_s f)\II(F\leq M)|\geq 2^{1-(1+\kappa)s}\!\cdot\!\sqrt\frac{M^{1-\kappa}}{n}\!\cdot\!\bigl(\sqrt{t}+\sqrt{2\log N_{s+1}}\bigr)\Big|\PP_n\Bigr)\\
        \leq&\PP_\varepsilon \Bigl(\sup_{f\in\Fcal}|\PP_n^\varepsilon (\pi_{s+1}f-\pi_s f)\II(F\leq M)|\geq 2^{1-(1+\kappa)s}\cdot\sqrt\frac{M^{1-\kappa}\cdot (t+2\log N_{s+1})}{2^\kappa\cdot n}\Big|\PP_n\Bigr)\\
        \leq& 2\exp(-t).
    \end{aligned}$$

    Define $$A_s=2\cdot 2^{-(1+\kappa)s}\cdot \sqrt\frac{M^{1-\kappa}}{n}\cdot\Bigl(\sqrt{(1+s)(1+t)}+\sqrt{2\log N_{s+1}}\Bigr),$$ then, when $S>s_0$,
    $$\begin{aligned}
        &\PP_\varepsilon \Bigl(\sup_{f\in\Fcal}|\sum_{s=s_0}^{S-1}\PP_n^\varepsilon (\pi_{s+1}f-\pi_s f)\II(F\leq M)|\geq \sum_{s=s_0}^{S-1} A_s\Big|\PP_n\Bigr)\\
        \leq&\sum_{s=s_0}^{S-1}\PP_\varepsilon \Bigl(\sup_{f\in\Fcal}|\PP_n^\varepsilon (\pi_{s+1}f-\pi_s f)\II(F\leq M)|\geq A_s\Big|\PP_n\Bigr)\\
        \leq& 2\sum_{s=s_0}^{S-1}\exp(-(1+s)(1+t))\leq 2\exp(-t).
    \end{aligned}$$
    If $S=s_0$, the chain sum is empty and $\EE^\ast[\|\PP_n^\varepsilon\|_{\Fcal_{s_0}^M}\mid\PP_n]\leq\epsilon$ directly; the subsequent bounds remain valid in this case.
    We need to give an upper bound of $\sum_{s=s_0}^{S-1} A_s$: for the first component
    $$\begin{aligned}
        \sum_{s=s_0}^{S-1}2^{-(1+\kappa)s}\cdot \sqrt{1+s}
        \leq& \int_{s_0-1}^{S}\frac{\sqrt{x+1}}{2^{(1+\kappa)x}}\, \dd x\\
        =&e^{(1+\kappa)\log 2} \int_{s_0}^{S+1}\frac{\sqrt{u}}{e^{(1+\kappa)\log 2\cdot u}}\, \dd u\\
        =& \frac{2^{1+\kappa}}{((1+\kappa)\cdot\log 2)^{3/2}}\int_{(1+\kappa)\log 2\cdot s_0}^{(1+\kappa)\log 2\cdot (S+1)}\frac{\sqrt{u}}{e^u}\, \dd u\\
        \leq & \frac{2^{1+\kappa}}{((1+\kappa)\cdot\log 2)^{3/2}}\int_{(1+\kappa)\log 2\cdot s_0}^{\infty}\frac{\sqrt{u}}{e^u}\, \dd u\\
        = & \frac{2^{1+\kappa}}{((1+\kappa)\cdot\log 2)^{3/2}}\Gamma(\frac{3}{2}, {(1+\kappa)\log 2\cdot s_0}),\\
    \end{aligned}$$
    where by Lemma \ref{lemma: upper bound of incomplete Gamma integral}, the upper bound of the incomplete Gamma function is $\Gamma(3/2,x)\leq 2\sqrt{x}e^{-x}$ for $x>1/2$. Meanwhile, if $s_0\geq 1$, we have $(1+\kappa)\cdot\log 2\cdot s_0\geq \log 2\geq 1/2$. So,
    $$\begin{aligned}
        &\sum_{s=s_0}^{S-1}2^{-(1+\kappa)s}\cdot \sqrt{1+s}
        \leq \frac{2^{1+\kappa}}{((1+\kappa)\cdot\log 2)^{3/2}}\Gamma(\frac{3}{2}, {(1+\kappa)\log 2\cdot s_0})\\
        \leq& \frac{2^{1+\kappa}}{((1+\kappa)\cdot\log 2)^{3/2}}\cdot 2 ((1+\kappa)\cdot\log 2\cdot s_0)^{1/2} 2^{-(1+\kappa)s_0}\\
        =& \frac{2^{\kappa+2}}{(\kappa+1)\log 2}\sqrt{s_0} \cdot 2^{-(1+\kappa)s_0}.
    \end{aligned}$$

    For the second component in $\sum_{s=s_0}^{S-1} A_s$:
    $$\begin{aligned}
        & \sum_{s=s_0}^{S-1}2^{-(1+\kappa)s}\cdot\sqrt{2\log \Ncal(2^{-2(s+1)},\Fcal,L^{1+\kappa}(\PP_n))}\\
        =& \sum_{s=s_0}^{S-1}2^{-(1+\kappa)s}\cdot\frac{16}{3}2^{2s} \cdot(2^{-2(s+1)}-2^{-2(s+2)})\sqrt{2\log \Ncal(2^{-2(s+1)},\Fcal,L^{1+\kappa}(\PP_n))}\\
        =& \frac{16}{3}\sum_{s=s_0}^{S-1}(2^{-2(s+1)}-2^{-2(s+2)})\sqrt{\frac{2\log \Ncal(2^{-2(s+1)},\Fcal,L^{1+\kappa}(\PP_n))}{(2^{-2s})^{1-\kappa}}}\\
        \leq& \frac{16}{3}\int_{2^{-2(S+1)}}^{2^{-2(s_0+1)}}\sqrt{\frac{2\log \Ncal(x,\Fcal,L^{1+\kappa}(\PP_n))}{x^{1-\kappa}}}\, \dd x,
    \end{aligned}$$
    where the inequality is because $x\mapsto \sqrt{\log\Ncal(x,\Fcal,L^{1+\kappa}(\PP_n))/x^{1-\kappa}}$ is a decreasing function.

    For any $\epsilon\in (0, \sigma\land 2^{-2})$, take $S=\lceil\log_2(1/\epsilon)\rceil/2$ so that $$\sup_{f\in\Fcal}\|f-\pi_S f\|_{L^1(\PP_n)}\leq\sup_{f\in\Fcal}\|f-\pi_S f\|_{L^{1+\kappa}(\PP_n)}\leq 2^{-2S}\leq \epsilon\quad\text{and}\quad 2^{-2(S+1)}\geq \epsilon/8.$$ Then, by the decomposition $f-\pi_{s_0}f=\sum_{s=s_0}^{S-1} (\pi_{s+1}f-\pi_sf)+(f-\pi_{S}f)$, we conclude that, for any $t>0$ and any $\epsilon\in (0, \sigma\land2^{-2})$,
    $$\begin{aligned}
        &\PP_\varepsilon ^\ast\Bigl(\|\PP_n^\varepsilon  \|_{\Fcal_{s_0}^M}\geq \epsilon+2\sqrt\frac{M^{1-\kappa}}{n}\Bigl(\frac{2^{\kappa+2}}{(\kappa+1)\log 2}\sqrt{s_0} \cdot 2^{-(1+\kappa)s_0}\cdot \sqrt{1+t}\\
        &\qquad\qquad +\frac{16}{3}\int_{\epsilon/8}^{2^{-2(s_0+1)}}\sqrt{\frac{2\log \Ncal(x,\Fcal,L^{1+\kappa}(\PP_n))}{x^{1-\kappa}}}\, \dd x\Bigr)\Big|\PP_n\Bigr)\\
        \leq& \PP_\varepsilon ^\ast\Biggl(\sup_{f\in\Fcal}|\PP_n^\varepsilon \sum_{s=s_0}^{S-1}(\pi_{s+1}f-\pi_s f)\II(F\leq M)|\\
        &\qquad \geq 2\sqrt\frac{M^{1-\kappa}}{n}\Bigl(\frac{2^{\kappa+2}}{(\kappa+1)\log 2}\sqrt{s_0} \cdot 2^{-(1+\kappa)s_0}\cdot \sqrt{1+t}\\
        &\qquad\qquad +\frac{16}{3}\int_{\epsilon/8}^{2^{-2(s_0+1)}}\sqrt{\frac{2\log \Ncal(x,\Fcal,L^{1+\kappa}(\PP_n))}{x^{1-\kappa}}}\, \dd x\Bigr)\Big|\PP_n\Biggr)\\
        \leq& 2\exp(-t).
    \end{aligned}$$

    By Lemma \ref{lemma: expectation from concentration inequality}, for any $\epsilon\in (0, \sigma\land2^{-2})$, if $s_0\geq 1$,
    $$\begin{aligned}
        &\EE^\ast[\|\PP_n^\varepsilon\|_{\Fcal_{s_0}^M}\mid\PP_n]\leq 2c_0\cdot \sqrt\frac{M^{1-\kappa}}{n}\cdot \frac{2^{\kappa+2}}{(\kappa+1)\log 2}\sqrt{s_0} \cdot 2^{-(1+\kappa)s_0}\\
        & \qquad +\Bigl[\epsilon +\frac{32}{3}\sqrt\frac{M^{1-\kappa}}{n}\int_{\epsilon/8}^{2^{-2(s_0+1)}} \sqrt{\frac{2\log \Ncal(x,\Fcal,L^{1+\kappa}(\PP_n))}{x^{1-\kappa}}}\, \dd x\Bigr]\\
        \leq&\frac{8c_0}{\log2}\cdot \sqrt\frac{M^{1-\kappa}}{n}\cdot \sqrt{s_0} \cdot 2^{-(1+\kappa)s_0} \\
        &\qquad+\Bigl[\epsilon +\frac{32}{3}\sqrt\frac{M^{1-\kappa}}{n}\int_{\epsilon/8}^{2^{-2(s_0+1)}} \sqrt{\frac{2\log \Ncal(x,\Fcal,L^{1+\kappa}(\PP_n))}{x^{1-\kappa}}}\, \dd x\Bigr].\\
    \end{aligned}$$
    Now, let $$s_0=\left\{\begin{aligned}
    &\frac{\lceil\log_2(\frac{1}{\sigma}\lor 2^2)\rceil}{2} & \text{if }S-\frac{\lceil\log_2(\frac{1}{\sigma}\lor 2^2)\rceil}{2}\in\ZZ\\
    & \frac{\lceil\log_2(\frac{1}{\sigma}\lor 2^2)\rceil+1}{2} & \text{otherwise}.
    \end{aligned}\right.$$
    In such a scenario, as $\epsilon\leq 2^{-2}\land \sigma$, we have $1\leq\lceil\log_2(\frac{1}{\sigma}\lor2^2)\rceil/2\leq s_0\leq S$ and $S-s_0\in\ZZ$. Besides, since $x\mapsto \sqrt{x} 2^{-(1+\kappa)x}$ is monotonically decreasing on $x\geq 1$ for all $\kappa\in[0,1]$,
    $$\begin{aligned}
        \frac{\sqrt{s_0}}{2^{(1+\kappa)s_0}}
        \leq& \frac{\sqrt{x}}{2^{(1+\kappa)x}}\Big|_{x=\frac{\lceil\log_2(\frac{1}{\sigma}\lor 2^2)\rceil}{2}}\leq\sqrt\frac{\lceil\log_2(\frac{1}{\sigma}\lor 2^2)\rceil}{2}\cdot 2^{-(1+\kappa)\frac{\log_2(\frac{1}{\sigma}\lor 2^2)}{2}}\\
        =&\sqrt\frac{\lceil\log_2(\frac{1}{\sigma}\lor 2^2)\rceil}{2}\cdot \Bigl(\sigma^{\frac{1+\kappa}{2}}\land 2^{-(1+\kappa)}\Bigr).
    \end{aligned}$$

    As $2^{-2(s_0+1)}\leq 2^{-\log_2(\frac{1}{\sigma}\lor 2^2)-1}=\frac{\sigma}{2}\land 2^{-3}$, we have derived that, for any $M>0$ and any $\epsilon\in(0,2^{-2}\land \sigma)$,
    \begin{equation}\label{eq: first term, convergence of EP with L^1 integrable functions, supp}
        \begin{aligned}
            &\EE^\ast\Bigl[\|\PP_n^\varepsilon\|_{\Fcal_{s_0}^M}\Bigm|\PP_n\Bigr]\\
            \leq{}&\frac{8c_0}{\sqrt{2}\log2}\sqrt\frac{M^{1-\kappa}}{n}
            \sqrt{\lceil\log_2(\tfrac1\sigma\lor2^2)\rceil}
            \Bigl(\sigma^{\frac{1+\kappa}{2}}\land2^{-(1+\kappa)}\Bigr)\\
            &\quad+\epsilon+\frac{32}{3}\sqrt\frac{M^{1-\kappa}}{n}
            \int_{\epsilon/8}^{\frac\sigma2\land2^{-3}}
            \sqrt{\frac{2\log\Ncal(x,\Fcal,L^{1+\kappa}(\PP_n))}{x^{1-\kappa}}}\,\dd x.
        \end{aligned}
    \end{equation}
    Taking expectation over $\PP_n$ and applying Tonelli's theorem and Jensen's inequality gives
    \[\begin{aligned}
        &\EE^\ast\Bigl[\|\PP_n^\varepsilon\|_{\Fcal_{s_0}^M}\Bigr]\\
        \leq{}&\frac{8c_0}{\sqrt{2}\log2}\sqrt\frac{M^{1-\kappa}}{n}
        \sqrt{\lceil\log_2(\tfrac1\sigma\lor2^2)\rceil}
        \Bigl(\sigma^{\frac{1+\kappa}{2}}\land2^{-(1+\kappa)}\Bigr)\\
        &\quad+\epsilon+\frac{32}{3}\sqrt\frac{M^{1-\kappa}}{n}
        \int_{\epsilon/8}^{\frac\sigma2\land2^{-3}}
        \sqrt{\frac{2\EE[\log\Ncal(x,\Fcal,L^{1+\kappa}(\PP_n))]}{x^{1-\kappa}}}\,\dd x.
    \end{aligned}\]

    It remains to bound $\EE^\ast[\|\PP_n^\varepsilon\|_{\Pi_{s_0}^M}\mid\PP_n]$ in Equation \eqref{eq: first trade-off, convergence of EP with L^1 integrable functions, supp}. Conditional on $\PP_n$, the net $\Pi_{s_0}^M$ is fixed and contains at most $N_{s_0}$ elements. For every $g\in\Pi_{s_0}^M$ and $t>0$, the Rademacher moment-generating-function bound gives
    \[
        \EE[e^{\pm t\PP_n^\varepsilon g}\mid\PP_n]
        \leq \exp\Bigl(\frac{t^2\PP_ng^2}{2n}\Bigr).
    \]
    Since $\Pi_{s_0}^M\subseteq\Fcal_M$, the empirical second moment $\PP_ng^2$ is bounded by $\sup_{g\in\Fcal_M}\PP_ng^2$. Consequently,     \[
        \begin{aligned}
            &\EE^\ast\Bigl[
                \|\PP_n^\varepsilon\|_{\Pi_{s_0}^M}
                \Bigm|\PP_n
            \Bigr]\\
            &\leq
            \frac1t\log\EE\Bigl[
                \exp\bigl(t\|\PP_n^\varepsilon\|_{\Pi_{s_0}^M}\bigr)
                \Bigm|\PP_n
            \Bigr]\\
            &\leq
            \frac1t\log\Biggl(
                \sum_{g\in\Pi_{s_0}^M}
                \Bigl\{
                    \EE[e^{t\PP_n^\varepsilon g}\mid\PP_n]
                    +\EE[e^{-t\PP_n^\varepsilon g}\mid\PP_n]
                \Bigr\}
            \Biggr)\\
            &\leq
            \frac{\log(2N_{s_0})}{t}
            +\frac{t}{2n}\sup_{g\in\Fcal_M}\PP_ng^2.
        \end{aligned}
    \]
    Taking the infimum over $t>0$ gives
    \begin{equation}\label{eq: second term, convergence of EP with L^1 integrable functions, supp}
        \EE^\ast\Bigl[
            \|\PP_n^\varepsilon\|_{\Pi_{s_0}^M}
            \Bigm|\PP_n
        \Bigr]
        \leq
        \sqrt{
            \frac{2\log(2N_{s_0})}{n}
            \sup_{g\in\Fcal_M}\PP_ng^2
        }.
    \end{equation}
    Since $Pg^2\leq M^{1-\kappa}\sigma^{1+\kappa}$ for every $g\in\Fcal_M$, symmetrization and the Rademacher contraction inequality give
    \[\begin{aligned}
        \EE^\ast\Bigl[\sup_{g\in\Fcal_M}\PP_ng^2\Bigr]
        &\leq M^{1-\kappa}\sigma^{1+\kappa}
        +\EE^\ast\Bigl[\sup_{g\in\Fcal_M}|(\PP_n-P)g^2|\Bigr]\\
        &\leq M^{1-\kappa}\sigma^{1+\kappa}
        +2\EE^\ast\Bigl[\sup_{g\in\Fcal_M}|\PP_n^\varepsilon g^2|\Bigr]\\
        &\leq M^{1-\kappa}\sigma^{1+\kappa}
        +8M\EE^\ast\Bigl[\|\PP_n^\varepsilon\|_{\Fcal_M}\Bigr],
    \end{aligned}\]
    where the last inequality is because the map $t\mapsto(\max\{-M,\min\{t,M\}\})^2$ vanishes at zero and is $2M$-Lipschitz. Applying $\sqrt{2ab}\leq\lambda a+b/(2\lambda)$ to~\eqref{eq: second term, convergence of EP with L^1 integrable functions, supp} and taking expectation gives, for every $\lambda>0$,
    \[\begin{aligned}
        &\EE^\ast\Bigl[\|\PP_n^\varepsilon\|_{\Pi_{s_0}^M}\Bigr]\\
        \leq{}&\lambda\EE^\ast\Bigl[\sup_{g\in\Fcal_M}\PP_ng^2\Bigr]
        +\frac{\EE[\log(2N_{s_0})]}{2\lambda n}\\
        \leq{}&\lambda M^{1-\kappa}\sigma^{1+\kappa}
        +8M\lambda\EE^\ast\Bigl[\|\PP_n^\varepsilon\|_{\Fcal_M}\Bigr]
        +\frac{\EE[\log(2N_{s_0})]}{2\lambda n}\\
        \leq{}&\lambda M^{1-\kappa}\sigma^{1+\kappa}
        +8M\lambda\Bigl(
        \EE^\ast[\|\PP_n^\varepsilon\|_{\Fcal_{s_0}^M}]
        +\EE^\ast[\|\PP_n^\varepsilon\|_{\Pi_{s_0}^M}]\Bigr)+\frac{\EE[\log(2N_{s_0})]}{2\lambda n}.
    \end{aligned}\]
    If $\EE[\log(2N_{s_0})]=\infty$, the stated bound is immediate. Otherwise take
    \[
        \lambda=\min\left\{\frac{1}{16M},\,
        \sqrt{\frac{\EE[\log(2N_{s_0})]}{nM^{1-\kappa}\sigma^{1+\kappa}}}\right\},
    \]
    so $8M\lambda\leq1/2$, which gives
    \[\begin{aligned}
        \EE^\ast\Bigl[\|\PP_n^\varepsilon\|_{\Pi_{s_0}^M}\Bigr]
        \leq{}&\lambda M^{1-\kappa}\sigma^{1+\kappa}
        +\frac{1}{2}\Bigl(
        \EE^\ast[\|\PP_n^\varepsilon\|_{\Fcal_{s_0}^M}]
        +\EE^\ast[\|\PP_n^\varepsilon\|_{\Pi_{s_0}^M}]\Bigr)+\frac{\EE[\log(2N_{s_0})]}{2\lambda n},
    \end{aligned}\]
    or equivalently,
    \[\begin{aligned}
        &\EE^\ast\Bigl[\|\PP_n^\varepsilon\|_{\Pi_{s_0}^M}\Bigr]\\
        \leq{}&\EE^\ast\Bigl[\|\PP_n^\varepsilon\|_{\Fcal_{s_0}^M}\Bigr]
        +2\lambda M^{1-\kappa}\sigma^{1+\kappa}
        +\frac{\EE[\log(2N_{s_0})]}{\lambda n}\\
        \leq{}&\EE^\ast\Bigl[\|\PP_n^\varepsilon\|_{\Fcal_{s_0}^M}\Bigr]
        +\frac{3M^{\frac{1-\kappa}{2}}\sigma^{\frac{1+\kappa}{2}}}{\sqrt n}
        \sqrt{\EE[\log(2N_{s_0})]}+\frac{16M}{n}\EE[\log(2N_{s_0})].
    \end{aligned}\]
    Finally, the choice of $s_0$ gives
    \[
        2^{-2s_0}\geq2^{-\lceil\log_2(\frac1\sigma\lor2^2)\rceil-1}
        \geq\frac\sigma4\land2^{-4},
    \]
    and hence $N_{s_0}\leq\Ncal(\sigma/4\land2^{-4},\Fcal,L^{1+\kappa}(\PP_n))$.

    Substituting this bound and the preceding estimate of $\EE^\ast[\|\PP_n^\varepsilon\|_{\Fcal_{s_0}^M}]$ into Equation \eqref{eq: first trade-off, convergence of EP with L^1 integrable functions, supp} and taking the infimum over $\epsilon$ yields, for any $M>0$,
    \begin{equation}\label{eq: final upper bound of theorem: convergence of EP with L^1 integrable functions}
        \begin{aligned}
            &\EE^\ast\|\PP_n-P\|_\Fcal
            \leq \frac{32c_0}{\sqrt{2}\log2}\cdot \sqrt\frac{M^{1-\kappa}}{n}\cdot \sqrt{\lceil\log_2(\frac{1}{\sigma}\lor 2^2)\rceil} \cdot \Bigl(\sigma^\frac{1+\kappa}{2}\land 2^{-(1+\kappa)}\Bigr)\\
            &\qquad +\inf_{\epsilon\in (0, \sigma\land 2^{-2})}\Bigl[4\epsilon +\frac{128\sqrt{2}}{3}\sqrt\frac{M^{1-\kappa}}{n}\int_{\epsilon/8}^{\frac{\sigma}{2}\land 2^{-3}} \sqrt{\frac{\EE[\log \Ncal(x,\Fcal,L^{1+\kappa}(\PP_n))]}{x^{1-\kappa}}}\, \dd x\Bigr]\\
            &\qquad + \frac{6\cdot M^\frac{1-\kappa}{2}}{\sqrt{n}} \sigma^\frac{1+\kappa}{2}\cdot  \sqrt{ \EE[\log(2\Ncal(\sigma/4\land2^{-4}, \Fcal, L^{1+\kappa}(\PP_n)))]}\\
            &\qquad +\frac{32 M}{n} \EE[\log(2\Ncal(\sigma/4\land2^{-4}, \Fcal, L^{1+\kappa}(\PP_n)))] + 2\EE\Bigl[F\cdot \II(F> M)\Bigr],
        \end{aligned}
    \end{equation}
    where we use Equation \eqref{eq: value of c0 in Lemma: expectation from concentration inequality},
    $$\begin{aligned}
        \frac{32c_0}{\sqrt{2}\log2}= \frac{16\sqrt2}{\log 2}\Bigl(\sqrt{\log(2e)}+2e \int_{\sqrt{\log(2e)}}^\infty \exp(-x^2)\, \dd x\Bigr)\leq 52.8172.
    \end{aligned}$$
    Also, $128\sqrt{2}/3\leq60.3400$, completing the proof.

\subsection{Proof of Corollary~\ref{corollary: convergence of EP with L^1 integrable functions}}
    Let $\widetilde\Fcal:=\{f/(4\sigma):f\in\Fcal\}$, with envelope
    $\widetilde F:=F/(4\sigma)$. Since
    $\sup_{\tilde f\in\widetilde\Fcal}\|\tilde f\|_{L^2(P)}\leq 1/4$,
    Theorem~\ref{theorem: convergence of EP with L^1 integrable functions},
    applied to $\widetilde\Fcal$ with $\kappa=1$ and radius bound $1/4$,
    gives, for every $\widetilde M>0$,
    \[
    \begin{aligned}
        \EE^\ast\|\PP_n-P\|_{\widetilde\Fcal}
        &\leq \frac{52.8172\sqrt{2}}{4\sqrt n}\\
        &\quad + \inf_{0<\tilde\epsilon<1/4}\Bigl[
            4\tilde\epsilon + \frac{60.3400}{\sqrt n}
            \int_{\tilde\epsilon/8}^{1/8}
            \sqrt{\EE[\log\Ncal(x,\widetilde\Fcal,L^2(\PP_n))]}\,\dd x
            \Bigr]\\
        &\quad + \frac{6}{4\sqrt n}
            \sqrt{\EE[\log(2\Ncal(1/16,\widetilde\Fcal,L^2(\PP_n)))]}\\
        &\quad + \frac{32\widetilde M}{n}
            \EE[\log(2\Ncal(1/16,\widetilde\Fcal,L^2(\PP_n)))]
            + 2\EE[\widetilde F\cdot\II(\widetilde F>\widetilde M)].
    \end{aligned}
    \]

    Since $\log(2\Ncal)\geq\log 2$, the first term is bounded by a universal constant times the third term. In particular,
    \[
        \frac{52.8172\sqrt{2}}{4\sqrt n}\leq \frac{52.8172\sqrt{2}}{4\sqrt{\log 2}\sqrt n}\sqrt{\EE[\log(2\Ncal(1/16,\widetilde\Fcal,L^2(\PP_n)))]}
    \]
    Thus, for every $0<\tilde\epsilon<1/4$ and $\widetilde M>0$,
    \[
    \begin{aligned}
        \EE^\ast\|\PP_n-P\|_{\widetilde\Fcal}
        &\leq \Bigl(\frac{52.8172\sqrt{2}}{4\sqrt{\log 2}} + \frac{3}{2} \Bigr)\frac{1}{\sqrt n}\sqrt{\EE[\log(2\Ncal(1/16,\widetilde\Fcal,L^2(\PP_n)))]}\\
        &\quad +
            4\tilde\epsilon + \frac{60.3400}{\sqrt n}
            \int_{\tilde\epsilon/8}^{1/8}
            \sqrt{\EE[\log\Ncal(x,\widetilde\Fcal,L^2(\PP_n))]}\,\dd x
            \\
        &\quad + \frac{32\widetilde M}{n}
            \EE[\log(2\Ncal(1/16,\widetilde\Fcal,L^2(\PP_n)))]
            + 2\EE[\widetilde F\cdot\II(\widetilde F>\widetilde M)].
    \end{aligned}
    \]
    Since $\log\Ncal\leq\log(2\Ncal)$, we may enlarge the integrand. To absorb the first term into the covering integral, note that
    $[1/32,1/16]\subseteq[\tilde\epsilon/8,1/8]$. By monotonicity of the
    covering number in its radius,
    \[
        \int_{\tilde\epsilon/8}^{1/8}
        \sqrt{\EE[\log(2\Ncal(x,\widetilde\Fcal,L^2(\PP_n)))]}\,\dd x
        \geq \frac{1}{32}
        \sqrt{\EE[\log(2\Ncal(1/16,\widetilde\Fcal,L^2(\PP_n)))]}.
    \]
    Thus,
    \[
    \begin{aligned}
        &\EE^\ast\|\PP_n-P\|_{\widetilde\Fcal}\\
        \leq &4\tilde\epsilon + \Bigl(32\bigl(\frac{52.8172\sqrt{2}}{4\sqrt{\log 2}} + \frac{3}{2} \bigr)+60.3400\Bigr)\frac{1}{\sqrt n}\int_{\tilde\epsilon/8}^{1/8}
            \sqrt{\EE[\log(2\Ncal(x,\widetilde\Fcal,L^2(\PP_n)))]}\,\dd x
            \\
        &+ \frac{32\widetilde M}{n}
            \EE[\log(2\Ncal(1/16,\widetilde\Fcal,L^2(\PP_n)))]
            + 2\EE[\widetilde F\cdot\II(\widetilde F>\widetilde M)],
    \end{aligned}
    \]
    where
    \[
        32\bigl(\frac{52.8172\sqrt{2}}{4\sqrt{\log 2}} + \frac{3}{2} \bigr)+60.3400\leq 826.0808.
    \]

    We now return to the original scale. By $\EE^\ast\|\PP_n-P\|_{\Fcal} = (4\sigma)\EE^\ast\|\PP_n-P\|_{\widetilde\Fcal}$ and $\Ncal(x,\widetilde\Fcal,L^2(\PP_n))
    =\Ncal(4\sigma x,\Fcal,L^2(\PP_n))$, set
    $\epsilon=4\sigma\tilde\epsilon$ and $M=4\sigma\widetilde M$. Hence, for every
    $0<\epsilon<\sigma$ and $M>0$,
    \[
    \begin{aligned}
        \EE^\ast\|\PP_n-P\|_\Fcal
        &\leq 4\epsilon
        + \frac{826.0808}{\sqrt n}\int_{\epsilon/8}^{\sigma/2}
          \sqrt{\EE[\log(2\Ncal(u,\Fcal,L^2(\PP_n)))]}\,\dd u\\
        &\quad + \frac{32M}{n}
          \EE[\log(2\Ncal(\sigma/4,\Fcal,L^2(\PP_n)))]
          + 2\EE[F\cdot\II(F>M)].
    \end{aligned}
    \]
    Taking the infima over $\epsilon$ and $M$ separately yields the claim.

\subsection{Proof of Proposition \ref{proposition: average covering entropy domminated for parametrized class}}

    Let $\{\theta_i\}_{i=1}^N$ be an $h$-net of $(\Theta,d)$ with $N=\Ncal(h,\Theta,d)$.
    For each $\theta\in\Theta$, choose $i$ with $d(\theta,\theta_i)\le h$. The Lipschitz condition gives
    \[
        \|f_\theta-f_{\theta_i}\|_{L^{1+\kappa}(\PP_n)}
        \le d(\theta,\theta_i)\|G\|_{L^{1+\kappa}(\PP_n)}
        \le h\|G\|_{L^{1+\kappa}(\PP_n)}.
    \]
    Thus,
    \[\begin{aligned}
        &\log\Ncal(h\|G\|_{L^{1+\kappa}(\PP_n)}, \Fcal, L^{1+\kappa}(\PP_n))\leq \log\Ncal(h, \Theta, d).
    \end{aligned}\]
    For $m>0$, on the event $\|G\|_{L^{1+\kappa}(\PP_n)}\le m\|G\|_{L^{1+\kappa}(P)}$, apply this bound at mesh $h/m$. By monotonicity and the assumed parameter entropy bound,
    \[
        \PP\!\left(\log\Ncal(h\|G\|_{L^{1+\kappa}(P)},\Fcal,L^{1+\kappa}(\PP_n))
        >D_\Theta h^{-\gamma}m^\gamma\right)\le\psi_n(m).
    \]
    For $\gamma>0$, integration with respect to $\dd(m^\gamma)$ gives the claimed expectation bound. For $\gamma=0$, the log covering number is at most $D_\Theta$ at every positive radius by the same net bound.

\subsection{Proof of Proposition \ref{proposition: average covering entropy dominated by L-infty covering entropy}}

    For any $f,g\in\Fcal$ and any empirical measure $\PP_n$, we have
    $$\begin{aligned}
        &\|f-g\|_{L^{1+\kappa}(\PP_n)}^{1+\kappa}\\
        =&\int |f(x)-g(x)|^{1+\kappa}\PP_n(\dd x)
        = \int \inner{x}^{(1+\kappa)\eta}\inner{x}^{-(1+\kappa)\eta} |f(x)-g(x)|^{1+\kappa}\PP_n(\dd x)\\
        \leq& \int \inner{x}^{(1+\kappa)\eta}\PP_n(\dd x)\cdot \|f-g\|_{L^\infty(\inner{\cdot}^{-\eta})}^{1+\kappa}=\PP_n\inner{x}^{(1+\kappa)\eta}\cdot \|f-g\|_{L^\infty(\inner{\cdot}^{-\eta})}^{1+\kappa}.\\
    \end{aligned}$$

    Thus, with probability $1-\psi_n(m)$, $\PP_n\inner{x}^{(1+\kappa)\eta}\leq (m\cdot \mu)^{1+\kappa}$, which indicates that for any $f,g\in\Fcal$, $$\|f-g\|_{L^{1+\kappa}(\PP_n)}\leq (m\cdot \mu)\cdot \|f-g\|_{L^\infty(\inner{\cdot}^{-\eta})}.$$ Consequently, if $\{f_i\}_{i=1}^N$ is an $h/(m\cdot \mu)$-covering of $\Fcal$ with norm $L^\infty(\inner{\cdot}^{-\eta})$, then it further an $h$-covering with norm $L^2(\PP_n)$.

    So, if $\gamma=0$ (i.e. $\Fcal$ is finite), we have derived $\log\Ncal(h,\Fcal,L^{1+\kappa}(\PP_n))\leq D_\Fcal$ for any $\PP_n$. Now, let us assume $\gamma>0$. In this case,
    $$\begin{aligned}
        1-\psi_n(m)\geq& \PP\Bigl(\log\Ncal(h,\Fcal,L^{1+\kappa}(\PP_n))\leq \log\Ncal(\frac{h}{m\cdot \mu},\Fcal,L^\infty(\inner{\cdot}^{-\eta}))\Bigr)\\
        \geq& \PP\Bigl(\log\Ncal(h,\Fcal,L^{1+\kappa}(\PP_n))\leq D_\Fcal\cdot \bigl(\frac{m\cdot \mu}{h}\bigr)^{\gamma}\Bigr),\\
    \end{aligned}$$
    which can be rewritten as for any $m>0$, $$\PP\Bigl(\log\Ncal(h,\Fcal,L^{1+\kappa}(\PP_n))\geq m\Bigr)\leq \psi_n\Bigl(\frac{h}{\mu}\cdot\bigl(\frac{m}{D_\Fcal}\bigr)^{1/\gamma}\Bigr).$$ Since $\log\Ncal(h,\Fcal,L^{1+\kappa}(\PP_n))$ is nonnegative, taking the integral gives
    $$\begin{aligned}
        \EE[\log\Ncal(h,\Fcal,L^{1+\kappa}(\PP_n))]
        =& \int_0^\infty \PP\Bigl(\log\Ncal(h,\Fcal,L^{1+\kappa}(\PP_n))\geq m\Bigr)\, \dd m\\
        \leq& \int_0^\infty \psi_n\Bigl(\frac{h}{\mu}\cdot\bigl(\frac{m}{D_\Fcal}\bigr)^{1/\gamma}\Bigr)\, \dd m\\
        =& D_\Fcal\cdot\bigl(\frac{\mu}{h}\bigr)^\gamma\cdot \int_0^\infty \psi_n(u)\, \dd (u^\gamma).
    \end{aligned}$$

\subsection{Proof of Theorem \ref{theorem: convergence of EP with L^infty integrable functions}}

    As discussed at the beginning of the proof of Theorem \ref{theorem: convergence of EP with L^1 integrable functions}, it suffices for us to consider the upper bound of $\EE^\ast\|\PP_n^\varepsilon\|_\Fcal$.

    \textbf{The first step} is to define a partition of $\Fcal$ based on two chains. For $s\in\{1,2,\cdots,S\}=:[S]$ ($S$ to be determined later), define $N_s=2^{2^s}$, and $$\epsilon_s=\inf\{\epsilon:\Ncal(\epsilon,\Fcal,L^\infty(w))\leq N_s\},\qquad \epsilon_s^\prime=\inf\{\epsilon:\Ncal(\epsilon,\Fcal,L^{1+\kappa}(P))\leq N_s\}.$$
    For any $\delta>0$, define $\tilde\epsilon_s=(1+\delta)\epsilon_s$ and $\tilde\epsilon_s^\prime=(1+\delta)\epsilon_s^\prime$. Then, we may take $\{f_{i,s}\}_{i=1}^{N_s}\subseteq \Fcal$ as an $\tilde\epsilon_s$-covering of $\Fcal$ with $L^\infty(w)$ norm; and $\{g_{j,s}\}_{j=1}^{N_s}\subseteq\Fcal$ as an $\tilde\epsilon_s^\prime$-covering of $\Fcal$ with $L^{1+\kappa}(P)$ norm. \footnote{Here, we pick $\tilde\epsilon_s$ ($\tilde\epsilon_s^\prime$) covering set of $\Fcal$ instead of $\epsilon_s$ ($\epsilon_s^\prime$) counterpart. This is because we do not assume the covering number function $x\mapsto \Ncal(x,\Fcal,d)$ is right continuous, despite it can be for sequentially compact $\Fcal$. As a result, it is possible that $\Ncal(\epsilon_s,\Fcal,L^\infty(w))> N_s$ ($\Ncal(\epsilon_s^\prime,\Fcal,L^{1+\kappa}(P))> N_s$), meaning that such covering sets of size $N_s$ does not exist.}

    Now, define a partition of $\Fcal$ as, $$\Acal_{ijs}=\big\{f\in\Fcal: \|f-f_{i,s}\|_{L^\infty(w)}\leq \tilde\epsilon_s,\|f-g_{js}\|_{L^{1+\kappa}(P)}\leq \tilde\epsilon_s^\prime\big\}.$$
    For those nonempty $\Acal_{ijs}$, pick $h_{ijs}\in\Acal_{ijs}$ and define the projector operator as, for $s\in[S]$ and $f\in\Fcal$: $$\pi_sf=h_{ijs},\quad\text{if }f\in\Acal_{ijs}.$$
    Then we have the decomposition $$f=(f-\pi_Sf)+\sum_{s=1}^{S-1}(\pi_{s+1}f-\pi_sf)+\pi_1f.$$
    Consequently,
    $$\begin{aligned}
        &\EE^\ast \|\PP_n^\varepsilon\|_{\Fcal}
        =\EE^\ast\Bigl[\sup_{f\in\Fcal}\Bigl|\PP_n^\varepsilon f\Bigr|\Bigr]\\
        \leq& \EE^\ast\Bigl[\sup_{f\in\Fcal}\Bigl|\PP_n^\varepsilon (f-\pi_Sf)\Bigr|\Bigr] + \sum_{s=1}^{S-1}\EE^\ast\Bigl[\sup_{f\in\Fcal}\Bigl|\PP_n^\varepsilon (\pi_{s+1}f-\pi_sf)\Bigr|\Bigr]+ \EE^\ast\Bigl[\sup_{f\in\Fcal}\Bigl|\PP_n^\varepsilon \pi_1f\Bigr|\Bigr].
    \end{aligned}$$
    For the first term on the RHS, note that
    $$\begin{aligned}
        |f(x)-\pi_Sf(x)|
        \leq& \|f-\pi_Sf\|_{L^\infty(w)}\cdot(1/w)(x)\\
        \leq& (\|f-g_{jS}\|_{L^\infty(w)}+\|g_{jS}-h_{ijS}\|_{L^\infty(w)})\cdot (1/w)(x)\\
        \leq& \frac{2\tilde\epsilon_S}{w(x)},
    \end{aligned}$$
    where the second inequality is due to the triangle inequality, and the last one is due to the definition of $\Acal_{ijS}$.

    Therefore, $$\lim_{S\to\infty}\EE^\ast\Bigl[\sup_{f\in\Fcal}\Bigl|\PP_n^\varepsilon (f-\pi_Sf)\Bigr|\Bigr]\leq \lim_{S\to\infty} 2\tilde\epsilon_S\cdot \EE^\ast[\PP_n1/w]=2(1+\delta)\|1/w\|_{L^1(P)}\cdot \lim_{S\to\infty}\epsilon_S=0.$$
    As a result, we conclude that
    \begin{equation}\label{eq: decomposition of f into summation of projection operators, theorem: convergence of EP with L^infty integrable functions}
        \begin{aligned}
            \EE^\ast \|\PP_n^\varepsilon\|_{\Fcal}\leq& \sum_{s=1}^{\infty}\EE^\ast\Bigl[\sup_{f\in\Fcal}\Bigl|\PP_n^\varepsilon (\pi_{s+1}f-\pi_sf)\Bigr|\Bigr]+ \EE^\ast\Bigl[\sup_{f\in\Fcal}\Bigl|\PP_n^\varepsilon \pi_1f\Bigr|\Bigr].
        \end{aligned}
    \end{equation}

    \textbf{Step 2.} Let us analyze the second term in Equation \eqref{eq: decomposition of f into summation of projection operators, theorem: convergence of EP with L^infty integrable functions}. Since $N_1=4$, $$|\{\pi_1f:f\in\Fcal\}|\leq N_1^2=16.$$ Moreover, $\pi_1f\in\Fcal$ implies that $\|\pi_1 f\|_{L^{1+\kappa}(P)}\leq \sigma$ and $F$ is the envelope function of $\{\pi_1f:f\in\Fcal\}$.

    Now, applying Corollary \ref{coro: maximal inequality for EP, unbounded function} gives,
    \begin{equation}\label{eq: upper bound of the empirical process of pi1 f, theorem: convergence of EP with L^infty integrable functions}
        \begin{aligned}
            \EE\Bigl[\sup_{f\in\Fcal}\Bigl|\PP_n^\varepsilon  \pi_1f\Bigr|\Bigr]
            \leq \sqrt{2} 3^{\frac{1-\kappa}{2(1+\kappa)}}\sigma^{\frac{1+\kappa}{2}}\|F\|_{L^{1+\kappa}(P)}^{\frac{1-\kappa}{2}}\Bigl(\frac{\log(32)}{n}\Bigr)^{\frac{\kappa}{1+\kappa}}
            +2\Bigl(\frac{\log(32)}{3n}\Bigr)^{1-\frac{1}{m}}\|F\|_{L^{m}(P)}.
        \end{aligned}
    \end{equation}

    \textbf{Step 3.} To control the first term in Equation \eqref{eq: decomposition of f into summation of projection operators, theorem: convergence of EP with L^infty integrable functions}, we first bound $\EE^\ast\Bigl[\sup_{f\in\Fcal}\Bigl|\PP_n^\varepsilon (\pi_{s+1}f-\pi_sf)\Bigr|\Bigr]$ for each $s\geq 1$. To apply Corollary \ref{coro: maximal inequality for EP, unbounded function}, we need the following results:
    \begin{enumerate}
        \item The cardinality can be controlled by: $$\Bigl|\Bigl\{\pi_{s+1}f-\pi_sf:f\in\Fcal\Bigr\}\Bigr|\leq N_s^2N_{s+1}^2=2^{3\cdot 2^{s+1}}=N_{s+1}^3.$$

        \item For any $f\in\Fcal$, the triangle inequality gives, for some $i^\prime, j^\prime\in [N_{s+1}]$ and $i, j\in [N_{s}]$,
        $$\begin{aligned}
            &\|\pi_{s+1}f-\pi_sf\|_{L^{1+\kappa}(P)}
            = \|h_{i^\prime j^\prime s+1}- h_{i j s}\|_{L^{1+\kappa}(P)}\\
            \leq& \|h_{i^\prime j^\prime s+1}- g_{j^\prime, s+1}\|_{L^{1+\kappa}(P)} + \|g_{j^\prime, s+1}- f\|_{L^{1+\kappa}(P)} \\
            &\qquad + \|f- g_{j, s}\|_{L^{1+\kappa}(P)} + \|g_{j, s}- h_{i j s}\|_{L^{1+\kappa}(P)}\\
            \leq& 2\tilde\epsilon_{s+1}^\prime + 2\tilde\epsilon_s^\prime\leq 4\tilde\epsilon_{s}^\prime.
        \end{aligned}$$

        Consequently, we denote
        \begin{equation}\label{eq: definition of sigma_s+1, theorem: convergence of EP with L^infty integrable functions}
            \sigma_{s}:=\sup_{f\in\Fcal} \|\pi_{s+1}f-\pi_sf\|_{L^{1+\kappa}(P)}\leq 4\tilde\epsilon_{s}^\prime.
        \end{equation}

        \item Similarly, for any $f\in\Fcal$ and any $x$, there exist some $i^\prime, j^\prime\in [N_{s+1}]$ and $i, j\in [N_{s}]$, such that
        $$\begin{aligned}
            &|\pi_{s+1}f(x)-\pi_sf(x)|
            = |h_{i^\prime j^\prime s+1}(x)- h_{i j s}(x)|\\
            \leq& |h_{i^\prime j^\prime s+1}(x)- f_{j^\prime, s+1}(x)| + |f_{i^\prime, s+1}(x)- f(x)|\\
            &\qquad\qquad + |f(x)- f_{i, s}(x)| + |f_{i, s}(x)(x)- h_{i j s}(x)|\\
            \leq& \frac{4\tilde\epsilon_{s}}{w(x)}.
        \end{aligned}$$
        Consequently, we denote the local envelope function of $\{\pi_{s+1}f-\pi_sf:f\in\Fcal\}$ as
        \begin{equation}\label{eq: definition of  local envelope function, theorem: convergence of EP with L^infty integrable functions}
            F_s(x):=\frac{4\tilde\epsilon_s}{w(x)}.
        \end{equation}
    \end{enumerate}

    Now, applying Corollary \ref{coro: maximal inequality for EP, unbounded function} gives
    $$\begin{aligned}
        &\EE\Bigl[\sup_{f\in\Fcal}\Bigl|\PP_n^\varepsilon (\pi_{s+1}f-\pi_sf)\Bigr|\Bigr]\\
        \leq& \sqrt{2} 3^{\frac{1-\kappa}{2(1+\kappa)}}\sigma_s^{\frac{1+\kappa}{2}}\|F_s\|_{L^{1+\kappa}(P)}^{\frac{1-\kappa}{2}}\Bigl(\frac{\log(2N_{s+1}^3)}{n}\Bigr)^{\frac{\kappa}{1+\kappa}}+2\Bigl(\frac{\log(2N_{s+1}^3)}{3n}\Bigr)^{1-\frac{1}{m}}\|F_s\|_{L^{m}(P)}\\
        \leq& \sqrt{2} 3^{\frac{1-\kappa}{2(1+\kappa)}}\cdot n^{-\frac{\kappa}{1+\kappa}}\Bigl(\sigma_s \log(2N_{s+1}^3)^{\frac{\kappa}{1+\kappa}}\Bigr)^{\frac{1+\kappa}{2}}\Bigl(\|F_s\|_{L^{1+\kappa}(P)}\log(2N_{s+1}^3)^{\frac{\kappa}{1+\kappa}}\Bigr)^{\frac{1-\kappa}{2}}\\
        &\qquad +2\cdot 3^{\frac{1}{m}-1} n^{\frac{1}{m}-1}\Bigl(\|F_s\|_{L^{m}(P)}\log(2N_{s+1}^3)^{1-\frac{1}{m}}\Bigr)\\
    \end{aligned}$$
    Substituting $\sigma_{s}$ and $F_s$ based on Equations \eqref{eq: definition of sigma_s+1, theorem: convergence of EP with L^infty integrable functions} and \eqref{eq: definition of  local envelope function, theorem: convergence of EP with L^infty integrable functions} yields
    $$\begin{aligned}
        &\EE\Bigl[\sup_{f\in\Fcal}\Bigl|\PP_n^\varepsilon (\pi_{s+1}f-\pi_sf)\Bigr|\Bigr]\\
        \leq& 2^{\frac{5}{2}} 3^{\frac{1-\kappa}{2(1+\kappa)}}\cdot n^{-\frac{\kappa}{1+\kappa}}\Bigl(\tilde\epsilon_s^\prime \log(2N_{s+1}^3)^{\frac{\kappa}{1+\kappa}}\Bigr)^{\frac{1+\kappa}{2}}\Bigl(\tilde\epsilon_s\log(2N_{s+1}^3)^{\frac{\kappa}{1+\kappa}}\Bigr)^{\frac{1-\kappa}{2}}\cdot \|1/w\|_{L^{1+\kappa}(P)}^{\frac{1-\kappa}{2}}\\
        &\qquad +8\cdot 3^{\frac{1}{m}-1} n^{\frac{1}{m}-1}\Bigl(\tilde\epsilon_s\log(2N_{s+1}^3)^{1-\frac{1}{m}}\Bigr)\cdot \|1/w\|_{L^{m}(P)}\\
    \end{aligned}$$

    Now, summing over $s$ yields
    \begin{equation}\label{eq: upper bound of the empirical process of pis=1-pis f, theorem: convergence of EP with L^infty integrable functions}
        \begin{aligned}
            &\sum_{s=1}^\infty\EE\Bigl[\sup_{f\in\Fcal}\Bigl|\PP_n^\varepsilon (\pi_{s+1}f-\pi_sf)\Bigr|\Bigr]\\
            \leq& \frac{2^{\frac{5}{2}} 3^{\frac{1-\kappa}{2(1+\kappa)}}}{n^{\frac{\kappa}{1+\kappa}}}\cdot \|1/w\|_{L^{1+\kappa}(P)}^{\frac{1-\kappa}{2}}\sum_{s=1}^\infty \Bigl(\tilde\epsilon_s^\prime \log(2N_{s+1}^3)^{\frac{\kappa}{1+\kappa}}\Bigr)^{\frac{1+\kappa}{2}}\Bigl(\tilde\epsilon_s\log(2N_{s+1}^3)^{\frac{\kappa}{1+\kappa}}\Bigr)^{\frac{1-\kappa}{2}}\\
            &\qquad +8\cdot 3^{\frac{1}{m}-1} n^{\frac{1}{m}-1} \cdot \|1/w\|_{L^{m}(P)} \sum_{s=1}^\infty\Bigl(\tilde\epsilon_s\log(2N_{s+1}^3)^{1-\frac{1}{m}}\Bigr)\\
            \leq& \frac{2^{\frac{5}{2}} 3^{\frac{1-\kappa}{2(1+\kappa)}}}{n^{\frac{\kappa}{1+\kappa}}}\cdot \|1/w\|_{L^{1+\kappa}(P)}^{\frac{1-\kappa}{2}} \Bigl(\underbrace{\sum_{s=1}^\infty\tilde\epsilon_s^\prime \log(2N_{s+1}^3)^{\frac{\kappa}{1+\kappa}}}_\text{term I}\Bigr)^{\frac{1+\kappa}{2}}\Bigl(\underbrace{\sum_{s=1}^\infty\tilde\epsilon_s\log(2N_{s+1}^3)^{\frac{\kappa}{1+\kappa}}}_\text{term II}\Bigr)^{\frac{1-\kappa}{2}}\\
            &\qquad +8\cdot 3^{\frac{1}{m}-1} n^{\frac{1}{m}-1} \cdot \|1/w\|_{L^{m}(P)} \Bigl(\underbrace{\sum_{s=1}^\infty\tilde\epsilon_s\log(2N_{s+1}^3)^{1-\frac{1}{m}}}_\text{term III}\Bigr),\\
        \end{aligned}
    \end{equation}
    where the second inequality is due to Hölder's inequality: for $a\in(0,1)$ and nonnegative $x_i,y_i$'s, define $\xbf=(x_1^a,x_2^a,\cdots)$ and $\ybf=(y_1^{1-a},y_2^{1-a},\cdots)$. Then,
    $$\begin{aligned}
        &\sum_{i=1}^\infty x_i^a\cdot y_i^{1-a}
        =\|\xbf\cdot \ybf\|_{\ell_1}\\
        \leq& \|\xbf\|_{\ell_{\frac{1}{a}}}\cdot \|\ybf\|_{\ell_{\frac{1}{1-a}}}
        = \Bigl(\sum_{i=1}^\infty (x_i^a)^{\frac{1}{a}}\Bigr)^a\cdot \Bigl(\sum_{i=1}^\infty (y_i^{1-a})^{\frac{1}{1-a}}\Bigr)^{1-a}
        = \Bigl(\sum_{i=1}^\infty x_i\Bigr)^a\cdot \Bigl(\sum_{i=1}^\infty y_i\Bigr)^{1-a}.
    \end{aligned}$$

    \textbf{Step 4.} Note that in Equation \eqref{eq: upper bound of the empirical process of pis=1-pis f, theorem: convergence of EP with L^infty integrable functions}, terms I, II and III share the same structure. Yet, term I relates to the covering entropy $\log\Ncal(\cdot,\Fcal,L^{1+\kappa}(P))$, terms II and III relate to the covering entropy $\log\Ncal(\cdot,\Fcal,L^{\infty}(w))$.

    For this reason, for metric $d=L^{1+\kappa}(P)$ or $L^\infty(w)$, define, for $s\geq 1$, $$\varepsilon_s=\inf\{\varepsilon>0: \Ncal(\varepsilon, \Fcal, d)\leq N_s\},$$ and for any $\delta>0$, $\tilde\varepsilon_s=(1+\delta)\varepsilon_s$. Now, it suffices for us to consider the upper bound of, for some $a\in(0,1)$, $$\sum_{s=1}^\infty \tilde\varepsilon_s \cdot \log(2N_{s+1}^3)^a.$$

    Firstly, for $N_s=2^{2^s}$, $2N_s^3=2^{1+3\cdot 2^s}$. Thus, $$\max_{s\geq 1}\frac{\log(2N_3^3)}{\log(2N_{s+1}^3)}=\max_{s\geq 1}\frac{\log_2(2N_3^3)}{\log_2(2N_{s+1}^3)}=\max_{s\geq 1}\frac{1+3\cdot 2^s}{1+3\cdot 2^{s+1}}=\frac{7}{13},$$ which indicates $$\log(2N_s^3)\leq \frac{7}{13}\log(2N_{s+1}^3).$$
    Thus,
    \begin{equation}\label{eq: first inequality in step 4, theorem: convergence of EP with L^infty integrable functions}
        \begin{aligned}
            \sum_{s=1}^\infty (\tilde\varepsilon_s-\tilde\varepsilon_{s+1})\log(2N_{s+1}^3)^a
            =& \sum_{s=1}^\infty \tilde\varepsilon_s\log(2N_{s+1}^3)^a - \sum_{s=2}^\infty \tilde\varepsilon_s\log(2N_{s}^3)^a\\
            \geq& \sum_{s=1}^\infty \tilde\varepsilon_s\Bigl(\log(2N_{s+1}^3)^a- \log(2N_{s}^3)^a\Bigr)\\
            \geq& \bigl(1-(7/13)^a\bigr)\sum_{s=1}^\infty \tilde\varepsilon_s\cdot \log(2N_{s+1}^3)^a.
        \end{aligned}
    \end{equation}

    Then, by the definition of $\varepsilon_s$, if $\varepsilon<\varepsilon_s$, $\Ncal(\varepsilon,\Fcal,d)>N_s$. For this reason,
    \begin{equation}\label{eq: second inequality in step 4, theorem: convergence of EP with L^infty integrable functions}
        (\log N_s)^a\cdot (\varepsilon_s-\varepsilon_{s+1})\leq \int_{\varepsilon_{s+1}}^{\varepsilon_s} (\log\Ncal(x,\Fcal,d))^a\, \dd x.
    \end{equation}

    It remains for us to bound $2N_{s+1}^3$ by $N_s$. Note that $$\max_{s\geq 1}\frac{\log(2N_{s+1}^3)}{\log N_s}=\max_{s\geq 1}\frac{1+3\cdot 2^{s+1}}{2^s}=\frac{13}{2}.$$
    So, we conclude that
    \begin{equation}\label{eq: third inequality in step 4, theorem: convergence of EP with L^infty integrable functions}
        \begin{aligned}
            \sum_{s=1}^\infty (\tilde\varepsilon_s-\tilde\varepsilon_{s+1})\log(2N_{s+1}^3)^a
            \leq& (13/2)^a \sum_{s=1}^\infty (\tilde\varepsilon_s-\tilde\varepsilon_{s+1})\log(N_{s})^a.
        \end{aligned}
    \end{equation}

    Now, combining Equations \eqref{eq: first inequality in step 4, theorem: convergence of EP with L^infty integrable functions}, \eqref{eq: second inequality in step 4, theorem: convergence of EP with L^infty integrable functions} and \eqref{eq: third inequality in step 4, theorem: convergence of EP with L^infty integrable functions} yields
    \begin{equation}\label{eq: intermediate result in step 4, theorem: convergence of EP with L^infty integrable functions}
        \begin{aligned}
            \sum_{s=1}^\infty \tilde\varepsilon_s\cdot \log(2N_{s+1}^3)^a
            \leq& \frac{(13/2)^a}{1-(7/13)^a}\sum_{s=1}^\infty (\tilde\varepsilon_s-\tilde\varepsilon_{s+1})\log(N_{s})^a\\
            =& (1+\delta)\frac{(13/2)^a}{1-(7/13)^a}\sum_{s=1}^\infty (\varepsilon_s-\varepsilon_{s+1})\log(N_{s})^a\\
            \leq& (1+\delta)\frac{(13/2)^a}{1-(7/13)^a}\sum_{s=1}^\infty\int_{\varepsilon_{s+1}}^{\varepsilon_s} (\log\Ncal(x,\Fcal,d))^a\, \dd x\\
            \leq& (1+\delta)\frac{(13/2)^a}{1-(7/13)^a} \int_{0}^{\infty} \log\Ncal(x,\Fcal,d)^a\, \dd x.\\
        \end{aligned}
    \end{equation}
    For this reason,
    \begin{equation}\label{eq: result in step 4, theorem: convergence of EP with L^infty integrable functions}
        \begin{aligned}
            \text{term I}\leq& (1+\delta)\frac{(13/2)^{\frac{\kappa}{1+\kappa}}}{1-(7/13)^{\frac{\kappa}{1+\kappa}}} \int_{0}^{\infty} \log\Ncal(x,\Fcal,L^{1+\kappa}(P))^\frac{\kappa}{1+\kappa}\, \dd x\\
            \text{term II}\leq& (1+\delta)\frac{(13/2)^{\frac{\kappa}{1+\kappa}}}{1-(7/13)^{\frac{\kappa}{1+\kappa}}} \int_{0}^{\infty} \log\Ncal(x,\Fcal,L^{\infty}(w))^\frac{\kappa}{1+\kappa}\, \dd x\\
            \text{term III}\leq& (1+\delta)\frac{(13/2)^{1-\frac{1}{m}}}{1-(7/13)^{1-\frac{1}{m}}} \int_{0}^{\infty} \log\Ncal(x,\Fcal,L^{\infty}(w))^{1-\frac{1}{m}}\, \dd x\\
        \end{aligned}
    \end{equation}

    \textbf{Step 5.} We are ready to wrap up the results from Step 2 to Step 4. Firstly, plugging Equation \eqref{eq: result in step 4, theorem: convergence of EP with L^infty integrable functions} into Equation \eqref{eq: upper bound of the empirical process of pis=1-pis f, theorem: convergence of EP with L^infty integrable functions} gives,
    \begin{equation}\label{eq: first upper bound in step 5, theorem: convergence of EP with L^infty integrable functions}
        \begin{aligned}
            &\frac{1}{1+\delta}\sum_{s=1}^\infty\EE\Bigl[\sup_{f\in\Fcal}\Bigl|\PP_n^\varepsilon (\pi_{s+1}f-\pi_sf)\Bigr|\Bigr]\\
            \leq& \frac{2^{\frac{5}{2}} 3^{\frac{1-\kappa}{2(1+\kappa)}}}{n^{\frac{\kappa}{1+\kappa}}}\cdot \|1/w\|_{L^{1+\kappa}(P)}^{\frac{1-\kappa}{2}}\cdot \frac{(13/2)^{\frac{\kappa}{1+\kappa}}}{1-(7/13)^{\frac{\kappa}{1+\kappa}}}
            \Bigl(\int_{0}^{\infty} \log\Ncal(x,\Fcal,L^{1+\kappa}(P))^\frac{\kappa}{1+\kappa}\, \dd x\Bigr)^{\frac{1+\kappa}{2}}\\
            &\qquad\cdot\Bigl(\int_{0}^{\infty} \log\Ncal(x,\Fcal,L^{\infty}(w))^\frac{\kappa}{1+\kappa}\, \dd x\Bigr)^{\frac{1-\kappa}{2}}\\
            &\quad +8\cdot 3^{\frac{1}{m}-1} n^{\frac{1}{m}-1} \cdot \|1/w\|_{L^{m}(P)}\cdot \frac{(13/2)^{1-\frac{1}{m}}}{1-(7/13)^{1-\frac{1}{m}}} \Bigl(\int_{0}^{\infty} \log\Ncal(x,\Fcal,L^{\infty}(w))^{1-\frac{1}{m}}\, \dd x\Bigr)\\
            =& \frac{\|1/w\|_{L^{1+\kappa}(P)}^{\frac{1-\kappa}{2}}}{n^{\frac{\kappa}{1+\kappa}}}\cdot \frac{2^{\frac{5}{2}} 3^{\frac{1-\kappa}{2(1+\kappa)}}\cdot (13/2)^{\frac{\kappa}{1+\kappa}}}{1-(7/13)^{\frac{\kappa}{1+\kappa}}}
            \Bigl(\int_{0}^{2\sigma} \log\Ncal(x,\Fcal,L^{1+\kappa}(P))^\frac{\kappa}{1+\kappa}\, \dd x\Bigr)^{\frac{1+\kappa}{2}}\\
            &\qquad\cdot\Bigl(\int_{0}^{2 \|F\|_{L^\infty(w)}} \log\Ncal(x,\Fcal,L^{\infty}(w))^\frac{\kappa}{1+\kappa}\, \dd x\Bigr)^{\frac{1-\kappa}{2}}\\
            &\quad + \frac{\|1/w\|_{L^m(P)}}{n^{1-\frac{1}{m}}}\cdot \frac{8\cdot(13/6)^{1-\frac{1}{m}}}{1-(7/13)^{1-\frac{1}{m}}} \Bigl(\int_{0}^{2 \|F\|_{L^\infty(w)}} \log\Ncal(x,\Fcal,L^{\infty}(w))^{1-\frac{1}{m}}\, \dd x\Bigr),\\
        \end{aligned}
    \end{equation}
    where we replace the infinite upper limit of integrations with finite (upper bounds of) diameters of $\Fcal$: $$\sup_{f,g\in\Fcal}\|f-g\|_{L^{1+\kappa}(P)}\leq 2\sup_{f\in\Fcal}\|f-0\|_{L^{1+\kappa}(P)}\leq 2\sigma,$$
    and
    \[
        \sup_{f,g\in\Fcal}\|f-g\|_{L^{\infty}(w)}\leq 2\sup_{f\in\Fcal}\|f-0\|_{L^{\infty}(w)}\leq 2\|F\|_{L^\infty(w)}.
    \]
    In the meantime, the upper bound of $\EE\Bigl[\sup_{f\in\Fcal}\Bigl|\PP_n^\varepsilon  \pi_1f\Bigr|\Bigr]$ can also be expressed as covering integrals: by Equation \eqref{eq: upper bound of the empirical process of pi1 f, theorem: convergence of EP with L^infty integrable functions},
    \begin{equation}\label{eq: second upper bound in step 5, theorem: convergence of EP with L^infty integrable functions}
        \begin{aligned}
            &\frac{1}{1+\delta}\EE\Bigl[\sup_{f\in\Fcal}\Bigl|\PP_n^\varepsilon  \pi_1f\Bigr|\Bigr]\\
            \leq& 2^{-\frac{1}{2}} 3^{\frac{1-\kappa}{2(1+\kappa)}}(2\sigma)^{\frac{1+\kappa}{2}}(2\|F\|_{L^{1+\kappa}(P)})^{\frac{1-\kappa}{2}}\Bigl(\frac{\log(32)}{n}\Bigr)^{\frac{\kappa}{1+\kappa}}
            +\Bigl(\frac{\log(32)}{3n}\Bigr)^{1-\frac{1}{m}}2\|F\|_{L^{m}(P)}\\
            \leq& 2^{-\frac{1}{2}} 3^{\frac{1-\kappa}{2(1+\kappa)}}(2\sigma)^{\frac{1+\kappa}{2}}(2\|F\|_{L^\infty(w)})^{\frac{1-\kappa}{2}}\Bigl(\frac{\log(32)}{n}\Bigr)^{\frac{\kappa}{1+\kappa}}\|1/w\|_{L^{1+\kappa}(P)}^{\frac{1-\kappa}{2}}\\
            &\quad +\Bigl(\frac{\log(32)}{3n}\Bigr)^{1-\frac{1}{m}}2\|F\|_{L^{\infty}(w)}\cdot \|1/w\|_{L^{m}(P)}\\
            =& \frac{\|1/w\|_{L^{1+\kappa}(P)}^{\frac{1-\kappa}{2}}}{n^{\frac{\kappa}{1+\kappa}}}\Bigl(2^{-\frac{1}{2}} 3^{\frac{1-\kappa}{2(1+\kappa)}} \log(32)^{\frac{\kappa}{1+\kappa}}\Bigr)\Bigl(\int_0^{2\sigma}\, \dd x\Bigr)^{\frac{1+\kappa}{2}}\Bigl(\int_0^{2\|F\|_{L^{\infty}(P)}}\, \dd x\Bigr)^{\frac{1-\kappa}{2}}\\
            &\quad +\frac{\|1/w\|_{L^m(P)}}{n^{1-\frac{1}{m}}}\Bigl(\frac{\log(32)}{3}\Bigr)^{1-\frac{1}{m}}\Bigl(\int_0^{2\|F\|_{L^{\infty}(P)}}\, \dd x\Bigr)\\
        \end{aligned}
    \end{equation}

    Now, substituting Equations \eqref{eq: first upper bound in step 5, theorem: convergence of EP with L^infty integrable functions} and \eqref{eq: second upper bound in step 5, theorem: convergence of EP with L^infty integrable functions} into Equation \eqref{eq: decomposition of f into summation of projection operators, theorem: convergence of EP with L^infty integrable functions} gives, for any $\delta>0$,
    $$\begin{aligned}
        \frac{1}{1+\delta}\EE^\ast \|\PP_n^\varepsilon\|_{\Fcal}
        \leq&
        \frac{\|1/w\|_{L^{1+\kappa}(P)}^{\frac{1-\kappa}{2}}}{n^{\frac{\kappa}{1+\kappa}}}\cdot
        \underbrace{\Bigl(\frac{2^{\frac{5}{2}} 3^{\frac{1-\kappa}{2(1+\kappa)}}\cdot (13/2)^{\frac{\kappa}{1+\kappa}}}{1-(7/13)^{\frac{\kappa}{1+\kappa}}} + 2^{-\frac{1}{2}} 3^{\frac{1-\kappa}{2(1+\kappa)}} \log(32)^{\frac{\kappa}{1+\kappa}}\Bigr)}_{=:c_1(\kappa)}\\
        &\qquad \Bigl(\int_{0}^{2\sigma}1+\log\Ncal(x,\Fcal,L^{1+\kappa}(P))^\frac{\kappa}{1+\kappa}\, \dd x\Bigr)^{\frac{1+\kappa}{2}}\\
        &\qquad\cdot\Bigl(\int_{0}^{2 \|F\|_{L^\infty(w)}} 1+\log\Ncal(x,\Fcal,L^{\infty}(w))^\frac{\kappa}{1+\kappa}\, \dd x\Bigr)^{\frac{1-\kappa}{2}}\\
        &+ \frac{\|1/w\|_{L^m(P)}}{n^{1-\frac{1}{m}}}\cdot \underbrace{\Bigl(\frac{8\cdot(13/6)^{1-\frac{1}{m}}}{1-(7/13)^{1-\frac{1}{m}}}  + \Bigl(\frac{\log(32)}{3}\Bigr)^{1-\frac{1}{m}}\Bigr)}_{=:c_2(m)}\\
        &\qquad\cdot\Bigl(\int_{0}^{2 \|F\|_{L^\infty(w)}} 1+\log\Ncal(x,\Fcal,L^{\infty}(w))^{1-\frac{1}{m}}\, \dd x\Bigr),\\
    \end{aligned}$$
    where $c_1(\cdot)$ and $c_2(\cdot)$ can be controlled by, since $1-a^x\geq (1-a)x$ for $a\in (0,1)$ and $x\in[0,1]$,
    $$\begin{aligned}
        c_1(\kappa)\leq& \frac{62.4963}{\kappa}+ 1.3164\leq \frac{63.8127}{\kappa};\qquad c_2(m)\leq \frac{51.0280}{\kappa}+1.1553\leq \frac{52.1833}{\kappa}.
    \end{aligned}$$
    So, by taking $\delta\to0+$, we conclude that
    $$\begin{aligned}
        &\EE^\ast\|\PP_n-P\|_\Fcal
        \leq 2\EE^\ast\|\PP_n^\varepsilon\|_\Fcal\\
        \leq& \frac{127.6254}{\kappa}\frac{\|1/w\|_{L^{1+\kappa}(P)}^{\frac{1-\kappa}{2}}}{n^{\frac{\kappa}{1+\kappa}}}\cdot  \Bigl(\int_{0}^{2\sigma}1+\log\Ncal(x,\Fcal,L^{1+\kappa}(P))^\frac{\kappa}{1+\kappa}\, \dd x\Bigr)^{\frac{1+\kappa}{2}}\\
        &\qquad\cdot\Bigl(\int_{0}^{2 \|F\|_{L^\infty(w)}} 1+\log\Ncal(x,\Fcal,L^{\infty}(w))^\frac{\kappa}{1+\kappa}\, \dd x\Bigr)^{\frac{1-\kappa}{2}}\\
        &+\frac{104.3666}{\kappa} \frac{\|1/w\|_{L^m(P)}}{n^{1-\frac{1}{m}}}\cdot\Bigl(\int_{0}^{2 \|F\|_{L^\infty(w)}} 1+\log\Ncal(x,\Fcal,L^{\infty}(w))^{1-\frac{1}{m}}\, \dd x\Bigr).\\
    \end{aligned}$$

\section{Proof of Results in Section \ref{sec: Robustness of Deep Huber Regression}}

\subsection{Properties of Huber Regression}\label{sec: Properties of Huber Regression}
\noindent\phantom{.}\par\vspace{-\baselineskip}

{

\begin{proposition}[Huber regression with optimization error]\label{prop: convergence rate of Huber estimator}
    Suppose Assumption~\ref{assumption: distributional assumption for Huber regression} holds, $\Fcal_n\subseteq L^\infty(M)$, and $\tau\geq 2\max\{2M,(2v_m)^{1/m}\}$. Let $\tilde f_n\in\Fcal_n$ be an $\varepsilon_n$-approximate empirical minimizer satisfying
    \[
        \PP_n\ell_\tau(y-\tilde f_n(x))
        \leq \inf_{f\in\Fcal_n}\PP_n\ell_\tau(y-f(x))+\varepsilon_n,
        \qquad \varepsilon_n\geq0.
    \]
    For any fixed $f_n^\ast\in\Fcal_n$, define
    \[
        \overline\Fcal_n=\{tf+(1-t)f_n^\ast:t\in[0,1],\ f\in\Fcal_n\}
    \]
    and, for $c>0$,
    \[
        \Lscr_c=\{\ell_\tau(y-f(x))-\ell_\tau(y-f_n^\ast(x)):f\in\overline\Fcal_n,\ \|f-f_n^\ast\|_{L^2(P)}\leq c\}.
    \]
    Suppose that, for every $c>0$ and $\delta\in(0,1)$, there is a function $\phi_n$ such that
    \[
        \PP\bigl(\|\PP_n-P\|_{\Lscr_c}\geq\phi_n(c,\delta)\bigr)\leq\delta,
    \]
    and that $c\mapsto\sqrt{\phi_n(c,\delta)}$ is non-decreasing and sub-linear in the following sense: for every $\delta\in(0,1)$, there is some $r_\delta>0$ such that
    \[
        \sqrt{\phi_n(c,\delta)}
        \begin{cases}
            \leq c/(8\sqrt2),&c\geq r_\delta,\\
            \geq c/(8\sqrt2),&0<c\leq r_\delta.
        \end{cases}
    \]
    Then there exist universal constants $c_1^\prime,c_2^\prime>0$ such that, for every $\delta\in(0,1)$,
    \[\begin{aligned}
        &\PP\Bigl(
        c_2^\prime\|\tilde f_n-f_0\|_{L^2(P)}
        \geq \inf\{r>0:\phi_n(r,\delta)\leq c_1^\prime r^2\}\\
        &\qquad\qquad\qquad+\|f_n^\ast-f_0\|_{L^2(P)}+v_m\tau^{1-m}+\sqrt{\varepsilon_n}
        \Bigr)\leq\delta.
    \end{aligned}\]
    In particular, taking $\varepsilon_n=0$ and $\tilde f_n=\hat f_n(\tau)$ recovers the exact-minimizer bound used in the proof of Theorem~\ref{thm: convergence rate of Huber regression}.
\end{proposition}
\begin{proof}
    Write
    \[
        L_\tau(f)=P\ell_\tau(y-f(x)),\qquad
        L_{n,\tau}(f)=\PP_n\ell_\tau(y-f(x)).
    \]
    Fix a localization radius $r>0$, and define
    \[
        t=\frac{r}{r+\|\tilde f_n-f_n^\ast\|_{L^2(P)}},
        \qquad \tilde f_{n,t}=t\tilde f_n+(1-t)f_n^\ast.
    \]
    Thus, $\tilde f_{n,t}\in\overline\Fcal_n$ and $\|\tilde f_{n,t}-f_n^\ast\|_{L^2(P)}\leq r$. By convexity of the Huber loss and the approximate-minimizer property,
    \[
        L_{n,\tau}(\tilde f_{n,t})
        \leq tL_{n,\tau}(\tilde f_n)+(1-t)L_{n,\tau}(f_n^\ast)
        \leq L_{n,\tau}(f_n^\ast)+t\varepsilon_n
        \leq L_{n,\tau}(f_n^\ast)+\varepsilon_n.
    \]
    Consequently,
    \begin{equation}\label{eq: oracle inequality for approximate Huber minimizer}
        L_\tau(\tilde f_{n,t})-L_\tau(f_0)
        \leq \|\PP_n-P\|_{\Lscr_r}+L_\tau(f_n^\ast)-L_\tau(f_0)+\varepsilon_n.
    \end{equation}

    Proposition~\ref{prop: stability link of Huber regression} yields two cases. If $\|\tilde f_{n,t}-f_0\|_{L^2(P)}>8v_m\tau^{1-m}$, then \eqref{eq: oracle inequality for approximate Huber minimizer} and the fact that $f_{0,\tau}$ minimizes $L_\tau$ give
    \[
        \|\tilde f_{n,t}-f_0\|_{L^2(P)}^2
        \leq 8\|\PP_n-P\|_{\Lscr_r}
        +8\{L_\tau(f_n^\ast)-L_\tau(f_{0,\tau})\}+8\varepsilon_n.
    \]
    Otherwise, $\|\tilde f_{n,t}-f_0\|_{L^2(P)}\leq8v_m\tau^{1-m}$. Hence, in either case,
    \begin{equation}\label{eq: localized approximate Huber bound}
        \begin{aligned}
        \|\tilde f_{n,t}-f_0\|_{L^2(P)}
        \leq{}&2\sqrt2\sqrt{\|\PP_n-P\|_{\Lscr_r}}
        +2\sqrt2\sqrt{L_\tau(f_n^\ast)-L_\tau(f_{0,\tau})}\\
        &+2\sqrt2\sqrt{\varepsilon_n}+8v_m\tau^{1-m}.
        \end{aligned}
    \end{equation}
    By Propositions~\ref{prop: approximation error of huber loss function} and~\ref{prop: huberization bias},
    \[
        L_\tau(f_n^\ast)-L_\tau(f_{0,\tau})
        \leq\tfrac12\|f_n^\ast-f_{0,\tau}\|_{L^2(P)}^2,
        \qquad \|f_{0,\tau}-f_0\|_{L^2(P)}\leq4v_m\tau^{1-m}.
    \]
    It follows from \eqref{eq: localized approximate Huber bound} and the triangle inequality that
    \begin{equation}\label{eq: localization inequality for approximate Huber minimizer}
        \|\tilde f_{n,t}-f_n^\ast\|_{L^2(P)}
        \leq 2\sqrt2\sqrt{\|\PP_n-P\|_{\Lscr_r}}+B_n,
    \end{equation}
    where
    \[
        B_n:=3\|f_n^\ast-f_0\|_{L^2(P)}+16v_m\tau^{1-m}+2\sqrt2\sqrt{\varepsilon_n}.
    \]

    Let $A(r)=2\sqrt2\sqrt{\phi_n(r,\delta)}$ and
    \[
        r_\delta=\inf\{r>0:A(r)\leq r/4\}.
    \]
    Take $r=r_\delta+4B_n$. By the assumed sub-linearity, $A(r)\leq r/4$, while $B_n\leq r/4$. On the event $\|\PP_n-P\|_{\Lscr_r}\leq\phi_n(r,\delta)$, which has probability at least $1-\delta$, Equation~\eqref{eq: localization inequality for approximate Huber minimizer} therefore gives
    \[
        \|\tilde f_{n,t}-f_n^\ast\|_{L^2(P)}\leq r/2.
    \]
    If $d=\|\tilde f_n-f_n^\ast\|_{L^2(P)}$, then the left-hand side equals $rd/(r+d)$. Thus $rd/(r+d)\leq r/2$ implies $d\leq r$. Applying the triangle inequality yields
    \[
        \|\tilde f_n-f_0\|_{L^2(P)}
        \leq r_\delta+4B_n+\|f_n^\ast-f_0\|_{L^2(P)}.
    \]
    Since $A(r)\leq r/4$ is equivalent to $\phi_n(r,\delta)\leq r^2/128$, the claimed result follows after absorbing universal constants into $c_1^\prime$ and $c_2^\prime$.
\end{proof}

}

\begin{proposition}[Proposition 3.1 in \cite{fan2024noise}]\label{prop: stability link of Huber regression}
    Assume $\|f_0\|_{L^\infty([0,1]^d)}\leq M$ for some $M\geq 1$, and let $\tau\geq 2\max\{2M, (2v_m)^{1/m}\}$. Then, $$L_\tau(f)-L_\tau(f_0)\geq \frac{1}{8}\|f-f_0\|_{L^2(P)}^2$$ for all $f$ satisfying $$\|f\|_{L^\infty([0,1]^d)}\leq M, \|f-f_0\|_{L^2(P)}>8v_m\tau^{1-m}.$$
\end{proposition}

\begin{proposition}\label{prop: approximation error of huber loss function}
    Let $f_{0,\tau}$ be the global minimizer of $f\mapsto \EE_P[\ell_\tau(y-f(x))]$. Then for any $f\in L^2(P)$, $$\EE[\ell_\tau(y-f(x))-\ell_\tau(y-f_{0,\tau}(x))]\leq \frac{1}{2}\|f-f_{0,\tau}\|_{L^2(P)}^2.$$
\end{proposition}
\begin{proof}
    For any $f$, applying the Taylor expansion gives,
    \begin{equation}\label{eq: first intermediate result in prop: approximation error of huber loss function}
        \begin{aligned}
            &\EE[\ell_\tau(y-f(x))-\ell_\tau(y-f_{0,\tau}(x))]\\
            =&\EE[\ell_\tau^\prime(y-f_{0,\tau}(x))(f_{0,\tau}(x)-f(x))]+\EE\Bigl[\int_{y-f_{0,\tau}(x)}^{y-f(x)}\ell_\tau^\pprime(t)(y-f(x)-t)\, \dd t\Bigr].
        \end{aligned}
    \end{equation}
    Since $\ell_\tau^\pprime(x)=\II(|x|\leq \tau)$,
    \[\begin{aligned}
        0\leq& \EE\Bigl[\int_{y-f_{0,\tau}(x)}^{y-f(x)}\ell_\tau^\pprime(t)(y-f(x)-t)\, \dd t\Bigr]\\
        \leq& \EE\Bigl[\int_{y-f_{0,\tau}(x)}^{y-f(x)}1\cdot (y-f(x)-t)\, \dd t\Bigr]= \frac{1}{2}\|f-f_{0,\tau}\|_{L^2(P)}^2.
    \end{aligned}\]

    Meanwhile, $\EE[\ell_\tau(y-f(x))-\ell_\tau(y-f_{0,\tau}(x))]\geq 0$ implies that
    \begin{equation}\label{eq: second intermediate result in prop: approximation error of huber loss function}
        \Bigl|\EE[\ell_\tau^\prime(y-f_{0,\tau}(x))(f_{0,\tau}(x)-f(x))]\Bigr|\leq \frac{1}{2}\|f-f_{0,\tau}\|_{L^2(P)}^2,
    \end{equation}
     which can be verified by combining the following two facts:
     \begin{enumerate}
        \item $\EE[\ell_\tau(y-f(x))-\ell_\tau(y-f_{0,\tau}(x))]\geq 0$ implies that there is no $f$, such that $$\EE[\ell_\tau^\prime(y-f_{0,\tau}(x))(f_{0,\tau}(x)-f(x))]<-\frac{1}{2}\|f-f_{0,\tau}\|_{L^2(P)}^2.$$ Otherwise,
        $$\begin{aligned}
            0\leq &\EE[\ell_\tau(y-f(x))-\ell_\tau(y-f_{0,\tau}(x))]\\
            =&\EE[\ell_\tau^\prime(y-f_{0,\tau}(x))(f_{0,\tau}(x)-f(x))]+\EE\Bigl[\int_{y-f_{0,\tau}(x)}^{y-f(x)}\ell_\tau^\pprime(t)(y-f(x)-t)\, \dd t\Bigr]\\
            <&-\frac{1}{2}\|f-f_{0,\tau}\|_{L^2(P)}^2+\frac{1}{2}\|f-f_{0,\tau}\|_{L^2(P)}^2=0.
        \end{aligned}$$
        \item Suppose there exists a function $g$, such that $$\EE[\ell_\tau^\prime(y-f_{0,\tau}(x))(f_{0,\tau}(x)-g(x))]>\frac{1}{2}\|g-f_{0,\tau}\|_{L^2(P)}^2.$$ Then, taking a function $f$, such that $f_{0,\tau}-f=g-f_{0,\tau}$, and we have: $$\EE[\ell_\tau^\prime(y-f_{0,\tau}(x))(f_{0,\tau}(x)-f(x))]<-\frac{1}{2}\|g-f_{0,\tau}\|_{L^2(P)}^2=-\frac{1}{2}\|f-f_{0,\tau}\|_{L^2(P)}^2,$$ which is impossible.
    \end{enumerate}
    Let us apply Equation~\eqref{eq: second intermediate result in prop: approximation error of huber loss function} again. If we select function $f$, such that $$f_{0,\tau}(x)-f(x)=\EE[\ell_\tau^\prime(y-f_{0,\tau}(x)|x],$$ then $$\EE\Bigl[\EE[\ell_\tau^\prime(y-f_{0,\tau}(x)|x]^2\Bigr]\leq\frac{1}{2}\EE\Bigl[\EE[\ell_\tau^\prime(y-f_{0,\tau}(x)|x]^2\Bigr]$$
    implies that $\EE[\ell_\tau^\prime(y-f_{0,\tau}(x)|x]=0$ a.e.

    Thus, Equation \eqref{eq: first intermediate result in prop: approximation error of huber loss function} can be rewritten as
    $$\begin{aligned}
            &\EE[\ell_\tau(y-f(x))-\ell_\tau(y-f_{0,\tau}(x))]\\
            =&\EE\Bigl[\EE[\ell_\tau^\prime(y-f_{0,\tau}(x))|x](f_{0,\tau}(x)-f(x))\Bigr]+\EE\Bigl[\int_{y-f_{0,\tau}(x)}^{y-f(x)}\ell_\tau^\pprime(t)(y-f(x)-t)\, \dd t\Bigr]\\
            =& \EE\Bigl[\int_{y-f_{0,\tau}(x)}^{y-f(x)}\ell_\tau^\pprime(t)(y-f(x)-t)\, \dd t\Bigr]
            \leq\EE\Bigl[\int_{y-f_{0,\tau}(x)}^{y-f(x)}1\cdot (y-f(x)-t)\, \dd t\Bigr]\\
            =&\frac{1}{2}\|f-f_{0,\tau}\|_{L^2(P)}^2.
    \end{aligned}$$
\end{proof}

\begin{proposition}[Proposition 3.2 in \cite{fan2024noise}]\label{prop: huberization bias}
    Under Assumption \ref{assumption: distributional assumption for Huber regression}, and let $\tau\geq 2\max\{2M, (2v_m)^{1/m}\}$. Then the global minimizer $f_{0,\tau}$ of the population Huber loss functional satisfies $$\|f_{0,\tau}-f_0\|_2\leq 4v_m\tau^{1-m}.$$
\end{proposition}

\begin{proposition}\label{prop: l2 norm of difference of Huber loss}
    Under Assumption \ref{assumption: distributional assumption for Huber regression} (if $m<2$, denote $v_2=\infty$). For any $f_1,f_2$ such that $\|f_1\|_{L^\infty}, \|f_2\|_{L^\infty}\leq M$, $$\|\ell_\tau(y-f_1(x))-\ell_\tau(y-f_2(x))\|_{L^2(P)}\leq \sqrt{2}\bigl((\sqrt{v_2}+2M)\land\tau+M\bigr)\|f_1-f_2\|_{L^2(P)}.$$
\end{proposition}
\begin{proof}
    Define $\Delta f=f_2-f_1$. Then, applying Taylor's expansion gives,
    $$\begin{aligned}
        &\EE\Bigl[\Bigl|\ell_\tau(y-f_1(x))-\ell_\tau(y-f_2(x))\Bigr|^2\Bigr]\\
        =& \EE\Bigl[\Bigl|\ell_\tau^\prime(y-f_2(x))\Delta f(x)+\int_{y-f_1(x)}^{y-f_2(x)}\ell_\tau^\pprime(t)\cdot (y-f_1(x)-t)\, \dd t\Bigr|^2\Bigr]\\
        \leq& 2\EE\Bigl[\ell_\tau^\prime(y-f_2(x))^2 |\Delta f(x)|^2\Bigr]+2\EE\Bigl[\Bigl|\int_{y-f_2(x)}^{y-f_1(x)}\underbrace{\ell_\tau^\pprime(t)}_{\in[0,1]}\cdot (y-f_1(x)-t)\, \dd t\Bigr|^2\Bigr]\\
        \leq& 2\EE\Bigl[\EE[\ell_\tau^\prime(y-f_2(x))^2|x]\cdot  |\Delta f(x)|^2\Bigr] + \frac{1}{2}\EE[\Delta f(x)^4].
    \end{aligned}$$
    \begin{enumerate}
        \item For the first term on the RHS, $\ell_\tau^\prime(x)=\sgn(x)\cdot (|x|\land \tau)$ implies that $|\ell_\tau^\prime(x)|\leq |x|\land \tau$. As a result,
        $$\begin{aligned}
            \EE[\ell_\tau^\prime(y-f_2(x))^2|x]
            =& \EE[\ell_\tau^\prime(\xi+f_0(x)-f_2(x))^2|x]
            \leq \EE[(\xi+f_0(x)-f_2(x))^2\land \tau^2|x]\\
            \leq & \EE[(\xi+f_0(x)-f_2(x))^2|x]\land \tau^2\\
            =& \Bigl(\EE[\xi^2|x]+(f_0(x)-f_2(x))^2\Bigr)\land \tau^2\\
            \leq& (v_2+4M^2)\land \tau^2.
        \end{aligned}$$
        where the second inequality is due to the Jensen's inequality.

        \item For the second term on the RHS, $|\Delta f(x)|\leq 2M$ implies that $$\EE[\Delta f(x)^4]\leq 4M^2 \|\Delta f\|_{L^2(P)}^2.$$
    \end{enumerate}

    Hence, we conclude that
    $$\begin{aligned}
        &\EE\Bigl[\Bigl|\ell_\tau(y-f_1(x))-\ell_\tau(y-f_2(x))\Bigr|^2\Bigr]\\
        \leq& 2\EE\Bigl[((v_2+4M^2)\land \tau^2)\cdot  |\Delta f(x)|^2\Bigr] + 2M^2 \|\Delta f\|_{L^2(P)}^2\\
        =& 2\Bigl((v_2+4M^2)\land \tau^2+M^2\Bigr)\|\Delta f\|_{L^2(P)}^2.
    \end{aligned}$$
\end{proof}

\subsection{Proof of Theorem \ref{thm: convergence rate of Huber regression}}

Note that Theorem \ref{thm: convergence rate of Huber regression} is a direct consequence of Lemma~\ref{lemma: expected covering entropy of neural network} and Lemma~\ref{lemma: convergence rate of Huber regression, with covering entropy}.

\begin{lemma}\label{lemma: expected covering entropy of neural network}
    For any $h>0$, the expected $L^2(\PP_n)$ covering entropy of function class $\Fcal:=\Fcal(d,D,W,M)$ satisfies $$\EE[\log(h,\Fcal,L^2(\PP_n))]\lesssim (DW)^2 \cdot \log(DW)\cdot \log_+(enM/h).$$
\end{lemma}
\begin{proof}
    By Theorem 7 of \cite{bartlett2019nearly}, the pseudo-dimension of ReLU neural network function class, $\Fcal$, satisfies $$\text{Pdim}(\Fcal)\lesssim \Nfrak \overline{D}\log(\Nfrak),$$ where $\overline{D}=\Ocal(D)$ is the effective depth and $\Nfrak=\Ocal(D\cdot W^2)$ is the number of parameters of $\Fcal$.

    Define the uniform covering number of function class $\Fcal=\{f:\Xcal\to\RR\}$ as, for $n\in\NN$, $$\Ncal_\infty(h,\Fcal,n)=\sup_{X=(x_1,\cdots, x_n)\in\Xcal^n}\Ncal(h,\Fcal_X,\ell^\infty),$$ where $\Fcal|_X:=\{(f(x_1),\cdots,f(x_n)):f\in\Fcal\}$. Then, by Theorem 12.2 in \cite{anthony2009neural}, since $\Fcal$ is uniformly bounded by $M$, for any $h>0$, $$\log\Ncal_\infty(h,\Fcal, n)\leq\text{Pdim}(\Fcal)\cdot\log(enM/h)\lesssim (DW)^2 \cdot \log(enM/h)\cdot \log(DW).$$

    Consequently, $$\EE[\log(h,\Fcal,L^2(\PP_n))]\leq \log\Ncal_\infty(h,\Fcal, n)\lesssim (DW)^2 \cdot \log(enM/h)\cdot \log(DW).$$
\end{proof}

\begin{lemma}\label{lemma: convergence rate of Huber regression, with covering entropy}
    Suppose Assumption \ref{assumption: distributional assumption for Huber regression} (if $m<2$, denote $v_2=\infty$) holds, and suppose the function class $\Fcal_n\subseteq L^\infty(M)$ for some $M\geq 0$, and that for any $h>0$, its expected $L^2(\PP_n)$ covering entropy satisfies
    \[
        \EE[\log(h,\Fcal,L^2(\PP_n))]\lesssim D_{\Fcal_n}\cdot \log_+(enM/h).
    \]
    Let $\tau\geq 2\max\{2M, (2v_m)^{1/m}\}$. Then for any $\delta\in(0,1)$, the deep Huber estimator:
    \[
        \hat{f}_n(\tau)\in\argmin_{f\in\Fcal_n}\frac{1}{n}\sum_{i=1}^n \ell_\tau\big(Y_i-f(X_i)\big)
    \]
    satisfies:
    $$\PP\Bigl(\|\hat{f}_n(\tau)-f_0\|_{L^2(P)}\gtrsim \sqrt{\tau}\cdot \sqrt{\tau\land (\sqrt{v_2}+M)}\cdot \sqrt\frac{\log(2/\delta)}{n} + \delta_\sfrak + \delta_\afrak+\delta_\bfrak\Bigr)\leq \delta,$$
    where $\tilde{n}= n/D_{\Fcal_n}$ and
    $$\begin{aligned}
        &\delta_\sfrak:=_{\log n}\left\{\begin{aligned}
            & \sqrt{\tau}\cdot \sqrt{\tau\land (\sqrt{v_2}+M)}\cdot \tilde{n}^{-\frac{1}{2}} &  \text{when }\tilde{n}^{\frac{1}{m}}\cdot(M+v_m^{1/m})\geq \tau\\
            & \sqrt{ M\cdot(M+v_m^{1/m}) }\cdot  \tilde{n}^{\frac{1}{2m}-\frac{1}{2}} & \text{when }\tilde{n}^{\frac{1}{m}}\cdot(M+v_m^{1/m})< \tau,\\
        \end{aligned}\right.\\
        &\delta_\afrak:=\inf_{f\in\Fcal_n}\|f-f_0\|_{L^2(P)},\quad
    \delta_\bfrak:=\frac{v_m}{\tau^{m-1}}.
    \end{aligned}$$
\end{lemma}
\begin{proof}
    Since the function $f_0$ in the data-generating process may not minimize the population Huber loss, Theorem~\ref{thm: Estimation Error of Sieved M-estimators} cannot be applied directly. We instead apply Proposition~\ref{prop: convergence rate of Huber estimator} with $\varepsilon_n=0$. Without loss of generality (WLOG), take $f_n^\ast\in\argmin_{f\in\Fcal_n}\|f-f_0\|_{L^2(P)}$. Then, for any $\delta\in(0,1)$,
    $$\PP\Bigl(c_2^\prime\cdot \|\hat{f}_n-f_0\|_{L^2(P)}\geq \inf\Bigl\{c>0:\phi_n(c,\delta)\leq c_1^\prime c^2\Bigr\}+\inf_{f\in\Fcal_n}\|f-f_{0}\|_{L^2(P)}+v_m\tau^{1-m}\Bigr)\leq \delta.$$

    Thus, it suffices to find the infimum over $c$ satisfying $\phi_n(c,\delta)\leq c_1^\prime c^2$. Note that $x\mapsto \ell_\tau(x)$ is a $\tau$-Lipschitz continuous function. Therefore,
    \begin{equation}\label{eq: Lipschitz property of Huber regression, thm: convergence rate of Huber regression}
        \frac{1}{\tau}\Bigl|\ell_\tau(y-f(x))-\ell_\tau(y-f_n^\ast(x))\Bigr|\leq |f(x)-f_n^\ast(x)|.
    \end{equation}

    Define function class:
    $$\frac{1}{\tau}\Lscr_c=\{\frac{1}{\tau}(\ell_\tau(y-f(x))-\ell_\tau(y-f_n^\ast(x))): f\in \overline\Fcal_n, \|f-f_n^\ast\|_{L^2(P)}\leq c\}.$$
    By Equation \eqref{eq: Lipschitz property of Huber regression, thm: convergence rate of Huber regression}, note that:
    \begin{enumerate}
        \item For any $\ell\in \frac{1}{\tau}\Lscr_c$, $\|\ell\|_{L^\infty}\leq 2M$.
        \item For any $\ell\in \frac{1}{\tau}\Lscr_c$, $\|\ell\|_{L^2(P)}\leq c$. Meanwhile, by Proposition \ref{prop: l2 norm of difference of Huber loss}, $$\|\ell\|_{L^2(P)}\leq \frac{\sqrt{2}}{\tau}\bigl((\sqrt{v_2}+2M)\land\tau+M\bigr)\cdot c\leq \sqrt{2}\Bigl((\sqrt{v_2}/\tau+2M/\tau)\land 1+M/\tau\Bigr)\cdot c.$$
        Combining these two bounds yields:
        $$\begin{aligned}
            \|\ell\|_{L^2(P)}
            \leq& \sqrt{2}\Bigl[\Bigl((\sqrt{v_2}/\tau+2M/\tau)\land 1+M/\tau\Bigr)\land 1\Bigr]\cdot c\\
            \leq& 2\sqrt{2}\Bigl[(\sqrt{v_2}/\tau+2M/\tau)\land 1\land 1+(M/\tau)\land 1\Bigr]\cdot c\\
            \leq& 4\sqrt{2}\Bigl[(\sqrt{v_2}/\tau+2M/\tau)\land 1\Bigr]\cdot c\\
            =:& \alpha\cdot c.
        \end{aligned}$$
        where the second inequality is due to for $a,b\geq 0$, $(a+b)\land 1\leq 2((a\land 1)+(b\land 1))$; and $\alpha\leq 4\sqrt{2}$.

        \item For any small $x>0$,
        \begin{equation}\label{eq: covering entropy of ReLU NN, thm: convergence rate of Huber regression}
            \begin{aligned}
                \EE[\log\Ncal(x, \frac{1}{\tau}\Lscr_c, L^2(\PP_n))]
                \leq& \EE[\log\Ncal(x, \overline\Fcal_n, L^2(\PP_n))]\\
                \leq& \EE[\log\Ncal(x/2, \Fcal_n, L^2(\PP_n))]+\log\Ncal(x/(4M), [0,1], |\cdot|)\\
                \lesssim& D_{\Fcal_n}\cdot \log_+(e nM/x) + \log(\frac{1}{4x}+1)\\
                \lesssim& D_{\Fcal_n}\cdot \log_+(e nM/x)\\
                =:& D_{\Fcal_n}\cdot \log_+(c_3 nM/x),
            \end{aligned}
        \end{equation}
        where $c_3=e$, and the second inequality is due to Proposition \ref{prop: covering number of extended parameter space}, the third one is due to Lemma \ref{lemma: expected covering entropy of neural network}.
    \end{enumerate}

    Denote $\tilde{n}=n/D_{\Fcal_n}$. Now, applying Theorem \ref{theorem: convergence of EP with L^1 integrable functions} gives, for $c\ll 1$,
    \begin{equation}\label{eq: first upper bound of expected EP, thm: convergence rate of Huber regression}
        \begin{aligned}
            &\EE^\ast\|\PP_n-P\|_{\frac{1}{\tau}\Lscr_c}\\
            \lesssim& \sqrt\frac{\log(1/\alpha c)}{n}\alpha c+\inf_{\epsilon\in(0,\alpha c)}\Bigl[\epsilon+n^{-\frac{1}{2}}\int_{\epsilon/8}^{\alpha c/2}\sqrt{\EE[\log\Ncal(x,\frac{1}{\tau}\Lscr_c, L^2(\PP_n))]}, \dd x\Bigr]\\
            &+n^{-\frac{1}{2}}\alpha c \cdot \sqrt{\EE[\log2\Ncal(\alpha c/2, \frac{1}{\tau}\Lscr_c, L^2(\PP_n))]}+ \frac{M}{n}\EE[\log2\Ncal(\alpha c/2, \frac{1}{\tau}\Lscr_c, L^2(\PP_n))]\\
            \leq& \sqrt\frac{\log(1/\alpha c)}{n}\alpha c+\tilde{n}^{-\frac{1}{2}}\int_{0}^{\alpha c/2}\log_+(c_3 nM/x)^{\frac{1}{2}}, \dd x\\
            &+\tilde{n}^{-\frac{1}{2}}\alpha c \cdot \log_+(2c_3 nM/\alpha c)^{\frac{1}{2}}+ \frac{M}{\tilde{n}}\log_+(2c_3 nM/\alpha c)\\
            \lesssim& \tilde{n}^{-\frac{1}{2}}\alpha c \cdot \log_+(2c_3 nM/\alpha c)^{\frac{1}{2}}+ \frac{M}{\tilde{n}}\log_+(2c_3 nM/\alpha c).
        \end{aligned}
    \end{equation}
    where the third inequality is due to Lemma \ref{lemma: integral of covering entropy with logarithm}, and the suppressed constant is some universal constant.

    In Equation \eqref{eq: first upper bound of expected EP, thm: convergence rate of Huber regression}, we get an upper bound of expected empirical process with the property that Huber loss is Lipschitz. Another upper bound is attainable with the fact that Huber loss resembles the least squares regression. In fact, denote $\ell(f):=\ell(f;x,y)=\frac{1}{2}(y-f(x))^2$. Then, for any functions $f,g$ and hyper-parameter $\tau>0$,
    \[\begin{aligned}
        &|\ell_\tau(f)-\ell_\tau(g)|
        \leq |\ell(f)-\ell(g)|\\
        =&|\frac{1}{2}(g(x)-f(x))^2 + (y-g(x))(g(x)-f(x))|\\
        \leq& 2M^2+2M\cdot(|\xi|+2M).
    \end{aligned}\]

     Thus, denote function $FL(x,y)=6M^2+2M\cdot |\xi|$, then $FL\in L^m(P)$ is another envelope function of function class $\Lscr_c$, with $\|FL\|_{L^m(P)}\lesssim M^2 + M\cdot v_m^{1/m}$. Now, applying Theorem \ref{theorem: convergence of EP with L^1 integrable functions} again gives, for $c\ll 1$,
    \begin{equation}\label{eq: second upper bound of expected EP, thm: convergence rate of Huber regression}
        \begin{aligned}
            &\EE^\ast\|\PP_n-P\|_{\frac{1}{\tau}\Lscr_c}\\
            \lesssim& \sqrt\frac{\log(1/\alpha c)}{n}\alpha c+\inf_{\epsilon\in(0,\alpha c)}\Bigl[\epsilon+n^{-\frac{1}{2}}\int_{\epsilon/8}^{\alpha c/2}\sqrt{\EE[\log\Ncal(x,\frac{1}{\tau}\Lscr_c, L^2(\PP_n))]}, \dd x\Bigr]\\
            &+n^{-\frac{1}{2}}\alpha c \cdot \sqrt{\EE[\log2\Ncal(\alpha c/2, \frac{1}{\tau}\Lscr_c, L^2(\PP_n))]}\\
            &+\inf_{B>0} \Bigl(\frac{B}{n}\EE[\log2\Ncal(\alpha c/2, \frac{1}{\tau}\Lscr_c, L^2(\PP_n))] + \EE[\frac{1}{\tau}FL\cdot \II(\frac{1}{\tau}FL\geq B)]\Bigr)\\
            \leq& \tilde{n}^{-\frac{1}{2}}\alpha c \cdot \log_+(2c_3 nM/\alpha c)^{\frac{1}{2}}+\inf_{B>0} \Bigl(\frac{B}{\tilde{n}}\log_+(2c_3 nM/\alpha c) + \frac{\|\frac{1}{\tau}FL\|_{L^m(P)}^m}{B^{m-1}}\Bigr)\\
            \lesssim& \tilde{n}^{-\frac{1}{2}}\alpha c \cdot \log_+(2c_3 nM/\alpha c)^{\frac{1}{2}}+\tilde{n}^{\frac{1}{m}-1}\cdot\|\frac{1}{\tau}FL\|_{L^m(P)}\cdot \log_+(2c_3 nM/\alpha c)\\
            \lesssim& \tilde{n}^{-\frac{1}{2}}\alpha c \cdot \log_+(2c_3 nM/\alpha c)^{\frac{1}{2}}+\tilde{n}^{\frac{1}{m}-1}\cdot\frac{M^2 + M\cdot v_m^{1/m}}{\tau}\cdot \log_+(2c_3 nM/\alpha c)\\
        \end{aligned}
    \end{equation}

    We further provide another bound, \eqref{eq: huber m moment expected process}, when $1<m<2$. We use the envelope $FL=6M^2+2M|\xi|$ from \eqref{eq: second upper bound of expected EP, thm: convergence rate of Huber regression}. For $M>0$ (otherwise the error is zero),
    \[
       \sup_{h\in\Lscr_c}\left\|\frac{h}{4\|FL\|_{L^m(P)}}\right\|_{L^m(P)}
       \le \frac14,\qquad
       6M^2\le\|FL\|_{L^m(P)}\le 6M^2+2M v_m^{1/m}.
    \]
    Since the empirical $L^m$ norm is at most the empirical $L^2$ norm, \eqref{eq: covering entropy of ReLU NN, thm: convergence rate of Huber regression} gives
    \[
    \begin{aligned}
       \EE\log\Ncal\left(x,\frac{\Lscr_c}{4\|FL\|_{L^m(P)}},L^m(\PP_n)\right)
       &\le \EE\log\Ncal\left(\frac{4\|FL\|_{L^m(P)}x}{\tau},\frac{\Lscr_c}{\tau},L^2(\PP_n)\right)\\
       &\lesssim D_{\Fcal_n}\log_+\left(\frac{c_3 nM\tau}{4\|FL\|_{L^m(P)}x}\right).
    \end{aligned}
    \]
    We may assume $\tau\le M\sqrt n$, since otherwise the deviation term in the claim already controls $\|\hat f_n-f_0\|_2\le2M$. Hence $nM\tau/\|FL\|_{L^m(P)}\le n^{3/2}$. In Proposition~\ref{proposition: convergence of EP with L^1 integrable functions}, take $\kappa=m-1$, $\gamma=0$, $\gamma'=1$ and $B=\tilde n^{1/m}$ as the cutoff. Multiplying the bound by $4\|FL\|_{L^m(P)}/\tau$ gives
    \begin{equation}\label{eq: huber m moment expected process}
    \begin{aligned}
       \EE^\ast\|\PP_n-P\|_{\Lscr_c/\tau}
       &\lesssim \frac{\|FL\|_{L^m(P)}}{\tau}\log_+(en)
       \left(\sqrt{\frac{B^{2-m}}{\tilde n}}+\frac{B}{\tilde n}+B^{1-m}\right)\\
       &\lesssim \frac{M(M+v_m^{1/m})}{\tau}\log_+(en)\,\tilde n^{1/m-1}.
    \end{aligned}
    \end{equation}

    For $m\ge2$, taking the minimum of \eqref{eq: first upper bound of expected EP, thm: convergence rate of Huber regression} and \eqref{eq: second upper bound of expected EP, thm: convergence rate of Huber regression} gives
    \[
    \begin{aligned}
       \EE^\ast\|\PP_n-P\|_{\Lscr_c/\tau}
       &\lesssim \frac{\alpha c}{\sqrt{\tilde n}}\log_+^{1/2}\!\left(\frac{2c_3 nM}{\alpha c}\right)
       \\ &\quad +\frac{M}{\tilde n}\log_+\!\left(\frac{2c_3 nM}{\alpha c}\right)
        \left(1\land\frac{\tilde n^{1/m}(M+v_m^{1/m})}{\tau}\right).
    \end{aligned}
    \]
    For $1<m<2$, taking the minimum of \eqref{eq: first upper bound of expected EP, thm: convergence rate of Huber regression} and \eqref{eq: huber m moment expected process} gives
    \[
    \begin{aligned}
       \EE^\ast\|\PP_n-P\|_{\Lscr_c/\tau}
       &\lesssim \left(\frac{\alpha c}{\sqrt{\tilde n}}\log_+^{1/2}\!\left(\frac{2c_3 nM}{\alpha c}\right)
                       +\frac{M}{\tilde n}\log_+\!\left(\frac{2c_3 nM}{\alpha c}\right)\right)
                 \\ &\quad \land \frac{M(M+v_m^{1/m})\log_+(en)}{\tau}\tilde n^{1/m-1}.
    \end{aligned}
    \]
    Since the fitted functions are bounded by $M$, it suffices to use $c\le2M$. Hence $\log_+(en)\lesssim\log_+(2c_3 nM/(\alpha c))$. Set
    \[\alpha_\ast:=\left\{\begin{aligned}
        & 0 \quad & \text{when }1<m<2\text{ and }\tau>\tilde n^{1/m}(M+v_m^{1/m})\\
        &\alpha & \text{otherwise.}
    \end{aligned}\right.\]
    Using the second term of the minimum in that case, we obtain for all $m>1$
    \[
    \begin{aligned}
       \EE^\ast\|\PP_n-P\|_{\Lscr_c/\tau}
       &\lesssim \frac{\alpha_* c}{\sqrt{\tilde n}}\log_+^{1/2}\!\left(\frac{2c_3 nM}{\alpha c}\right)
       \\ &\quad +\frac{M}{\tilde n}\log_+\!\left(\frac{2c_3 nM}{\alpha c}\right)
          \left(1\land\frac{\tilde n^{1/m}(M+v_m^{1/m})}{\tau}\right).
    \end{aligned}
    \]

    By Corollary \ref{corollary: talagrand's inequality}, for each fixed $\delta\in(0,1)$ and large $n$, the following bound holds with high-probability of $1-\delta$:
    \[\begin{aligned}
        &\|\PP_n-P\|_{\frac{1}{\tau}\Lscr_c}\\
        \lesssim& \EE^\ast \|\PP_n-P\|_{\frac{1}{\tau}\Lscr_c} + \frac{\alpha c}{\sqrt{n}}\cdot\sqrt{\log(2/\delta)} + \frac{M}{n}\cdot \log(2/\delta)\\
        \lesssim& \tilde{n}^{-\frac{1}{2}}\alpha_* c \cdot \log_+(2c_3 nM/\alpha c)^{\frac{1}{2}}+ \frac{M}{\tilde{n}}\log_+(2c_3 nM/\alpha c)\cdot\Bigl(1\land \tilde{n}^{\frac{1}{m}}\cdot\frac{M+v_m^{1/m}}{\tau}\Bigr) \\
        & + \frac{\alpha c}{\sqrt{n}}\cdot\sqrt{\log(2/\delta)} + \frac{M}{n}\cdot \log(2/\delta).
    \end{aligned}\]
    By Theorem~\ref{thm: Estimation Error of Sieved M-estimators}, we may select
    $$\begin{aligned}
        &\phi_n(c,\delta)\\
        \asymp& \tau\cdot \Bigl( \tilde{n}^{-\frac{1}{2}}\alpha_* c \cdot \log_+(2c_3 nM/\alpha c)^{\frac{1}{2}}+ \frac{M}{\tilde{n}}\log_+(2c_3 nM/\alpha c)\cdot\Bigl(1\land \tilde{n}^{\frac{1}{m}}\cdot\frac{M+v_m^{1/m}}{\tau}\Bigr)\Bigr)\\
        & \qquad\qquad + \tau\cdot \Bigl( \frac{\alpha c}{\sqrt{n}}\cdot\sqrt{\log(2/\delta)} + \frac{M}{n}\cdot \log(2/\delta)\Bigr).
    \end{aligned}$$

    By Lemma \ref{lemma: convergence rate involving logarithm}, $\phi_n(c,\delta)\asymp c^2$ is solved by
    $$\begin{aligned}
        &\Bigl(\alpha_*\tau\cdot\tilde{n}^{-\frac{1}{2}}\cdot\log\Bigl(\frac{4c_3^2\cdot M^2}{\tau^2\alpha ^4}\tilde{n} n^2\Bigr)^{\frac{1}{2}}\Bigr)\\
        &\lor \Bigl(\sqrt{\tau} M^{\frac{1}{2}}\tilde{n}^{-\frac{1}{2}}\cdot\log\Bigl(\frac{4c_3^2\cdot M}{\tau \alpha^2}\tilde{n}n^2\Bigr)^{\frac{1}{2}}\cdot\Bigl(1\land \tilde{n}^{\frac{1}{m}}\cdot\frac{M+v_m^{1/m}}{\tau}\Bigr)^{\frac{1}{2}}\Bigr)\\
        &\lor \frac{\alpha \tau}{\sqrt{n}}\cdot\sqrt{\log(2/\delta)} \lor \sqrt\frac{\tau M}{n}\cdot \log(2/\delta)^{\frac{1}{2}}\\
        \lesssim& \Bigl(\alpha_*\tau\cdot\tilde{n}^{-\frac{1}{2}}\cdot\log\Bigl(\frac{4c_3^2\cdot M^2}{\tau^2\alpha ^4}\tilde{n} n^2\Bigr)^{\frac{1}{2}}\Bigr)\\
        &\lor \Bigl(\sqrt{\tau} M^{\frac{1}{2}}\tilde{n}^{-\frac{1}{2}}\cdot\log\Bigl(\frac{4c_3^2\cdot M}{\tau \alpha^2}\tilde{n}n^2\Bigr)^{\frac{1}{2}}\cdot\Bigl(1\land \tilde{n}^{\frac{1}{m}}\cdot\frac{M+v_m^{1/m}}{\tau}\Bigr)^{\frac{1}{2}}\Bigr)\\
        & +\frac{\alpha\tau\lor\sqrt{\tau M}}{\sqrt{n}}\cdot\sqrt{\log(2/\delta)}\\
    \end{aligned}$$
    Since
    \[
        \alpha\geq 4\sqrt{2}[(2M/\tau)\land 1]\geq 8\sqrt{2}M/\tau,\qquad \tilde{n}\leq n,
    \]
    and we require $\tau\geq 2\max\{2M, (2v_m)^{1/m}\}$, \[ \log\Bigl(\frac{4c_3^2\cdot M^2}{\tau^2\alpha ^4}\tilde{n} n^2\Bigr) \leq \log((\tau/M)^2 n^3). \]
    Meanwhile,
    \begin{equation}\label{eq: term 2 to the solution of fixed point iteration in lemma: expected covering entropy of neural network}
        \begin{aligned}
            \alpha\tau\lor \sqrt{\tau M}
            \asymp& \bigl((\sqrt{v_2}+2M)\land \tau\bigr)\lor \sqrt{\tau M}\\
            \lesssim&\sqrt{\tau}\cdot\sqrt{\tau\land(\sqrt{v_2}+M)},
        \end{aligned}
    \end{equation}

    where $M\le\tau$ and $v_2=\infty$ for $1<m<2$.
    Thus, the solution to $\phi_n(c,\delta)\asymp c^2$ can be formulated as
    \[\begin{aligned}
        &\Bigl(\alpha_*\tau\cdot\tilde{n}^{-\frac{1}{2}}\cdot\log((\tau/M)^2 n^3)^{\frac{1}{2}}\Bigr)\\
        &\lor \Bigl(\sqrt{\tau} M^{\frac{1}{2}}\tilde{n}^{-\frac{1}{2}}\cdot\log((\tau/M)^2 n^3)^{\frac{1}{2}}\cdot\Bigl(1\land \tilde{n}^{\frac{1}{m}}\cdot\frac{M+v_m^{1/m}}{\tau}\Bigr)^{\frac{1}{2}}\Bigr)\\
        & +\frac{\alpha\tau\lor\sqrt{\tau M}}{\sqrt{n}}\cdot\sqrt{\log(2/\delta)}.\\
        :=& \text{Term I} + \text{Term II}.
    \end{aligned}\]
    We first simplify $\text{Term I}$ above. By \eqref{eq: term 2 to the solution of fixed point iteration in lemma: expected covering entropy of neural network},
    \[\begin{aligned}
        &\text{Term I}\\
        \lesssim&\left\{\begin{aligned}
            & (\alpha_*\tau\lor\sqrt{\tau M}) \cdot \tilde{n}^{-\frac{1}{2}}\cdot\log((\tau/M)^2 n^3)^{\frac{1}{2}}\\
            &\qquad\qquad \qquad\qquad \text{when }\tilde{n}^{\frac{1}{m}}\cdot\frac{M+v_m^{1/m}}{\tau}\geq 1\\
            & \Bigl[(\alpha_*\tau)\lor \Bigl( M^{\frac{1}{2}}\cdot\tilde{n}^{\frac{1}{2m}}\cdot\sqrt{ M+v_m^{1/m} }\Bigr)\Bigr]\cdot\tilde{n}^{-\frac{1}{2}}\cdot\log((\tau/M)^2 n^3)^{\frac{1}{2}}\\
            &\qquad\qquad \qquad\qquad \text{when }\tilde{n}^{\frac{1}{m}}\cdot\frac{M+v_m^{1/m}}{\tau}< 1\\
        \end{aligned}\right.\\
        \lesssim&\left\{\begin{aligned}
            & \sqrt{\tau}\cdot \sqrt{\tau\land (\sqrt{v_2}+M)}\cdot \sqrt\frac{\log((\tau/M)^2\cdot n^3)}{\tilde{n}}\\
            &\qquad\qquad \qquad\qquad \text{when }\tilde{n}^{\frac{1}{m}}\cdot\frac{M+v_m^{1/m}}{\tau}\geq 1\\
            & \Bigl[0\lor \tilde{n}^{\frac{1}{2m}-\frac{1}{2}}\cdot\sqrt{ M\cdot(M+v_m^{1/m}) }\Bigr]\cdot\log((\tau/M)^2 n^3)^{\frac{1}{2}}  \\
            &\qquad\qquad \qquad\qquad \text{when }\tilde{n}^{\frac{1}{m}}\cdot\frac{M+v_m^{1/m}}{\tau}< 1\text{ and }1<m<2\\
            & \Bigl[(\sqrt{v_2}+M)\lor \Bigl( M^{\frac{1}{2}}\cdot\tilde{n}^{\frac{1}{2m}}\cdot\sqrt{ M+v_m^{1/m} }\Bigr)\Bigr]\cdot\tilde{n}^{-\frac{1}{2}}\cdot\log((\tau/M)^2 n^3)^{\frac{1}{2}}  \\
            &\qquad\qquad \qquad\qquad \text{when }\tilde{n}^{\frac{1}{m}}\cdot\frac{M+v_m^{1/m}}{\tau}< 1\text{ and }m\geq 2\\
        \end{aligned}\right.\\
        \lesssim&\left\{\begin{aligned}
            & \sqrt{\tau}\cdot \sqrt{\tau\land (\sqrt{v_2}+M)}\cdot \sqrt\frac{\log((\tau/M)^2\cdot n^3)}{\tilde{n}}\\
            &\qquad\qquad \qquad\qquad \text{when }\tilde{n}^{\frac{1}{m}}\cdot\frac{M+v_m^{1/m}}{\tau}\geq 1\\
            & \tilde{n}^{\frac{1}{2m}-\frac{1}{2}}\cdot\sqrt{ M\cdot(M+v_m^{1/m}) }\cdot\log((\tau/M)^2 n^3)^{\frac{1}{2}}  \\
            &\qquad\qquad \qquad\qquad \text{when }\tilde{n}^{\frac{1}{m}}\cdot\frac{M+v_m^{1/m}}{\tau}< 1\\
        \end{aligned}\right.\\
        =:&\delta_\sfrak.
    \end{aligned}\]
    As for $\text{Term II}$, \eqref{eq: term 2 to the solution of fixed point iteration in lemma: expected covering entropy of neural network} yields,
    \[
        \text{Term II}\lesssim\sqrt{\tau}\cdot \sqrt{\tau\land (\sqrt{v_2}+M)}\cdot \sqrt\frac{\log(2/\delta)}{n}.
    \]
    Thus, the solution to $\phi_n(c,\delta)\asymp c^2$ can be bounded by
    \[
        \delta_\sfrak+ \sqrt{\tau}\cdot \sqrt{\tau\land (\sqrt{v_2}+M)}\cdot \sqrt\frac{\log(2/\delta)}{n}.
    \]

    Therefore, we conclude that, for some suppressed universal constant, the follow bound holds for any $\delta\in(0,1)$
    $$\begin{aligned}
        \PP\Bigl(\|\hat{f}_n-f_0\|_{L^2(P)}\gtrsim & \sqrt{\tau}\cdot \sqrt{\tau\land (\sqrt{v_2}+M)}\cdot \sqrt\frac{\log(2/\delta)}{n} + \delta_\sfrak\\
        &+\inf_{f\in\Fcal_n}\|f-f_{0}\|_{L^2(P)}+v_m\tau^{1-m}\Bigr)\leq \delta.
    \end{aligned}$$
\end{proof}

\begin{lemma}[Talagrand's concentration inequality; Theorem 7.3 in \cite{bousquet2003concentration}]\label{lemma: talagrand's inequality}
    Suppose $Pf=0$, $\|f\|_{L^\infty}\leq B$ and $\|f\|_{L^2(P)}\leq \sigma$ for all $f\in\Fcal$. Denote \[Z=\sup_{f\in\Fcal}\frac{1}{n}\sum_{i=1}^n f(X_i),\quad v=\sigma^2+2B\EE(Z).\]

    Then for any $t\geq 0$, \[\PP\Bigl(Z\geq \EE(Z) + \sqrt\frac{2vt}{n} + \frac{Bt}{3n}\Bigr)\leq \exp(-t).\]
\end{lemma}

Apply Lemma~\ref{lemma: talagrand's inequality} to $\{\pm(f-Pf):f\in\Fcal\}$, with envelope $2B$, and use $2\sqrt{B\EE^\ast[Z]t/n}\le\EE^\ast[Z]+Bt/n$ for $Z=\|\PP_n-P\|_{\Fcal}$.
\begin{corollary}\label{corollary: talagrand's inequality}
    Suppose $\|f\|_{L^\infty}\leq B$ and $\|f\|_{L^2(P)}\leq \sigma$ for all $f\in\Fcal$. Then there is some suppressed universal constant, such that for any $t\geq 0$, \[\PP\Bigl(\|\PP_n-P\|_{\Fcal}\gtrsim \EE^\ast\|\PP_n-P\|_{\Fcal} + \sigma\cdot \sqrt\frac{t}{n} + \frac{Bt}{n}\Bigr)\leq 2\exp(-t).\]
\end{corollary}

\subsection{Proof of Theorem \ref{thm: quantile regression convergence rate}}

\begin{lemma}\label{lemma: map stability of quantile regression}
    For any $f, f_n^\ast$ such that $\|f\|_{L^\infty}, \|f_n^\ast\|_{L^\infty}\leq M$, $$\|f-f_n^\ast\|_{L^2(P)}^2\lesssim M\cdot \bigl(\EE\rho_\tau(Y-f(X))-\EE\rho_\tau(Y-f(X))\bigr)+M^2\|f_n^\ast-f_0\|_{L^2(P)}^2$$
\end{lemma}
\begin{proof}
    We refer readers to Lemma 4 in \cite{feng2024deep} for the proof of Lemma \ref{lemma: map stability of quantile regression}.
\end{proof}

\begin{proof}[Proof of Theorem \ref{thm: quantile regression convergence rate}]
    Denote
    \[
        \Fcal_n:=\Fcal(d,D,W,M),\qquad \tilde{n}:=\frac{n}{(DW)^2\log(DW)}.
    \]
    WLOG, assume $f_n^\ast\in\min_{f\in\Fcal_n}\|f-f_0\|_{L^2(P)}$. Let $\hat{f}_{n,t}=t\cdot \hat{f}_n+(1-t)\cdot f_n^\ast$, where
    \[
        t=\frac{\tau_n}{\tau_n+\|\hat{f}_n-f_n^\ast\|_{L^2(P)}}
    \]
    for some $\tau_n$ to be determined.

    Define the extended function space as $$\overline\Fcal_n=\{tf+(1-t)f_n^\ast:t\in[0,1], f\in\Fcal_n\}.$$
    Then by Lemma \ref{lemma: map stability of quantile regression},
    $$\begin{aligned}
        &\|\hat{f}_{n,t}-f_0\|_{L^2(P)}^2\\
        \lesssim& \bigl(P\rho_\tau(Y-\hat{f}_{n,t}(X))-P\rho_\tau(Y-f_n^\ast(X))\bigr)+\|f_n^\ast-f_0\|_{L^2(P)}^2\\
        \leq& \Big|(\PP_n-P)(\rho_\tau(Y-\hat{f}_{n,t}(X))-\rho_\tau(Y-f_n^\ast(X)))\Bigr|+\|f_n^\ast-f_0\|_{L^2(P)}^2\\
        \leq& \sup_{f\in\overline\Fcal_n:\|f-f_n^\ast\|_{L^2(P)}\leq \tau_n}\Big|(\PP_n-P)(\rho_\tau(Y-f(X))-\rho_\tau(Y-f_n^\ast(X)))\Bigr|+\|f_n^\ast-f_0\|_{L^2(P)}^2,\\
    \end{aligned}$$
    where the second inequality is because $\PP_n \rho_\tau(Y-\hat{f}_{n,t}(X))\leq \PP_n \rho_\tau(Y-f_n^\ast(X))$, owning to the convexity of $\rho_\tau$ and the definition of $\hat{f}_n$ as the empirical loss minimizer.

    For any $c>0$, define the loss function class as $$\Lscr_c=\{(x,y)\mapsto (\rho_\tau(y-f(x))-\rho_\tau(y-f_n^\ast(x)):f\in\overline\Fcal_n,\|f-f_n^\ast\|_{L^2(P)}\leq c\}.$$
    Suppose for any $c>0$ and any $\delta\in(0,1)$, there is a function $\phi_n$, such that
    $$\PP\Bigl(\sup_{\ell\in\Lscr_c}|(\PP_n-P)\ell|\geq \phi_n(c,\delta)\Bigr)\leq\delta.$$
    Then as shown in the proof of Theorem \ref{thm: Estimation Error of Sieved M-estimators}, the estimation error of $\hat{f}_n$ can be written as, for any $\delta\in(0,1)$,
    $$\PP\Bigl(\|\hat{f}_n-f_0\|_{L^2(P)}\gtrsim \inf\{c>0:\sqrt{\phi_n(c,\delta)}\lesssim c\}+\|f_n^\ast-f_0\|_{L^2(P)}\Bigr)\leq\delta.$$

    Therefore, it remains to figure out $\phi_n$, and solve the corresponding inequality. Since $\rho_\tau$ is a $1$-Lipschitz continuous function:
    \begin{enumerate}
        \item For any $\ell\in\Lscr_c$, $\|\ell\|_{L^\infty}\leq 2M$ and $\|\ell\|_{L^2(P)}\leq c$.
        \item Similar as in Equation \eqref{eq: covering entropy of ReLU NN, thm: convergence rate of Huber regression}, for any $x>0$, there is some universal constant $c_3>0$, such that $$\EE[\log\Ncal(x, \Lscr_c, L^2(\PP_n))]\lesssim D_{\Fcal_n}\cdot \log_+(c_3 nM/x).$$
    \end{enumerate}
    Now, applying Theorem \ref{theorem: convergence of EP with L^1 integrable functions} gives, for $c\ll 1$, by taking $\epsilon\to 0$,
    $$\begin{aligned}
        \EE^\ast\|\PP_n-P\|_{\Lscr_c}
        \lesssim & \sqrt\frac{\log(1/c)}{n}c+\tilde{n}^{-\frac{1}{2}}\int_{0}^{c/2}\log_+(c_3 nM/x)^{\frac{1}{2}}, \dd x\\
        &+\tilde{n}^{-\frac{1}{2}}c \cdot \log_+(2c_3 nM/c)^{\frac{1}{2}}+ \frac{M}{\tilde{n}}\log_+(2c_3 nM/c)\\
        \lesssim& \tilde{n}^{-\frac{1}{2}}c \cdot \log_+(2c_3 nM/c)^{\frac{1}{2}}+ \frac{M}{\tilde{n}}\log_+(2c_3 nM/c).
    \end{aligned}$$

    By Corollary \ref{corollary: talagrand's inequality}, for each fixed $\delta\in(0,1)$ and large $n$, the following bound holds with high-probability of $1-\delta$:
    \[\begin{aligned}
        \|\PP_n-P\|_{\frac{1}{\tau}\Lscr_c}
        \lesssim& \EE^\ast \|\PP_n-P\|_{\frac{1}{\tau}\Lscr_c} + \frac{c}{\sqrt{n}}\cdot\sqrt{\log(2/\delta)} + \frac{M}{n}\cdot \log(2/\delta)\\
        \lesssim& \underbrace{\sqrt{\log(2/\delta)}\cdot \bigl( \tilde{n}^{-\frac{1}{2}} c \cdot \log_+(2c_3 nM/ c)^{\frac{1}{2}}+ \frac{M}{\tilde{n}}\log_+(2c_3 nM/ c)\bigr)}_{=:\phi_n(c,\delta)}
    \end{aligned}\]

    Now, solving the inequality $c:\sqrt{\phi_n(c,\delta)}\lesssim c$ and we get, there is some $c_4>0$ such that for any $\delta\in(0,1)$,
    $$\PP\Bigl(\|\hat{f}_n-f_0\|_{L^2(P)}\gtrsim \sqrt{\log(2/\delta)}\tilde{n}^{-\frac{1}{2}}\log(c_4 n^3 M^2)+\|f_n^\ast-f_0\|_{L^2(P)}\Bigr)\leq\delta,$$
    where by the definition of $f_n^\ast$ at the beginning of this proof, $$\|f_n^\ast-f_0\|_{L^2(P)}=\inf_{f\in\Fcal_n}\|f-f_0\|_{L^2(P)}.$$
\end{proof}

\section{Proof of Results in Section \ref{sec: Heavy-tailed Nonparametric Least Squares Regression}}

\subsection{Proof of Theorem \ref{theorem: least squares estimation convergence rate with Linfty covering entropy}}

    \textit{Outline of the proof.} We present a high-level overview based on the theoretical tools developed earlier:
    \begin{itemize}
        \item \textbf{Step 1.} Recast least squares regression into the framework of Theorem \ref{thm: Estimation Error of Sieved M-estimators}.
        \item \textbf{Step 2.} Verify the conditions in Propositions \ref{proposition: convergence of EP with L^infty integrable functions} and \ref{prop: convergence rate expression in terms of weighted covering entropy} to obtain a bound on the empirical process.
        \item \textbf{Step 3.} Apply Proposition \ref{prop: convergence rate expression in terms of weighted covering entropy} to get the estimation error.
    \end{itemize}

    \textbf{Step 1.} Given data generating process $$y=f_0(x)+\xi,$$ the empirical least squares loss and NPLSE, as the minimizer of the convex functional, read $$\hat{f}_n\in\argmin_{f\in\Fcal_n}\frac{1}{n}\sum_{i=1}^n (y_i-f(x_i))^2=\frac{1}{n}\sum_{i=1}^n \xi_i^2 +(f_0(x_i)-f(x_i))^2+2\xi_i(f_0(x_i)-f(x_i)).$$
    Since $\frac{1}{n}\sum_{i=1}^n \xi_i^2$ is independent of $f$, we define the least squares loss as $$\ell(f)=\ell(f;x,\xi)=(f_0(x)-f(x))^2-2\xi(f(x)-f_0(x)).$$ Then the ERM becomes $$\hat{f}_n\in\argmin_{f\in\Fcal_n}\PP_n\ell(f).$$
    By $\EE[\xi|x]=0$, the target function, as the population minimizer, reads $$f_0\in\argmin_{f\in\Fcal_n\cup\{f_0\}}P\ell(f)=:L(f).$$ Clearly, $L(f)-L(f_0)=\|f-f_0\|_{L^2(P)}^2$ implies that the stability link function is $\wfrak_n(x)=x^2$.

    Define the approximator of $f_0$ in $\Fcal_n$ as $$f_n^\ast\in\argmin_{f\in\Fcal_n}\|f-f_0\|_{L^2(P)}.$$ If such a minimizer does not exist, we can take an $\varepsilon$-approximator: for any $\varepsilon>0$, take $f_n^\ast\in\Fcal_n$ such that $$\|f_n^\ast-f_0\|\leq\varepsilon+\inf_{f\in\Fcal_n}\|f-f_0\|_{L^2(P)}.$$ Then the statistical error remains the same, and the approximation error term is inflated by a negligible term $\varepsilon$.

    Define extended function space $$\overline\Fcal_n=\Bigl\{t f+(1-t)f_n^\ast:t\in[0,1], f\in\Fcal_n\Bigr\}.$$ By Proposition \ref{prop: covering number of extended parameter space}, since $\Fcal_n$ is uniformly bounded by $M$, for any $x>0$,
    \begin{equation}\label{eq: covering entropy of extended function space, theorem: least squares estimation convergence rate with Linfty covering entropy}
        \begin{aligned}
            \log\Ncal(x,\overline\Fcal_n,L^\infty)
            \leq& \log\Ncal(x/2,\Fcal_n,L^\infty)+\log\Ncal(\frac{x}{4M}, [0,1], |\cdot|)\\
            \lesssim& D_{\Fcal_n}\cdot (x/2)^{-\gamma} + \log_+(4M/x)\\
            \lesssim_{\log}& D_{\Fcal_n}\cdot x^{-\gamma}.
        \end{aligned}
    \end{equation}

    \textbf{Step 2.}
    For any $c>0$, define function class $$\Lscr_c=\Big\{\ell(f)-\ell(f_n^\ast):f\in\Fcal,\|f-f_n^\ast\|_{L^2(P)}\leq c\Big\}.$$
    As discussed in Section \ref{sec: Application to ERMs with Empirical Process Theory}, to apply Theorem \ref{thm: Estimation Error of Sieved M-estimators}, we need to verify:
    \begin{enumerate}
        \item The envelope function $LF_c$ of the function class $\Lscr_c$.
        \item The upper bound of $\sup_{\ell\in\Lscr_c}\|\ell\|_{L^{1+\kappa}(P)}$.
        \item The covering entropy of $\Lscr_c$ in the weighted uniform norm.
    \end{enumerate}

    \textbf{Step 2.A (Envelope function $LF_c$).} Let $\Delta f= f-f_0$. then
    \begin{equation}\label{eq: decomposition of excess loss function, theorem: least squares estimation convergence rate with Linfty covering entropy}
        \begin{aligned}
            \ell(f)-\ell(f_n^\ast)
            =& (\Delta f)^2 - 2\xi \Delta f-(\Delta f_n^\ast)^2 + 2\xi \Delta f_n^\ast\\
            =& (\Delta f + \Delta f_n^\ast-2\xi)(\Delta f- \Delta f_n^\ast)\\
            =& (f + f_n^\ast - 2 f_0 - 2\xi)(f- f_n^\ast).
        \end{aligned}
    \end{equation}

    Let $F_{\Gcal_c}$ be the envelope function of $\Gcal_c=\{f-f_n^\ast:\|f-f_n^\ast\|_{L^2(P)}\leq c\}$. Since $\Fcal_n\cup\{f_0\}$ is uniformly bounded by $M>0$, triangle inequality gives $$\|f + f_n^\ast - 2 f_0\|_{L^\infty}\leq 4M.$$
    Consequently, for $\ell(f)-\ell(f_n^\ast)\in\Lscr_c$,
    \begin{equation}\label{eq: definition of envelope function of loss function class, theorem: least squares estimation convergence rate with Linfty covering entropy}
        |\ell(f)-\ell(f_n^\ast)|\leq (2|\xi|+4M)\cdot F_{\Gcal_c}=:LF_c.
    \end{equation}

    Besides, define weight function $$Lw_c = \frac{1}{2|\xi|+4M}$$ and we have,
    \begin{equation}\label{eq: weighted uniform bound of local envelope function, theorem: least squares estimation convergence rate with Linfty covering entropy}
        \|LF_c\|_{L^\infty(Lw_c)}=\|F_{\Gcal_c}\|_{L^\infty}\lesssim c^s,
    \end{equation}
    where the last inequality is due to the assumption of interpolation condition of $\Fcal_n$. Meanwhile, $\xi\in L^m (P)$ implies $1/(Lw_c)\in L^m(P)$.

    \textbf{Step 2.B ($L^{1+\kappa}(P)$ norm of $\Lscr_c$).}
    Since $\|f + f_n^\ast - 2 f_0\|_{L^\infty}\leq 4M$, by Equation \eqref{eq: decomposition of excess loss function, theorem: least squares estimation convergence rate with Linfty covering entropy},
    $$\begin{aligned}
        |\ell(f)-\ell(f_n^\ast)|\leq (2|\xi|+4M)|f- f_n^\ast|.
    \end{aligned}$$
    Taking $L^{1+\kappa}(P)$ norm gives,
    $$\begin{aligned}
        \|\ell(f)-\ell(f_n^\ast)\|_{L^{1+\kappa}(P)}^{1+\kappa}
        \leq& \EE\Bigl[  (2|\xi|+4M)^{1+\kappa}|f(X)- f_n^\ast(X)|^{1+\kappa}\Bigr]\\
        \lesssim& \EE\Bigl[  \Bigl(|\xi|^{1+\kappa}+M^{1+\kappa}\Bigr)|f(X)- f_n^\ast(X)|^{1+\kappa}\Bigr]\\
        =& \EE\Bigl[  \Bigl(\EE[|\xi|^{1+\kappa}|X]+M^{1+\kappa}\Bigr)|f(X)- f_n^\ast(X)|^{1+\kappa}\Bigr]\\
        \lesssim& \EE[|f(X)- f_n^\ast(X)|^{1+\kappa}]=\|f-f_n\|_{L^{1+\kappa}(P)}^{1+\kappa}.
    \end{aligned}$$
    So, we conclude that, for any $\ell(f)-\ell(f_n^\ast)\in\Lscr_c$,
    \begin{equation}\label{eq: L1+kappa norm of lscrc, theorem: least squares estimation convergence rate with Linfty covering entropy}
        \|\ell(f)-\ell(f_n^\ast)\|_{L^{1+\kappa}(P)}\lesssim c.
    \end{equation}

    \textbf{Step 2.C (Covering entropy of $\Lscr_c$).}
    Using Equation \eqref{eq: decomposition of excess loss function, theorem: least squares estimation convergence rate with Linfty covering entropy} again, for any $f,g\in\overline\Fcal_n$,
    $$\begin{aligned}
        (\ell(f)-\ell(f_n^\ast)) - (\ell(g)-\ell(f_n^\ast))
        =& \ell(f)-\ell(g)
        = (f + g - 2 f_0 - 2\xi)(f- g),
    \end{aligned}$$
    where $|f + g - 2 f_0 - 2\xi|\leq 2|\xi|+4M$. This implies that $$\|\ell(f)-\ell(g)\|_{L^\infty(Lw_c)}\leq \|f-g\|_{L^\infty}.$$
    Therefore, for any $x>0$, the covering entropy of $\Lscr_c$ satisfies,
    \begin{equation}\label{eq: weighted Linfty covering entropy of Lscrc, theorem: least squares estimation convergence rate with Linfty covering entropy}
        \begin{aligned}
            \log\Ncal(x,\Lscr_c,L^\infty(Lw_c))\leq \log\Ncal(x/2, \overline\Fcal_n,L^\infty)\lesssim_{\log} D_{\Fcal_n}\cdot x^{-\gamma},
        \end{aligned}
    \end{equation}
    where the last inequality is due to Equation \eqref{eq: covering entropy of extended function space, theorem: least squares estimation convergence rate with Linfty covering entropy}.

    Now, based on Equations \eqref{eq: weighted uniform bound of local envelope function, theorem: least squares estimation convergence rate with Linfty covering entropy}, \eqref{eq: L1+kappa norm of lscrc, theorem: least squares estimation convergence rate with Linfty covering entropy} and \eqref{eq: weighted Linfty covering entropy of Lscrc, theorem: least squares estimation convergence rate with Linfty covering entropy}, the assumptions in Proposition \ref{proposition: convergence of EP with L^infty integrable functions} and Proposition \ref{prop: convergence rate expression in terms of weighted covering entropy} hold.

    \textbf{Step 3.} We next use Propositions \ref{proposition: convergence of EP with L^infty integrable functions} and \ref{prop: convergence rate expression in terms of weighted covering entropy} to derive the convergence rates of $\hat{f}_n$.

    Let $\phi_n(c)$ be the upper bound of $\|\|\PP_n-P\|_{\Lscr_c}\|_{L^m(P)}$ given in Proposition \ref{proposition: convergence of EP with L^infty integrable functions}, and define $\sigma_n(\delta)=\inf\{\sigma>0: \delta^{-\frac{1}{m}}\cdot \phi_n(\sigma)\leq \sigma^2\}$.
    Then, by Theorem \ref{thm: Estimation Error of Sieved M-estimators}, for any $\delta\in(0,1)$,
    $$\PP\Bigl(\|\hat{f}_n-f_0\|_{L^2(P)}\geq \sigma_n(\delta)+\|f_n^\ast-f_0\|_{L^2(P)}\Bigr)\leq\delta,$$ where by Proposition \ref{prop: convergence rate expression in terms of weighted covering entropy},
    $$\sigma_n(\delta)\lesssim_{\log n}\delta^{-\frac{1}{m}}\cdot\Bigl(\tilde{n}^{-(\frac{1}{2+\gamma}\land\frac{1}{2\gamma})} + \tilde{n}^{-\frac{1}{\frac{2-s}{1-1/m}+s\gamma}}\Bigr).$$

    Lastly, applying Lemma \ref{lemma: connection between high-probability bound and convergence in mean rate} gives,
    $$\EE\|\hat{f}_n-f_0\|_{L^2(P)}\lesssim_{\log n}\tilde{n}^{-(\frac{1}{2+\gamma}\land\frac{1}{2\gamma})} + \tilde{n}^{-\frac{1}{\frac{2-s}{1-1/m}+s\gamma}}+\underbrace{\|f_n^\ast-f_0\|_{L^2(P)}}_{=:\inf_{f\in\Fcal}\|f-f_0\|_{L^2(P)}}.$$

\subsection{Proof of Proposition \ref{prop: quadratic stability link function for NPGLM}}

By Taylor's expansion,
$$\begin{aligned}
    &\EE_{P}[\ell(f)]-\EE_{P}[\ell(f_0)]\\
    =& \EE\Bigl[(-y)(f(\xbf)-f_0(\xbf))+\psi(f(\xbf))-\psi(f_0(\xbf))\Bigr]\\
    =& \EE\Bigl[(-y)(f(\xbf)-f_0(\xbf))+\psi^\prime(f_0(\xbf))(f(\xbf)-f_0(\xbf))+\int_{f_0(\xbf)}^{f(\xbf)}\psi^\pprime(t)(f(\xbf)-t)\, \dd t\Bigr]\\
    =&\EE\Bigl[\Bigl(f(\xbf)-f_0(\xbf)\Bigr)\underbrace{\Bigl(\psi^\prime(f_0(\xbf))-\EE[y|\xbf]\Bigr)}_{=0}\Bigr]+\EE\Bigl[\int_{f_0(\xbf)}^{f(\xbf)}\psi^\pprime(t)(f(\xbf)-t)\, \dd t\Bigr]\\
    =&\EE\Bigl[\int_{f_0(\xbf)}^{f(\xbf)}\psi^\pprime(t)(f(\xbf)-t)\, \dd t\Bigr].\\
\end{aligned}$$

\subsection{Proof of Theorem \ref{theorem: general ERM convergence rate with Linfty covering entropy}}

The proof follows the same steps as that of Theorem \ref{theorem: least squares estimation convergence rate with Linfty covering entropy}, with only minimal modifications needed to account for the general form of the loss function. The two main substitutions are:

\begin{itemize}
\item \textbf{Weight function:} Replace the loss-specific weight with
$$
Lw_c = \frac{1}{U(x,y)}.
$$
\item \textbf{Envelope function:} Define the envelope of the localized loss class $\Lscr_c$ as
$$
LF_c = \frac{1}{Lw_c} \cdot F_{\Gcal_c} = U(x,y)\cdot F_{\Gcal_c},
$$
where recall that $F_{\Gcal_c}$ is the envelope function of the localized function difference class $\Gcal_c := \{f - f_n^\ast : \|f - f_n^\ast\|_{L^2(P)} \leq c\}$.
\end{itemize}

With these substitutions, all conditions in Proposition \ref{proposition: convergence of EP with L^infty integrable functions} and Proposition \ref{prop: convergence rate expression in terms of weighted covering entropy} are satisfied under Assumption \ref{assumption: regularity conditions for general ERM}, thus concluding the proof.

\subsection{Proof of Theorem \ref{theorem: indicator regression convergence rate with L2 covering entropy}}

\begin{proof}[Proof of Theorem \ref{theorem: indicator regression convergence rate with L2 covering entropy}]
    Denote least squares loss as $$\ell(f)=\ell(f;x,\xi)=(f_0(x)-f(x))^2-2\xi(f(x)-f_0(x)).$$ Then the empirical and population loss functions are denoted as $L_n(f)=\PP_n\ell(f)$ and $L(f)=P\ell(f)$, respectively. Clearly, the stability link function is $\wfrak_n(x)=x^2$.

    Suppose $f_n^\ast\in\Fcal_n$ is the minimizer of $\inf_{f\in\Fcal_n}\|f-f_0\|_{L^2(P)}$. If the infimum is not attained, take $f_n^\ast$ within $n^{-1}$ of it. Then, for the extended parameter space $\overline\Fcal_n=\{tf+(1-t)f_n^\ast:f\in\Fcal_n,t\in[0,1]\}$, Proposition \ref{prop: covering number of extended parameter space} implies that the covering entropy of $\overline\Fcal_n$ can be dominated by that of $\Fcal_n$ plus a logarithmic term.

    Denote $\Delta f= f-f_0$. For any $c>0$, define $$\Lscr_c=\{\ell(f)-\ell(f_n^\ast):f\in\overline\Fcal_n,\|f-f_n^\ast\|_{L^2(P)}\leq c\}.$$
    Then by the assumption that $\inner{x}^\eta$ is an envelope function of $\Fcal_n$ up to some multiplicative factor,
    $$\begin{aligned}
        \ell(f)-\ell(f_n^\ast)
        =& (\Delta f)^2 - 2\xi \Delta f-(\Delta f_n^\ast)^2 + 2\xi \Delta f_n^\ast\\
        =& (\Delta f + \Delta f_n^\ast-2\xi)(\Delta f- \Delta f_n^\ast).
    \end{aligned}$$
    Thus, an envelope function of $\Lscr_c$ is (proportional to) $(|\xi|+\inner{x}^\eta)\inner{x}^\eta=:LF(\xi,x)$, which is sub-Weibull based on the assumption.

    We still need to derive the expected $L^1(\PP_n)$ and $L^2(\PP_n)$ covering entropies of the truncated $\Lscr_c$. By Equation \eqref{eq: decomposition of excess loss function, theorem: least squares estimation convergence rate with Linfty covering entropy}, for any $f_1,f_2\in\overline\Fcal_n$,
    $$\|\ell(f_1)-\ell(f_2)\|_{L^1(\PP_n)}\lesssim \EE_{\PP_n}[|\xi|\cdot|f_1-f_2|(X)]+\EE_{\PP_n}[\inner{X}^\eta\cdot|f_1-f_2|(X)].$$

    Take $A=\{|\xi|+\inner X^\eta\le M\}$ with $M\ge1$ a power of $\log(en)$ such that $P[LF\,\II(A^c)]\le n^{-2}$.
    For $\Lscr_c^M=\{g\II(A):g\in\Lscr_c\}$ and $k=1,2$, the loss factorization gives
    \[\begin{gathered}
        \|g\II(A)\|_\infty\lesssim M^2,\qquad \|g\II(A)\|_{L^2(P)}\lesssim Mc,\\
        \|(\ell(f_1)-\ell(f_2))\II(A)\|_{L^k(\PP_n)}
        \lesssim M\|f_1-f_2\|_{L^k(\PP_n)}.
    \end{gathered}\]
    For $k=1,2$ and $0<x\le1$, restricting a cover to the localized class costs at most a factor of two in radius. Hence, by Proposition~\ref{prop: covering number of extended parameter space},
    \[\begin{aligned}
       &\log\bigl(2\Ncal(x,\Lscr_c^M,L^k(\PP_n))\bigr)\\
       &\quad\le \log\bigl(2\Ncal(x/(CM),\overline\Fcal_n,L^k(\PP_n))\bigr)\\
       &\quad\le \log\bigl(2\Ncal(x/(2CM),\Fcal_n,L^k(\PP_n))\bigr)
            +\log\bigl(1+CM\|\inner X^\eta\|_{L^k(\PP_n)}/x\bigr),
    \end{aligned}\]
    where $C$ may increase from line to line. Taking expectations and applying Jensen's inequality gives
    \[\begin{aligned}
       \EE\log\bigl(2\Ncal(x,\Lscr_c^M,L^k(\PP_n))\bigr)
       &\le D_{\Fcal_n}(2CM/x)^{k\gamma/2}
          +\log\bigl(1+CM\|\inner X^\eta\|_{L^k(P)}/x\bigr)\\
       &\lesssim_{\log n} D_{\Fcal_n}x^{-k\gamma/2}.
    \end{aligned}\]

    With $\tilde n=n/D_{\Fcal_n}$, it suffices to treat $\tilde n^{-1/(2+\gamma)}<1/4$ and $\tilde n^{-1/(2+\gamma)}\le c\lesssim1$: the common envelope bounds all prediction distances by $2C\|\inner X^\eta\|_{L^2(P)}$.
    In Theorem~\ref{theorem: convergence of EP with L^1 integrable functions}, use radius $CMc$ and truncation level $CM^2$, since
    \[\begin{aligned}
       \sup_{g\in\Lscr_c^M}\|g\|_{L^{1+\kappa}(P)}
       &\le\sup_{g\in\Lscr_c^M}\|g\|_{L^2(P)}\le CMc,\qquad
          \sup_{g\in\Lscr_c^M}\|g\|_\infty\le CM^2,\\
       \EE^\ast\|\PP_n-P\|_{\Lscr_c}
       &\le\EE^\ast\|\PP_n-P\|_{\Lscr_c^M}+2P[LF\,\II(A^c)]
          \le\EE^\ast\|\PP_n-P\|_{\Lscr_c^M}+2n^{-2},
    \end{aligned}\]
    taking $LF$ with its envelope constant. In both cases below, set $\sigma=CMc$ in Theorem~\ref{theorem: convergence of EP with L^1 integrable functions}. Its envelope-tail term is zero for $\Lscr_c^M$ at level $CM^2$; the $2n^{-2}$ above is the cost of returning to $\Lscr_c$. We use
    \[
      \begin{cases}
        \kappa=1,\quad k=2,\quad \epsilon\downarrow0, &0<\gamma<2,\\
        \kappa=1/\log(en),\quad k=1,\quad \epsilon=\tilde n^{-2/(2+\gamma)}, &\gamma\ge2.
      \end{cases}
    \]

    \textbf{Case 1: $0<\gamma<2$.} With $\kappa=1$, Theorem~\ref{theorem: convergence of EP with L^1 integrable functions} gives
    \[\begin{aligned}
      &\EE^\ast\|\PP_n-P\|_{\Lscr_c}\\
      &\lesssim
        \frac{\sigma\land 2^{-2}}{\sqrt n}
          \sqrt{\left\lceil\log_2(1/\sigma\lor 4)\right\rceil}\\
      &\quad+\inf_{0<\epsilon<\sigma\land 2^{-2}}\left[
        \epsilon+\frac1{\sqrt n}\int_{\epsilon/8}^{(\sigma/2)\land 2^{-3}}
        \sqrt{\EE\log\Ncal(x,\Lscr_c^M,L^2(\PP_n))}\,\dd x\right]\\
      &\quad+\frac{\sigma}{\sqrt n}
        \sqrt{\EE\log\bigl(2\Ncal((\sigma/4)\land 2^{-4},\Lscr_c^M,L^2(\PP_n))\bigr)}\\
      &\quad+\frac{CM^2}{n}
        \EE\log\bigl(2\Ncal((\sigma/4)\land 2^{-4},\Lscr_c^M,L^2(\PP_n))\bigr)+2n^{-2}.
    \end{aligned}\]
    Insert the $k=2$ entropy bound and let $\epsilon\downarrow0$. Since $(\sigma/4)\land2^{-4}\gtrsim c$ and $M$ is polylogarithmic in $n$, we obtain
    \[\begin{aligned}
       \EE^\ast\|\PP_n-P\|_{\Lscr_c}
       &\lesssim_{\log n} \frac{1}{\sqrt{\tilde n}}
          \left(\int_0^c x^{-\gamma/2}\,\dd x+c\sqrt{c^{-\gamma}}\right)
          +\frac{c^{-\gamma}}{\tilde n}\\
       &\lesssim_\gamma \frac{c^{1-\gamma/2}}{\sqrt{\tilde n}}
          +\frac{c^{-\gamma}}{\tilde n}.
    \end{aligned}\]
    Since $c\ge\tilde n^{-1/(2+\gamma)}$,
    \[
       \frac{c^{-\gamma/2}}{\sqrt{\tilde n}}
         \le \tilde n^{-1/2+\gamma/(2(2+\gamma))}=\tilde n^{-1/(2+\gamma)},
       \qquad
       \EE^\ast\|\PP_n-P\|_{\Lscr_c}
         \lesssim_{\log n} c\,\tilde n^{-1/(2+\gamma)} .
    \]

    \textbf{Case 2: $\gamma\ge2$.} Take $\kappa=1/\log(en)$. For every empirical vector $v$,
    \[
       \|v\|_{L^{1+\kappa}(\PP_n)}\le n^{\kappa/(1+\kappa)}\|v\|_{L^1(\PP_n)}
       \le e\|v\|_{L^1(\PP_n)} .
    \]
    For $\epsilon=\tilde n^{-2/(2+\gamma)}<\sigma\land 2^{-2}$, Theorem~\ref{theorem: convergence of EP with L^1 integrable functions} gives
    \[\begin{aligned}
      &\EE^\ast\|\PP_n-P\|_{\Lscr_c}\\
      &\lesssim \sqrt{\frac{(CM^2)^{1-\kappa}}{n}}\,
         \bigl(\sigma^{(1+\kappa)/2}\land 2^{-(1+\kappa)}\bigr)
         \sqrt{\left\lceil\log_2(1/\sigma\lor4)\right\rceil}\\
      &\quad+\epsilon+\sqrt{\frac{(CM^2)^{1-\kappa}}{n}}
         \int_{\epsilon/8}^{(\sigma/2)\land2^{-3}}
         \sqrt{\frac{\EE\log\Ncal(x,\Lscr_c^M,L^{1+\kappa}(\PP_n))}
                    {x^{1-\kappa}}}\,\dd x\\
      &\quad+\sqrt{\frac{(CM^2)^{1-\kappa}}{n}}\,\sigma^{(1+\kappa)/2}
         \sqrt{\EE\log\bigl(2\Ncal((\sigma/4)\land2^{-4},\Lscr_c^M,L^{1+\kappa}(\PP_n))\bigr)}\\
      &\quad+\frac{CM^2}{n}
         \EE\log\bigl(2\Ncal((\sigma/4)\land2^{-4},\Lscr_c^M,L^{1+\kappa}(\PP_n))\bigr)+2n^{-2}.
    \end{aligned}\]
    The constants are uniform in $\kappa$. For $0<x\le1$, the empirical norm comparison and the $k=1$ bound yield
    \[\begin{aligned}
       \EE\log\bigl(2\Ncal(x,\Lscr_c^M,L^{1+\kappa}(\PP_n))\bigr)
       &\le \EE\log\bigl(2\Ncal(x/e,\Lscr_c^M,L^1(\PP_n))\bigr)
       \lesssim_{\log n} D_{\Fcal_n}x^{-\gamma/2},\\
       x^{(\kappa-1)/2}\sqrt{x^{-\gamma/2}}&\le x^{-(2+\gamma)/4}.
    \end{aligned}\]
    Substitution into the preceding display gives
    \[\begin{aligned}
        \EE^\ast\|\PP_n-P\|_{\Lscr_c}
        &\lesssim_{\log n} \epsilon+\tilde n^{-1/2}
          \int_{\epsilon/8}^{(CMc/2)\land 2^{-3}}x^{-(2+\gamma)/4}\,\dd x
          +\tilde n^{-1/2}c^{(2-\gamma)/4}+\tilde n^{-1}c^{-\gamma/2}\\
        &\lesssim_{\log n} \tilde n^{-2/(2+\gamma)}
           +\tilde n^{-1/2}c^{(2-\gamma)/4}+\tilde n^{-1}c^{-\gamma/2}
           =:\phi_n(c).
    \end{aligned}\]
    For $\gamma=2$ the integral is $O(\log(en))$; for $\gamma>2$ it is $O(\epsilon^{(2-\gamma)/4})$.
    Solving $\delta_n^2\asymp\phi_n(\delta_n)$ gives, up to logarithmic factors for $\gamma\ge2$,
    \[
       \delta_n\asymp\tilde n^{-1/(2+\gamma)}\lor\tilde n^{-2/(6+\gamma)}
       \lor\tilde n^{-2/(4+\gamma)}=\tilde n^{-1/(2+\gamma)}.
    \]

    In both cases $\EE^\ast\|\PP_n-P\|_{\Lscr_c}\lesssim c\,\tilde n^{-1/(2+\gamma)}\log(en)^B$ for $c\ge\tilde n^{-1/(2+\gamma)}$ and some fixed $B$; enlarging $c$ beyond the diameter does not change $\Lscr_c$.
    If $\|\hat f_n-f_0\|_{L^2(P)}>\|f_n^\ast-f_0\|_{L^2(P)}+u$, choose $f$ on the segment from $f_n^\ast$ to $\hat f_n$ with $\|f-f_0\|_{L^2(P)}=\|f_n^\ast-f_0\|_{L^2(P)}+u$. By convexity $L_n(f)\le L_n(f_n^\ast)$, while
    \[
       \|f-f_n^\ast\|_{L^2(P)}\le 2\|f_n^\ast-f_0\|_{L^2(P)}+u,\qquad L(f)-L(f_n^\ast)=u(2\|f_n^\ast-f_0\|_{L^2(P)}+u).
    \]
    Hence, by Markov's inequality, for $u=K\tilde n^{-1/(2+\gamma)}\log(en)^B$, $K\ge1$,
    \[\begin{aligned}
       &\PP\bigl(\|\hat f_n-f_0\|_{L^2(P)}>\|f_n^\ast-f_0\|_{L^2(P)}+u\bigr)
       \le \frac{\EE^\ast\|\PP_n-P\|_{\Lscr_{2\|f_n^\ast-f_0\|_{L^2(P)}+u}}}{u(2\|f_n^\ast-f_0\|_{L^2(P)}+u)}
       \lesssim K^{-1}.
    \end{aligned}\]

    Hence, $$\|\hat{f_n}-f_0\|_{L^2(P)}=_{\log n}\Ocal_{\PP}(\tilde{n}^{-\frac{1}{2+\gamma}})+\inf_{f\in\Fcal_n}\|f-f_0\|_{L^2(P)}.$$
\end{proof}

\stopcontents[supp]
\end{document}